\documentclass[11pt,reqno,english]{amsart}
\usepackage[T1]{fontenc}
\usepackage[latin9]{inputenc}
\usepackage{geometry}
\usepackage[svgnames]{xcolor}
\usepackage{babel}
\usepackage{mathrsfs}
\usepackage{mathtools}
\usepackage{bm}
\usepackage{amsbsy}
\usepackage{amstext}
\usepackage{amsthm}
\usepackage{amssymb}
\usepackage{stmaryrd}
\usepackage{graphicx}
\usepackage{setspace}
\usepackage{esint}

\usepackage{tikz}
\usetikzlibrary[patterns]
\usetikzlibrary{arrows.meta}

\usepackage[unicode=true,pdfusetitle,
 bookmarks=true,bookmarksnumbered=false,bookmarksopen=false,
 breaklinks=false,pdfborder={0 0 1},backref=false,colorlinks=false]
 {hyperref}

\makeatletter
\numberwithin{equation}{section}
\numberwithin{figure}{section}
\theoremstyle{plain}
\newtheorem{thm}{\protect\theoremname}[section]
\theoremstyle{remark}
\newtheorem{rem}[thm]{\protect\remarkname}
\theoremstyle{definition}
\newtheorem{defn}[thm]{\protect\definitionname}
\theoremstyle{plain}
\newtheorem{prop}[thm]{\protect\propositionname}
\theoremstyle{definition}
\newtheorem*{example*}{\protect\examplename}
\theoremstyle{plain}
\newtheorem{cor}[thm]{\protect\corollaryname}
\theoremstyle{plain}
\newtheorem{lem}[thm]{\protect\lemmaname}

\usepackage{babel}
\usepackage{mathrsfs}
\usepackage{bm}
\usepackage{amsbsy}
\usepackage{amstext}

\let\SF@@footnote\footnote
\def\footnote{\ifx\protect\@typeset@protect
    \expandafter\SF@@footnote
  \else
    \expandafter\SF@gobble@opt
  \fi
}
\expandafter\def\csname SF@gobble@opt \endcsname{\@ifnextchar[
  \SF@gobble@twobracket
  \@gobble
}
\edef\SF@gobble@opt{\noexpand\protect
  \expandafter\noexpand\csname SF@gobble@opt \endcsname}
\def\SF@gobble@twobracket[#1]#2{}

\usepackage{cite}
\usepackage{bbm}

\providecommand{\leftsquigarrow}{%
  \mathrel{\mathpalette\reflect@squig\relax}%
}
\newcommand{\reflect@squig}[2]{%
  \reflectbox{$\m@th#1\rightsquigarrow$}%
}

\makeatother

\providecommand{\corollaryname}{Corollary}
\providecommand{\definitionname}{Definition}
\providecommand{\examplename}{Example}
\providecommand{\lemmaname}{Lemma}
\providecommand{\propositionname}{Proposition}
\providecommand{\remarkname}{Remark}
\providecommand{\theoremname}{Theorem}

\begin{document}
\title[Metastability for parabolic equations with drift 2]{Metastability and time scales for parabolic equations with drift
2: the general time scale}
\author{Claudio Landim, Jungkyoung Lee, and Insuk Seo}

\address{IMPA, Estrada Dona Castorina 110, J. Botanico, 22460 Rio de
Janeiro, Brazil and Univ. Rouen Normandie, CNRS,
LMRS UMR 6085,  F-76000 Rouen, France. \\
e-mail: \texttt{landim@impa.br} }

\address{June E Huh Center for Mathematical Challenges, Korea Institute for Advanced Study,
Republic of Korea. \\
e-mail: \texttt{jklee@kias.re.kr} }

\address{Department of Mathematical Sciences, Seoul National
University and
Research Institute of Mathematics, Republic of Korea. \\
e-mail: \texttt{insuk.seo@snu.ac.kr} }

\begin{abstract}
Consider the elliptic operator given by
\[
\mathscr{L}_{\epsilon}f\,=\,\boldsymbol{b}\cdot\nabla f\,+\,\epsilon\,\Delta f
\]
for some a smooth vector field $\boldsymbol{b}\colon\mathbb{R}^{d}\to\mathbb{R}^{d}$
and a small parameter $\epsilon>0$. Consider the initial-valued problem
on $\mathbb{R}^{d}$
\[
\left\{ \begin{aligned} & \partial_{t}\,u_{\epsilon}\,=\,\mathscr{L}_{\epsilon}\,u_{\epsilon}\;,\\
 & u_{\epsilon}(0,\,\cdot)=u_{0}(\cdot)\;,
\end{aligned}
\right.
\]
for some bounded continuous function $u_{0}$. Under the hypothesis that the diffusion
on $\mathbb{R}^{d}$ induced by generator $\mathscr{L}_{\epsilon}$
has a Gibbs invariant measure of the form $Z_{\epsilon}^{-1}\exp\left\{ -U(x)/\epsilon\right\} dx$
for some smooth Morse potential function $U$ with finitely many critical
points, we provide the complete characterization of the multi-scale
behavior of the solution $u_{\epsilon}$ in the regime $\epsilon\rightarrow0$.
More precisely, we find the critical time scales $\theta_{\epsilon}^{(1)},\,\dots,\,\theta_{\epsilon}^{(\mathfrak{q})}$
such that $\theta_{\epsilon}^{(1)}\to\infty$, $\theta_{\epsilon}^{(p+1)}/\theta_{\epsilon}^{(p)}\to\infty$
for all $1\le p\le\mathfrak{q}-1$ as $\epsilon\rightarrow0$, and
the kernels $R_{t}^{(p)}\colon\mathcal{M}_{0}\times\mathcal{M}_{0}\to\mathbb{R}_{+}$,
where $\mathcal{M}_{0}$ represents the set of local minima of $U$,
such that
\[
\lim_{\epsilon\to0}u_{\epsilon}(t\theta_{\epsilon}^{(p)},\,\boldsymbol{x})\;=\;\sum_{\boldsymbol{m}'\in\mathcal{M}_{0}}R_{t}^{(p)}(\boldsymbol{m},\,\boldsymbol{m}')\,u_{0}(\boldsymbol{m}')\;,
\]
for all $t>0$ and $\boldsymbol{x}$ in the domain of attraction of
$\boldsymbol{m}$ for the dynamical system described by the ordinary
differential equation $\dot{\boldsymbol{x}}(t)=\boldsymbol{b}(\boldsymbol{x}(t))$.
We then complete the characterization of the solution $u_{\epsilon}$
by computing the exact asymptotic limit of the solution $u_{\epsilon}(\varrho_{\epsilon},\,\boldsymbol{x})$
for the intermediate time-scales $\varrho_{\epsilon}$ such that $\varrho_{\epsilon}/\theta_{\epsilon}^{(p)}\to\infty$,
$\varrho_{\epsilon}/\theta_{\epsilon}^{(p+1)}\to0$ for some $0\le p\le\mathfrak{q}$,
where $\theta_{\epsilon}^{(0)}=1$, $\theta_{\epsilon}^{(\mathfrak{q}+1)}=+\infty$.

Our analysis makes essential use of the hierarchical tree structure underlying the metastable behavior in different time-scales of the diffusion on $\mathbb{R}^{d}$
induced by generator $\mathscr{L}_{\epsilon}$. This result can be
regarded as the precise refinement of Freidlin-Wentzell theory which
was not known for more than a half century. The kernels $R_{t}^{(p)}(\,\cdot\,,\,\cdot\,)$
are expressed in terms of Markov semigroups $\{p_{t}^{(p)}:t\ge0\}$,
defined on partitions of $\mathcal{M}_{0}$ appearing in the description
of the tree structure.
\end{abstract}

\maketitle

\section{Introduction\label{sec1}}

In this article, we examine the asymptotic behavior of the solution
of the parabolic initial-valued problem on $\mathbb{R}^{d}$
\begin{equation}
\left\{ \begin{aligned} & \partial_{t}\,u_{\epsilon}\ =\ \mathscr{L}_{\epsilon}\,u_{\epsilon}\;,\\
 & u_{\epsilon}(0,\,\cdot)\ =\ u_{0}(\cdot)\;,
\end{aligned}
\right.\label{eq:pde1}
\end{equation}
where the elliptic operator $\mathscr{L}_{\epsilon}$, $\epsilon>0$,
is given by
\begin{equation}
\mathscr{L}_{\epsilon}f\ =\ \boldsymbol{b}\cdot\nabla f\,+\,\epsilon\,\Delta f\;,\ \ f\in C^{2}(\mathbb{R}^{d})\;,\label{eq:gen}
\end{equation}
for some smooth vector field $\boldsymbol{b}:\mathbb{R}^{d}\rightarrow\mathbb{R}^{d}$.
We assume that the initial condition $u_{0}\colon\mathbb{R}^{d}\to\mathbb{R}$
is a bounded and continuous function.

It is well-known (e.g., \cite{Miclo2, fk10a,fk10b,is15,is17}) that the
solution $u_{\epsilon}$ exhibits a multi time-scale structure when the
dynamical system
\begin{equation}
\dot{\boldsymbol{x}}(t)\,=\,\boldsymbol{b}(\boldsymbol{x}(t))\label{eq:ode}
\end{equation}
has more than one stable equilibrium. Namely, there exist several
different time-scales along which the solution $u_{\epsilon}$ converges
pointwisely in the domain of attraction of each stable equilibrium.
Quantitative precise estimates of the asymptotic behavior were, however,
not known so far.

We assume that the diffusion process $\{\boldsymbol{x}_{\epsilon}(t)\}_{t\ge0}$
induced by the generator $\mathscr{L}_{\epsilon}$ has a Gibbs invariant
measure of the form
\begin{equation}
d\mu_{\epsilon}(\boldsymbol{x})\ =\ \frac{1}{Z_{\epsilon}}e^{-U(\boldsymbol{x})/\epsilon}\,d\boldsymbol{x}\;,\label{eq:invmeas}
\end{equation}
where $U\colon\mathbb{R}^{d}\to\mathbb{R}$ is a smooth potential
and $Z_{\epsilon}$ is the partition function turning $\mu_{\epsilon}$
into a probability measure on $\mathbb{R}^{d}$. We further assume
that $U$ has a finite number of local minima, that it satisfies a
suitable growth condition (cf. Section \ref{sec2}), and that the
critical points of $U$ are non-degenerate (which entails that $U$
is a Morse function). Under these hypotheses, we present a full characterization
of the multi-scale structure of the solution $u_{\epsilon}(\cdot)$.

It has been shown in \cite{LS-22} that the diffusion process $\{\boldsymbol{x}_{\epsilon}(t)\}_{t\ge0}$
has a Gibbs invariant measure \eqref{eq:invmeas} if, and only if,
the vector field $\boldsymbol{b}:\mathbb{R}^{d}\rightarrow\mathbb{R}^{d}$
appearing in \eqref{eq:gen} has a decomposition of the form
\begin{equation}
\boldsymbol{b}\ =\ -\,(\nabla U+\boldsymbol{\ell})\text{\;\;where\;\;}\nabla\cdot\boldsymbol{\ell}\ \equiv\ 0\;,\;\;\nabla U\cdot\boldsymbol{\ell}\ \equiv\ 0\;,\label{eq:decb}
\end{equation}
where $U$ is the potential function introduced in \eqref{eq:invmeas}
and $\boldsymbol{\ell}\colon\mathbb{R}^{d}\to\mathbb{R}$ is an incompressible
vector field orthogonal to the gradient of $U$. In the special case
$\boldsymbol{\ell}=\boldsymbol{0}$, $\mathscr{L}_{\epsilon}$ is
the generator of the overdamped Langevin dynamics under the potential
$U$ at temperature $\epsilon>0$.

\subsection{Multi-scale structure of solution $u_{\epsilon}$ }

We turn to the main result of the article. Denote by \textcolor{blue}{$\mathcal{M}_{0}$}
the set of the local minima of $U$, assumed to have at least two
elements, $|\mathcal{M}_{0}|\ge2$. Denote by $\mathcal{D}(\boldsymbol{m})$,
$\boldsymbol{m}\in\mathcal{M}_{0}$, the domain of attraction of $\boldsymbol{m}$
for the dynamical system \eqref{eq:ode}.

For two positive sequences $(\alpha_{\epsilon})_{\epsilon>0}$ and
$(\beta_{\epsilon})_{\epsilon>0}$, write ${\color{blue}\alpha_{\epsilon}\prec\beta_{\epsilon}}$
if $\alpha_{\epsilon}/\beta_{\epsilon}\to0$ as $\epsilon\rightarrow0$.
The main result of this article identifies a positive integer $\mathfrak{q}$,
time-scales $\theta_{\epsilon}^{(j)}$, $j\in\llbracket0,\,\mathfrak{q}+1\rrbracket$\footnote{In this article, we write $\llbracket a,\,b\rrbracket=[a,\,b]\cap\mathbb{Z}$
where $a,\,b\in\mathbb{R}$ and $[a,\,b]$ denotes the closed interval.}, depending only on the potential function $U$, real numbers $c^{(k)}(\boldsymbol{m})$, $k\in\llbracket0,\,\mathfrak{q}\rrbracket$,
and continuous functions $f^{(l)}(\cdot\,;\boldsymbol{m})\colon(0,+\infty)\rightarrow\mathbb{R}$,
$l\in\llbracket1,\,\mathfrak{q}\rrbracket$, $\bm{m}\in\mathcal{M}_{0}$,
such that
\begin{equation}
1\equiv\theta_{\epsilon}^{(0)}\prec\theta_{\epsilon}^{(1)}\prec\cdots\prec\theta_{\epsilon}^{(\mathfrak{q})}\prec\theta_{\epsilon}^{(\mathfrak{q}+1)}\equiv\infty\;,\label{eq:timescale}
\end{equation}
and
\begin{equation}
\lim_{\epsilon\to0}\,u_{\epsilon}(\varrho_{\epsilon},\boldsymbol{x})\ =\ \begin{cases}
c^{(p)}(\boldsymbol{m}) & \text{if}\ \theta_{\epsilon}^{(p)}\prec\varrho_{\epsilon}\prec\theta_{\epsilon}^{(p+1)}\;,\;\;p\in\llbracket0,\,\mathfrak{q}\rrbracket\;,\\
f^{(p)}(t,\,\boldsymbol{m}) & \text{if}\ \varrho_{\epsilon}=t\theta_{\epsilon}^{(p)}\;,\;\;p\in\llbracket1,\,\mathfrak{q}\rrbracket\;,
\end{cases}\label{eq:mes1}
\end{equation}
for all $\boldsymbol{m}\in\mathcal{M}_{0}$ and $\boldsymbol{x}\in\mathcal{D}(\boldsymbol{m})$.
Furthermore,
\[
\lim_{t\rightarrow0}f^{(p)}(t;\,\boldsymbol{m})\ =\ c^{(p-1)}(\boldsymbol{m})\;\;\;\text{and}\;\;\;\lim_{t\rightarrow\infty}f^{(p)}(t;\,\boldsymbol{m})\ =\ c^{(p)}(\boldsymbol{m})\;.
\]
In other words, the function $f^{(p)}(t;\boldsymbol{m})$ interpolates
the constants $c^{(p-1)}(\boldsymbol{m})$ and $c^{(p)}(\boldsymbol{m})$.

By \eqref{eq:timescale}, the equations \eqref{eq:mes1} fully characterize
the asymptotic behavior of the solution $u_{\epsilon}$ for all time
scales (cf. Figure \ref{fig:multiscale}). Moreover, it will be seen
in the proof that the constants $c^{(p)}(\boldsymbol{m})$ and $f^{(p)}(t;\boldsymbol{m})$,
$t\ge0$, are convex combinations of $\{u_{0}(\boldsymbol{m}'),\,\boldsymbol{m}'\in\mathcal{M}_{0}\}$,
whose weights depend only on $U$ and $\boldsymbol{\ell}$. In particular,
$c^{(0)}(\boldsymbol{m})=u_{0}(\boldsymbol{m})$. The cases $p=0$,
$p=1$ have been considered in \cite{LLS-1st}.

The arguments presented in this article also apply to the case where
equation \eqref{eq:pde1} is considered on a bounded domain with
Dirichlet or Neumann boundary conditions, see Remark \ref{rem:bdry}.

\begin{figure}
\includegraphics[scale=0.082]{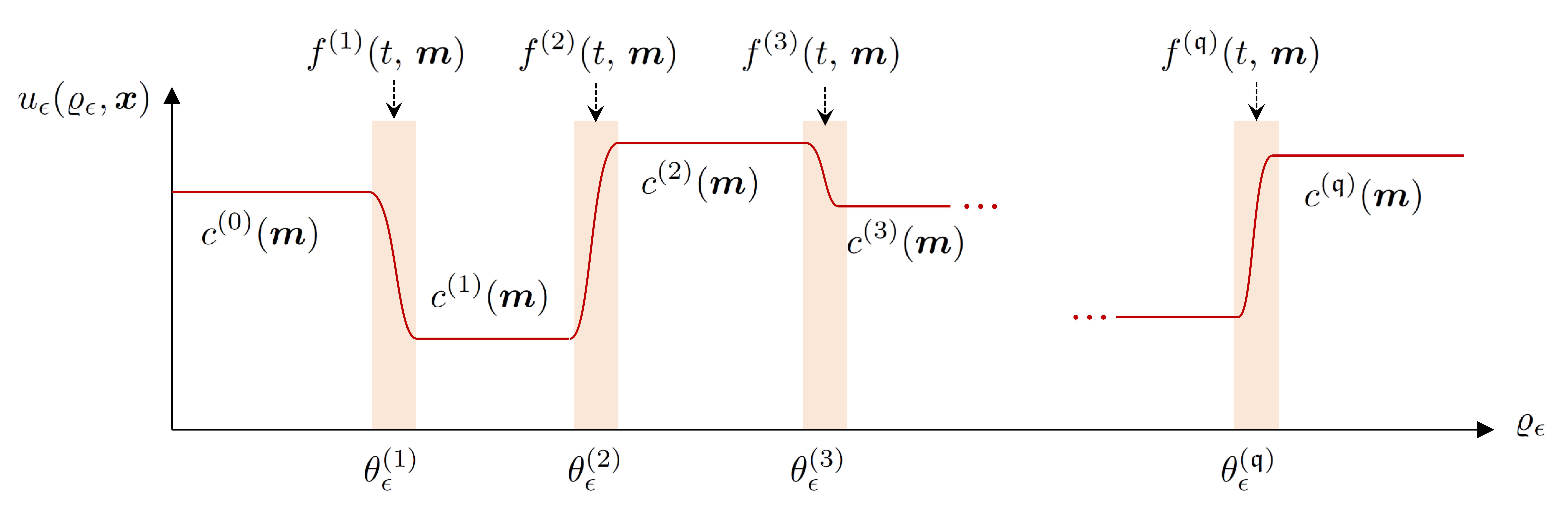}\caption{An illustration of the multi-scale structure of the solution $u_{\epsilon}$
explained in \eqref{eq:mes1}.}
\label{fig:multiscale}
\end{figure}

\subsection{Connection to metastability of diffusion processes }

To understand the multi-scale structure \eqref{eq:mes1}, we examine
the metastable behavior of the diffusion process $\boldsymbol{x}_{\epsilon}(\cdot)$
associated with the generator $\mathscr{L}_{\epsilon}$, and described
by the stochastic differential equation
\begin{equation}
d\boldsymbol{x}_{\epsilon}(t)\ =\ \boldsymbol{b}(\boldsymbol{x}_{\epsilon}(t))\,dt\,+\,\sqrt{2\epsilon}\,dW_{t}\;,\label{sde}
\end{equation}
where $W_{t}$ represents a $d$-dimensional Brownian motion and $\epsilon>0$
a small constant corresponding to the temperature. As $\epsilon$
is small, we regard this diffusion process as a small perturbation
of the dynamical system $\boldsymbol{x}(\cdot)$ given by \eqref{eq:ode}.
In particular, when the dynamical system \eqref{eq:ode} has more
than one stable equilibria, the diffusion process $\boldsymbol{x}_{\epsilon}(\cdot)$
is expected to exhibit a metastable behavior among the stable equilibria.

This problem has been investigated thoroughly during the last decades.
Freidlin and Wentzell \cite{FW} considered it from the point of view
of the large deviations, and set the bases of the theory. The estimates
obtained with large deviations arguments \cite{ov} have been refined
in the case $\boldsymbol{b}=-\nabla U$ \cite{BEGK} and in the general
case \cite{LLS-1st} with methods relying on the Markov chain potential
theory \cite{l-review}. Recently, \cite{RS,LS-22b} considered the
tunneling of the process $\boldsymbol{x}_{\epsilon}(\cdot)$ among
the global minima of $U$.

In spite of the efforts toward the understanding of the metastable
behavior of the process $\boldsymbol{x}_{\epsilon}(\cdot)$, a full
description remained an open problem in the last decades. This problem
is tackled here by constructing a \emph{tree} where each generation
encloses the metastable behavior of the process $\boldsymbol{x}_{\epsilon}(\cdot)$
at a certain time-scale. We refer to Figure \ref{fig:example_tree}
for an illustration of the construction in a specific example.

In Section \ref{sec4}, for each time scale $\theta_{\epsilon}^{(p)}$,
$p\in\llbracket1,\,\mathfrak{q}\rrbracket$, we construct a partition
\begin{equation}
\left\{ \,\mathcal{M}_{1}^{(p)},\,\dots,\,\mathcal{M}_{a_{p}}^{(p)},\,\mathcal{N}_{p}\,\right\} \label{eq:part}
\end{equation}
of the set $\mathcal{M}_{0}$, and a $\{\mathcal{M}_{1}^{(p)},\dots,\mathcal{M}_{a_{p}}^{(p)}\}$-valued
continuous-time Markov chain, denoted by $\{\mathbf{y}^{(p)}(t)\}_{t\ge0}$,
which describes the metastable behavior of the process $\boldsymbol{x}_{\epsilon}(\cdot)$
among the wells defined by the local minima of $U$. Mind that the
element $\mathcal{N}_{p}$ of the partition does not belong to the
state space of the chain $\mathbf{y}^{(p)}(\cdot)$ as, in the scale
$\theta_{\epsilon}^{(p)}$, the process remains only a negligible
amount of time in a neighborhood of the local minima of $U$ in $\mathcal{N}_{p}$.
In contrast, each set $\mathcal{M}_{i}^{(p)}$ corresponds to a metastable
or stable set of the process $\boldsymbol{x}_{\epsilon}(\cdot)$ in
the time-scale $\theta_{\epsilon}^{(p)}$. A similar structure has
been introduced in \cite{bl11,lx16} in the context of metastable
finite-state Markov chains.

The multi-scale structure of the solution $u_{\epsilon}$ explained
in \eqref{eq:mes1} is connected to the metastable behavior of the
process $\boldsymbol{x}_{\epsilon}(\cdot)$ by the stochastic representation
of $u_{\epsilon}$ in terms of $\boldsymbol{x}_{\epsilon}(\cdot)$:
\begin{equation}
u_{\epsilon}(t,\,\boldsymbol{x})\;=\;\mathbb{E}\left[u_{0}(\boldsymbol{x}_{\epsilon}(t))\,\Big|\,\boldsymbol{x}_{\epsilon}(0)=\boldsymbol{x}\right]\;.\label{eq:probexp}
\end{equation}
This relation along with the metastable results for the process $\boldsymbol{x}_{\epsilon}(\cdot)$
yields the following expressions for the functions and constants appearing
in \eqref{eq:mes1}:
\begin{align}
f^{(p)}(t,\,\boldsymbol{m})\ =\  & \sum_{k,\,\ell=1}^{a_{p}}\,\sum_{\boldsymbol{m}'\in\mathcal{M}_{\ell}^{(p)}}\mathfrak{a}^{(p-1)}(\boldsymbol{m},\,\mathcal{M}_{k}^{(p)})\,p_{t}^{(p)}(\mathcal{M}_{k}^{(p)},\,\mathcal{M}_{\ell}^{(p)})\,\frac{\nu(\boldsymbol{m}')}{\nu(\mathcal{M}_{\ell}^{(p)})}\,u_{0}(\boldsymbol{m}')\;,\label{eq:fp}\\
c^{(p)}(\boldsymbol{m})\ =\  & \sum_{k=1}^{a_{p}}\,\sum_{\boldsymbol{m}'\in\mathcal{M}_{k}^{(p)}}\,\mathfrak{a}^{(p-1)}(\boldsymbol{m},\,\mathcal{M}_{k}^{(p)})\,\frac{\nu(\boldsymbol{m}')}{\nu(\mathcal{M}_{\ell}^{(p)})}\,u_{0}(\boldsymbol{m}')\;,\label{eq:cp}
\end{align}
for all $\boldsymbol{m}\in\mathcal{M}_{0}$ and $p\in\llbracket1,\,\mathfrak{q}\rrbracket$,
where
\begin{itemize}
\item $\mathfrak{a}^{(p-1)}(\boldsymbol{m},\,\mathcal{M}_{k}^{(p)})$ is
a constant determined by $\mathbf{y}^{(1)}(\cdot),\,\dots,\,\mathbf{y}^{(p-1)}(\cdot)$
describing the probability that the process $\boldsymbol{x}_{\epsilon}(\cdot)$
starting from $\boldsymbol{m}$ hits a neighborhood of $\mathcal{M}_{k}^{(p)}$
before hitting those of $\mathcal{M}_{j}^{(p)}$, $j\neq k$,
\item $p_{t}^{(p)}(\,\cdot\,,\,\cdot\,)$ is the semigroup associated with
the process $\mathbf{y}^{(p)}(\cdot)$, and
\item $\nu(\cdot)/\nu(\mathcal{M}_{\ell}^{(p)})$ can be regarded as the
probability measure on $\mathcal{M}_{\ell}^{(p)}$ describing the
stationary state of the process $\mathbf{y}^{(p-1)}(\cdot)$ restricted
to the well $\mathcal{M}_{\ell}^{(p)}$ (where we regard $\mathbf{y}^{(0)}(\cdot)$
as a trivial Markov chain that does not move at all). We refer to
\eqref{eq:nu} for the exact definition.
\end{itemize}

\subsection{Resolvent approach to metastability}

The proof of the results described above is purely probabilistic and
relies on the theory of metastable Markov processes developed in \cite{BL1,BL2,LLM,LMS2,RS,LMS}.
The crucial point consists in showing that the solution of a resolvent
equation is asymptotically constant in the wells. More precisely,
fix a constant $\lambda>0$ and $p\in\llbracket1,\,\mathfrak{q}\rrbracket$.
For $\boldsymbol{g}\colon\{\mathcal{M}_{1}^{(p)},\,\dots,\,\mathcal{M}_{a_{p}}^{(p)}\}\rightarrow\mathbb{R}$
denote by $\phi_{\epsilon}^{\boldsymbol{g}}$ the unique solution
of the resolvent equation
\begin{equation}
\left(\lambda-\theta_{\epsilon}^{(p)}\mathscr{L}_{\epsilon}\right)\,\phi_{\epsilon}\ =\ \sum_{j=1}^{a_{p}}\boldsymbol{g}(\mathcal{M}_{j}^{(p)})\,\chi_{\mathcal{E}(\mathcal{M}_{j}^{(p)})}\;,\label{eq:res}
\end{equation}
where $\chi_{\mathcal{A}}$, $\mathcal{A}\subset\mathbb{R}^{d}$,
represents the indicator function of the set $\mathcal{A}$ and $\mathcal{E}(\mathcal{M}_{j}^{(p)})$
a neighborhood of the set $\mathcal{M}_{j}^{(p)}$ (e.g., the union
of small open balls around each $\boldsymbol{m}\in\mathcal{M}_{j}^{(p)}$).
We prove that the solution $\phi_{\epsilon}^{\boldsymbol{g}}$ is
asymptotically constant in each set $\mathcal{E}(\mathcal{M}_{j}^{(p)})$
in the sense that
\begin{equation}
\lim_{\epsilon\rightarrow0}\,\max_{j=1,\,\dots,\,a_{p}}\,\sup_{\boldsymbol{x}\in\mathcal{E}(\mathcal{M}_{j}^{(p)})}\left|\,\phi_{\epsilon}^{\boldsymbol{g}}(\boldsymbol{x})-\boldsymbol{f}(\mathcal{M}_{j}^{(p)})\,\right|\ =\ 0\;,\label{fx3}
\end{equation}
where $\boldsymbol{f}$ is the solution of the reduced resolvent equation
\begin{equation}
\left(\lambda-\mathfrak{L}^{(p)}\right)\,\boldsymbol{f}\ =\ \boldsymbol{g}\;,\label{fx1}
\end{equation}
and $\mathfrak{L}^{(p)}$ is the generator of the Markov chain $\mathbf{y}^{(p)}(\cdot)$
introduced above.

The resolvent approach to metastability developed recently in \cite{LMS}
asserts that establishing \eqref{fx3}, \eqref{fx1} for all functions
$\boldsymbol{g}\colon\{\mathcal{M}_{1}^{(p)},\,\dots,\,\mathcal{M}_{a_{p}}^{(p)}\}\rightarrow\mathbb{R}$
is a necessary and sufficient condition for the metastability behavior
of the process $\boldsymbol{x}_{\epsilon}(\cdot)$ at the time-scale
$\theta_{\epsilon}^{(p)}$. A rigorous formulation of this statement
is presented in Section \ref{sec3}.

Therefore, since the derivation of \eqref{eq:mes1} is reduced to
the proof of the metastable behavior of the diffusion $\boldsymbol{x}_{\epsilon}(\cdot)$,
and this latter result follows from \eqref{fx3}, \eqref{fx1}, the
proof is reduced to this last property of the resolvent equation.

\subsection{Related works}

As mentioned above, the asymptotic behavior of the solution
$u_{\epsilon}$ of the equation \eqref{eq:pde1} in the regime
$\epsilon\rightarrow0$ is closely connected to the metastable behavior
of the diffusion process $\boldsymbol{x}_{\epsilon}(\cdot)$. The
analysis of $u_{\epsilon}$ based on this connection has been examined
before.

Miclo \cite{Miclo2} proved the existence of a finite number of
disjoint cycles which absorb the trajectories in the non-critical
time-scales. This results corresponds to the limits \eqref{eq:mes1} in
the case
$\theta_{\epsilon}^{(p)}\prec\varrho_{\epsilon}\prec\theta_{\epsilon}^{(p+1)}$
for some $p$. 

Freidlin and Koralov
\cite{fk10a,fk10b} found a critical depth $D>0$ and showed that
the the solution $u_{\epsilon}(t,\,x)$ in the interval $t\in[0,\,e^{(D-\eta)/\epsilon}]$
and $t\in[e^{(D+\eta)/\epsilon},\,\infty)$ differ significantly for
all $\eta>0$. Therefore, a dramatic phase transition occurs at the
scale $\theta_{\epsilon}=e^{D/\epsilon}$.

This result has been extended by Koralov and Tcheuko \cite{kt16} to
cases which exhibit multiple metastable time-scales. Ishii and
Souganidis \cite{is15,is17} derived similar results with purely
analytical methods. The main advantage exploited here in respect to
\cite{fk10a,fk10b} is that the stationary state is explicitly
known. Miclo \cite{Miclo4} and Bertini, Gabrielli and Landim
\cite{bgl2} considered the same problem in the context of finite-state
Markov chains.

The metastable behavior of the process
$\boldsymbol{x}_{\epsilon}(\cdot)$ has been recently studied in
several articles: \cite{LM} provided sharp asymptotics on the
low-lying spectra which is closely related with the metastability of
the process $\boldsymbol{x}_{\epsilon}(\cdot)$, \cite{BEGK,LS-22}
established Eyring-Kramers law precisely estimating the mean
transition time from a local minimum of $U$ to another one, and
\cite{LS-22b,RS} investigated the metastability among the global
minima (i.e., ground states) of $U$. The last work can be regarded as
the analysis of the metastability at the final time-scale
$\theta_{\epsilon}^{(\mathfrak{q})}$ considered here.

Uniform estimates, similar to \eqref{fx3}, for solutions of Dirichlet
problems go back at least to Devinatz and Friedman \cite{DF78}, and
Day \cite{Day82}. The convergence to a constant is called in the
literature the leveling property of the equation. We refer to Leli\`evre,
Le Peutrec and Nectoux \cite{LLP22} for a recent account and further
references.

\subsection{Novelty of article}

We conclude this section with some comments on the
novelties of the article.

Firstly, we established a robust inductive scheme
to analyze a complex energy landscape hierarchy structure. This scheme
is completely new and has potential to be applied to other models
with complex landscape, such as the Ising and Potts models \cite{BC,KS1,KS2},
spin glass models \cite{MPV}, or to the loss function appearing in
deep learning \cite{BDM}.

The metastability analysis of the current article is based on the
martingale approach \cite{BL1,LLM} and the resolvent appoach
\cite{LMS}, but we get highly refined results than the previous ones
based on a multitude of technical novelties. For instance, we
generalize the result of Barrera-Jara \cite{BJ} to prove the local
ergodicity of the process $\boldsymbol{x}_{\epsilon}(\cdot)$ to the
case where $U$ has multiple local minima, we introduce a new type of
test functions to characterize the values of the resolvent equation in
the wells, we design an auxiliary process, denoted by
$\widehat{\mathbf{y}}^{(p)}(\cdot)$, to define the Markov chain
$\mathbf{y}^{(p)}(\cdot)$ which describes the metastable behavior of
$\boldsymbol{x}_{\epsilon}(\cdot)$ at the time-scale
$\theta_{\epsilon}^{(p)}$.  These innovations permit to extend the
analysis of the metastable behavior of the process
$\boldsymbol{x}_{\epsilon}(\cdot)$ to the case where it starts from a
domain of attraction of a local minimum which can even be negligible at
each time scale; extending the previous results which required the
process to start from a microscopic well around a local minimum whose
depth corresponds to the time scale.

\section{Main results on partial differential equations\label{sec2}}

In this section, we present the main results on the solution to partial
differential equations (PDEs). Theorem \ref{t00}, explained in Subsection
\ref{sec2.3}, provides the complete multi-scale structure of the
solution $u_{\epsilon}$ to the parabolic equation \eqref{eq:pde1}
in the $\epsilon\rightarrow0$ regime. The underlying analysis allowing
us to establish this multi-scale structure is the demonstration of
the leveling property of the solution to the resolvent equations associated
with the operator $\mathscr{L}_{\epsilon}$. This second result is
explained in Subsection \ref{sec2.4}.

The third result is the rigorous verification of the tree structure formed by the metastable states of the diffusion process
\textcolor{blue}{$\bm{x}_{\epsilon}(\cdot)$} associated with the
generator $\mathscr{L}_{\epsilon}$ introduced in \eqref{sde} in the
$\epsilon\rightarrow0$ regime.  This result is explained in the next
section.

\subsection{\label{sec2.1}Potential function and zero-noise dynamics}

Before starting to state the main results, we summarize here the assumptions
on the potential function $U\colon\mathbb{R}^{d}\to\mathbb{R}$ and
vector field $\boldsymbol{\ell}\colon\mathbb{R}^{d}\to\mathbb{R}^{d}$
which appeared in the decomposition \eqref{eq:decb} of $\boldsymbol{b}$.
We assume that they are smooth enough, $U\in C^{3}(\mathbb{R}^{d})$
and $\boldsymbol{\ell}\in C^{2}(\mathbb{R}^{d},\,\mathbb{R}^{d})$,
and satisfy the relation \eqref{eq:decb}.

We assume the following control of $U$ at infinity:
\begin{equation}
\begin{gathered}\lim_{n\to\infty}\,\inf_{|\boldsymbol{x}|\geq n}\,\frac{U(\boldsymbol{x})}{|\boldsymbol{x}|}\ =\ \infty\;,\ \ \lim_{|\boldsymbol{x}|\to\infty}\,\frac{\boldsymbol{x}}{|\boldsymbol{x}|}\cdot\nabla U(\boldsymbol{x})\ =\ \infty\;,\\
\lim_{|\boldsymbol{x}|\to\infty}\,\left\{ \,\left|\,\nabla U(\boldsymbol{x})\,\right|\,-\,2\,\Delta U(\boldsymbol{x})\,\right\} \ =\ \infty\;.
\end{gathered}
\label{eq:growth}
\end{equation}
In this formula and below, ${\color{blue}|\boldsymbol{x}|}$ represents
the Euclidean norm of $\boldsymbol{x}\in\mathbb{R}^{d}$. These are
routine assumptions (cf. \cite{BEGK,LMS,LS-22}) guaranteeing the
positive recurrence of the process $\boldsymbol{x}_{\epsilon}(\cdot)$
and hence that the Gibbs measure $\mu_{\epsilon}$ given in \eqref{eq:invmeas}
is the invariant probability measure of the process $\boldsymbol{x}_{\epsilon}(\cdot)$
(cf. \cite[Theorem 2.2]{LS-22})

Denote by ${\color{blue}\mathcal{C}}$ the set of critical points
of $U$. We assume that $U$ is a Morse function (cf. \cite[Definition 1.7]{Nic18}),
i.e., all the critical points of $U$ are non-degenerate in the sense
that the Hessian of $U$ at $\boldsymbol{c}$, denoted by\textcolor{blue}{
$(\nabla^{2}U)(\boldsymbol{c})$,} is invertible for all $\boldsymbol{c}\in\mathcal{C}$.

Denote by $\{\boldsymbol{x}(t)\}_{t\ge0}$ the zero-noise dynamics
associated with \eqref{sde} which is described by the ODE
\begin{equation}
\dot{\boldsymbol{x}}(t)\ =\ \boldsymbol{b}(\boldsymbol{x}(t))\;\;\;;\ \;t\in\mathbb{R}\;.\label{eq:x(t)}
\end{equation}
Namely, we can regard the diffusion process $\boldsymbol{x}_{\epsilon}(\cdot)$
given in \eqref{sde} as a small random perturbation of the dynamical
system \eqref{eq:x(t)}. In particular, the metastability of the process
$\boldsymbol{x}_{\epsilon}(\cdot)$ is closely related with the phase
diagram of the dynamical system \eqref{eq:x(t)}. The following is
the basic results on the dynamical system \eqref{eq:x(t)} under \eqref{eq:decb}.
Recall that we denote by $\mathcal{M}_{0}$ the set of local minima
of $U$.
\begin{thm}[{{\cite[Theorem 2.1]{LS-22}}}]
The set $\mathcal{C}$ of critical points of $U$ is exactly the
set of all equilibria of the dynamical system $\boldsymbol{x}(\cdot)$,
and the set $\mathcal{M}_{0}$ is exactly the set of all the stable
equilibria of $\boldsymbol{x}(\cdot)$.
\end{thm}

In view of the previous theorem, the metastability of the process
$\boldsymbol{x}_{\epsilon}(\cdot)$, and therefore the multi-scale
structure of the solution $u_{\epsilon}$, appears only when $|\mathcal{M}_{0}|\ge2$,
i.e., when there are multiple stable equilibria of the dynamical system
$\boldsymbol{x}(\cdot)$. Therefore, we henceforth assume that $|\mathcal{M}_{0}|\ge2$.
We refer to Remark \ref{rem:M0=00003D00003D1} for a comment on the
trivial case with $|\mathcal{M}_{0}|=1$.

For two critical points $\boldsymbol{c}_{1},\,\boldsymbol{c}_{2}\in\mathcal{C}$
of $U$, a heteroclinic orbit $\phi$ from $\boldsymbol{c}_{1}$ to
$\boldsymbol{c}_{2}$ is the path $\phi\colon\mathbb{R}\to\mathbb{R}^{d}$
satisfying \eqref{eq:x(t)}, i.e., $\dot{\phi}(t)\,=\,\boldsymbol{b}(\phi(t))$
for all $t\in\mathbb{R}$, and such that
\begin{equation}
\lim_{t\to-\infty}\phi(t)\ =\ \boldsymbol{c}_{1}\,,\quad\lim_{t\to\infty}\phi(t)\ =\ \boldsymbol{c}_{2}\;.\label{2-13}
\end{equation}

For $\boldsymbol{c}_{1},\,\boldsymbol{c}_{2}\in\mathcal{C}$, if there
exists a heteroclinic orbit from $\boldsymbol{c}_{1}$ to $\boldsymbol{c}_{2}$,
we write ${\color{blue}\boldsymbol{c}_{1}\curvearrowright\boldsymbol{c}_{2}}$.
We note that the potential $U$ is decreasing along the heteroclinic
orbit $\phi$ because
\[
\frac{d}{dt}U(\phi(t))\ =\ -\,|\,\nabla U(\phi(t))\,|^{2}\;.
\]
We used here that the field $\boldsymbol{\ell}$ is orthogonal to
$\nabla U$. Thus, reparametrising $\phi$, there exists a continuous
path $\boldsymbol{z}:[0,\,1]\rightarrow\mathbb{R}^{d}$ such that
\begin{equation}
\boldsymbol{z}(0)\,=\,\bm{c}_{1}\;,\;\;\boldsymbol{z}(1)\,=\,\bm{c}_{2}\;,\text{ and \;\;\;}U(\boldsymbol{z}(t))\,<\,U(\bm{c}_{1})\;\;\text{ for all }\;\;t\in(0,\,1]\label{z(t)_exp}
\end{equation}
when $\boldsymbol{c}_{1}\curvearrowright\boldsymbol{c}_{2}$.

We denote by ${\color{blue}\mathcal{S}_{0}}$ the set of saddle points
of $U$. Since $U$ is a Morse function, $\mathcal{S}_{0}$ is the
set of critical points $\boldsymbol{\sigma}$ of $U$ at which the
Hessian $(\nabla^{2}U)(\boldsymbol{\sigma})$ of $U$ at $\boldsymbol{\sigma}$
has only one negative eigenvalue, all the other ones being strictly
positive. In particular, by the Hartman-Grobman theorem, there exist
two heteroclinic orbits starting from $\boldsymbol{\sigma}$. We assume
that these two heteroclinic orbits end at local minima, i.e., for
all $\boldsymbol{\sigma}\in\mathcal{S}_{0}$, there exist $\boldsymbol{m}_{\boldsymbol{\sigma}}^{+},\,\boldsymbol{m}_{\boldsymbol{\sigma}}^{-}\in\mathcal{M}_{0}$
such that
\begin{equation}
\boldsymbol{\sigma}\curvearrowright{\color{blue}\boldsymbol{m}_{\boldsymbol{\sigma}}^{+}}\;\;\text{and}\;\;\boldsymbol{\sigma}\curvearrowright{\color{blue}\boldsymbol{m}_{\boldsymbol{\sigma}}^{-}}\;.\label{hyp2}
\end{equation}
We do not assume that the limit points are distinct, we may have that
$\boldsymbol{m}_{\boldsymbol{\sigma}}^{+}=\boldsymbol{m}_{\boldsymbol{\sigma}}^{-}$.
We refer to Figure \ref{fig:hetero} for an illustration.

\begin{figure}
\includegraphics[scale=0.072]{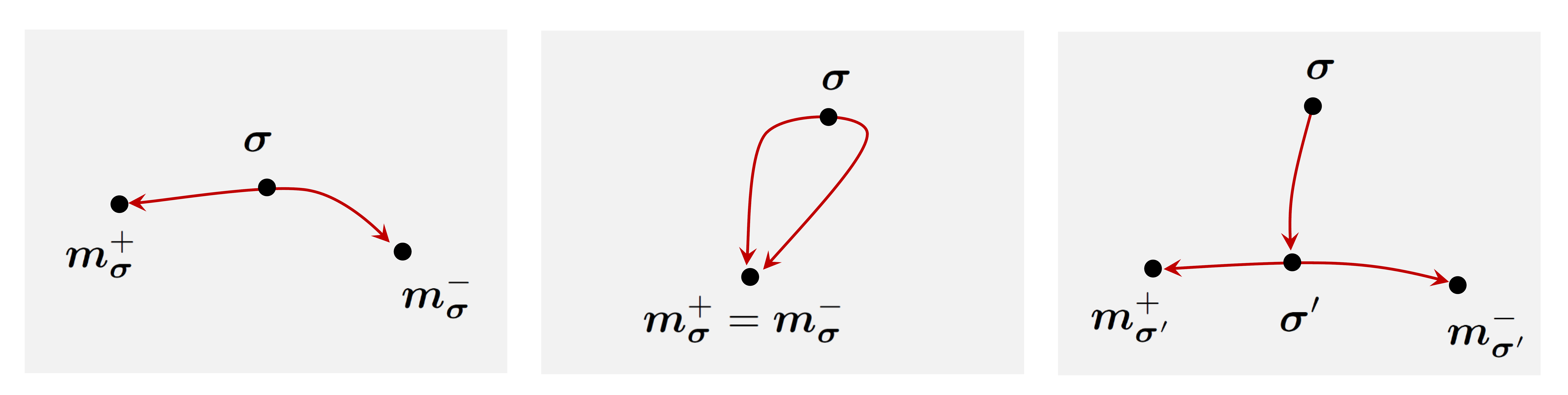}

\caption{\label{fig:hetero}Heteroclinic orbits starting from a saddle $\boldsymbol{\sigma}$.
We allow cases (1) and (2), but do not allow case (3). In particular,
in case (3), both $\bm{\sigma}$ and $\bm{\sigma}'$ are saddle. The
unstable manifold of $\bm{\sigma}$ and stable manifold of $\bm{\sigma}'$
do not intersect transversally in this case, and therefore the system
is not Morse--Smale.}
\end{figure}

\begin{rem}
The assumption \eqref{hyp2} holds when the dynamical system $\boldsymbol{x}(\cdot)$
defined in \eqref{eq:x(t)} is a Morse-Smale system. In the Morse-Smale
system, for two critical points $\boldsymbol{c}_{1},\,\boldsymbol{c}_{2}\in\mathcal{C}$,
the unstable manifold of $\boldsymbol{c}_{1}$ and the stable manifold
of $\boldsymbol{c}_{2}$ intersect transversally and thus when $\boldsymbol{c}_{1}\curvearrowright\boldsymbol{c}_{2}$,
the index (the number of negative eigenvalue of the Hessian) must
strictly decrease along heteroclinic orbits; hence \eqref{hyp2} follows
naturally.
\end{rem}

\begin{rem}
The assumption \eqref{hyp2} needs not to be satisfied for all saddle
points, cf. Remark \ref{rem_assu_saddle}.
\end{rem}

\subsection{\label{sec2.2}Tree structure of metastable behavior}

In this subsection, we introduce the tree structure associated with
the potential $U$ which is used in all the main statements. Since
a detailed presentation is very long, we postpone it to Section \ref{sec4}
and just sketch it here.
\begin{defn}[Tree structure]
\label{def:tree}Recall that we have assumed $|\mathcal{M}_{0}|\ge2$.
Under this assumption, we can construct the following objects in Section
\ref{sec4}.
\begin{enumerate}
\item A positive integer ${\color{blue}\mathfrak{q}\ge1}$ denoting the
number of scales.
\item A sequence of depths $0<d^{(1)}<\dots<d^{(\mathfrak{q})}<\infty$
and time-scales
\[
{\color{blue}\theta_{\epsilon}^{(p)}}\ :=\ \exp\frac{d^{(p)}}{\epsilon}\;\;;\ \;p\in\llbracket1,\,\mathfrak{q}\rrbracket\;.
\]
\item A sequence of finite-state Markov chains ${\bf y}^{(p)}(\cdot)$,
$p\in\llbracket1,\,\mathfrak{q}\rrbracket$, constructed inductively
as follows.
\end{enumerate}
\end{defn}

Let
\begin{equation}
{\color{blue}\mathscr{S}^{(1)}\ =\ \mathscr{V}^{(1)}}\,:=\,\left\{ \,\{\boldsymbol{m}\}:\boldsymbol{m}\in\mathcal{M}_{0}\,\right\} \;\;\text{and}\;\;{\color{blue}\mathscr{N}^{(1)}}\ :=\ \varnothing\;.\label{eq:v1}
\end{equation}
In Section \ref{sec4}, we construct a $\mathscr{S}^{(1)}$-valued
continuous-time Markov chain \textcolor{blue}{$\{\widehat{\mathbf{y}}^{(1)}(t)\}_{t\ge0}$},
and a $\mathscr{V}^{(1)}$-valued continuous-time Markov chain ${\color{blue}\{\mathbf{y}^{(1)}(t)\}_{t\ge0}}$.
For $p=1$ they coincide.

Denote by $\mathfrak{n}_{0}$ the number of local minima of $U$,
${\color{blue}\mathfrak{n}_{0}=|\mathcal{M}_{0}|}$, and by ${\color{blue}\mathscr{R}_{1}^{(1)},\dots,\mathscr{R}_{\mathfrak{n}_{1}}^{(1)}}$,
${\color{blue}\mathscr{T}^{(1)}}$ the closed irreducible classes
and the transient states of the Markov chain ${\bf y}^{(1)}$, respectively.
If $\mathfrak{n}_{1}=1$, the construction is complete, $\mathfrak{q}=1$,
and $d^{(1)}$, $\mathscr{V}^{(1)}$, $\mathscr{N}^{(1)}$, $\mathbf{y}^{(1)}(\cdot)$
have been defined. If $\mathfrak{n}_{1}>1$, we add a new layer to
the construction.

Fix $p\ge1$. Assume that the construction has been carried out up
to layer $p$, and that \textcolor{blue}{$\mathfrak{n}_{p}$}, the
number of recurrent classes of the Markov chain $\mathbf{y}^{(p)}(\cdot)$,
is strictly larger than $1$: $\mathfrak{n}_{p}>1$. For $j\in\llbracket1,\,\mathfrak{n}_{p}\rrbracket$,
let
\begin{equation}
\mathcal{M}_{j}^{(p+1)}\ :=\ \bigcup_{\mathcal{M}\in\mathscr{R}_{j}^{(p)}}\mathcal{M}\;,\ \ {\color{blue}\mathscr{V}^{(p+1)}}\ :=\ \left\{ \,\mathcal{M}_{1}^{(p+1)}\,\dots,\mathcal{M}_{\mathfrak{n}_{p}}^{(p+1)}\,\right\} \;,\ \ {\color{blue}\mathscr{N}^{(p+1)}}\ :=\ \mathscr{N}^{(p)}\cup\mathscr{T}^{(p)}\;,\label{2-14}
\end{equation}
${\color{blue}\mathscr{S}^{(p+1)}}:=\mathscr{V}^{(p+1)}\cup\mathscr{N}^{(p+1)}$.
Note that the elements of $\mathscr{S}^{(p+1)}$ are subsets of
$\mathcal{M}_{0}$.  In Section \ref{sec4}, we construct a
$\mathscr{S}^{(p+1)}$-valued continuous-time Markov chain
${\color{blue}\{\widehat{\mathbf{y}}^{(p+1)}(t)\}_{t\ge0}}$, and
define ${\color{blue}\{\mathbf{y}^{(p+1)}(t)\}_{t\ge0}}$ as its trace
on the set $\mathscr{V}^{(p+1)}$ (cf. Appendix \ref{app:trace}). The
$\mathscr{S}^{(p+1)}$-valued Markov chain
$\widehat{\mathbf{y}}^{(p+1)}(\cdot)$ is an artificial process
introduced to define in a simple way the Markov chain
$\mathbf{y}^{(p+1)}(\cdot)$.  It follows from this construction that
\begin{itemize}
\item[(a)] If $A\in\mathscr{V}^{(q)}$, $q\in\llbracket1,\,\mathfrak{q}\rrbracket$,
then either $A\in\mathscr{N}^{(q+1)}$ or there exists a unique $B\in\mathscr{V}^{(q+1)}$
such that $A\subset B$. In contrast, if $A\in\mathscr{N}^{(q)}$
then $A\in\mathscr{N}^{(q+1)}$.
\item[(b)] By Theorem \ref{t:tree}-(3), $|\mathscr{V}^{(q)}|>|\mathscr{V}^{(q+1)}|$
if $|\mathscr{V}^{(q)}|>1$. In particular, there exists $p\ge1$,
denoted by $\mathfrak{q}$, such that $\mathfrak{n}_{\mathfrak{q}}=1$,
and this construction ends.
\item[(c)] For each $1\in\llbracket1,\,\mathfrak{q}\rrbracket$, the elements
of $\mathscr{S}^{(\mathfrak{q})}$ form a partition of $\mathcal{M}_{0}$.
\item[(d)] The transient states of the chain $\mathbf{y}^{(q)}(\cdot)$ are
the sets in $\mathscr{V}^{(q)}$ which are transfered to $\mathscr{N}^{(q+1)}$.
The closed irreducible classes form the elements of $\mathscr{V}^{(q+1)}$,
$q\in\llbracket1,\,\mathfrak{q}\rrbracket$.
\end{itemize}
By (a) and (b), at each step $q$ there is at least one element $A\in\mathscr{V}^{(q)}$
which is transferred to $\mathscr{N}^{(q+1)}$ or which merges with
at least another element of $\mathscr{V}^{(q)}$ to become an element
of $\mathscr{V}^{(q+1)}$.

One can visualize the $q$-th partition of $\mathcal{M}_{0}$ as a
generation of a tree. The root is the set $\mathcal{M}_{0}$. The
first generation consists of two types of elements: the unique element of $\mathscr{V}^{(\mathfrak{q}+1)}$
and the elements of $\mathscr{N}^{(\mathfrak{q}+1)}$.
The $n$-th generation, $n\in\llbracket2,\,\mathfrak{q}\rrbracket$,
is formed by the elements of $\mathscr{V}^{(\mathfrak{q}+2-n)}$ and
the union of the elements of $\mathscr{N}^{(\mathfrak{q}+2-n)}$.
The $(\mathfrak{q}+1)$-th and last generation is composed by the
singletons $\{\bm{m}\}$, $\bm{m}\in\mathcal{M}_{0}$. In this interpretation,
a set $A$ in the generation $n+1$ is a child of a set $B$ in the
generation $n$ if $A\subset B$. In the tree language, the construction
has been carried out from the leaves to the root. The example given
in Section \ref{sec2_example} illustrates this structure.
\begin{rem}
\label{rem:tree} It turns out from the construction carried out in
Section \ref{sec4} that:
\begin{enumerate}
\item The depths $d^{(1)}<\cdots<d^{(\mathfrak{q})}$ are defined together
with the processes $\widehat{\mathbf{y}}^{(p+1)}(\cdot)$ in the inductive
argument. These depths and the partitions $\mathscr{S}^{(p)}$, $p\in\llbracket1,\,\mathfrak{q}+1\rrbracket$,
are defined only in terms of the potential function $U$ and independent
of $\boldsymbol{\ell}$. However, the Markov chain ${\bf y}^{(p)}(\cdot)$
depends on $\boldsymbol{\ell}$.
\item By Proposition \ref{prop_global_min}, $\mathscr{V}^{(\mathfrak{q}+1)}$
consists of all global minima of $U$.
\item The main properties of the tree structure are stated as $\mathfrak{P}_{1}(p),\,\dots,\,\mathfrak{P}_{4}(p)$,
$p\in\llbracket1,\,\mathfrak{q}\rrbracket$, in Section \ref{sec5}.
\end{enumerate}
\end{rem}

\subsection{\label{sec2_example}Example of tree structure}

In this subsection, we present an example to illustrate the previous
construction. We refer to Figure \ref{fig:potential} for the graph
of the one-dimensional potential $U:\mathbb{R}\rightarrow\mathbb{R}$.
The ten local minima are represented by $\mathcal{M}_{0}=\{\bm{m}_{1},\,\bm{m}_{2},\,\dots,\,\bm{m}_{10}\}$,
and three depths $d^{(j)}$, $j\in\llbracket1,\,3\rrbracket$, are
indicated by red, blue and green arrows.

\begin{figure}
\includegraphics[scale=0.27]{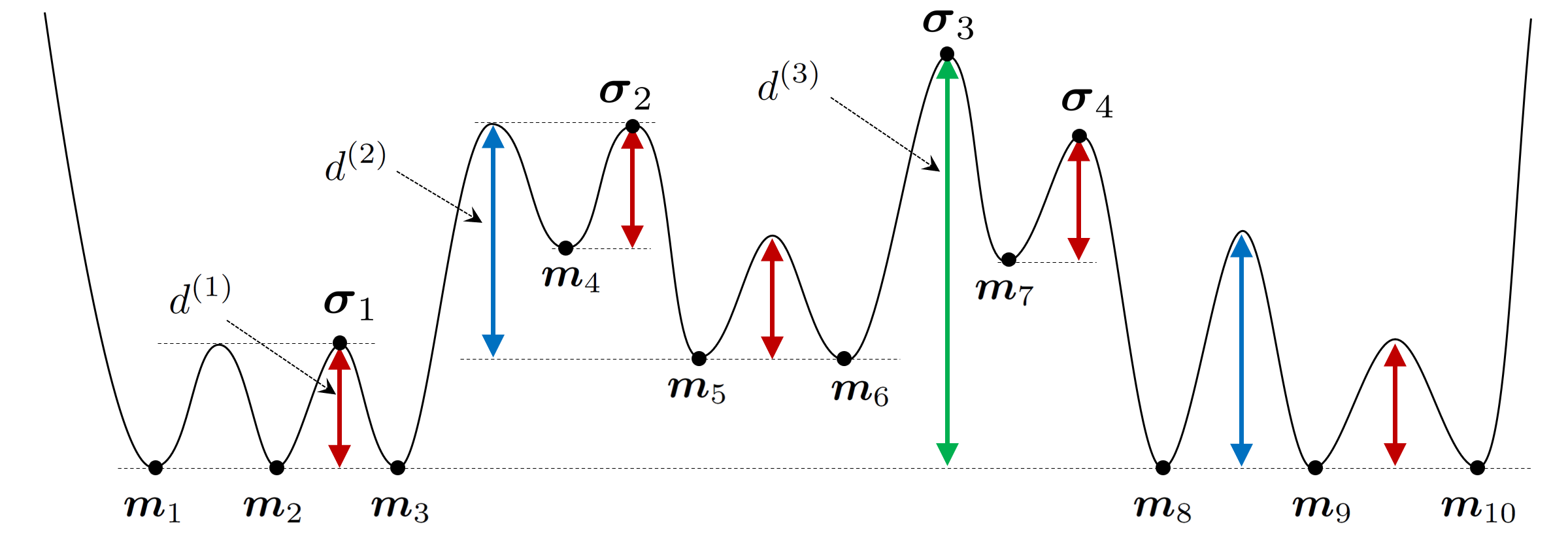}\caption{An example of potential function $U$}
\label{fig:potential}
\end{figure}

As explained in the previous subsection, $\mathscr{S}^{(1)}=\mathscr{V}^{(1)}=\{\,\{\bm{m}_{1}\},\,\{\bm{m}_{2}\},\,\dots,\,\{\bm{m}_{10}\}\}$,
$\mathscr{N}^{(1)}=\varnothing$. The depth $d^{(1)}$ is defined
as the smallest energy barrier between two local minima. In Figure
\ref{fig:potential}, it is indicated by the red arrow.

The $\mathscr{S}^{(1)}$-valued Markov chain ${\bf y}^{(1)}(\cdot)$
can be informally described as follows. It may jump from a state $\{\bm{m}_{i}\}$
to a state $\{\bm{m}_{j}\}$ if $\bm{m}_{j}$ is a neighbour of $\bm{m}_{i}$
and if the potential barrier between $\bm{m}_{i}$ and $\bm{m}_{j}$
is $d^{(1)}$. The specific jump rates will be given in Section \ref{sec4}.

By the description of the jump-rates, there are four recurrent classes
\begin{equation}
\{\boldsymbol{m}_{1},\,\boldsymbol{m}_{2},\,\boldsymbol{m}_{3}\},\,\{\boldsymbol{m}_{8}\},\,\{\boldsymbol{m}_{9},\,\boldsymbol{m}_{10}\},\,\{\boldsymbol{m}_{5},\,\boldsymbol{m}_{6}\}\label{eq:recc}
\end{equation}
and two transient states $\{\boldsymbol{m}_{4}\}$ and $\{\boldsymbol{m}_{7}\}$.
By construction, $\mathscr{V}^{(2)}$ is the set formed by the four
recurrent classes and $\mathscr{N}^{(2)}$ by the two transient states.
The figure \ref{fig:example_tree} illustrates the last two generations
of the tree, formed by the partitions $\mathscr{S}^{(1)}$ and $\mathscr{S}^{(2)}$.

\begin{figure}
\includegraphics[scale=0.072]{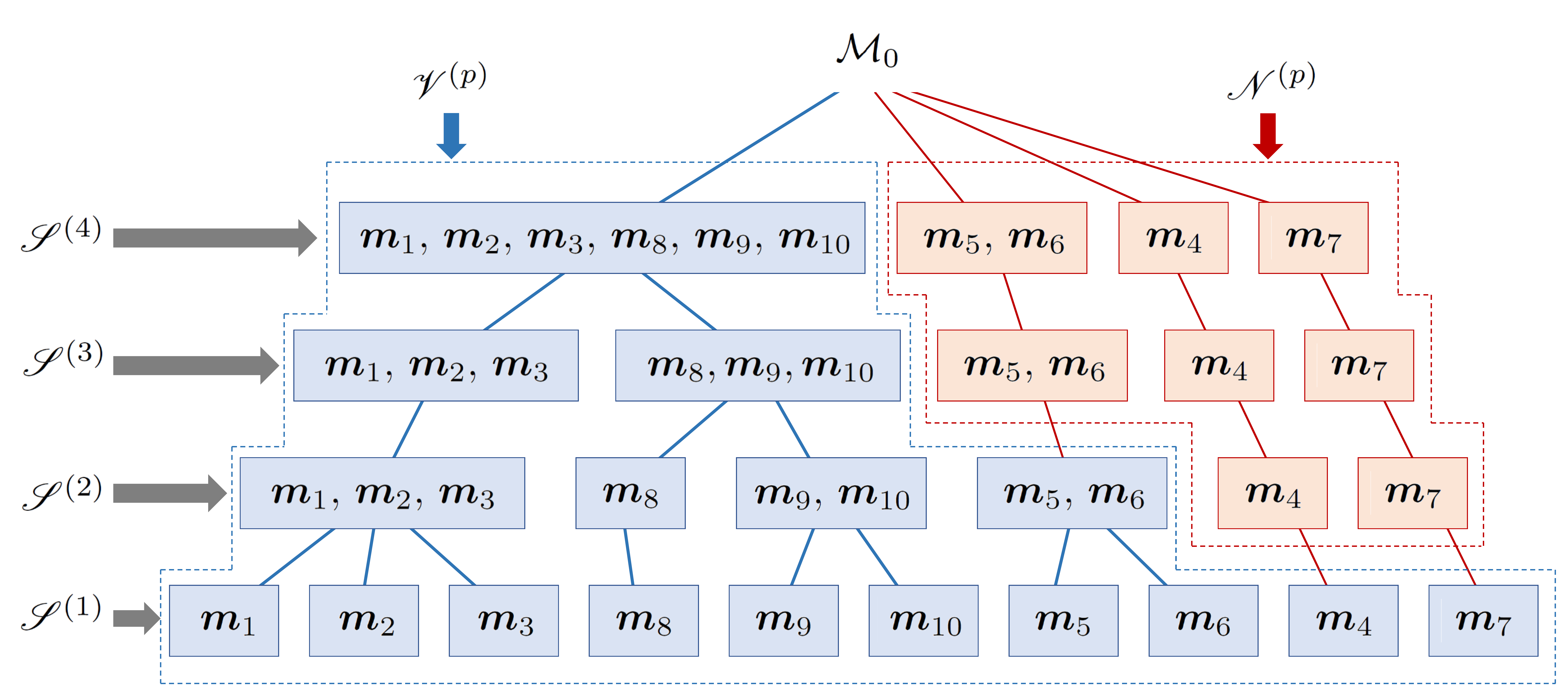}\label{fig:example_tree}\caption{Tree structure associated with the potential given in Figure \ref{fig:potential}}
\end{figure}

The second depth $d^{(2)}$ is given by the smallest energy barrier
between the sets in $\mathscr{V}^{(2)}$. In the example, it is illustrated
by the blue arrows. The $\mathscr{V}^{(2)}$-valued Markov chain ${\bf y}^{(2)}(\cdot)$
can be informally described as follows. It may jump from a set $\mathcal{M}\in\mathscr{V}^{(2)}$
to a set $\mathcal{M}'\in\mathscr{V}^{(2)}$ if $\mathcal{M}'$ is
a neighbour of $\mathcal{M}$ (that is, if there are no sets $\mathcal{M}''\in\mathscr{V}^{(2)}$
between $\mathcal{M}$ and $\mathcal{M}'$) and if the potential barrier
between $\mathcal{M}$ and $\mathcal{M}'$ is $d^{(2)}$. In particular,
the Markov chain ${\bf y}^{(2)}(\cdot)$ has two recurrent classes
\[
\{\boldsymbol{m}_{1},\,\boldsymbol{m}_{2},\,\boldsymbol{m}_{3}\}\;\;\text{and}\;\;\left\{ \{\boldsymbol{m}_{8}\},\,\{\boldsymbol{m}_{9},\,\boldsymbol{m}_{10}\}\right\} \;,
\]
and one transient state $\{\boldsymbol{m}_{5},\,\boldsymbol{m}_{6}\}$.

By definition, the set $\mathscr{V}^{(3)}$ is a pair formed by the
two recurrent classes. Each element of $\mathscr{V}^{(3)}$ is obtained
by taking the union of the sets in the recurrent class. On the other
hand, $\mathscr{N}^{(3)}$ is the union of the set $\mathscr{N}^{(2)}$
with the transient states:
\begin{equation}
\begin{gathered}\mathscr{V}^{(3)}\ =\ \left\{ \,\{\boldsymbol{m}_{1},\,\boldsymbol{m}_{2},\,\boldsymbol{m}_{3}\},\,\{\boldsymbol{m}_{8},\,\boldsymbol{m}_{9},\,\boldsymbol{m}_{10}\}\,\right\} \;,\\
\mathscr{N}^{(3)}\ =\ \mathscr{N}^{(2)}\cup\left\{ \,\{\boldsymbol{m}_{5},\,\boldsymbol{m}_{6}\}\,\right\} \ =\ \left\{ \,\{\boldsymbol{m}_{5},\,\boldsymbol{m}_{6}\},\,\{\boldsymbol{m}_{4}\},\,\{\boldsymbol{m}_{7}\}\,\right\} \;.
\end{gathered}
\label{eq:v3}
\end{equation}
This construction adds a new generation to the tree being constructed
in Figure \ref{fig:example_tree}.

We may proceed and define a new depth $d^{(3)}$, illustrated by a
green arrow in Figure \ref{fig:potential}, and a $\mathscr{V}^{(3)}$-valued
Markov chain ${\bf y}^{(3)}(\cdot)$. This process has a unique recurrent
class formed by the two elements of $\mathscr{V}^{(3)}$. Hence, $\mathscr{V}^{(4)}$
is a singleton, $\mathfrak{n}_{3}=|\mathscr{V}^{(4)}|=1$, and the
construction ends so that $\mathfrak{q}=3$. Note that the unique
element of $\mathscr{V}^{(4)}$ is the set of all the global minima
of $U$. Since there is no transient state, $\mathscr{N}^{(4)}=\mathscr{N}^{(3)}$
as indicated in Figure \ref{fig:example_tree}.

\subsection{\label{sec2.3}Multi-scale structure of parabolic equations}

The statement of the main result of this article requires some more
notation. Denote by \textcolor{blue}{$\widehat{\mathfrak{L}}^{(p)}$},
\textcolor{blue}{$\mathfrak{L}^{(p)}$}, $p\in\llbracket1,\,\mathfrak{q}\rrbracket$,
the infinitesimal generators of the Markov chains $\mathbf{\widehat{y}}^{(p)}(\cdot)$,
${\bf y}^{(p)}(\cdot)$, respectively. Let ${\color{blue}\mathcal{\widehat{Q}}_{\mathcal{M}}^{(p)}}$,
${\color{blue}\mathcal{Q}_{\mathcal{M}'}^{(p)}}$, $\mathcal{M}\in\mathscr{S}^{(p)}$,
$\mathcal{M}'\in\mathscr{V}^{(p)}$, be the law of the process $\mathbf{\widehat{y}}^{(p)}(\cdot)$,
$\mathbf{y}^{(p)}(\cdot)$ starting at $\mathcal{M}$, $\mathcal{M}'$,
respectively.

Recall that $\mathcal{M}_{0}$ denotes the set of local minima of
$U$.
\begin{enumerate}
\item For $\boldsymbol{m}\in\mathcal{M}_{0}$ and $\mathcal{M}\subset\mathcal{M}_{0}$,
let the weights $\nu(\boldsymbol{m})$ and $\nu(\mathcal{M})$ be
given by
\begin{equation}
{\color{blue}\nu(\boldsymbol{m})}\ :=\ \frac{1}{\sqrt{\det(\nabla^{2}U)(\boldsymbol{m})}}\;,\ \ {\color{blue}\nu(\mathcal{M})}\ :=\ \sum_{\boldsymbol{m}\in\mathcal{M}}\nu(\boldsymbol{m})\;.\label{eq:nu}
\end{equation}
\item For $\boldsymbol{m}\in\mathcal{M}_{0}$ and $p\in\llbracket1,\,\mathfrak{q}\rrbracket$,
denote by ${\color{blue}\mathcal{M}(p,\,\boldsymbol{m})}$ the element
in $\mathscr{S}^{(p)}$ which contains $\boldsymbol{m}$. In particular,
by \eqref{eq:v1}, we have $\mathcal{M}(1,\,\boldsymbol{m})=\{\boldsymbol{m}\}$.
\item For $p\in\llbracket1,\,\mathfrak{q}+1\rrbracket$, $\mathcal{M}\in\mathscr{V}^{(p)}$
and $\boldsymbol{m}\in\mathcal{M}_{0}$, let $\mathfrak{a}^{(p-1)}(\boldsymbol{m},\,\mathcal{M})$
be the probability that the auxiliary Markov chain $\widehat{\mathbf{y}}^{(p)}(\cdot)$
starting from $\mathcal{M}(p,\,\boldsymbol{m})$ firstly reaches an
element of $\mathscr{V}^{(p)}$ exactly at $\mathcal{M}$ (cf. Remark
\ref{rem:a(p-1)}):
\begin{equation}
{\color{blue}\mathfrak{a}^{(p-1)}(\boldsymbol{m},\,\mathcal{M})}\ :=\ \widehat{\mathcal{Q}}_{\mathcal{M}(p,\,\boldsymbol{m})}^{(p)}\left[\,\tau_{\mathscr{V}^{(p)}}=\tau_{\mathcal{M}}\,\right]\;,\label{noteq1}
\end{equation}
where $\tau_{\mathscr{A}}$, $\mathscr{A}\subset\mathscr{S}^{(p)}$
represents the hitting time of the set $\mathscr{A}$:
\[
{\color{blue}\tau_{\mathscr{A}}}\ :=\ \inf\left\{ \,t>0\,:\,\widehat{\mathbf{y}}^{(p)}(t)\in\mathscr{A}\,\right\} \;.
\]
Note that there is an abuse of notation since we write $\tau_{\mathcal{M}'}$
for $\tau_{\{\mathcal{M}'\}}$. Clearly, if $\boldsymbol{m}\in\mathcal{M}$
for some $\mathcal{M}\in\mathscr{V}^{(p)}$, we have
\begin{equation}
\mathfrak{a}^{(p-1)}(\boldsymbol{m},\,\mathcal{M}')\ =\ \mathbf{1}\left\{ \,\mathcal{M}'=\mathcal{M}\,\right\} \;.\label{eq:a0-1}
\end{equation}
Thus, in view of \eqref{eq:v1}, we have
\begin{equation}
{\color{blue}\mathfrak{a}^{(0)}(\boldsymbol{m},\,\boldsymbol{m}')}\ =\ \mathbf{1}\left\{ \,\boldsymbol{m}=\boldsymbol{m}'\,\right\} \;.\label{eq:a0}
\end{equation}
\item For $t>0$, denote by $p_{t}^{(p)}(\cdot,\,\cdot)$ the transition
probability of the Markov chain $\mathbf{y}^{(p)}(\cdot)$:
\begin{equation}
{\color{blue}p_{t}^{(p)}(\mathcal{M}',\,\mathcal{M}'')}\ :=\ \mathcal{Q}_{\mathcal{M}'}^{(p)}\left[\,{\bf y}^{(p)}(t)=\mathcal{M}''\,\right]\;,\ \ \mathcal{M}',\,\mathcal{M}''\in\mathscr{V}^{(p)}\;.\label{noteq2}
\end{equation}
\item For $\boldsymbol{m}\in\mathcal{M}_{0}$, denote by ${\color{blue}\mathcal{D}(\boldsymbol{m})}\subset\mathbb{R}^{d}$
the domain of attraction of $\boldsymbol{m}$ for the dynamical system
described by \eqref{eq:x(t)}.
\end{enumerate}
We are now ready to state the main result. Let \textcolor{blue}{$\theta_{\epsilon}^{(0)}\equiv1$}
and \textcolor{blue}{$\theta_{\epsilon}^{(\mathfrak{q}+1)}\equiv\infty$}
for convenience.
\begin{thm}
\label{t00}Fix a bounded and continuous function $u_{0}\colon\mathbb{R}^{d}\to\mathbb{R}$
and denote by $u_{\epsilon}$ the solution of the parabolic equation
\eqref{eq:pde1}. For all $\boldsymbol{m}\in\mathcal{M}_{0}$, the
following hold.
\begin{enumerate}
\item[(a)] For all $p\in\llbracket1,\,\mathfrak{q}\rrbracket$, $\boldsymbol{x}\in\mathcal{D}(\boldsymbol{m})$
and $t>0$,
\begin{equation}
\lim_{\epsilon\to0}\,u_{\epsilon}(\theta_{\epsilon}^{(p)}t,\,\boldsymbol{x})\ =\ \sum_{\mathcal{M}',\,\mathcal{M}''\in\mathscr{V}^{(p)}}\,\mathfrak{a}^{(p-1)}(\boldsymbol{m},\,\mathcal{M}')\,p_{t}^{(p)}(\mathcal{M}',\,\mathcal{M}'')\,\sum_{\boldsymbol{m}''\in\mathcal{M}''}\frac{\nu(\boldsymbol{m}'')}{\nu(\mathcal{M}'')}\,u_{0}(\boldsymbol{m}'')\ .\label{eq:t00-a}
\end{equation}
\item[(b)] For all $p\in\llbracket1,\,\mathfrak{q}\rrbracket$, $\boldsymbol{x}\in\mathcal{D}(\boldsymbol{m})$
and sequence $(\varrho_{\epsilon})_{\epsilon>0}$ such that $\theta_{\epsilon}^{(p-1)}\,\prec\,\varrho_{\epsilon}\,\prec\,\theta_{\epsilon}^{(p)}$,
\begin{equation}
\lim_{\epsilon\to0}\,u_{\epsilon}(\varrho_{\epsilon},\,\boldsymbol{x})\ =\ \sum_{\mathcal{M}'\in\mathscr{V}^{(p)}}\,\mathfrak{a}^{(p-1)}(\boldsymbol{m},\,\mathcal{M}')\,\sum_{\boldsymbol{m}'\in\mathcal{M}'}\frac{\nu(\boldsymbol{m}')}{\nu(\mathcal{M}')}\,u_{0}(\boldsymbol{m}')\ .\label{eq:t00-b}
\end{equation}
\item[(c)] For all $\bm{x}\in\mathbb{R}^{d}$ and sequence $(\varrho_{\epsilon})_{\epsilon>0}$
such that $\theta_{\epsilon}^{(\mathfrak{q})}\,\prec\,\varrho_{\epsilon}<\,\infty$,
\begin{equation}
\lim_{\epsilon\to0}\,u_{\epsilon}(\varrho_{\epsilon},\,\boldsymbol{x})\ =\ \sum_{\boldsymbol{m}'\in\mathcal{M}_{\star}}\frac{\nu(\boldsymbol{m}')}{\nu(\mathcal{M}_{\star})}\,u_{0}(\boldsymbol{m'})\ .\label{eq:t00-c}
\end{equation}
\end{enumerate}
\end{thm}

The previous theorem provides the full characterization of multi-scale
behavior of the solution $u_{\epsilon}$ (cf. Figure \ref{fig:multiscale}).
Part (b) with $p=1$ and $p=\mathfrak{q}+1$ correspond to the initial
and terminal behavior of the solution, respectively, and the case
$p\in\llbracket2,\,\mathfrak{q}\rrbracket$ corresponds to the intermediate
scale between two critical scales $\theta_{\epsilon}^{(p-1)}$ and
$\theta_{\epsilon}^{(p)}$. The precise asymptotic behavior at the
critical scale is explained by part (a).

The proof of this theorem is based on the metastability result explained
in next section.
\begin{rem}[Initial cases]
In our companion paper \cite{LLS-1st}, the case $p=1$ (for both
(a) and (b)) has been investigated.
\end{rem}

\begin{rem}[Simplification of \eqref{eq:t00-b} for initial and terminal cases]
Equation \eqref{eq:t00-b} for the initial case
$p=1$ can be stated in a more concise form. Indeed, for each sequence
$(\varrho_{\epsilon})_{\epsilon>0}$ such that $1\prec\varrho_{\epsilon}\prec\theta_{\epsilon}^{(1)}$,
by \eqref{eq:a0}, we can simplify \eqref{eq:t00-b} into
\begin{align*}
\lim_{\epsilon\to0}\,u_{\epsilon}(\varrho_{\epsilon},\,\boldsymbol{x})\  & =\ u_{0}(\boldsymbol{m})\ .
\end{align*}
This is the contents of \cite[Theorem 2.1]{LLS-1st}. We also emphasize
that the terminal case $p=\mathfrak{q}+1$ is proven for all $\boldsymbol{x}\in\mathbb{R}^{d}$,
instead for the points in a domain of attraction of a stable equilibrium.
\end{rem}

\begin{rem}[The case $|\mathcal{M}_{0}|=1$]
\label{rem:M0=00003D00003D1}Suppose that $|\mathcal{M}_{0}|=1$
so that we can write $\mathcal{M}_{0}=\{\boldsymbol{m}\}$. Then,
by the growth condition \eqref{eq:growth}, $\boldsymbol{m}$ is the
unique global minimum of $U$, and our proof yields that
\begin{equation}
\lim_{\epsilon\to0}\,u_{\epsilon}(\varrho_{\epsilon},\,\boldsymbol{x})\ =\ u_{0}(\boldsymbol{m})\label{eq:case1}
\end{equation}
for all $\boldsymbol{x}\in\mathcal{D}(\boldsymbol{m})$ and time scale
$\varrho_{\epsilon}$ such that $\varrho_{\epsilon}\rightarrow\infty$
as $\epsilon\rightarrow0$. Thus, \eqref{eq:case1} holds for all
$\boldsymbol{x}\in\mathbb{R}^{d}$ if $\boldsymbol{m}$ is the unique
critical point of $U$. Otherwise, by \cite[Theorem 6.1]{LLS-1st},
we can show that \eqref{eq:case1} holds for all time scales $\varrho_{\epsilon}\succ\frac{1}{\epsilon}$.
We expect, however, that \eqref{eq:case1} holds for time scales $\varrho_{\epsilon}\succ\log\frac{1}{\epsilon}$.
\end{rem}

\begin{rem}[Dependency of results with $U$ and $\boldsymbol{\ell}$]
The jump rates of the processes $\mathbf{y}^{(p)}(\cdot)$ determine
the terms $\mathfrak{a}^{(p-1)}(\cdot,\cdot)$, $p_{t}^{(p)}(\cdot,\cdot)$
appearing in the statement of Theorem \ref{t00}. The weight $\nu(\cdot)$,
defined in \eqref{eq:nu}, only depends on the behavior of $U$ in
small neighbourhoods of local minima, and does not depend on the field
$\boldsymbol{\ell}$. It accounts for the time spent at a neighbourhood
of local minima between jumps. The rates $\omega(\cdot,\cdot)$, introduced
below in \eqref{eq:rate_30}, only depends on the behavior of $U$
and $\boldsymbol{\ell}$ in small neighbourhoods of saddle points.
It measures the difficulty to overpass the saddle point. Hence, the
asymptotic behavior of the solution $u_{\epsilon}$ described in Theorem
\ref{t00} only depends on the behavior of $U$ around local minima
and the one of $U$ and $\boldsymbol{\ell}$ around saddle points.
\end{rem}

\begin{rem}[\textcolor{black}{Analysis on bounded domains}]
\label{rem:bdry} Our approach covers the case of bounded domains
with smooth boundaries. Denote by $\Omega$ such a domain, and let
$\partial\Omega$ be its boundary. Assume that $U$ is a Morse function
which satisfies \eqref{hyp2}, and that there are finitely many critical
points of $U$ in $\Omega$. The hypotheses \eqref{eq:growth} on
the asymptotic behavior of the potential are replaced by the condition
that $\bm{n}(\bm{x})\cdot(\nabla U)(\bm{x})>0$ for all $\bm{x}\in\partial\Omega$,
where $\bm{n}$ represents the exterior normal vector to the boundary.

Let
\[
H\ :=\ \min_{\bm{x}\in\partial\Omega}U(\bm{x})\,-\,\min_{\bm{y}\in\Omega}U(\bm{y})\;.
\]
Denote by $\mathcal{M}_{0}$ the set of all local minima of $U$ in
$\Omega$, and assume that $|\mathcal{M}_{0}|\ge2$.

The method presented in this article applies to the parabolic equation
\eqref{eq:pde1} on $\Omega$ with Neumann boundary condition
\[
\nabla u_{\epsilon}(t,\,\boldsymbol{x})\cdot\boldsymbol{n}(\boldsymbol{x})\ \equiv\ 0\;\;\;\text{for }(t,\,\boldsymbol{x})\in(0,\,\infty)\times\partial\Omega\;,
\]
or Dirichlet boundary condition
\[
u_{\epsilon}(t,\,\boldsymbol{x})\ =\ 0\;\;\;\text{for }(t,\,\boldsymbol{x})\in(0,\,\infty)\times\partial\Omega\;.
\]
In both cases, it applies to time scales $\theta_{\epsilon}^{(p)}$
associated to depths $d^{(p)}<H$, and to time-scales $\varrho_{\epsilon}\prec e^{H/\epsilon}$.

For the Neumann, Dirichlet boundary conditions, the diffusion process
$\boldsymbol{x}_{\epsilon}(\cdot)$ introduced in \eqref{eq:probexp}
is reflected, absorbed at the boundary $\partial\Omega$, respectively.
\end{rem}

\begin{rem}[Explanation on $\mathfrak{a}^{(p-1)}(\boldsymbol{m},\,\mathcal{M})$]
\label{rem:a(p-1)} It might be unnatural at first
glance that we use index $p-1$ instead of $p$ at the definition
\eqref{noteq1}. To see why it is a natural selection of the index, let
$p\in\llbracket2,\,\mathfrak{q}\rrbracket$ and let
$\boldsymbol{m}\in\mathcal{M}_{0}$ such that
$\mathcal{M}(p,\,\bm{m})\in\mathscr{N}^{(p)}$ so that the probability
at the right-hand side of \eqref{noteq1} is non-trivial.  Then, from
Lemma \ref{l: tau_V}, it turns out that
\[
\mathfrak{a}^{(p-1)}(\boldsymbol{m},\,\mathcal{M})\ =\ \widehat{\mathcal{Q}}_{\mathcal{M}(p-1,\,\boldsymbol{m})}^{(p-1)}\left[\,\tau_{\mathscr{R}^{(p-1)}}=\tau_{\mathscr{V}^{(p-1)}(\mathcal{M})}\,\right]\;.
\]
Namely, $\mathfrak{a}^{(p-1)}(\boldsymbol{m},\,\mathcal{M})$ denotes
the probability of an event that happened already at the time scale
$\theta_{\epsilon}^{(p-1)}$. This is the reason that we use the index
$p-1$ here.

Moreover, we can even express $\mathfrak{a}^{(p-1)}(\boldsymbol{m},\,\mathcal{M})$
in terms of the Markov chains $\mathbf{y}^{(1)}(\cdot),\,\dots,\,\mathbf{y}^{(p-1)}(\cdot)$
(without using the auxiliary process $\mathbf{\widehat{y}}^{(p-1)}(\cdot)$).
To see this, for $p\in\llbracket1,\,\mathfrak{q}\rrbracket$, $\mathcal{M}\in\mathscr{V}^{(p-1)}$,
and $\mathcal{M}'\in\mathscr{V}^{(p)}$, let $\mathfrak{A}^{(p-1)}(\mathcal{M},\,\mathcal{M}')$
be the probability that the Markov chain ${\bf y}^{(p-1)}(\cdot)$
starting from $\mathcal{M}$ is absorbed at the closed irreducible
class of $\mathbf{y}^{(p-1)}(\cdot)$ which forms the set $\mathcal{M}'$:
\[
\mathfrak{A}^{(p-1)}(\mathcal{M},\,\mathcal{M}')\ =\ \mathcal{Q}_{\mathcal{M}}^{(p-1)}\left[\,\tau_{\mathscr{V}^{(p)}}=\tau_{\mathcal{M}'}\,\right]\;.
\]
Decompose the probability on the right-hand side of \eqref{noteq1}
according to the events $\{\tau_{\mathscr{V}^{(p-1)}}=\tau_{\mathcal{M}''}\}$,
$\mathcal{M}''\in\mathscr{V}^{(p-1)}$. This decomposition, the strong
Markov property and the ideas presented in the previous paragraph
yield that
\[
\mathfrak{a}^{(p-1)}(\boldsymbol{m},\,\mathcal{M}')\ =\ \sum_{\mathcal{M}''\in\mathscr{V}^{(p-1)}}\,\mathfrak{a}^{(p-2)}(\boldsymbol{m},\,\mathcal{M}'')\,\mathfrak{A}^{(p-1)}(\mathcal{M}'',\,\mathcal{M}')
\]
for all $\boldsymbol{m}\in\mathcal{M}_{0}$,
$\mathcal{M}'\in\mathscr{V}^{(p)}$, and
$p\in\llbracket2,\,\mathfrak{q}\rrbracket$. This identity provides a
cascade interpretation of the weights
$\mathfrak{a}^{(p-1)}(\boldsymbol{m},\,\mathcal{M})$.  It represents
the probability that starting from $\boldsymbol{m}$ the chain is
absorbed at a closed irreducible class of the process
$\mathbf{y}^{(1)}(\cdot)$, which at the second layer is absorbed at a
closed irreducible class of $\mathbf{y}^{(2)}(\cdot)$, and so on until
at layer $p-1$ it is absorbed at the closed irreducible class
$\mathbf{y}^{(p-1)}(\cdot)$ which forms the set
$\mathcal{M}\in\mathscr{V}^{(p)}$.  With this cascade expression, we
can represent $\mathfrak{a}^{(p-1)}(\boldsymbol{m},\,\mathcal{M})$
only in terms of Markov chains
$\mathbf{y}^{(1)}(\cdot),\,\dots,\,\mathbf{y}^{(p-1)}(\cdot)$.
\end{rem}

\begin{rem}[Right-hand sides of \eqref{eq:t00-a} and \eqref{eq:t00-b}]
The expression on the right-hand side of \eqref{eq:t00-a} has a
very simple interpretation. The previous remark explains the term
$\mathfrak{a}^{(p-1)}(\boldsymbol{m},\,\mathcal{M}')$. Note that
it refers to an event which occurs in a time-scale much shorter than
the one at which jumps between sets in $\mathscr{V}^{(p)}$ are observed.
The term $p_{t}^{(p)}(\mathcal{M}',\,\mathcal{M}'')$ corresponds
to the probability that the reduced Markov chain ${\bf y}^{(p)}(\cdot)$
is at $\mathcal{M}''$ at time $t$ if it started from $\mathcal{M}'$
and hence explains what happens at the scale $\theta_{\epsilon}^{(p)}$.

Finally, the term $\nu(\boldsymbol{m})/\nu(\mathcal{M}'')$ corresponds
to the proportion of time the diffusion $\boldsymbol{x}_{\epsilon}(\cdot)$
reflected at the boundary of $\mathcal{M}''$ remains close to $\boldsymbol{m}$.
More precisely, since $\mathcal{M}(1,\,\boldsymbol{m})=\{\boldsymbol{m}\}$
and $\mathcal{M}(p,\,\boldsymbol{m})=\mathcal{M}''$ we can write
\[
\frac{\nu(\boldsymbol{m})}{\nu(\mathcal{M}'')}\ =\ \prod_{q=1}^{p-1}\frac{\nu(\mathcal{M}(q,\,\boldsymbol{m}))}{\nu(\mathcal{M}(q+1,\,\boldsymbol{m}))}\;\cdot
\]
The term $\nu(\cdot)/\nu(\mathcal{M}(q,\boldsymbol{m}))$ represents
the stationary state of the reduced Markov chain ${\bf y}^{(q-1)}(\cdot)$
restricted to its irreducible class $\{\mathcal{M}\in\mathscr{V}^{(q-1)}:\mathcal{M}\subset\mathcal{M}(q,\,\boldsymbol{m})\}$.
We can provide a similar explanation on the right-hand side of \eqref{eq:t00-b}.
\end{rem}

\subsection{\label{sec2.4}Leveling property of resolvent equations }

The proof of Theorem \ref{t00} is based on the leveling properties
of the resolvent equation presented in this subsection. To explain
this result, we first define the valleys around each minimum $\boldsymbol{m}\in\mathcal{M}_{0}$.

Let
\begin{equation}
{\color{blue}\{\,U<H\,\}}\ :=\ \left\{ \,\boldsymbol{x}\in\mathbb{R}^{d}\,:\,U(\boldsymbol{x})<H\,\right\} \;\;\text{and}\ \ {\color{blue}\{\,U\le H\,\}}\ :=\ \left\{ \,\boldsymbol{x}\in\mathbb{R}^{d}\,:\,U(\boldsymbol{x})\le H\,\right\} \;.\label{eq:level}
\end{equation}
For each $\boldsymbol{m}\in\mathcal{M}_{0}$, denote by ${\color{blue}\mathcal{W}^{r}(\boldsymbol{m})}$
the connected component of the set $\{U\le U(\boldsymbol{m})+r\}$
containing $\boldsymbol{m}$.

For each $\boldsymbol{m}\in\mathcal{M}_{0}$, take $r_{0}=r_{0}(\boldsymbol{m})>0$
small so that $\mathcal{W}^{3r_{0}}(\boldsymbol{m})\subset\mathcal{D}(\boldsymbol{m})$
and satisfies \cite[condition (a)-(e) at the paragraph before (2.12)]{LLS-1st}
related with the control of the error term of the second order Taylor
expansion of $U$ around $\boldsymbol{m}$. In particular, $\boldsymbol{m}$
is the unique critical point of $U$ in $\mathcal{W}^{3r_{0}}(\boldsymbol{m})$.
Then, let us define the\emph{ valley }around $\boldsymbol{m}$ as
\begin{equation}
{\color{blue}\mathcal{E}(\boldsymbol{m}})\ :=\ \mathcal{W}^{r_{0}}(\boldsymbol{m})\;.\label{e_Em}
\end{equation}
For $\mathcal{M}\subset\mathcal{M}_{0}$, write $\mathcal{E}(\mathcal{M})$
for the union of the valleys around local minima of $\mathcal{M}$:
\begin{equation}
{\color{blue}\mathcal{E}(\mathcal{M})}\ :=\ \bigcup_{\boldsymbol{m}\in\mathcal{M}}\mathcal{E}(\boldsymbol{m})\;.\label{e_E_M}
\end{equation}

Fix a constant $\lambda>0$ and $p\in\llbracket1,\,\mathfrak{q}\rrbracket$.
For $\boldsymbol{g}\colon\mathscr{V}^{(p)}\rightarrow\mathbb{R}$,
denote by \textcolor{blue}{ $\phi_{\epsilon}=\phi_{\epsilon}^{p,\boldsymbol{g},\lambda}$
}the unique solution of the resolvent equation
\begin{equation}
\left(\lambda-\theta_{\epsilon}^{(p)}\mathscr{L}_{\epsilon}\right)\,\phi_{\epsilon}\ =\ \sum_{\mathcal{M}\in\mathscr{V}^{(p)}}\boldsymbol{g}(\mathcal{M})\,\chi_{_{\mathcal{E}(\mathcal{M})}}\;,\label{e_res}
\end{equation}
where \textcolor{blue}{ $\chi_{\mathcal{A}}$}, $\mathcal{A}\subset\mathbb{R}^{d}$,
denotes the indicator function on $\mathcal{A}$. The function on
the right-hand side vanishes at $[\cup_{\mathcal{M}\in\mathscr{V}^{(p)}}\mathcal{E}(\mathcal{M})]^{c}$
and is constant on each well $\mathcal{E}(\mathcal{M})$, $\mathcal{M}\in\mathscr{V}^{(p)}$,
where the constant value is $\boldsymbol{g}(\mathcal{M})$.

Recall from Section \ref{sec2.3} that $\mathfrak{L}^{(p)}$
denotes the generator associated with the Markov chain ${\bf y}^{(p)}(\cdot)$.
The second main result of this article reads as follows.
\begin{thm}
\label{t_res}Fix a constant $\lambda>0$, $p\in\llbracket1,\,\mathfrak{q}+1\rrbracket$
and $\boldsymbol{g}\colon\mathscr{V}^{(p)}\rightarrow\mathbb{R}$.
Then, for all $\boldsymbol{m}\in\mathcal{M}_{0}$ and compact set
$\mathcal{K}\subset\mathcal{D}(\boldsymbol{m})$, the solution $\phi_{\epsilon}$
to the resolvent equation \eqref{e_res} satisfies
\[
\lim_{\epsilon\rightarrow0}\,\sup_{\boldsymbol{x}\in\mathcal{K}}\,\bigg|\,\phi_{\epsilon}(\boldsymbol{x})\,-\sum_{\mathcal{M}'\in\mathscr{V}^{(p)}}\mathfrak{a}^{(p-1)}(\boldsymbol{m},\,\mathcal{M}')\,\boldsymbol{f}(\mathcal{M}')\,\bigg|\ =\ 0\;,
\]
where $\boldsymbol{f}:\mathscr{V}^{(p)}\rightarrow\mathbb{R}$ denotes
the unique solution of the reduced resolvent equation
\begin{equation}
\left(\lambda-\mathfrak{L}^{(p)}\right)\,\boldsymbol{f}\ =\ \boldsymbol{g}\;.\label{eq:res_y}
\end{equation}
In particular, if $\mathcal{M}(p,\,\bm{m})\in\mathscr{V}^{(p)}$,
by \eqref{eq:a0-1}, for all compact set $\mathcal{K}\subset\mathcal{D}(\boldsymbol{m})$,
we have
\[
\lim_{\epsilon\rightarrow0}\,\sup_{\boldsymbol{x}\in\mathcal{K}}\,\left|\,\phi_{\epsilon}(\boldsymbol{x})-\boldsymbol{f}(\mathcal{M})\,\right|\ =\ 0\;.
\]
\end{thm}

The previous theorem asserts that the solution to the resolvent equation
\eqref{e_res} is flat on each domain of attraction of each minimum
$\boldsymbol{m}\in\mathcal{M}_{0}$ at each time-scale $\theta_{\epsilon}^{(p)}$,
$p\in\llbracket1,\,\mathfrak{q}\rrbracket$. The value of the solution
depends on the time-scale. It is remarkable that this flat value is
determined by the solution to the resolvent equation \eqref{eq:res_y}
associated with the generator $\mathfrak{L}^{(p)}$. Such a phenomenon
is called leveling property of the equation (cf. \cite{Day82,DF78}).
\begin{rem}
We believe that Theorem \ref{t_res} holds when we define the metastable
valley $\mathcal{E}(\boldsymbol{m})$ as a large compact subset of
$\mathcal{D}(\boldsymbol{m})$ containing a neighborhood of $\boldsymbol{m}$.
We use the specific form of the valley $\mathcal{E}(\boldsymbol{m})$
mainly because the mixing properties of the diffusion process $\boldsymbol{x}_{\epsilon}(\cdot)$
were analyzed in \cite[Section 3]{LLS-1st} only for this set.
\end{rem}

\subsection{\label{sec25}Probablistic representations of solutions to PDEs }

The main theorems' proofs are based on the stochastic representation
of the resolvent equation solution \eqref{e_res} and the parabolic
equation solution \eqref{eq:pde1}.

Denote by ${\color{blue}\mathbb{P}_{\bm{x}}^{\epsilon}}$ the law
of $\bm{x}_{\epsilon}(\cdot)$ starting from $\bm{x}\in\mathbb{R}^{d}$
and by ${\color{blue}\mathbb{E}_{\bm{x}}^{\epsilon}}$ the corresponding
expectation. On the one hand,
\[
\phi_{\epsilon}(\bm{x})\ =\ \mathbb{E}_{\bm{x}}^{\epsilon}\left[\,\int_{0}^{\infty}\,e^{-\lambda s}\,G(\bm{x}_{\epsilon}(s\,\theta_{\epsilon}^{(p)}))\,ds\,\right]\;,
\]
where $G=\sum_{\mathcal{M}\in\mathscr{V}^{(p)}}\boldsymbol{g}(\mathcal{M})\,\chi_{_{\mathcal{E}(\mathcal{M})}}$.
We refer to \cite[Chapter 6, Section 5]{Fri} for a proof of this
result in the case of bounded domains with $G$ H\"older continuous.
The result can be extended to a sum of indicators, as in \eqref{e_res},
by approximating this function from below and from above by sequences
of H\"older continuous functions, by recalling the maximum principle
and by using the dominated convergence theorem. It can be extended
to $\mathbb{R}^{d}$ because the potential $U$ is confining in view
of the hypotheses \eqref{eq:growth}.

Similarly, the solution of \eqref{eq:pde1} can be expressed as
\[
u_{\epsilon}(t,\,\bm{x})\ =\ \mathbb{E}_{\bm{x}}^{\epsilon}\left[\,u_{0}(\bm{x}_{\epsilon}(t))\,\right]\;.
\]
We refer to \cite[Theorem 6.5.3]{Fri} for a proof of this result
in the case of a bounded drift. As for the resolvent equation, this
result can be extended to the present context because the potential
$U$ is confining. It is indeed enough to modify the drift outside
a ball of radius $R$ centered at the origin to make it bounded, and
to show that the time needed to hit the boundary of the ball increases
to $+\infty$ as the radius of the ball increases to $+\infty$ because
the potential is confining.

\section{\label{sec3}Main results on metastability of diffusion processes }

In this section, we explain our third main result which is on the
full description of the metastability of the diffusion process $\boldsymbol{x}_{\epsilon}(\cdot)$
defined in \eqref{sde}.

\subsection{Full description of metastable behavior}

The following theorem is the main result regarding the metastable
behavior of the process $\bm{x}_{\epsilon}(\cdot)$.
\begin{thm}
\label{t_meta}Let $p\in\llbracket1,\,\mathfrak{q}\rrbracket$, $\boldsymbol{m}\in\mathcal{M}_{0}$,
and let $\boldsymbol{x}\in\mathcal{D}(\boldsymbol{m})$. Then, for
all $n\ge1$, $0<t_{1}<\cdots<t_{n}$, $\mathcal{M}_{k_{1}},\,\dots,\,\mathcal{M}_{k_{n}}\in\mathscr{V}^{(p)}$,
and sequences $(t_{j,\,\epsilon})_{\epsilon>0}$, $j\in\llbracket1,\,n\rrbracket$,
such that $t_{j,\,\epsilon}\to t_{j}$ as $\epsilon\rightarrow0$,
we have
\begin{equation}
\lim_{\epsilon\rightarrow0}\,\mathbb{P}_{\boldsymbol{x}}^{\epsilon}\bigg[\,\bigcap_{j=1}^{n}\big\{\,\bm{x}_{\epsilon}(\theta_{\epsilon}^{(p)}t_{j,\,\epsilon})\in\mathcal{E}(\mathcal{M}_{k_{j}})\,\big\}\,\bigg]\ =\ \mathcal{Q}_{\mathfrak{a}^{(p-1)}(\boldsymbol{m},\,\cdot)}^{(p)}\Big[\,\bigcap_{j=1}^{n}\big\{\,{\bf y}^{(p)}(t_{j})=\mathcal{M}_{k_{j}}\,\big\}\,\Big]\;,\label{eq:fdd}
\end{equation}
where\textcolor{blue}{ $\mathcal{Q}_{\mathfrak{a}^{(p-1)}(\boldsymbol{m},\,\cdot)}^{(p)}$}
denotes the law of the process $\mathbf{y}^{(p)}(\cdot)$ starting
from the probability distribution $\mathfrak{a}^{(p-1)}(\boldsymbol{m},\,\cdot)$
on $\mathscr{V}^{(p)}$ defined in \eqref{noteq1}.
\end{thm}

We first mention that, when $\boldsymbol{m}\in\mathcal{M}$ for some
$\mathcal{M}\in\mathscr{V}^{(p)}$, the estimate \eqref{eq:fdd} can
be simplified into, for all $\boldsymbol{x}\in\mathcal{D}(\boldsymbol{m})$,
\[
\lim_{\epsilon\rightarrow0}\,\mathbb{P}_{\boldsymbol{x}}^{\epsilon}\bigg[\,\bigcap_{j=1}^{n}\big\{\,\bm{x}_{\epsilon}(\theta_{\epsilon}^{(p)}t_{j,\,\epsilon})\in\mathcal{E}(\mathcal{M}_{k_{j}})\,\big\}\,\bigg]\ =\ \mathcal{Q}_{\mathcal{M}}^{(p)}\bigg[\,\bigcap_{j=1}^{n}\big\{\,{\bf y}^{(p)}(t_{j})=\mathcal{M}_{k_{j}}\,\big\}\,\bigg]\;.
\]
This estimates demonstrates that the Markov chain $\mathbf{y}^{(p)}(\cdot)$
is the one describing the metastable behavior of the process $\bm{x}_{\epsilon}(\cdot)$
speeded-up by the time-scale $\theta_{\epsilon}^{(p)}$. The more
general estimate \eqref{eq:fdd} implies that, when the process starts
from a neighborhood of $\boldsymbol{m}\in\mathcal{M}$ for some $\mathcal{M}\in\mathscr{N}^{(p)}$,
the process first arrives instanteneously at a valley $\mathcal{M}'\in\mathscr{V}^{(p)}$
according to the probability distribution $\mathfrak{a}^{(p-1)}(\boldsymbol{m},\,\cdot)$,
and then behaves as the Markov chain $\mathbf{y}^{(p)}(\cdot)$.

We explain in Section \ref{sec12} the proof of Theorem \ref{t00}
based on the metastability result stated in Theorem \ref{t_meta}.

\subsection{General framework}

In this section, we introduce important notions to investigate the
metastable behavior of the process $\boldsymbol{x}_{\epsilon}(\cdot)$
at a certain time-scale $1\prec{\color{blue}\theta_{\epsilon}}\prec\infty$.
We remark that the result summarized in this section is not restricted
to the diffusion processes considered in the current article, but
can be applied to any Markov process exhibiting metastable-type behavior.

To explain this general framework, let ${\color{blue}\mathscr{V}}$
be a finite set, and ${\color{blue}\mathscr{E}=\{\mathcal{E}_{i}:i\in\mathscr{V}\}}$
be a collection of subsets of $\mathbb{R}^{d}$, each one representing
a metastable or stable set. We assume that $\mathcal{E}_{i}\cap\mathcal{E}_{j}=\varnothing$
for $i\ne j$, and let ${\color{blue}\mathcal{E}=\bigcup_{i\in\mathscr{V}}\mathcal{E}_{i}}$.

Denote by $\boldsymbol{y}_{\epsilon}(\cdot)$ the process $\boldsymbol{x}_{\epsilon}(\cdot)$
speeded-up by $\theta_{\epsilon}$: ${\color{blue}\boldsymbol{y}_{\epsilon}(t)=\boldsymbol{x}_{\epsilon}(\theta_{\epsilon}t)}$.
We expect that the speeded-up process $\boldsymbol{y}_{\epsilon}(\cdot)$
jumps among the sets in $\mathscr{E}$. We wish to describe this hopping
dynamics among the sets in $\mathscr{E}$ by a certain $\mathscr{V}$-valued
continuous-time Markov chain ${\color{blue}\{{\bf y}(t)\}_{t\ge0}}$.
As in Theorem \ref{t_meta}, we can formulate this as a convergence
of finite-dimensional distributions.
\begin{defn}[Finite-dimensional distributions]
\label{def_FC} We say that the condition $\mathfrak{C}_{{\rm fdd}}(\theta_{\epsilon},\,\mathscr{E},\,{\bf y})$
holds if for all $i\in\mathscr{V}$, $n\ge1$, $0<t_{1}<\cdots<t_{n}$,
sequences $(t_{j,\,\epsilon})_{\epsilon>0}$, $j\in\llbracket1,\,n\rrbracket$
such that $t_{j,\,\epsilon}\to t_{j}$, and $k_{1},\,\dots,\,k_{n}\in\mathscr{V}$,
we have
\begin{equation}
\lim_{\epsilon\rightarrow0}\,\sup_{\boldsymbol{x}\in\mathcal{E}_{i}}\,\bigg|\,\mathbb{Q}_{\boldsymbol{x}}^{\epsilon}\Big[\,\bigcap_{j=1}^{n}\big\{\,\boldsymbol{y}_{\epsilon}(t_{j,\,\epsilon})\in\mathcal{E}_{k_{j}}\,\big\}\,\Big]\,-\,\mathcal{Q}_{i}\Big[\,\bigcap_{j=1}^{n}\big\{\,{\bf y}(t_{j})=k_{j}\,\big\}\,\Big]\,\bigg|\ =\ 0\label{eq:fddc}
\end{equation}
where ${\color{blue}\mathbb{Q}_{\boldsymbol{x}}^{\epsilon}}$ and
${\color{blue}\mathcal{Q}_{i}}$ denote laws of the process $\boldsymbol{y}_{\epsilon}(\cdot)$
starting at $\boldsymbol{x}$ and of the process ${\bf y}(\cdot)$
starting at $i$, respectively.
\end{defn}

\begin{rem}
\label{rem:fdd}Let $(\gamma_{\epsilon})_{\epsilon>0}$ be the sequence
such that $\gamma_{\epsilon}\prec\theta_{\epsilon}$. Then, since
\eqref{eq:fddc} holds for all sequences $(\boldsymbol{z}_{\epsilon})_{\epsilon>0}$
in $\mathcal{E}_{i}$ and sequences $(t_{j,\,\epsilon})_{\epsilon>0}$,
$j\in\llbracket1,\,n\rrbracket$ such that $t_{j,\,\epsilon}\to t_{j}$,
we immediately have
\[
\lim_{\epsilon\rightarrow0}\,\sup_{\boldsymbol{x}\in\mathcal{E}_{i}}\,\sup_{s_{1},\,\dots,\,s_{n}\in[-\gamma_{\epsilon},\,\gamma_{\epsilon}]}\,\bigg|\,\mathbb{Q}_{\boldsymbol{x}}^{\epsilon}\Big[\,\bigcap_{j=1}^{n}\big\{\,\boldsymbol{y}_{\epsilon}(t_{j}+s_{j}/\theta_{\epsilon})\in\mathcal{E}_{k_{j}}\,\big\}\,\Big]\,-\,\mathcal{Q}_{i}\Big[\,\bigcap_{j=1}^{n}\big\{\,{\bf y}(t_{j})=k_{j}\,\big\}\,\Big]\,\bigg|\ =\ 0
\]
under the same conditions.
\end{rem}

\begin{rem}
\label{rem:fddneg}By letting $n=1$ in \eqref{eq:fddc} and summing
over $k_{1}\in\mathscr{V}$, we get
\[
\lim_{\epsilon\rightarrow0}\,\mathbb{Q}_{\boldsymbol{z}_{\epsilon}}^{\epsilon}\left[\,\boldsymbol{y}_{\epsilon}(t_{\epsilon})\notin\mathcal{E}\,\right]\ =\ 0
\]
for all sequence $(t_{\epsilon})_{\epsilon>0}$ such that $t_{\epsilon}\rightarrow t>0$.
\end{rem}

The verification of the condition $\mathfrak{C}_{{\rm fdd}}(\theta_{\epsilon},\,\mathscr{E},\,{\bf y})$
is usually based on the alternative formulation of the convergence
developed in \cite{BL1}. To explain this, we first define ${\color{blue}\{\boldsymbol{y}_{\epsilon}^{\mathcal{E}}(t)\}_{t\ge0}}$
as the trace process  (cf. Appendix \ref{app:trace})
of $\boldsymbol{y}_{\epsilon}(\cdot)$ on the set
$\mathcal{E}$. Note that the process $\boldsymbol{x}_{\epsilon}(\cdot)$,
and hence $\boldsymbol{y}_{\epsilon}(\cdot)$, is positive recurrent
by \cite[Theorem 2.2]{LS-22} so that we can define the trace process
provided that $\mathcal{E}$ contains an open ball.

this is the process obtained from $\boldsymbol{y}_{\epsilon}(\cdot)$
by neglecting the excursions outside of the sets in $\mathscr{E}$.
We then define the $\mathscr{V}$ process $\{\mathbf{y}_{\epsilon}(t)\}_{t\ge0}$
by
\[
{\color{blue}\mathbf{y}_{\epsilon}(t)}\ =\ i\;\;\iff\;\;\boldsymbol{y}_{\epsilon}^{\mathcal{E}}(t)\in\mathcal{E}_{i}\;\;\;;\;t\ge0\;,
\]
so that the process $\mathbf{y}_{\epsilon}(\cdot)$ indicates the
index of the set in $\mathscr{E}$ visited by the process $\boldsymbol{y}_{\epsilon}^{\mathcal{E}}(\cdot).$
We call $\mathbf{y}_{\epsilon}(\cdot)$ as the \textit{order process.
}The order process captures the hopping dynamics of the speeded-up
process $\boldsymbol{y}_{\epsilon}(\cdot)$ among the sets in $\mathscr{E}$.
Hence, if the order process $\mathbf{y}_{\epsilon}(\cdot)$ converges
to the Markov chain $\mathbf{y}(\cdot)$, this latter process describes
the metastable behavior of the process $\boldsymbol{x}_{\epsilon}(\cdot)$
at the scale $\theta_{\epsilon}$. This can be formalized in the following
manner.
\begin{defn}[Convergence in law]
\label{def_C} We say that the condition $\mathfrak{C}(\theta_{\epsilon},\,\mathscr{E},\,{\bf y})$
($\mathfrak{C}$ for convergence) holds if, for all $i\in\mathscr{V}$,
and for all sequences $(\boldsymbol{z}_{\epsilon})_{\epsilon>0}$
in $\mathcal{E}_{i}$, the process $\mathbf{y}_{\epsilon}(\cdot)$
with $\boldsymbol{y}_{\epsilon}(0)=\boldsymbol{z}_{\epsilon}$ converges
in the Skorohod topology to the $\mathscr{V}$-valued continuous-time
Markov chain ${\bf y}(\cdot)$ starting from $i$, and for all $T>0$,
\begin{equation}
\lim_{\epsilon\rightarrow0}\,\sup_{\boldsymbol{x}\in\mathcal{E}}\,\mathbb{Q}_{\boldsymbol{x}}^{\epsilon}\bigg[\,\int_{0}^{T}\,\chi_{_{\mathbb{R}^{d}\setminus\mathcal{E}}}(\,\boldsymbol{y}_{\epsilon}(t)\,)\,dt\,\bigg]\ =\ 0\;.\label{eq:neg}
\end{equation}
\end{defn}

The condition \eqref{eq:neg} assert that the speeded-up process $\boldsymbol{y}_{\epsilon}(\cdot)$
spends most of the time inside the set in $\mathscr{E}$; this condition
guarantees that $\mathscr{E}$ contains all the relevant metastable
or stable sets. We also not that, by Remark \ref{rem:fddneg}, the
condition $\mathfrak{C}_{{\rm fdd}}(\theta_{\epsilon},\,\mathscr{E},\,{\bf y})$
implies \eqref{eq:neg} (via Fubini's theorem).

To explain the relation between the conditions $\mathfrak{C}(\theta_{\epsilon},\,\mathscr{E},\,{\bf y})$
and $\mathfrak{C}_{{\rm fdd}}(\theta_{\epsilon},\,\mathscr{E},\,{\bf y})$,
we introduce one more condition.
\begin{defn}[Non-escaping condition]
\label{def_NE} We say that the condition $\mathfrak{G}(\theta_{\epsilon},\,\mathscr{E},\,{\bf y})$
($\mathfrak{G}$ for glued) holds if
\[
\limsup_{b\rightarrow0}\,\limsup_{\epsilon\rightarrow0}\,\sup_{i\in\mathscr{V}}\,\sup_{\boldsymbol{x}\in\mathcal{E}_{i}}\,\sup_{t\in[b,\,2b]}\,\mathbb{Q}_{\boldsymbol{x}}^{\epsilon}\left[\,\boldsymbol{y}_{\epsilon}(t)\notin\mathcal{E}_{i}\,\right]\ =\ 0\;.
\]
\end{defn}

Namely, starting from a set $\mathcal{E}_{i}\in\mathscr{E}$, the
process $\boldsymbol{y}_{\epsilon}(\cdot)$ remains at the same set
during a very short time scale. The next result is obtained in \cite[Proposition 10.2]{LLS-1st}
and \cite[Proposition 2.1]{LLM}. The proof is omitted since it is essentially the same.
\begin{prop}
\label{prop_C_F0}Conditions $\mathfrak{C}(\theta_{\epsilon},\,\mathscr{E},\,{\bf y})$
and $\mathfrak{G}(\theta_{\epsilon},\,\mathscr{E},\,{\bf y})$ imply
the condition $\mathfrak{C}_{{\rm fdd}}(\theta_{\epsilon},\,\mathscr{E},\,{\bf y})$.
\end{prop}

Hence, in order to prove $\mathfrak{C}_{{\rm fdd}}(\theta_{\epsilon},\,\mathscr{E},\,{\bf y})$,
we need to establish $\mathfrak{C}(\theta_{\epsilon},\,\mathscr{E},\,{\bf y})$
and $\mathfrak{G}(\theta_{\epsilon},\,\mathscr{E},\,{\bf y})$. Of
course, the verification of the condition $\mathfrak{G}(\theta_{\epsilon},\,\mathscr{E},\,{\bf y})$
is a technical matter and there are many possible ways to do this.
Hence, we need to establish a way to verify the condition $\mathfrak{C}(\theta_{\epsilon},\,\mathscr{E},\,{\bf y})$.The
verification of $\mathfrak{C}(\theta_{\epsilon},\,\mathscr{E},\,{\bf y})$
is based on the resolvent approach developed in \cite{LMS} which
will be explained below.

\subsection{Resolvent approach to metastability }

Let $\lambda>0$ be a constant which will be fixed through the article.
For a function $\mathbf{g}\colon\mathscr{V}\rightarrow\mathbb{R}$,
denote by $\phi_{\epsilon}^{\mathbf{g}}$ the unique solution of the
resolvent equation
\begin{equation}
\left(\lambda-\theta_{\epsilon}\mathscr{L}_{\epsilon}\right)\,\phi_{\epsilon}^{\mathbf{g}}\ =\ \sum_{i\in\mathscr{V}}{\bf g}(i)\,\chi_{_{\mathcal{E}_{i}}}\;.\label{e_res0}
\end{equation}

\begin{defn}
We say that the condition $\mathfrak{R}(\theta_{\epsilon},\,\mathscr{E},\,{\bf y})$
($\mathfrak{R}$ for resolvent) holds if for all ${\bf g}\colon\mathscr{V}\to\mathbb{R}$,
the unique solution $\phi_{\epsilon}^{{\bf g}}$ of \eqref{e_res0}
exists and satisfies
\begin{equation}
\lim_{\epsilon\rightarrow0}\,\max_{i\in\mathscr{V}}\,\parallel\,\phi_{\epsilon}^{{\bf g}}-\mathbf{f}(i)\,\Vert_{L^{\infty}(\mathcal{E}_{i})}\ =\ 0\;,\label{e_conphi}
\end{equation}
where ${\bf f}\colon\mathscr{V}\to\mathbb{R}$ is the solution of
the reduced resolvent equation
\[
\left(\lambda-\mathfrak{L}\right)\,{\bf f}\ =\ {\bf g}\;,
\]
and $\mathfrak{L}$ is the infinitesimal generator of the continuous-time
Markov chain ${\bf y}(\cdot)$.
\end{defn}

The next theorem is \cite[Theorem 2.3]{LMS}.
\begin{thm}
\label{t_res0} Conditions $\mathfrak{C}(\theta_{\epsilon},\,\mathscr{E},\,{\bf y})$
and $\mathfrak{R}(\theta_{\epsilon},\,\mathscr{E},\,{\bf y})$ are
equivalent.
\end{thm}

\subsection{Structure of inductive proof}

We turn to the approach adopted in this article. For $p\in\llbracket1,\,\mathfrak{q}\rrbracket$,
let
\begin{equation}
\begin{gathered}{\color{blue}\mathfrak{R}^{(p)}}\ :=\ \mathfrak{R}(\theta_{\epsilon}^{(p)},\,\mathscr{E}^{(p)},\,{\bf y}^{(p)})\;,\quad{\color{blue}\mathfrak{\mathfrak{C}}^{(p)}}\ :=\ \mathfrak{C}(\theta_{\epsilon}^{(p)},\,\mathscr{E}^{(p)},\,{\bf y}^{(p)})\\
{\color{blue}\mathfrak{C}_{{\rm fdd}}^{(p)}}\ :=\ \mathfrak{C}_{{\rm fdd}}(\theta_{\epsilon}^{(p)},\,\mathscr{E}^{(p)},\,{\bf y}^{(p)})\;,\quad{\color{blue}\mathfrak{G}^{(p)}}\ :=\ \mathfrak{G}(\theta_{\epsilon}^{(p)},\,\mathscr{E}^{(p)},\,{\bf y}^{(p)})\;.
\end{gathered}
\label{12-1}
\end{equation}
where ${\color{blue}\mathscr{E}^{(p)}=\{\mathcal{E}(\mathcal{M}):\mathcal{M}\in\mathscr{V}^{(p)}\}}$.
In addition, write
\begin{equation}
{\color{blue}\mathcal{E}^{(p)}}\ :=\ \bigcup_{\mathcal{M}\in\mathscr{V}^{(p)}}\mathcal{E}(\mathcal{M})\;,\quad p\in\llbracket1,\,\mathfrak{q}+1\rrbracket\;.\label{07}
\end{equation}
Define one more condition which plays a crucial role in the proof.
\begin{defn}
For $p\in\llbracket1,\,\mathfrak{q}+1\rrbracket$, we say that the
condition $\mathfrak{H}^{(p)}$ ($\mathfrak{H}$ for hitting) holds
if both of the following conditions hold.
\begin{enumerate}
\item For all sequence $(\alpha_{\epsilon})_{\epsilon>0}$ such that $\alpha_{\epsilon}\succ\theta_{\epsilon}^{(p-1)}$
\begin{equation}
\lim_{\epsilon\to0}\,\max_{\mathcal{M}\in\mathscr{N}^{(p)}}\,\sup_{\boldsymbol{x}\in\mathcal{E}(\mathcal{M})}\,\mathbb{P}_{\boldsymbol{x}}^{\epsilon}\left[\,\tau_{\mathcal{E}^{(p)}}\,>\,\alpha_{\epsilon}\,\right]\ =\ 0\ .\label{eq:Hp-1}
\end{equation}
Since $\mathscr{N}^{(1)}=\varnothing$, this condition is automatically
in force for $p=1$.
\item For all $\mathcal{M}\in\mathscr{N}^{(p)}$ and $\mathcal{M}'\in\mathscr{V}^{(p)}$,
\begin{equation}
\lim_{\epsilon\to0}\,\sup_{\boldsymbol{x}\in\mathcal{E}(\mathcal{M})}\,\bigg|\,\mathbb{P}_{\boldsymbol{x}}^{\epsilon}\left[\,\tau_{\mathcal{E}^{(p)}}\,=\,\tau_{\mathcal{E}(\mathcal{M}')}\,\right]\,-\,\widehat{\mathcal{Q}}_{\mathcal{M}}^{(p)}\left[\,\tau_{\mathscr{V}^{(p)}}\,=\,\tau_{\mathcal{M}'}\,\right]\,\bigg|\ =\ 0\;,\label{eq:Hp-2}
\end{equation}
where we recall that $\widehat{\mathcal{Q}}_{\mathcal{M}}^{(p)}$
represents the law of the process $\widehat{{\bf y}}^{(p)}(\cdot)$
starting from $\mathcal{M}\in\mathscr{S}^{(p)}$. Note that, for $p=\mathfrak{q}+1$,
the estimate \eqref{eq:Hp-2} is immediate since $\mathscr{V}^{(\mathfrak{q}+1)}$
is singleton and hence both probabilities at \eqref{eq:Hp-2} are
$1$.
\end{enumerate}
\end{defn}

Summing up, the condition $\mathfrak{H}^{(p)}$ implies that the process
$\boldsymbol{x}_{\epsilon}(\cdot)$ starting at a negligible valley
in $\mathscr{N}^{(p)}$ hits a non-negligible valley in $\mathscr{V}^{(p)}$
at the time scale not larger than $\theta_{\epsilon}^{(p-1)}$, and
the hitting distribution is explained by that of the process $\widehat{\mathbf{y}}^{(p)}(\cdot)$,
respectively.
\begin{rem}
The estimate \eqref{eq:Hp-2} trivially holds when $\mathcal{M}\in\mathscr{V}^{(p)}$.
\end{rem}

\begin{rem}
With the notation \eqref{noteq1}, we can rewrite \eqref{eq:Hp-2}
as, for all $\boldsymbol{m}\in\mathcal{M}_{0}$ (cf. previous remark)
\[
\lim_{\epsilon\to0}\,\sup_{\boldsymbol{x}\in\mathcal{E}(\boldsymbol{m})}\,\left|\,\mathbb{P}_{\boldsymbol{x}}^{\epsilon}\left[\,\tau_{\mathcal{E}^{(p)}}\,=\,\tau_{\mathcal{E}(\mathcal{M}')}\,\right]\,-\,\mathfrak{a}^{(p-1)}(\boldsymbol{m},\,\mathcal{M}')\,\right|\ =\ 0\;.
\]
This estimate provides another explanation on $\mathfrak{a}^{(p-1)}(\boldsymbol{m},\,\mathcal{M}')$,
and in fact is the main reason that this term appears in Theorem \ref{t00}.
\end{rem}

Now we turn to the proof of Theorems \ref{t_res} and \ref{t_meta}.
The following three propositions are the main steps in the complicated
inductive proof of these results. We refer to Figure \ref{fig_pf_t_res}
for an illustration of the structure of our inductive argument.

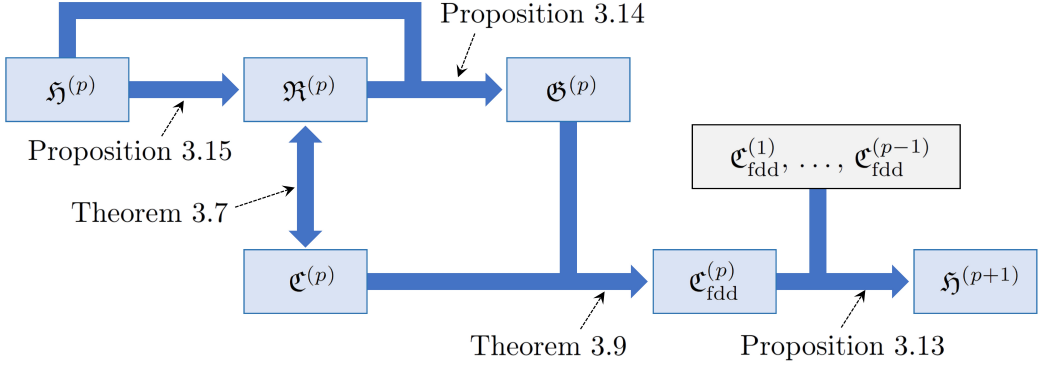
\begin{figure}
\begin{tikzpicture}

\fill[CornflowerBlue!30!white] (0,2.5) rectangle ++(1.5,0.8); \draw[CornflowerBlue] (0,2.5) rectangle ++(1.5,0.8); \draw (0.75,2.9) node{$\mathfrak{H}^{(p)}$};
\fill[CornflowerBlue!30!white] (3,2.5) rectangle ++(1.5,0.8); \draw[CornflowerBlue] (3,2.5) rectangle ++(1.5,0.8); \draw (3.75,2.9) node{$\mathfrak{R}^{(p)}$};
\fill[CornflowerBlue!30!white] (6,2.5) rectangle ++(1.5,0.8); \draw[CornflowerBlue] (6,2.5) rectangle ++(1.5,0.8); \draw (6.75,2.9) node{$\mathfrak{G}^{(p)}$};
\fill[CornflowerBlue!30!white] (3,0) rectangle ++(1.5,0.8); \draw[CornflowerBlue] (3,0) rectangle ++(1.5,0.8); \draw (3.75,0.4) node{$\mathfrak{C}^{(p)}$};
\fill[CornflowerBlue!30!white] (8,0) rectangle ++(1.5,0.8); \draw[CornflowerBlue] (8,0) rectangle ++(1.5,0.8); \draw (8.75,0.4) node{$\mathfrak{C}^{(p)}_{\rm{fdd}}$};
\fill[CornflowerBlue!30!white] (11,0) rectangle ++(1.5,0.8); \draw[CornflowerBlue] (11,0) rectangle ++(1.5,0.8); \draw (11.75,0.4) node{$\mathfrak{H}^{(p+1)}$};
\fill[CornflowerBlue!10!white] (8.5,1.7) rectangle ++(3.5,0.8); \draw[black] (8.5,1.7) rectangle ++(3.5,0.8); \draw (10.25,2.1) node{$\mathfrak{C}^{(1)}_{\rm{fdd}},\ \dots,\ \mathfrak{C}^{(p-1)}_{\rm{fdd}}$};

\draw [-latex, CornflowerBlue!70!Navy, line width=1mm] (1.5,2.9) to (2.95,2.9);
\draw [-latex, CornflowerBlue!70!Navy, line width=1mm] (4.5,2.9) to (5.95,2.9);
\draw [-latex, CornflowerBlue!70!Navy, line width=1mm] (4.5,0.4) to (7.95,0.4);
\draw [-latex, CornflowerBlue!70!Navy, line width=1mm] (9.5,0.4) to (10.95,0.4);
\draw [latex-latex, CornflowerBlue!70!Navy, line width=1mm] (3.75,0.8) to (3.75,2.5);
\draw [CornflowerBlue!70!Navy, line width=1mm] (0.75,3.3) to (0.75,3.8) to (5,3.8) to (5,2.9);
\draw [CornflowerBlue!70!Navy, line width=1mm] (6.75,2.5) to (6.75,0.4);
\draw [CornflowerBlue!70!Navy, line width=1mm] (10.25,1.7) to (10.25,0.4);

\draw (1.4,2) node{Proposition 3.14};
\draw (1.8,1.3) node{Theorem 3.9};
\draw (6.7,3.8) node{Proposition 3.15};
\draw (6.75,-0.4) node{Proposition 3.7};
\draw (10.25,-0.4) node{Proposition 3.13};

\draw [-latex, dash pattern=on 1.5pt off 1pt] (1.8,2.3) to (2.1,2.8);
\draw [-latex, dash pattern=on 1.5pt off 1pt] (3.1,1.5) to (3.65,1.65);
\draw [-latex, dash pattern=on 1.5pt off 1pt] (5.55,3.55) to (5.1,3);
\draw [-latex, dash pattern=on 1.5pt off 1pt] (6.75,-0.2) to (6.75,0.25);
\draw [-latex, dash pattern=on 1.5pt off 1pt] (10.25,-0.2) to (10.25,0.25);

\end{tikzpicture}
\caption{The structure of inductive proof}
\label{fig_pf_t_res}
\end{figure}

\begin{prop}
\label{prop_H}For $p\in\llbracket1,\,\mathfrak{q}\rrbracket$, conditions
$\mathfrak{C}_{{\rm fdd}}^{(1)},\,\dots,\,\mathfrak{C}_{{\rm fdd}}^{(p)}$
imply condition $\mathfrak{H}^{(p+1)}$.
\end{prop}

\begin{prop}
\label{prop_R}For $p\in\llbracket1,\,\mathfrak{q}\rrbracket$, the
condition $\mathfrak{H}^{(p)}$ implies $\mathfrak{R}^{(p)}$.
\end{prop}

\begin{prop}
\label{prop_NE} For $p\in\llbracket1,\,\mathfrak{q}\rrbracket$,
conditions $\mathfrak{R}^{(p)}$ and $\mathfrak{H}^{(p)}$ together
imply condition $\mathfrak{G}^{(p)}$.
\end{prop}

The proof of Propositions \ref{prop_H}, \ref{prop_R}, \ref{prop_NE}
are given in Sections \ref{sec7}, \ref{sec10}, \ref{sec11}, respectively.
Now we conclude the proof of Theorems \ref{t_res} and \ref{t_meta}.
\begin{proof}[Proof of Theorems \ref{t_res} and \ref{t_meta}]
We prove by induction that all conditions $\mathfrak{R}^{(p)}$,
$\mathfrak{C}^{(p)}$, $\mathfrak{G}^{(p)}$, and $\mathfrak{C}_{{\rm fdd}}^{(p)}$
are in force for all $p\in\llbracket1,\,\mathfrak{q}\rrbracket$.

The case $p=1$ is the content of \cite{LLS-1st}. More precisely,
condition $\mathfrak{R}^{(1)}$ is stated in \cite[Theorem 2.2]{LLS-1st},
condition $\mathfrak{C}^{(1)}$ in \cite[Theorem 9.3]{LLS-1st}, condition
$\mathfrak{G}^{(1)}$ in \cite[Proposition 10.4]{LLS-1st}, and condition
$\mathfrak{C}_{{\rm fdd}}^{(1)}$ in \cite[Proposition 10.2]{LLS-1st}.

Next, we fix $p\in\llbracket1,\,\mathfrak{q}-1\rrbracket$ and assume
that conditions $\mathfrak{R}^{(q)}$, $\mathfrak{C}^{(q)}$,
$\mathfrak{G}^{(q)}$, $\mathfrak{C}_{{\rm fdd}}^{(q)}$ hold for all
$q\in\llbracket1,\,p\rrbracket$.  By Proposition \ref{prop_H},
condition $\mathfrak{H}^{(p+1)}$ holds.  Then, by Proposition
\ref{prop_R} and Theorem \ref{t_res0}, conditions
$\mathfrak{R}^{(p+1)}$ and $\mathfrak{C}^{(p+1)}$ are inf force.  At
this point, by Proposition \ref{prop_NE}, condition
$\mathfrak{G}^{(p+1)}$ is fulfilled. Finally, by Proposition
\ref{prop_C_F0}, Condition $\mathfrak{C}_{{\rm fdd}}^{(p+1)}$
holds. This completes the induction step. Hence, the conditions
$\mathfrak{C}^{(p)}$, $\mathfrak{C}_{{\rm fdd}}^{(p)}$,
$\mathfrak{G}^{(p)}$, $\mathfrak{H}^{(p)}$, and $\mathfrak{R}^{(p)}$
hold for all $p\in\llbracket1,\,\mathfrak{q}\rrbracket$.  This
completes the proof of Theorems \ref{t_res} and \ref{t_meta} when
$\boldsymbol{m}\in\mathcal{M}_{0}$ is such that
$\mathcal{M}(p,\,\bm{m})\in\mathscr{V}^{(p)}$.

The proof of Theorems \ref{t_res} in the case where
$\boldsymbol{m}\in\mathcal{M}_{0}$ is such that
$\mathcal{M}(p,\,\bm{m})\in\mathscr{N}^{(p)}$ is the content of
Propositions \ref{p_flat}, \ref{p_char_neg}.  On the other hand, since
we proved $\mathfrak{C}_{{\rm fdd}}^{(p)}$ and $\mathfrak{H}^{(p)}$
for all $p\in\llbracket1,\,\mathfrak{q}\rrbracket$, Theorem
\ref{t_meta} follows from Proposition \ref{p_meta}.
\end{proof}

\subsection{Organization of article}

The remainder of the article is organized as follows.  In Sections
\ref{sec4}--\ref{sec6}, we explain the robust scheme leading to the
tree structure associated with the potential $U$. In Sections
\ref{sec7}--\ref{sec11}, we complete the induction argument used in
the proof of Theorems \ref{t_res} and \ref{t_meta} above.  In
particular, Sections \ref{sec8}--\ref{sec10} are devoted to the
analysis of the resolvent equation. In Sections \ref{sec12} and
\ref{sec13}, we complete the proof of Theorem \ref{t_meta} by allowing
the starting point to belong to the domain of attraction of a
negligible valley.

\section{\label{sec4}Construction of tree structure}

In this section, we present the explicit construction of the tree structure
introduced in Definition \ref{def:tree}. The construction is carried
out by inductively defining quintuples
\[
{\color{blue}\Lambda^{(n)}}\ :=\ \big(\,d^{(n)},\,\mathscr{V}^{(n)},\,\mathscr{N}^{(n)},\,\mathbf{\widehat{y}}^{(n)}(\cdot),\,\mathbf{y}^{(n)}(\cdot)\,\big)\;,\;\;\;n\in\llbracket1,\,\mathfrak{q}\rrbracket\;,
\]
consisting of the objects introduced in Definition \ref{def:tree} in
an manner, where the $n$th time scale will be defined by
$\theta_{\epsilon}^{(n)}=e^{d^{(n)}/\epsilon}$.  Note that
$\mathfrak{q}$, the number of different scales, is also a part of
construction. We continuously refer to Figure \ref{fig:potential} in
Section \ref{sec2_example} to explain concretely numerous notation
introduced below.

\subsection{\label{sec4.1}The first metastable scale}

To start the inductive construction, we have to explicitly define
$\Lambda^{(1)}$. To that end, we first introduce some notions which
are widely used in the study of energy landscape.

For each pair $\boldsymbol{m}'\neq\boldsymbol{m}''\in\mathcal{M}_{0}$,
denote by $\Theta(\boldsymbol{m}',\,\boldsymbol{m}'')$ the \emph{communication
height} between $\boldsymbol{m}'$ and $\boldsymbol{m}''$:
\begin{equation}
{\color{blue}\Theta(\boldsymbol{m}',\,\boldsymbol{m}'')}\ :=\ \inf_{\substack{\boldsymbol{z}:[0\,1]\rightarrow\mathbb{R}^{d}}
}\,\max_{t\in[0,\,1]}\,U(\boldsymbol{z}(t))\;,\label{Theta}
\end{equation}
where the infimum is carried over all continuous paths $\boldsymbol{z}(\cdot)$
such that $\boldsymbol{z}(0)=\boldsymbol{m}'$ and $\boldsymbol{z}(1)=\boldsymbol{m}''$.
Clearly, $\Theta(\boldsymbol{m}',\,\boldsymbol{m}'')=\Theta(\boldsymbol{m}'',\,\boldsymbol{m}')$.
For instance, in Figure \ref{fig:potential}, $\Theta(\boldsymbol{m}_{2},\,\boldsymbol{m}_{3})=U(\boldsymbol{\sigma}_{1})$
and $\Theta(\boldsymbol{m}_{1},\,\boldsymbol{m}_{10})=U(\boldsymbol{\sigma}_{3})$.

Define
\begin{equation}
{\color{blue}\Xi(\boldsymbol{m})}\ :=\ \inf\left\{ \,\Theta(\boldsymbol{m},\,\boldsymbol{m}'):\boldsymbol{m}'\in\mathcal{M}_{0}\setminus\{\boldsymbol{m}\}\text{ such that }U(\boldsymbol{m}')\le U(\boldsymbol{m})\,\right\} -U(\boldsymbol{m})\label{eq:Xi(m)}
\end{equation}
where we set $\Xi(\boldsymbol{m})=+\infty$ when $\boldsymbol{m}$
is the unique global minimum of $U$ so that we cannot find $\boldsymbol{m}'\neq\boldsymbol{m}$
such that $U(\boldsymbol{m}')\le U(\boldsymbol{m})$.

Now the minimal depth $d^{(1)}$ and the first time-scale $\theta_{\epsilon}^{(1)}$
are defined by
\begin{equation}
{\color{blue}d^{(1)}}\ :=\ \min_{\boldsymbol{m}\in\mathcal{M}_{0}}\Xi(\boldsymbol{m})\;,\quad{\color{blue}\theta_{\epsilon}^{(1)}}\ :=\ e^{d^{(1)}/\epsilon}\;.\label{eq:init1}
\end{equation}
This definition explain $d^{(1)}$in Figure \ref{fig:potential}.
For instance, in Figure \ref{fig:potential},
\[
\Xi(\bm{m}_{3})\;=\;\inf\left\{ \Theta(\boldsymbol{m}_{3},\,\boldsymbol{m}'):\bm{m}'\in\{\boldsymbol{m}_{1},\,\boldsymbol{m}_{2},\,\boldsymbol{m}_{8},\,\boldsymbol{m}_{9},\,\boldsymbol{m}_{10}\}\right\} -U(\boldsymbol{m}_{3})\;=\;U(\boldsymbol{\sigma}_{1})-U(\boldsymbol{m}_{3})\;.
\]
We can notice that $\Xi(\boldsymbol{m}_{i})\ge\Xi(\bm{m}_{3})$ for
all $i\in\llbracket1,\,10\rrbracket$ in this figure and therefore
$d^{(1)}=U(\boldsymbol{\sigma}_{1})-U(\boldsymbol{m}_{3})$.
\begin{rem}
\label{rem_Xi_Gamma}The definition of $d^{(1)}$ in \eqref{eq:init1}
looks slightly different from the one in \cite{LLS-1st} at which
\[
\Gamma(\boldsymbol{m})\ :=\ \min_{\boldsymbol{m}'\in\mathcal{M}_{0}\setminus\{\boldsymbol{m}\}}\,\Theta(\boldsymbol{m},\boldsymbol{m}')\,-\,U(\boldsymbol{m})
\]
is used instead of $\Xi(\boldsymbol{m})$. However, it can be easily
confirmed that the two definitions are equivalent.
\end{rem}

For the first-scale, we set (as in \eqref{eq:v1})
\begin{equation}
{\color{blue}\mathscr{V}^{(1)}}\ :=\ \left\{ \,\{\boldsymbol{m}\}:\boldsymbol{m}\in\mathcal{M}_{0}\,\right\} \;,\;\;{\color{blue}\mathscr{N}^{(1)}}\ :=\ \varnothing\;.\label{eq:init2}
\end{equation}
Note that, in this scale ${\color{blue}\mathscr{S}^{(1)}}:=\mathscr{V}^{(1)}\cup\mathscr{N}^{(1)}=\mathscr{V}^{(1)}$.

Now it remains to construct continuous Markov chains $\widehat{\mathbf{y}}^{(1)}(\cdot)$
on $\mathscr{S}^{(1)}$ and $\mathbf{y}^{(1)}(\cdot)$ on $\mathscr{V}^{(1)}$.
As $\mathscr{S}^{(1)}=\mathscr{V}^{(1)}$, they are defined on the
same space. In fact, they are the same Markov chain for the first
time-scale.

Let $\mathcal{S}^{(1)}(\boldsymbol{m})$ be the set of saddle points
connected to the local minimum $\boldsymbol{m}$:
\begin{gather*}
{\color{blue}\mathcal{S}^{(1)}(\boldsymbol{m})}\ :=\ \big\{\,\bm{\sigma}\in\mathcal{S}_{0}:\boldsymbol{\sigma}\curvearrowright\boldsymbol{m}\ ,\ \ U(\bm{\sigma})=U(\bm{m})+\Xi(\bm{m})\,\big\}\;.
\end{gather*}
In view of Remark \ref{rem_Xi_Gamma} and \cite[Lemma 9.2]{LLS-1st},
$\mathcal{S}^{(1)}(\bm{m})\ne\varnothing$ when $\Xi(\bm{m})=d^{(1)}$. Denote by
$\mathcal{S}(\boldsymbol{m},\boldsymbol{m}')$,
$\boldsymbol{m}'\neq\boldsymbol{m}$, the set of saddle points which
separate $\boldsymbol{m}$ from $\boldsymbol{m}'$:
\[
{\color{blue}\mathcal{S}(\boldsymbol{m},\boldsymbol{m}')}\ :=\ \big\{\,\boldsymbol{\sigma}\in\mathcal{S}^{(1)}(\boldsymbol{m}):\boldsymbol{\sigma}\curvearrowright\boldsymbol{m}\;,\;\;\boldsymbol{\sigma}\curvearrowright\boldsymbol{m}'\,\big\}\;.
\]
For instance, in Figure \ref{fig:potential},
$\mathcal{S}^{(1)}(\bm{m}_{7})=$$\mathcal{S}(\boldsymbol{m}_{7},\,\boldsymbol{m}_{8})=\{\boldsymbol{\sigma}_{1}\}$.

By \cite[Lemma 3.3]{LS-22}, for each saddle point $\boldsymbol{\sigma}\in\mathcal{S}_{0}$,
we observe that the matrix $(\nabla^{2}U)(\boldsymbol{\sigma})+(D\boldsymbol{\ell})(\boldsymbol{\sigma})$
(where $\nabla^{2}U$ and $D\boldsymbol{\ell}$ represents the Hessian
of $U$ and the Jacobian of $\boldsymbol{\ell}$, respectively) has
one negative eigenvalue, represented by ${\color{blue}-\mu_{\boldsymbol{\sigma}}<0}$.
For $\boldsymbol{\sigma}\in\mathcal{S}_{0}$, let the weight $\omega(\boldsymbol{\sigma})$,
so-called \textit{Eyring--Kramers constant}, be given by
\begin{equation}
{\color{blue}\omega(\boldsymbol{\sigma})}\ :=\ \frac{\mu(\boldsymbol{\sigma})}{2\pi\sqrt{-\,\det(\nabla^{2}U)(\boldsymbol{\sigma})}}\;,\label{eq:omega}
\end{equation}
and denote by $\omega(\boldsymbol{m},\boldsymbol{m}')$ the sum of
the weights of the saddle points in $\mathcal{S}(\boldsymbol{m},\boldsymbol{m}')$:
\begin{equation}
{\color{blue}\omega(\boldsymbol{m},\,\boldsymbol{m}')}\ :=\ \sum_{\boldsymbol{\sigma}\in\mathcal{S}(\boldsymbol{m},\boldsymbol{m}')}\,\omega(\boldsymbol{\sigma})\;.\label{eq:omega_2}
\end{equation}
Define
\begin{equation}
r^{(1)}(\{\boldsymbol{m}\},\{\boldsymbol{m}'\})\ :=\ \frac{1}{\nu(\boldsymbol{m})}\,\omega_{1}(\boldsymbol{m},\boldsymbol{m}')\;\;\text{where}\;\;{\color{blue}\omega_{1}(\boldsymbol{m},\boldsymbol{m}')}\ :=\ \omega(\boldsymbol{m},\boldsymbol{m}')\,\boldsymbol{1}\left\{ \,\Xi(\boldsymbol{m})=d^{(1)}\,\right\} \;,\label{40}
\end{equation}
and $\nu(\boldsymbol{m})$ is the weight defined in \eqref{eq:nu}.
Then, denote by \textcolor{blue}{$\mathfrak{L}^{(1)}$} the generator
given by
\[
(\mathfrak{L}^{(1)}\boldsymbol{h})\,(\{\boldsymbol{m}\})\;:=\;\sum_{\boldsymbol{m}'\in\mathcal{M}_{0}}\,r^{(1)}(\{\boldsymbol{m}\},\,\{\boldsymbol{m}'\})\,[\,\boldsymbol{h}(\{\boldsymbol{m}'\})\,-\,\boldsymbol{h}(\{\boldsymbol{m}\})\,]
\]
for $\boldsymbol{h}\colon\mathscr{V}^{(1)}\to\mathbb{R}$. Finally,
let $\widehat{\mathbf{y}}^{(1)}(\cdot)=\mathbf{y}^{(1)}(\cdot)$ be
the $\mathscr{S}^{(1)}=\mathscr{V}^{(1)}$-valued Markov chains induced
by the generator $\mathfrak{L}^{(1)}$.

Denote by $\mathfrak{n}_{0}$ the number of local minima of $U$,
${\color{blue}\mathfrak{n}_{0}=|\mathcal{M}_{0}|}$, and by ${\color{blue}\mathscr{R}_{1}^{(1)},\dots,\mathscr{R}_{\mathfrak{n}_{1}}^{(1)}}$,
${\color{blue}\mathscr{T}^{(1)}}$ be the closed irreducible classes
and the transient states of the Markov chain ${\bf y}^{(1)}(\cdot)$,
respectively. If $\mathfrak{n}_{1}=1$, the construction is completed,
$\mathfrak{q}=1$, and $\Lambda^{(1)}$ have been defined. If $\mathfrak{n}_{1}>1$,
we add a new layer to the tree according to the induction procedure
explained in the next subsection.

\subsection{\label{sec4.2}General metastable scale}

Now we suppose that the quintiples $\Lambda^{(1)},\,\dots,\,\Lambda^{(n)}$
are successfully constructed. Denote by \textcolor{blue}{$\mathscr{R}_{1}^{(n)},\dots,\mathscr{R}_{\mathfrak{n}_{n}}^{(n)}$,
$\mathscr{T}^{(n)}$} the closed irreducible classes and the transient
states of the Markov chain ${\bf y}^{(n)}(\cdot)$ on $\mathscr{V}^{(n)}$,
respectively. If $\mathfrak{n}_{n}=1$, the construction is completed
and $\mathfrak{q}=n$.

If $\mathfrak{n}_{n}\ge2$, we create a new layer.

\subsubsection*{Construction of $\mathscr{V}^{(n+1)}$ and $\mathscr{N}^{(n+1)}$}

We let
\begin{equation}
\mathcal{M}_{j}^{(n+1)}\ :=\ \bigcup_{\mathcal{M}\in\mathscr{R}_{j}^{(n)}}\mathcal{M}\;\;\;;\;j\in\llbracket1,\,\mathfrak{n}_{n}\rrbracket\;,\label{eq:V_p+1}
\end{equation}
and then define
\begin{equation}
{\color{blue}\mathscr{V}^{(n+1)}}\ :=\ \{\mathcal{M}_{1}^{(n+1)}\,\dots,\mathcal{M}_{\mathfrak{n}_{n}}^{(n+1)}\}\;,\quad{\color{blue}\mathscr{N}^{(n+1)}}\ :=\ \mathscr{N}^{(n)}\cup\mathscr{T}^{(n)}\;.\label{eq:VN_p+1}
\end{equation}

Then, we set ${\color{blue}\mathscr{S}^{(n+1)}}=\mathscr{V}^{(n+1)}\cup\mathscr{N}^{(n+1)}$.
It is immediate from this definition that, if $\mathscr{S}^{(n)}=\mathscr{V}^{(n)}\cup\mathscr{N}^{(n)}$
was a partition of $\mathcal{M}_{0}$, then $\mathscr{S}^{(n+1)}$
also is a partition of $\mathcal{M}_{0}$. The following remarks are
important and used throughout article frequently.
\begin{rem}
\label{rem:N_p+1}The inductive construction \eqref{eq:V_p+1} guarantees
that $\mathscr{N}^{(n+1)}\subset\mathscr{S}^{(n)}$. Moreover, by
\eqref{eq:VN_p+1}, we also have
\begin{equation}
\mathscr{N}^{(n+1)}\ =\ \mathscr{T}^{(1)}\cup\cdots\cup\mathscr{T}^{(n)}\;.\label{eq:N_p+1}
\end{equation}
Let $\mathcal{M}\in\mathscr{N}^{(n+1)}$. Then, \eqref{eq:N_p+1}
implies that we can find $k\in\llbracket1,\,n\rrbracket$ such that
$\mathcal{M}\in\mathscr{T}^{(k)}\subset\mathscr{V}^{(k)}$. Since
$\mathscr{T}^{(k)}\subset\mathscr{N}^{(k+1)}$, we also have $\mathcal{M}\in\mathscr{V}^{(k)}\cap\mathscr{N}^{(k+1)}$.
Namely, a set $\mathcal{M}\in\mathscr{N}^{(n+1)}$ was a member of
$\mathscr{V}^{(k)}$ but became a negligible one from the $(k+1)$-th
scale.
\end{rem}

\begin{rem}
\label{rem:stddec}By construction, each $\mathcal{M}\in\mathscr{V}^{(n+1)}$
is a union of elements of $\mathscr{V}^{(n)}$. We denote by $\mathscr{V}^{(n)}(\mathcal{M})$
these elements so that
\begin{equation}
{\color{blue}\mathscr{V}^{(n)}(\mathcal{M})}\ :=\ \{\mathcal{M}_{1},\,\dots,\,\mathcal{M}_{k}\}\;,\label{2-04}
\end{equation}
where $\{\mathcal{M}_{1},\,\dots,\,\mathcal{M}_{k}\}\subset\mathscr{V}^{(n)}$
is a recurrent class of chain $\mathbf{y}^{(n)}(\cdot)$ and $\mathcal{M}=\cup_{i\in\llbracket1,\,n\rrbracket}\mathcal{M}_{i}$.
\end{rem}

\begin{example*}
[Continuation of the example in Section \ref{sec2_example}]
We return to the example given in Figure \ref{fig:potential}.
The construction above implies that the jumps of the Markov chain
$\mathbf{y}^{(1)}(\cdot)=\widehat{\mathbf{y}}^{(1)}(\cdot)$ occur
along the following arrows.
\[
\boldsymbol{m}_{1}\leftrightarrow\boldsymbol{m}_{2}\leftrightarrow\boldsymbol{m}_{3}\leftarrow\boldsymbol{m}_{4}\rightarrow\boldsymbol{m}_{5}\leftrightarrow\boldsymbol{m}_{6}\;\;\;\;\boldsymbol{m}_{7}\rightarrow\boldsymbol{m}_{8}\;\;\;\;\boldsymbol{m}_{9}\leftrightarrow\boldsymbol{m}_{10}
\]
Hence, there are four recurrent classes $\{\boldsymbol{m}_{1},\,\boldsymbol{m}_{2},\,\boldsymbol{m}_{3}\}$,
$\{\boldsymbol{m}_{5},\,\boldsymbol{m}_{6}\}$, $\{\boldsymbol{m}_{8}\}$,
and $\{\boldsymbol{m}_{9},\,\boldsymbol{m}_{10}\}$, and therefore
\begin{align}
\mathscr{V}^{(2)} & \ =\ \left\{ \,\{\boldsymbol{m}_{1},\,\boldsymbol{m}_{2},\,\boldsymbol{m}_{3}\},\,\{\boldsymbol{m}_{5},\,\boldsymbol{m}_{6}\},\,\{\boldsymbol{m}_{8}\},\,\{\boldsymbol{m}_{9},\,\boldsymbol{m}_{10}\}\,\right\} \;,\label{eq:v2}\\
\mathscr{N}^{(2)} & \ =\ \left\{ \,\{\boldsymbol{m}_{4}\},\,\{\boldsymbol{m}_{7}\}\,\right\} \;.\nonumber
\end{align}
\end{example*}

\subsubsection*{Construction of $d^{(n+1)}$ and $\theta_{\epsilon}^{(n+1)}$}

We introduce several notions for analyzing the landscape of $U$.
\begin{itemize}
\item For two disjoint non-empty subsets $\mathcal{M}$ and $\mathcal{M}'$
of $\mathcal{M}_{0}$, let $\Theta(\mathcal{M},\,\mathcal{M}')$ be
the communication height between the two sets:
\[
{\color{blue}\Theta(\mathcal{M},\,\mathcal{M}')}\ :=\ \min_{\boldsymbol{m}\in\mathcal{M},\,\boldsymbol{m}'\in\mathcal{M}'}\,\Theta(\boldsymbol{m},\,\boldsymbol{m}')\;,
\]
with the convention that $\Theta(\mathcal{\mathcal{M}},\,\varnothing)=+\infty$.\smallskip{}
\item A set $\mathcal{M}\subset\mathcal{M}_{0}$ is said to be \textcolor{blue}{\emph{simple}}
if all its elements have the same depth, $U(\boldsymbol{m})=U(\boldsymbol{m}')$
for all $\boldsymbol{m}$, $\boldsymbol{m}'\in\mathcal{M}$, and that
we denote by \textcolor{blue}{$U(\mathcal{M})$ }the common value.
For instance, in Figure \ref{fig:potential}, $\{\boldsymbol{m}_{5},\,\boldsymbol{m}_{6}\}$
and $\{\boldsymbol{m}_{1},\,\boldsymbol{m}_{2},\,\boldsymbol{m}_{3}\}$
are simple sets. \smallskip
\item For a simple set $\mathcal{\mathcal{M}}\subset\mathcal{M}_{0}$, denote
by $\widetilde{\mathcal{M}}$ the set of local minima of $U$ which
do not belong to $\mathcal{M}$ and which have lower or equal energy
than $\mathcal{M}$:
\[
{\color{blue}\widetilde{\mathcal{M}}}\ :=\ \big\{\,\boldsymbol{m}\in\mathcal{M}_{0}\setminus\mathcal{M}:U(\boldsymbol{m})\le U(\mathcal{M})\,\big\}\;.
\]
Note that $\widetilde{\mathcal{\mathcal{M}}}=\varnothing$ if and
only if $\mathcal{M}$ contains all the global minima of $U$.\smallskip{}
\item For a simple set $\mathcal{M}\subset\mathcal{M}_{0}$, let $\Xi(\mathcal{M})$
be the depth of set $\mathcal{M}$:
\begin{equation}
{\color{blue}\Xi(\mathcal{M})}\ :=\ \Theta(\mathcal{\mathcal{M}},\,\widetilde{\mathcal{\mathcal{M}}})-U(\mathcal{M})\;.\label{e:Gamma(M)}
\end{equation}
\smallskip{}
\item A simple set $\mathcal{M}$ is said to be\textit{ }\textit{\textcolor{blue}{bound}}
if either $\mathcal{M}$ is a singleton or
\begin{equation}
\max_{\boldsymbol{m},\,\boldsymbol{m}'\in\mathcal{M}}\,\Theta(\boldsymbol{m},\,\boldsymbol{m}')\ <\ \Theta(\mathcal{M},\,\widetilde{\mathcal{M}})\ .\label{2-07}
\end{equation}
\end{itemize}
\emph{We will prove in Theorem \ref{t:tree} that all sets $\mathcal{M}\in\mathscr{S}^{(n+1)}$
are simple and bound }(cf. property $\mathfrak{P}_{1}(n+1)$ defined
below). As $\mathfrak{n}_{n}\ge2$, there exists $\mathcal{M}\in\mathscr{V}^{(n+1)}$
such that $\widetilde{\mathcal{\mathcal{M}}}\neq\varnothing$, so
that $\Xi(\mathcal{M})<\infty$. Based on this observation, define
the depth $d^{(n+1)}$ and the corresponding time-scale $\theta_{\epsilon}^{(n+1)}$
by
\begin{equation}
{\color{blue}d^{(n+1)}}\ :=\ \min_{\mathcal{M}\in\mathscr{V}^{(n+1)}}\,\Xi(\mathcal{M})\;,\quad{\color{blue}\theta_{\epsilon}^{(n+1)}}\ :=\ e^{d^{(n+1)}/\epsilon}\;.\label{eq:depth_p+1}
\end{equation}
Namely, we define $d^{(n+1)}$ as the depth of the shallowest valley
among $\mathscr{V}^{(n+1)}$.

\begin{example*} [Continuation of example in Section \ref{sec2_example}]
We note that the elements of $\mathscr{V}^{(2)}$
and $\mathscr{N}^{(2)}$ in \eqref{eq:v2} are indeed simple and bound.
This is not a coincidence but a consequence of property $\mathfrak{P}_{1}$
in Definition \ref{def:prop} below. As $\widetilde{\{\boldsymbol{m}_{5},\,\boldsymbol{m}_{6}\}}=\{\boldsymbol{m}_{1},\,\boldsymbol{m}_{2},\,\boldsymbol{m}_{3},\,\boldsymbol{m}_{8},\,\boldsymbol{m}_{9},\,\boldsymbol{m}_{10}\}$,
we can check that
\[
\Xi(\{\boldsymbol{m}_{5},\,\boldsymbol{m}_{6}\})\ =\ U(\boldsymbol{\sigma_{2}})-U(\boldsymbol{m}_{5})\;.
\]
Similarly, we can compute the depth of the elements of $\mathscr{V}^{(2)}$
given in \eqref{eq:v2}, and then check that $d^{(2)}=\Xi(\{\boldsymbol{m}_{5},\,\boldsymbol{m}_{6}\})=U(\boldsymbol{\sigma_{2}})-U(\boldsymbol{m}_{5})$
as illustrated in Figure \ref{fig:potential}.
\end{example*}

\subsubsection*{Construction of $\widehat{\mathbf{y}}^{(n+1)}(\cdot)$ and $\mathbf{y}^{(n+1)}(\cdot)$}

A saddle point $\boldsymbol{\sigma}\in\mathcal{S}_{0}$ is said to
be\emph{ }\textit{\textcolor{blue}{connected}} to a local minimum
$\boldsymbol{m}\in\mathcal{M}_{0}$ if $\boldsymbol{\sigma}\curvearrowright\boldsymbol{m}$
or if there exist $k\ge1$, $\boldsymbol{\sigma}_{1},\,\dots,\,\boldsymbol{\sigma}_{k}\in\mathcal{S}_{0}$
and $\boldsymbol{m}_{1}\,\,\dots,\,\boldsymbol{m}_{k}\in\mathcal{M}_{0}$
such that
\begin{equation}
\max\{U(\boldsymbol{\sigma}_{1}),\,\dots,\,U(\boldsymbol{\sigma}_{k})\,\}\ <\ U(\boldsymbol{\sigma})\;\;\;\text{and\;\;\;}\boldsymbol{\sigma}\curvearrowright\boldsymbol{m}_{1}\curvearrowleft\boldsymbol{\sigma}_{1}\curvearrowright\cdots\curvearrowright\boldsymbol{m}_{k}\curvearrowleft\boldsymbol{\sigma}_{k}\curvearrowright\boldsymbol{m}\;.\label{eq:approx}
\end{equation}
If $\boldsymbol{\sigma}\in\mathcal{S}_{0}$ is connected to
$\boldsymbol{m}\in\mathcal{M}_{0}$, we
write\textit{\textcolor{blue}{
$\boldsymbol{\sigma}\rightsquigarrow\boldsymbol{m}$}}.  For example,
in Figure \ref{fig:potential}, we have
$\boldsymbol{\sigma}_{3}\rightsquigarrow\boldsymbol{m}_{8}$ as
$\boldsymbol{\sigma}_{3}\curvearrowright\boldsymbol{m}_{7}\curvearrowleft\boldsymbol{\sigma}_{4}\curvearrowright\boldsymbol{m}_{8}$
and
$U(\boldsymbol{\sigma}_{4})<U(\boldsymbol{\sigma}_{3})$. Similarly, we
have $\boldsymbol{\sigma}_{3}\rightsquigarrow\boldsymbol{m}_{3}$.

For $\mathcal{\mathcal{M}}\subset\mathcal{M}_{0}$, write ${\color{blue}\boldsymbol{\sigma}\rightsquigarrow\mathcal{\mathcal{M}}}$
and \textcolor{blue}{$\bm{\sigma}\curvearrowright\mathcal{M}$} if
for some $\boldsymbol{m}\in\mathcal{\mathcal{M}}$, $\boldsymbol{\sigma}\rightsquigarrow\boldsymbol{m}$
and $\bm{\sigma}\curvearrowright\bm{m}$ , respectively.

Fix a non-empty simple bound set $\mathcal{\mathcal{M}}\subset\mathcal{M}_{0}$
such that $\widetilde{\mathcal{M}}\neq\varnothing$. A set $\mathcal{\mathcal{M}}'\subset\mathcal{M}_{0}$
such that $\mathcal{M}'\cap\mathcal{M}=\varnothing$ is said to be\textit{
}\textit{\textcolor{blue}{adjacent}} to $\mathcal{M}$ if there exists
$\boldsymbol{\sigma}\in\mathcal{S}_{0}$ such that
\begin{equation}
U(\boldsymbol{\sigma})\,=\,\Theta(\mathcal{M},\,\widetilde{\mathcal{M}})\,=\,\Theta(\mathcal{M},\,\mathcal{M}')\;\;\;\text{and}\;\;\;\mathcal{M}\,\leftsquigarrow\,\boldsymbol{\sigma}\,\curvearrowright\,\mathcal{M}'\;.\label{eq:con_gate}
\end{equation}
If $\mathcal{M}'$ is adjacent to $\mathcal{M}$, we
write\textit{\textcolor{blue}{
$\mathcal{M}\rightarrow\mathcal{M}'$}}. To emphasize the saddle point
$\boldsymbol{\sigma}$ between $\mathcal{M}$ and $\mathcal{M}'$ we
sometimes write
${\color{blue}\mathcal{M}\rightarrow_{\boldsymbol{\sigma}}\mathcal{M}'}$.
For example, in Figure \ref{fig:potential},
$\{\boldsymbol{m}_{1},\,\boldsymbol{m}_{2},\,\boldsymbol{m}_{3}\}\rightarrow_{\boldsymbol{\sigma}_{3}}\{\boldsymbol{m}_{7}\}$.

Denote by $\mathcal{S}(\mathcal{M},\,\mathcal{M}')$ the set of saddle
points $\boldsymbol{\sigma}\in\mathcal{S}_{0}$ satisfying \eqref{eq:con_gate},
\begin{equation}
{\color{blue}\mathcal{S}(\mathcal{M},\,\mathcal{M}')}\ :=\ \{\,\boldsymbol{\sigma}\in\mathcal{S}_{0}:\mathcal{M}\to_{\boldsymbol{\sigma}}\mathcal{M}'\,\}\ .\label{eq:saddle_gate}
\end{equation}
The set $\mathcal{S}(\mathcal{M},\,\mathcal{M}')$ represents the
collection of lowest connection points which separate $\mathcal{M}$
from $\mathcal{M}'$.

We denote by $\widehat{r}^{(n)}(\cdot,\,\cdot)$ the jump rates of
the $\mathscr{S}^{(n)}$-valued Markov chain $\widehat{\mathbf{y}}^{(n)}(\cdot)$.
Since $\mathscr{S}^{(n+1)}=\mathscr{V}^{(n+1)}\cup\mathscr{N}^{(n+1)}$,
we can divide the definition of the jump rate \textcolor{blue}{\emph{$\widehat{r}^{(n+1)}(\cdot,\,\cdot)$
}}of $\widehat{\mathbf{y}}^{(n+1)}(\cdot)$ into three cases:
\begin{itemize}
\item {[}\textbf{Case }1: $\mathcal{M}\in\mathscr{N}^{(n+1)}$ and $\mathcal{M}'\in\mathscr{N}^{(n+1)}${]}
Since $\mathcal{M},\,\mathcal{M}'\in\mathscr{S}^{(n)}$ by Remark
\ref{eq:N_p+1}, we set
\begin{equation}
\widehat{r}^{(n+1)}(\mathcal{M},\,\mathcal{M}')\ :=\ \widehat{r}^{(n)}(\mathcal{M},\,\mathcal{M}')\;.\label{eq:rate_1}
\end{equation}
\item {[}\textbf{Case }2: $\mathcal{M}\in\mathscr{N}^{(n+1)}$ and $\mathcal{M}'\in\mathscr{V}^{(n+1)}${]}
Since $\mathcal{M}\in\mathscr{S}^{(n)}$ by Remark \ref{eq:N_p+1}
and since $\mathcal{M}'$ is the union of elements (may be just one)
in $\mathscr{V}^{(n)}$ by \eqref{eq:V_p+1}, we set
\begin{equation}
\widehat{r}^{(n+1)}(\mathcal{M},\,\mathcal{M}')\ :=\ \sum_{\mathcal{M}''\in\mathscr{V}^{(n)}(\mathcal{M}')}\,\widehat{r}^{(n)}(\mathcal{M},\,\mathcal{M}'')\;.\label{eq:rate_2}
\end{equation}
\item {[}\textbf{Case 3}: $\mathcal{M}\in\mathscr{V}^{(n+1)}$ and $\mathcal{M}'\in\mathscr{S}^{(n+1)}${]}
Let
\begin{equation}
\boldsymbol{\omega}(\mathcal{M},\,\mathcal{M}')\ =\ \sum_{\boldsymbol{\sigma}\in\mathcal{S}(\mathcal{M},\,\mathcal{M}')}\omega(\boldsymbol{\sigma})\;,\quad{\color{blue}\omega_{n+1}(\mathcal{M},\,\mathcal{M}')}\ :=\ \omega(\mathcal{M},\,\mathcal{M}')\,\boldsymbol{1}\{\,\Xi(\mathcal{M})=d^{(n+1)}\,\}\;.\label{eq:rate_30}
\end{equation}
It is understood here that $\omega(\mathcal{M},\,\mathcal{M}')=0$
if the set $\mathcal{S}(\mathcal{M},\,\mathcal{M}')$ is empty, that
is if $\mathcal{M}'$ is not adjacent to $\mathcal{M}$. Then, we
set
\begin{equation}
\widehat{r}^{(n+1)}(\mathcal{M}',\,\mathcal{M}'')\ :=\ \frac{1}{\nu(\mathcal{M}')}\,\omega_{n+1}(\mathcal{M}',\,\mathcal{M}'')\;,\label{eq:rate_3}
\end{equation}
where the weight $\nu(\mathcal{M}')$ is given by \eqref{eq:nu}.
\end{itemize}
Denote by ${\color{blue}\widehat{\mathbf{y}}^{(n+1)}(\cdot)}$ the
$\mathscr{S}^{(n+1)}$-valued, continuous-time Markov process with jump
rates $\widehat{r}^{(n+1)}(\cdot,\,\cdot)$, and by
${\color{blue}\mathbf{y}^{(n+1)}(\cdot)}$ the trace of the process
$\widehat{\mathbf{y}}^{(n+1)}(\cdot)$ on $\mathscr{V}^{(n+1)}$
(cf. Appendix \ref{app:trace}. We note that the trace on
$\mathscr{V}^{(n+1)}$ can be defined because of Lemma
\ref{l_T_diverge}.

The jump rates
of the Markov chain $\mathbf{y}^{(n+1)}(\cdot)$ are represented by
${\color{blue}r^{(n+1)}(\cdot,\,\cdot)}$.

Let ${\color{blue}\widehat{\mathfrak{L}}^{(n+1)}}$ and ${\color{blue}\mathfrak{L}^{(n+1)}}$
be the generators of the Markov chains $\widehat{\mathbf{y}}^{(n+1)}(\cdot)$
and $\mathbf{y}^{(n+1)}(\cdot)$, respectively:
\begin{equation}
\begin{gathered}(\widehat{\mathfrak{L}}^{(n+1)}\boldsymbol{h})(\mathcal{M})\;:=\;\sum_{\mathcal{M}'\in\mathscr{S}^{(n+1)}}\,\widehat{r}^{(n+1)}(\mathcal{M},\mathcal{M}')\,[\,\boldsymbol{h}(\mathcal{M}')\,-\,\boldsymbol{h}(\mathcal{M})\,]\;,\\
(\mathfrak{L}^{(n+1)}\boldsymbol{h})(\mathcal{M})\;:=\;\sum_{\mathcal{M}'\in\mathscr{V}^{(n+1)}}\,r^{(n+1)}(\mathcal{M},\mathcal{M}')\,[\,\boldsymbol{h}(\mathcal{M}')\,-\,\boldsymbol{h}(\mathcal{M})\,]\;.
\end{gathered}
\label{eq:chain_p+1}
\end{equation}
By summing up \eqref{eq:VN_p+1}, \eqref{eq:depth_p+1}, and \eqref{eq:chain_p+1},
we can complete the inductive construction of $\Lambda^{(n+1)}$.

Denote by \textcolor{blue}{$\mathfrak{n}_{n+1}$}
the number of the irreducible classes of the Markov
chain ${\bf y}^{(n+1)}(\cdot)$ and by ${\color{blue}\mathscr{R}_{1}^{(n+1)},\dots,\mathscr{R}_{\mathfrak{n}_{n+1}}^{(n+1)}}$,
${\color{blue}\mathscr{T}^{(n+1)}}$ be the closed irreducible classes
and the transient states of ${\bf y}^{(n+1)}(\cdot)$, respectively.
If $\mathfrak{n}_{n+1}=1$, the construction is completed. If $\mathfrak{n}_{n+1}>1$,
we add a new layer in the same way.

\begin{example*}
[Conclusion of example in Section \ref{sec2_example}]
Based on the construction above, we can check that
the jumps of the Markov chain $\widehat{\mathbf{y}}^{(2)}(\cdot)$
occur along the following arrows:
\[
\{\boldsymbol{m}_{1},\,\boldsymbol{m}_{2},\,\boldsymbol{m}_{3}\}\leftarrow\{\boldsymbol{m}_{4}\}\leftrightarrow\{\boldsymbol{m}_{5},\,\boldsymbol{m}_{6}\}\;\;\;\;\{\boldsymbol{m}_{7}\}\rightarrow\{\boldsymbol{m}_{8}\}\leftrightarrow\{\boldsymbol{m}_{9},\,\boldsymbol{m}_{10}\}\;,
\]
and hence the jumps of the Markov chain $\mathbf{y}^{(2)}(\cdot)$
occur along
\[
\{\boldsymbol{m}_{1},\,\boldsymbol{m}_{2},\,\boldsymbol{m}_{3}\}\leftarrow\{\boldsymbol{m}_{5},\,\boldsymbol{m}_{6}\}\;\;\;\;\{\boldsymbol{m}_{8}\}\leftrightarrow\{\boldsymbol{m}_{9},\,\boldsymbol{m}_{10}\}\;.
\]
There are two recurrent classes $\{\boldsymbol{m}_{1},\,\boldsymbol{m}_{2},\,\boldsymbol{m}_{3}\}$
and $\left\{ \{\boldsymbol{m}_{8}\},\,\{\boldsymbol{m}_{9},\,\boldsymbol{m}_{10}\}\right\} $
and we get
\begin{align}
\mathscr{V}^{(3)} & \ =\ \left\{ \,\{\boldsymbol{m}_{1},\,\boldsymbol{m}_{2},\,\boldsymbol{m}_{3}\},\,\{\boldsymbol{m}_{8},\,\boldsymbol{m}_{9},\,\boldsymbol{m}_{10}\}\,\right\} \;,\label{eq:v^3}\\
\mathscr{N}^{(3)} & \ =\ \left\{ \,\{\boldsymbol{m}_{4}\},\,\{\boldsymbol{m}_{5},\,\boldsymbol{m}_{6}\},\,\{\boldsymbol{m}_{7}\}\,\right\} \;.\nonumber
\end{align}
One can readily check that $\Xi(\{\boldsymbol{m}_{1},\,\boldsymbol{m}_{2},\,\boldsymbol{m}_{3}\})=\Xi(\{\boldsymbol{m}_{8},\,\boldsymbol{m}_{9},\,\boldsymbol{m}_{10}\})=U(\boldsymbol{\sigma_{3}})-U(\boldsymbol{m}_{1})$,
and hence this depth is $d^{(3)}$, as illustrated in Figure \ref{fig:potential}.

We can further notice that the Markov chain $\mathbf{y}^{(3)}(\cdot)$
can oscillate between $\{\boldsymbol{m}_{1},\,\boldsymbol{m}_{2},\,\boldsymbol{m}_{3}\}$
and $\{\boldsymbol{m}_{8},\,\boldsymbol{m}_{9},\,\boldsymbol{m}_{10}\}$,
and hence there is only one recurrent class at this level and thus
we cannot proceed further. This implies that $\theta_{\epsilon}^{(3)}$
is the last metastable scale.
\end{example*}

\subsection{\label{sec4.3}Properties of tree structure}

We now present the main properties of the recursive construction carried
out above. We wish to prove the following properties for each metastable
scale.
\begin{defn}[Properties of $\Lambda^{(n)}$]
\label{def:prop}Suppose that $\Lambda^{(1)},\,\dots,\,\Lambda^{(n)}$
are constructed according to the procedure explained in Sections \ref{sec4.1}
and \ref{sec4.2}. Then, we define property $\mathfrak{P}_{i}(n)$,
$i\in\llbracket1,\,4\rrbracket$, of $\Lambda^{(n)}$ as follows.
\begin{itemize}
\item \textcolor{blue}{$\mathfrak{P}_{1}(n)$}: For all $\mathcal{M}\in\mathscr{S}^{(n)}$,
the set $\mathcal{M}$ is a simple bound set.
\item \textcolor{blue}{$\mathfrak{P}_{2}(n)$}: $d^{(n-1)}<d^{(n)}<\infty$,
where $d^{(0)}=0$.
\item \textcolor{blue}{$\mathfrak{P}_{3}(n)$}: For $\mathcal{M},\,\mathcal{M}'\in\mathscr{S}^{(n)}$,
we have $\widehat{r}^{(n)}(\mathcal{M},\,\mathcal{M}')>0$ if and
only if $\Xi(\mathcal{M})\le d^{(n)}$ and $\mathcal{M}\to\mathcal{M}'$.
\item \textcolor{blue}{$\mathfrak{P}_{4}(n)$}: $\mathcal{M}\in\mathscr{V}^{(n)}$\textbf{
}is an absorbing state of the Markov chain $\mathbf{y}^{(n)}(\cdot)$
if and only if $\Xi(\mathcal{M})>d^{(n)}$.
\end{itemize}
If the property $\mathfrak{P}_{i}(k)$, $i\in\llbracket1,\,4\rrbracket$,
holds for all $k\in\llbracket1,\,n\rrbracket$, we declare that the
property \textcolor{blue}{$\mathfrak{P}_{i}\llbracket n\rrbracket$
}holds. In addition, we say that the property \textcolor{blue}{$\mathfrak{P}_{i_{1},\,\dots,\,i_{k}}\llbracket n\rrbracket$}
holds if properties$\mathfrak{P}_{i_{1}}\llbracket n\rrbracket,\,\dots,\,\mathfrak{P}_{i_{k}}\llbracket n\rrbracket$
hold simultaneously. Finally, we use a shorthand ${\color{blue}\mathfrak{P}\llbracket n\rrbracket}=\mathfrak{P}_{1,\,2,\,3,\,4}\llbracket n\rrbracket$.
\end{defn}

\begin{rem}
\label{rem:nab}We can observe from definition \eqref{eq:depth_p+1}
that, for all $\mathcal{M}\in\mathscr{V}^{(n)}$, we have $\Xi(\mathcal{M})\ge d^{(n)}$.
Hence, $\mathfrak{P}_{4}(n)$ implies that $\mathcal{M}\in\mathscr{V}^{(n)}$\textbf{
}is a non-absorbing state of the Markov chain $\mathbf{y}^{(n)}(\cdot)$
if and only if $\Xi(\mathcal{M})=d^{(n)}$.
\end{rem}

We can first show that the properties given in Definition \ref{def:prop}
hold for the first scale.
\begin{prop}
\label{prop_init}The property $\mathfrak{P}(1)$ holds.
\end{prop}

The proof of this proposition is given at the last subsection of the
current section. Now we prove the properties given in Definition \ref{def:prop}
for any scale in an inductive manner.

We prove in Section \ref{sec6} the following theorem. Recall that
$\mathfrak{n}_{n}$ denotes the number of recurrent classes of the
Markov chain $\mathbf{y}^{(n)}(\cdot)$. If $\mathfrak{n}_{n}\ge2$
we can construct the $(n+1)$-th scale as in the previous section,
and $\mathfrak{n}_{n}=|\mathscr{V}^{(n+1)}|$.
\begin{thm}
\label{t:tree}Suppose that $\Lambda^{(1)},\,\dots,\,\Lambda^{(n)}$
have been constructed by the procedure explained in Sections \ref{sec4.1}
and \ref{sec4.2}, and that the property $\mathfrak{P}\llbracket n\rrbracket$
is in force. Suppose that $\mathfrak{n}_{n}\ge2$. Then,
\begin{enumerate}
\item The property $\mathfrak{P}_{1}(n+1)$ holds and therefore we can construct
$\Lambda^{(n+1)}$
\item The property $\mathfrak{P}_{2,\,3,\,4}\llbracket n+1\rrbracket$ is
in force.
\item $\mathfrak{n}_{n+1}<\mathfrak{n}_{n}$
\end{enumerate}
\end{thm}

In particular, by the induction, we can verify that the properties
$\mathfrak{P}_{1}(n)$-$\mathfrak{P}_{4}(n)$ hold for all metastable
scale constructed above. The proof of this theorem will be given in
Sections \ref{sec5} and \ref{sec6}. In the remainder of the section,
we assume this theorem and discuss its important consequences.

According to Theorem \ref{t:tree}-(3), we can find $n_{0}$ such
that, we can construct $\Lambda^{(1)},\,\dots,\,\Lambda^{(n_{0})}$
as in in Sections \ref{sec4.1} and \ref{sec4.2}, and
\[
\mathfrak{n}_{1}\,>\,\cdots\,>\,\mathfrak{n}_{n_{0}-1}\,>\,\mathfrak{n}_{n_{0}}\,=\,1\;.
\]
We let \textcolor{blue}{$\mathfrak{q}=n_{0}$} so that $\Lambda^{(\mathfrak{q})}$
corresponds to the \emph{last metastable scale}. For this scale, as
$\mathfrak{n}_{\mathfrak{q}}=1$, the chain $\mathbf{y}^{(\mathfrak{q})}(\cdot)$
has only one closed irreducible class and the construction explained
in Section \ref{sec4.2} cannot proceed further (cf. Section \ref{sec4.5}
below). Now, we can combine Proposition \ref{prop_init} and Theorem
\ref{t:tree} in the following manner.
\begin{cor}
\label{cor:tree}We can construct $\Lambda^{(1)},\,\dots,\,\Lambda^{(\mathfrak{q})}$
and the property $\mathfrak{P}\llbracket\mathfrak{q}\rrbracket$ holds.
\end{cor}

\begin{proof}
The proof is straightforward from the induction based on Proposition
\ref{prop_init} and Theorem \ref{t:tree}.
\end{proof}
The next result which fully characterizes the elements in $\mathcal{M}\in\mathscr{S}^{(n)}$
in terms of the depth $\Xi(\mathcal{M})$ is a straightforward consequence
of the previous corollary.
\begin{prop}
\label{prop:depth} Let $n\in\llbracket1,\,\mathfrak{q}\rrbracket$
and $\mathcal{M}\in\mathscr{S}^{(n)}$. Then,
\begin{equation}
\begin{cases}
\Xi(\mathcal{M})\ <\ d^{(n)} & \text{iff }\text{\ensuremath{\mathcal{M}\in}}\ensuremath{\mathscr{N}^{(n)}}\;,\\
\Xi(\mathcal{M})\ =\ d^{(n)} & \text{iff }\text{\ensuremath{\mathcal{M}\in}}\ensuremath{\mathscr{V}^{(n)}}\;\text{and}\ \mathcal{M}\ \text{is not an absorbing state of \ensuremath{\mathbf{y}^{(n)}(\cdot)}}\;,\\
\Xi(\mathcal{M})\ >\ d^{(n)} & \text{iff }\text{\ensuremath{\mathcal{M}\in}}\ensuremath{\mathscr{V}^{(n)}}\;\text{and}\ \mathcal{M}\ \text{is an absorbing state of \ensuremath{\mathbf{y}^{(n)}(\cdot)}}\;.
\end{cases}\label{char}
\end{equation}
In particular, the set of absorbing states of ${\bf y}^{(n)}(\cdot)$
and that of $\widehat{{\bf y}}^{(n)}(\cdot)$ are identical.
\end{prop}

\begin{proof}
Fix $n\in\llbracket1,\,\mathfrak{q}\rrbracket$ and $\mathcal{M}\in\mathscr{S}^{(n)}$.
Assume that $\mathcal{M}\in\mathscr{V}^{(n)}$. By \eqref{eq:depth_p+1},
$\Xi(\mathcal{M})\ge d^{(n)}$. The second and third claims of \eqref{char}
follow from $\text{\ensuremath{\mathfrak{P}_{4}(n)}}$. Suppose that
$\mathcal{M}\in\mathscr{N}^{(n)}$. By Remark \ref{rem:N_p+1}, $\mathcal{M}\in\mathscr{T}^{(k)}$
for some $k<n$. Hence, $\mathcal{M}\in\mathscr{V}^{(k)}$ and $\mathcal{M}$
is not an absorbing state of ${\bf y}^{(k)}(\cdot)$, and thus by
the second claim of \eqref{char} which is already proven, we have
\[
\Xi(\mathcal{M})\ =\ d^{(k)}\ <\ d^{(n)}\;,
\]
where the last inequality follows from $\mathfrak{P}_{2}\llbracket\mathfrak{q}\rrbracket$.
The last assertion of the Corollary follows from the first one.
\end{proof}
\begin{rem}
We can readily check that, if $\mathcal{M}\subset\mathcal{M}_{0}$
is a simple bound set such that $\widetilde{\mathcal{M}}\ne\varnothing$,
then there exists $p\in\llbracket1,\,\mathfrak{q}\rrbracket$ such
that $\mathcal{M}\in\mathscr{V}^{(p)}$ and $\Xi(\mathcal{M})=d^{(p)}$.
This explains that our construction completely captures all the metastable
regime.
\end{rem}

\subsection{\label{sec4.5} Construction at level $\mathfrak{q}+1$}

One can notice that, the construction carried out in Section \ref{sec4.2}
can be continued to $n=\mathfrak{q}$ to some extent, since the fact
$\mathfrak{n}_{n}\ge2$ was not used in the construction of $\mathscr{V}^{(n+1)}$
and $\mathscr{N}^{(n+1)}$. We denote by $\mathscr{R}_{1}^{(\mathfrak{q})}$
and $\mathscr{T}^{(\mathfrak{q})}$ the unique recurrent class (as
$n_{\mathfrak{q}}=1$) and the transient states of the Markov chain
${\bf y}^{(\mathfrak{q})}(\cdot)$. Then, we set
\begin{equation}
{\color{blue}\mathscr{V}^{(\mathfrak{q}+1)}}\ :=\ \Big\{\,\cup_{\mathcal{M}\in\mathscr{R}_{1}^{(\mathfrak{q})}}\mathcal{M}\,\Big\}\;\;\;\text{and}\;\;\;{\color{blue}\mathscr{N}^{(\mathfrak{q}+1)}}\ :=\ \mathscr{N}^{(\mathfrak{q})}\cup\mathscr{T}^{(\mathfrak{q})}\label{eq:def_q+1_1}
\end{equation}
as in Section \ref{sec4.2}. Write ${\color{blue}\mathscr{S}^{(\mathfrak{q}+1)}}=\mathscr{V}^{(\mathfrak{q}+1)}\cup\mathscr{N}^{(\mathfrak{q}+1)}$.
With these definitions, we can define the property $\mathfrak{P}_{1}(\mathfrak{q}+1)$
as in Definition \ref{def:prop}.

One can note from \eqref{eq:v^3} that the unique
element of $\mathscr{V}^{(4)}$ in that example is $\{\boldsymbol{m}_{1},\,\boldsymbol{m}_{2},\,\boldsymbol{m}_{3},\,\boldsymbol{m}_{8},\,\boldsymbol{m}_{9},\,\boldsymbol{m}_{10}\}$
which is the set of all global minima of $U$. This is the consequence
of the next proposition.
\begin{prop}
\label{prop_global_min}The following hold.
\begin{enumerate}
\item The property $\mathfrak{P}_{1}(\mathfrak{q}+1)$ holds.
\item $\mathscr{V}^{(\mathfrak{q}+1)}=\{\mathcal{M}_{\star}\}$, where we
recall that $\mathcal{M}_{\star}$ denotes the set of all global minima
of $U$.
\end{enumerate}
\end{prop}

The proof of this proposition is given in Section \ref{sec6.1}. In
view of assertion (2) of the last proposition, we cannot define the
depth $d^{(\mathfrak{q}+1)}$ as in \eqref{eq:depth_p+1}, as $\Xi(\mathcal{M}_{\star})=\infty$.

\subsection{Proof of Proposition \ref{prop_init}}

We conclude this section with proof of Proposition \ref{prop_init}.
We remark that the proof replies on several results proved in our
companion paper \cite{LLS-1st} which indeed handles the first time
scale of the current problem. We summarized the mainly referred results
at Section \ref{sec:paper-1}.
\begin{proof}[Proof of Proposition \ref{prop_init}]
The property $\mathfrak{P}_{1}(1)$ is in force because the elements
of $\mathscr{V}^{(1)}$ are singletons (as we declared that singletons
are bound). The property $\mathfrak{P}_{2}(1)$ is also immediate
from the definition of $d^{(1)}$ and non-degeneracy of the potential
function $U$.

For the property $\mathfrak{P}_{3}(1)$, we can observe
from \eqref{eq:omega_2} and \eqref{40} that $\widehat{r}^{(1)}(\{\boldsymbol{m}\},\{\boldsymbol{m}'\})>0$
if and only if $\Xi(\boldsymbol{m})=d^{(1)}$ and $\boldsymbol{\sigma}\curvearrowright\boldsymbol{m}\;,\;\;\boldsymbol{\sigma}\curvearrowright\boldsymbol{m}'$
for some $\bm{\sigma}\in\mathcal{S}^{(1)}(\bm{m})$. Therefore, it
suffices to prove $\Xi(\boldsymbol{m})=d^{(1)}$ and $\bm{m}\curvearrowleft\boldsymbol{\sigma}\curvearrowright\boldsymbol{m}'$
for some $\bm{\sigma}\in\mathcal{S}^{(1)}(\bm{m})$ if and only if
$\Xi(\boldsymbol{m})=d^{(1)}$ and $\bm{m}\to\bm{m}'$.

Suppose that $\Xi(\boldsymbol{m})=d^{(1)}$ and $\bm{m}\curvearrowleft\boldsymbol{\sigma}\curvearrowright\boldsymbol{m}'$.
We claim $U(\bm{m})<U(\bm{m}')$. Indeed, if not,
\[
\Xi(\bm{m}')\ =\ \Theta(\bm{m}',\,\widetilde{\bm{m}'})-U(\bm{m}')\ <\ \Theta(\bm{m}',\,\bm{m})-U(\bm{m})\ .
\]
By Lemma \ref{lem_not}-(3), $\Theta(\bm{m}',\,\bm{m})\le U(\bm{\sigma})$
so that since $\bm{\sigma}\in\mathcal{S}^{(1)}(\bm{m})$ satisfies
$U(\bm{\sigma})=U(\bm{m})+\Xi(\bm{m})$,
\[
\Theta(\bm{m}',\,\bm{m})-U(\bm{m})\ \le\ U(\bm{\sigma})-U(\bm{m})\ =\ \Xi(\bm{m})\ ,
\]
which contradicts to the definition of $d^{(1)}$. Therefore, $U(\bm{m})\ge U(\bm{m}')$.
Since $\bm{\sigma}\in\mathcal{S}^{(1)}(\bm{m})$ and $U(\bm{m})\ge U(\bm{m}')$,
we have
\[
U(\bm{\sigma})\ =\ \Theta(\bm{m},\,\widetilde{\bm{m}})\ \le\ \Theta(\bm{m},\,\bm{m}')\ \le\ U(\bm{\sigma})\ ,
\]
which implies $\bm{m}\to\bm{m}'$.

On the other hand, suppose that $\Xi(\boldsymbol{m})=d^{(1)}$
and $\bm{m}\to\bm{m}'$. Then, by definition, $\bm{m}\to_{\bm{\sigma}}\bm{m}'$
for some $\bm{\sigma}\in\mathcal{S}_{0}$ so that
\[
\boldsymbol{\sigma}\curvearrowright\boldsymbol{m}\;\text{\;\;and}\ \;\ U(\bm{\sigma})\ =\ \Theta(\bm{m},\,\widetilde{\bm{m}})\ =\ U(\bm{m})+\Xi(\bm{m})\ .
\]
Therefore, we can conclude that $\bm{\sigma}\in\mathcal{S}^{(1)}(\bm{m})$.

It remains to prove $\mathfrak{P}_{4}(1)$. It is clear from \eqref{40}
that $\{\boldsymbol{m}\}\in\mathscr{V}^{(1)}$ is an absorbing state
of the Markov chain $\mathbf{y}^{(1)}(\cdot)=\mathbf{\widehat{y}}^{(1)}(\cdot)$
if $\Xi(\{\boldsymbol{m}\})>d^{(1)}$. Thus, the next claim completes
the proof of the proposition.

\smallskip{}
\textbf{\textit{Claim}}: Suppose that $\{\boldsymbol{m}\}\in\mathscr{V}^{(1)}$
satisfies $\Xi(\{\boldsymbol{m}\})=d^{(1)}$. Then, $\{\boldsymbol{m}\}$
is not an absorbing state of chain $\mathbf{y}^{(1)}(\cdot)$. \smallskip{}

\noindent Let $\mathcal{W}_{1}$ be the connected component $\{\,U<\Theta(\boldsymbol{m},\,\widetilde{\{\boldsymbol{m}\}})\,\}$
containing $\boldsymbol{m}$. We affirm that $\boldsymbol{m}$ is
the unique local minimum in $\mathcal{W}_{1}$. Suppose that there
is another local minimum $\boldsymbol{m}'$ in $\mathcal{W}_{1}$.
By Lemma \ref{lap01}-(1),
\begin{equation}
\Theta(\boldsymbol{m},\,\boldsymbol{m}')\ <\ \Theta(\boldsymbol{m},\,\widetilde{\{\boldsymbol{m}\}})\;.\label{etta}
\end{equation}
If $U(\boldsymbol{m})\ge U(\boldsymbol{m}')$, we get a contradiction
to \eqref{etta} in view of Lemma \ref{lem_not}-(1) since $\boldsymbol{m}'\in\widetilde{\{\boldsymbol{m}\}}$.
On the other hand, suppose that $U(\boldsymbol{m})<U(\boldsymbol{m}')$.
In this case, $\boldsymbol{m}\in\widetilde{\{\boldsymbol{m}'\}}$
so that, by Lemma \ref{lem_not}-(1),
\begin{align*}
\Xi(\boldsymbol{m}')\ =\ \Theta(\boldsymbol{m}',\,\widetilde{\{\boldsymbol{m}'\}})-U(\boldsymbol{m}') & \ <\ \Theta(\boldsymbol{m}',\,\boldsymbol{m})-U(\boldsymbol{m})\\
 & \ <\ \Theta(\boldsymbol{m},\,\widetilde{\{\boldsymbol{m}\}})-U(\boldsymbol{m})\,=\,\Xi(\boldsymbol{m})\,=\,d^{(1)}\;,
\end{align*}
where the last inequality follows from \eqref{etta}. Therefore, we
get $\Xi(\boldsymbol{m}')<d^{(1)}$, in contradiction with the definition
of $d^{(1)}$. Hence, $\boldsymbol{m}$ is the unique local minimum
in $\mathcal{W}_{1}$.

By \cite[Lemma A.2]{LLS-1st}, there exists a connected component
of $\{U<\Theta(\boldsymbol{m},\,\widetilde{\boldsymbol{m}})\}$, denoted
by of $\mathcal{W}_{2}$, such that
\[
\mathcal{W}_{1}\cap\mathcal{W}_{2}\ =\ \varnothing\ \ ,\ \ \overline{\mathcal{W}_{1}}\cap\overline{\mathcal{W}_{2}}\ \ne\ \varnothing\ .
\]
Let $\boldsymbol{\sigma}\in\overline{\mathcal{W}_{1}}\cap\overline{\mathcal{W}_{2}}$.
By Lemma \ref{l_cap_saddle}, $\boldsymbol{\sigma}$ is a saddle point
connecting $\mathcal{W}_{1}$ and $\mathcal{W}_{2}$ such that $U(\boldsymbol{\sigma})=\Theta(\boldsymbol{m},\,\widetilde{\boldsymbol{m}})$.
Furthermore, by Lemma \ref{l_assu_saddle}-(1), there are two local
minima $\boldsymbol{m}_{1}\in\mathcal{W}_{1}$ and $\boldsymbol{m}_{2}\in\mathcal{W}_{2}$
such that $\boldsymbol{m}_{1}\curvearrowleft\boldsymbol{\sigma}\curvearrowright\boldsymbol{m}_{2}$.
Since $\boldsymbol{m}$ is the unique local minimum in $\mathcal{W}_{1}$
and $\bm{m}_{2}\in\mathcal{W}_{2}$, $\boldsymbol{m}_{1}=\boldsymbol{m}$
and $\boldsymbol{m}_{2}\ne\bm{m}$. Hence, by $\mathfrak{P}_{3}(1)$
proved above, we get $r^{(1)}(\boldsymbol{m},\,\boldsymbol{m}_{2})=\widehat{r}^{(1)}(\boldsymbol{m},\,\boldsymbol{m}_{2})>0$
and thus $\boldsymbol{m}$ is not an absorbing state of chain $\mathbf{y}^{(1)}(\cdot)$.
This proves the claim.
\end{proof}

\section{Analysis of level sets\label{sec5}}

In this section, we fix $n\ge2$ and suppose that $\Lambda^{(1)},\,\dots,\,\Lambda^{(n)}$
are constructed and the property $\mathfrak{P}_{1,\,2,\,3}\llbracket n\rrbracket$
holds and investigate several consequences of these assumptions. In
the next section, under the same assumptions, we construct $\Lambda^{(n+1)}$
and prove that the property $\mathfrak{P}\llbracket n+1\rrbracket$
holds as well.

More precisely, our purpose in this section is, before proceeding
to prove $\mathfrak{P}_{i}(n+1)$, $i\in\llbracket1,\,4\rrbracket$,
to analyze the level sets of the form \eqref{eq:level}.
\begin{defn}
\label{def1}The following notions will be frequently used throughout
the article and are main concern of the current section.
\begin{enumerate}
\item We say that a set $\mathcal{A\subset\mathbb{R}}^{d}$\textcolor{blue}{\emph{
does not separates $(n)$-states}} if, for all $\mathcal{M}\in\mathscr{S}^{(n)}$,
either $\mathcal{M}\subset\mathcal{A}$ or $\mathcal{M}\subset\mathcal{A}^{c}$.
\item For a set $\mathcal{A\subset\mathbb{R}}^{d}$, we write $\mathcal{M}^{*}(\mathcal{A})$
the set of minima of $U$ in $\mathcal{A}$, i.e.,
\[
{\color{blue}\mathcal{M}^{*}(\mathcal{A})}\ :=\ \bigg\{\,\boldsymbol{m}\in\mathcal{M}_{0}:U(\boldsymbol{m})=\min_{\boldsymbol{x}\in\mathcal{A}}U(\boldsymbol{x})\,\bigg\}\;.
\]
Note that, if $\mathcal{A}$ is a non-empty level set given in \eqref{eq:level},
the set $\mathcal{M}^{*}(\mathcal{A})$ is non-empty.
\end{enumerate}
\end{defn}

\subsection{Preliminary results}

The next lemma investigate properties of recurrent classes of Markov
chain $\mathbf{y}^{(n)}(\cdot)$.
\begin{lem}
\label{lem_stddec}Let $\{\mathcal{M}_{1},\,\dots,\,\mathcal{M}_{a}\}$
for some $a\ge2$ be a recurrent class of Markov chain $\mathbf{y}^{(n)}(\cdot)$.
Then, the followings hold.
\begin{enumerate}
\item We have $\Xi(\mathcal{M}_{1})=\cdots=\Xi(\mathcal{M}_{a})=d^{(n)}$.
\item We have $U(\mathcal{M}_{1})=\cdots=U(\mathcal{M}_{a})$. In particular,
$\mathcal{M}$ is simple.
\item For all $i\neq j\in\llbracket1,\,a\rrbracket$,
\[
\Theta(\mathcal{M}_{i},\,\mathcal{M}_{j})\ =\ \Theta(\mathcal{M}_{i},\,\widetilde{\mathcal{M}_{i}})\ =\ U(\mathcal{M}_{i})+d^{(n)}\;.
\]
\end{enumerate}
\end{lem}

\begin{proof}
(1) Since $a\ge2$, each $\mathcal{M}_{i}$ is not an absorbing state
of $\mathbf{y}^{(n)}(\cdot)$ and consequently, it is not an absorbing
state for $\widehat{{\bf y}}^{(n)}(\cdot)$ either. Hence, $\Xi(\mathcal{M}_{i})\le d^{(n)}$
by $\mathfrak{P}_{3}(n)$. The assertion follows the definition \eqref{eq:depth_p+1}
of $d^{(n)}$.

\smallskip{}
(2, 3) Since $\{\mathcal{M}_{1},\,\dots,\,\mathcal{M}_{a}\}$ forms
a recurrent class of chain $\mathbf{y}^{(n)}(\cdot)$, there exist
$\mathcal{M}_{1}',\,\dots,\,\mathcal{M}_{b}'\in\mathscr{S}^{(n)}$
such that
\[
\widehat{r}^{(n)}(\mathcal{M}_{1},\mathcal{M}'_{1})\,>\,0\;,\;\;\widehat{r}^{(n)}(\mathcal{M}'_{i},\mathcal{M}'_{i+1})\,>\,0\text{ for }i\in\llbracket1,\,b-1\rrbracket\;,\text{ and \;\;}\widehat{r}^{(n)}(\mathcal{M}'_{b},\mathcal{M}_{2})\,>\,0\;.
\]
Thus, by $\mathfrak{P}_{3}(n)$, we have
\[
\mathcal{M}_{1}\,\to\,\mathcal{M}_{1}'\,\to\,\cdots\,\to\,\mathcal{M}_{b}'\,\to\,\mathcal{M}_{2}\;,
\]
and $\Xi(\mathcal{M}_{i}')\le d^{(n)}$ for all $i\in\llbracket1,\,b\rrbracket$.
Since $\Xi(\mathcal{M}_{1})=\Xi(\mathcal{M}_{2})=d^{(n)}$ by assertion
(1), by Lemma \ref{l_M->M}-(2),
\[
U(\mathcal{M}_{1})\ \ge\ U(\mathcal{M}_{2})\;\;\;\;\text{and\;\;\;\;}\Theta(\mathcal{M}_{1},\,\widetilde{\mathcal{M}_{1}})\ =\ \Theta(\mathcal{M}_{1},\,\mathcal{M}_{2})\;.
\]
Assertion (3) follows from the last one. Assertion (2) follows from
the first one since we are free to interchange the role of $\mathcal{M}_{1}$
and $\mathcal{M}_{2}$ to get $U(\mathcal{M}_{1})\le U(\mathcal{M}_{2})$.
\end{proof}
Next we prove that $\mathcal{M}\in\mathscr{V}^{(n)}$ is not only
bound but also satisfies a stronger inequality under $\mathfrak{P}_{2}(n)$.
\begin{lem}
\label{lem:bound_dn}For $\mathcal{M}\in\mathscr{V}^{(n)}$ such that
$|\mathcal{M}|\ge2$, we have that
\begin{equation}
\max_{\boldsymbol{m},\,\boldsymbol{m}'\in\mathcal{M}}\,\Theta(\boldsymbol{m},\,\boldsymbol{m}')\ \le\ U(\mathcal{M})+d^{(n-1)}\;.\label{eq:bound_dn}
\end{equation}
\end{lem}

\begin{proof}
Fix $\mathcal{M}\in\mathscr{V}^{(n)}$ such that $|\mathcal{M}|\ge2$.
By the construction, there is $k\in\llbracket1,\,n-1\rrbracket$ such
that $\mathcal{M}\notin\mathscr{V}^{(k)}$ and $\mathcal{M}\in\mathscr{V}^{(k+1)}$.
Write $\mathscr{V}^{(k)}(\mathcal{M})=\{\mathcal{M}_{1},\,\dots,\,\mathcal{M}_{a}\}$
(cf. Remark \ref{rem:stddec}). Since $\mathcal{M}\notin\mathscr{V}^{(k)}$,
we have $a\ge2$. Then,
\[
\max_{\boldsymbol{m},\boldsymbol{m}'\in\mathcal{M}}\,\Theta(\boldsymbol{m},\,\boldsymbol{m}')\ =\ \max\bigg\{\,\max_{i,\,j\in\llbracket1,\,a\rrbracket}\,\Theta(\mathcal{M}_{i},\,\mathcal{M}_{j}),\,\max_{i\in\llbracket1,\,a\rrbracket}\,\max_{\boldsymbol{m},\boldsymbol{m}'\in\mathcal{M}_{i}}\,\Theta(\boldsymbol{m},\,\boldsymbol{m}')\,\bigg\}\;.
\]
By Lemma \ref{lem_stddec}-(3) (with $k$ instead of $n$), we have
\[
\Theta(\mathcal{M}_{i},\,\mathcal{M}_{j})\ =\ U(\mathcal{M})+d^{(k)}\;\;\;\text{for all }i,\,j\in\llbracket1,\,a\rrbracket\;.
\]
On the other hand, by $\mathfrak{P}_{2}\llbracket n\rrbracket$ and
Lemma \ref{lem_stddec}-(2, 3), we have, for all $i\in\llbracket1,\,a\rrbracket$
and $\boldsymbol{m},\,\boldsymbol{m}'\in\mathcal{M}_{i}$,
\[
\Theta(\boldsymbol{m},\,\boldsymbol{m}')\ <\ \Theta(\mathcal{M}_{i},\,\widetilde{\mathcal{M}_{i}})\ =\ U(\mathcal{M}_{i})+d^{(k)}\ \le\ U(\mathcal{M})+d^{(n-1)}\;.
\]
Summing up three displays above completes the proof.
\end{proof}

\subsection{Analysis of open level sets}

Recall that $\mathbb{R}^{d}$ is a locally connected space. Therefore,
the connected components of an open set are open. On the other hand,
open subsets of $\mathbb{R}^{d}$ are connected if, and only if, they
are path connected. These facts are used below without further notice.

In this subsection, we fix $H\in\mathbb{R}$ such that the level set
$\{U<H\}$ is not empty. Then, we fix a connected component\textcolor{blue}{
$\mathcal{H}$ }of the level set $\{U<H\}$. The
purpose is to analyze the properties of $\mathcal{H}$. We note here
that $\mathcal{M}^{*}(\mathcal{H})\neq\varnothing$ by the definition
of $\mathcal{H}$.

The first lemma analyze the situation where $\mathcal{H}$ separates
an $(n)$-state. Remind that we assume that the property $\mathfrak{P}\llbracket n\rrbracket$
is in force throughout the section.
\begin{lem}
\label{lem_levmin}Suppose that $\mathcal{H}$ separate an $(n)$-state,
i.e., there exists $\mathcal{M}\in\mathscr{S}^{(n)}$ such that $\mathcal{M}\cap\mathcal{H}\neq\varnothing$
and $\mathcal{M}\cap\mathcal{H}^{c}\neq\varnothing$. Then, $\mathcal{M}^{*}(\mathcal{H})\subset\mathcal{M}.$
\end{lem}

\begin{proof}
Let $\boldsymbol{m}\in\mathcal{M}\cap\mathcal{H}$ and let $\boldsymbol{m}'\in\mathcal{M}\cap\mathcal{H}^{c}$.
Suppose, by contradiction, that there exists $\boldsymbol{m}''\in\mathcal{M}^{*}(\mathcal{H})$
that is not an element of $\mathcal{M}$. By $\mathfrak{P}_{1}(n)$,
the set $\mathcal{M}$ is simple and we have
\[
U(\boldsymbol{m}'')\,\le\,U(\boldsymbol{m})\,=\,U(\mathcal{M})
\]
where the first inequality holds since $\boldsymbol{m}\in\mathcal{H}$
and $\boldsymbol{m}''\in\mathcal{M}^{*}(\mathcal{H})$. Thus, Lemma
\ref{lem_not}-(1), we have
\[
\Theta(\mathcal{M},\,\widetilde{\mathcal{M}})\ \le\ \Theta(\mathcal{M},\,\boldsymbol{m}'')\ \le\ \Theta(\boldsymbol{m},\,\boldsymbol{m}'')\ <\ H\ \le\ \Theta(\boldsymbol{m},\,\boldsymbol{m}')\;,
\]
where the last two inequalities follow from Lemma \ref{lap01}-(1)
and (2), respectively. Thus, we get a contradiction to the fact that
$\mathcal{M}$ is bound by $\mathfrak{P}_{1}(n)$.
\end{proof}
Next, we seek sufficient conditions for $\mathcal{H}$ not to separate
the $(n)$-states.
\begin{lem}
\label{lem_separate}The level set $\mathcal{H}$ does not separate
$(n)$-states if one of the following hold:
\begin{enumerate}
\item $\min_{\boldsymbol{x}\in\mathcal{H}}U(\boldsymbol{x})\le H-d^{(n)}$
\item $\mathcal{M}\subset\mathcal{H}$ for some $\mathcal{M}\in\mathscr{V}^{(n)}$
\end{enumerate}
\end{lem}

\begin{proof}
Suppose to the contrary that there exists $\mathcal{M}'\in\mathscr{S}^{(n)}$
containing two minima $\boldsymbol{m},\,\boldsymbol{m}'\in\mathcal{M}'$
such that $\boldsymbol{m}\in\mathcal{H}$ and $\boldsymbol{m}'\notin\mathcal{H}$.
By Lemma \ref{lem_levmin}, $\mathcal{M}^{*}(\mathcal{H})\subset\mathcal{M}'$.

\smallskip{}

\noindent (1) Since $\mathcal{M}'$ is simple by $\mathfrak{P}_{1}(n)$
and $\mathcal{M}^{*}(\mathcal{H})\subset\mathcal{M}'$, we have
\[
U(\mathcal{M}')\ =\ U(\boldsymbol{m})\ \le\ H-d^{(n)}
\]
by the assumption of part (1). Then, since $\Theta(\boldsymbol{m},\,\boldsymbol{m}')\ge H$
by Lemma \ref{lap01}-(2), by $\mathfrak{P}_{2}(n)$, we have
\[
\Theta(\boldsymbol{m},\,\boldsymbol{m}')-U(\mathcal{M}')\ \ge\ d^{(n)}\ >\ d^{(n-1)}
\]
which contradicts to Lemma  \ref{lem:bound_dn}.

\smallskip{}

\noindent (2) Since $\mathcal{M}^{*}(\mathcal{H})\subset\mathcal{M}'$,
we must have $U(\mathcal{M})>U(\mathcal{M}')$. As $\mathcal{M}\in\mathscr{V}^{(n)}$,
by definition \eqref{eq:depth_p+1} and by Lemma \ref{lap01}-(1),
\[
U(\mathcal{M}')+d^{(n)}\ \le\ U(\mathcal{M})+\Xi(\mathcal{M})\ =\ \Theta(\mathcal{M},\,\widetilde{\mathcal{M}})\ \le\ \Theta(\mathcal{M},\,\boldsymbol{m})\ <\ H\;.
\]
On the other hand, by Lemma \ref{lap01}-(2) and the previous bound
and $\mathfrak{P}_{2}(n)$, we get
\[
\Theta(\boldsymbol{m},\,\boldsymbol{m}')-U(\mathcal{M}')\ \ge\ H-U(\mathcal{M}')\ >\ d^{(n)}\ >\ d^{(n-1)}
\]
which contradicts to Lemma  \ref{lem:bound_dn}.
\end{proof}
Finally, we investigate property of $\mathcal{H}$ when it does not
separate $(n)$-states.

For $\mathcal{A}\subset\mathbb{R}^{d}$, let
\[
{\color{blue}\mathscr{S}^{(n)}(\mathcal{A})}\;:=\;\{\,\mathcal{M}\in\mathscr{S}^{(n)}\,:\,\mathcal{M}\subset\mathcal{A}\,\}
\]
and then define ${\color{blue}\mathscr{V}^{(n)}(\mathcal{A})}$ and
\textcolor{blue}{$\mathscr{N}^{(n)}(\mathcal{A})$ }in the same manner.
Note that this definition includes the one defined in Remark \ref{rem:stddec}.
\begin{lem}
\label{lem_escape}Suppose that $\mathcal{H}$ does not separate $(n)$-states.
\begin{enumerate}
\item Suppose that $\mathcal{M},\,\mathcal{M}'\in\mathscr{S}^{(n)}$ satisfy
$\mathcal{M}\subset\mathcal{H}$, $\mathcal{M}'\subset\mathcal{H}^{c}$,
and $\mathcal{M}\rightarrow\mathcal{M}'$. Then, $\mathcal{M}=\mathcal{M}^{*}(\mathcal{H})$.
\item Let $\mathcal{M}=\mathcal{M}^{*}(\mathcal{H})$ and suppose that $\Xi(\mathcal{M})\le d^{(n)}$.
Then, $\mathscr{V}^{(n)}(\mathcal{H})=\{\mathcal{M}\}$ or $\mathscr{V}^{(n)}(\mathcal{H})=\varnothing$.
If $\mathscr{V}^{(n)}(\mathcal{H})=\{\mathcal{M}\}$, $\Xi(\mathcal{M})=d^{(n)}$.
\item If $\mathscr{V}^{(n)}(\mathcal{H})=\varnothing$, $\mathcal{M}^{*}(\mathcal{H})\in\mathscr{N}^{(n)}(\mathcal{H})$.
\end{enumerate}
\end{lem}

The proof of this lemma requires two lemmata.
\begin{lem}
\label{lem_charDelta}Let $\mathcal{M}\in\mathscr{S}^{(n)}$. Then,
$\mathcal{M}\in\mathscr{N}^{(n)}$ if and only if $\Xi(\mathcal{M})\le d^{(n-1)}$.
In addition, the last condition can be relaxed to $\Xi(\mathcal{M})<d^{(n)}$,
instead of $\Xi(\mathcal{M})\le d^{(n-1)}$.
\end{lem}

\begin{proof}
Let $\mathcal{M}\in\mathscr{S}^{(n)}$ and suppose that $\Xi(\mathcal{M})\le d^{(n-1)}$
so that $\Xi(\mathcal{M})<d^{(n)}$ by $\mathfrak{P}_{2}(n)$. Then,
it is clear from the definition of \eqref{eq:depth_p+1} of $d^{(n)}$
that $\mathcal{M}\not\in\mathscr{V}^{(n)}$, and hence $\mathcal{M}\in\mathscr{N}^{(n)}$.

Conversely, assume that $\mathcal{M}\in\mathscr{N}^{(n)}$ so that
by Remark \ref{rem:N_p+1}, $\mathcal{M}\in\mathscr{T}^{(k)}$ for
some $k\in\llbracket1,\,n-1\rrbracket$. This means that $\mathcal{M}$
belongs to $\mathscr{V}^{(k)}$ (and thus is not an element of $\mathscr{N}^{(k)}$),
and is a transient state for the $\mathscr{V}^{(k)}$-valued reduced
Markov chain ${\bf y}^{(k)}(\cdot)$. This implies that $\mathcal{M}$
is also a transient state for the $\mathscr{S}^{(k)}$-valued Markov
chain $\widehat{{\bf y}}^{(k)}(\cdot)$. In particular, there exists
$\mathcal{M}'\in\mathscr{S}^{(k)}$ such that $\widehat{r}^{(k)}(\mathcal{M},\,\mathcal{M}')>0$.
As $\mathcal{M}\in\mathscr{V}^{(k)}$, by \eqref{eq:rate_30} and
\eqref{eq:rate_3}, we should have $\Xi(\mathcal{M})=d^{(k)}$. By
$\mathfrak{P}_{2}\llbracket n\rrbracket$, since $k\in\llbracket1,\,n-1\rrbracket$
we can conclude that $\Xi(\mathcal{M})\le d^{(n-1)}$. We note that
a careful reading of the proof reveals that the second assertion of
the lemma is also proven.
\end{proof}
\begin{lem}
\label{l_irred_negli}The process $\widehat{{\bf y}}^{(n)}(\cdot)$
has no recurrent class consisting only of elements of $\mathscr{N}^{(n)}$.
\end{lem}

\begin{proof}
Suppose that $\mathscr{U}=\{\mathcal{M}_{1},\,\dots,\,\mathcal{M}_{a}\}\subset\mathscr{N}^{(n)}$
is a recurrent class of $\widehat{{\bf y}}^{(n)}(\cdot)$. Since $\mathscr{N}^{(1)}=\varnothing$,
there exists $j\in\llbracket1,\,n-1\rrbracket$ such that
\[
\mathscr{U}\not\subset\mathscr{N}^{(j)}\;\;\text{and}\;\;\mathscr{U}\subset\mathscr{N}^{(j+1)}\;.
\]
Also, since $\mathscr{N}^{(j+1)}=\mathscr{N}^{(j)}\cup\mathscr{T}^{(j)}\subset\mathscr{S}^{(j)}$,
$\mathscr{U}\subset\mathscr{S}^{(j)}$.\smallskip{}
\textbf{\emph{Claim:}} \emph{If $\mathscr{U}$ is a recurrent class
of $\widehat{{\bf y}}^{(i+1)}(\cdot)$ for some $i\in\llbracket j,\,n-1\rrbracket$,
then it is a recurrent class of $\widehat{{\bf y}}^{(i)}(\cdot)$. }

\noindent To prove this claim, first note that, by \eqref{eq:rate_1}
and the fact that $\mathscr{U}\subset\mathscr{N}^{(i+1)}$,
\begin{equation}
\widehat{r}^{(i)}(\mathcal{M}',\,\mathcal{M}'')\ =\ \widehat{r}^{(i+1)}(\mathcal{M}',\,\mathcal{M}'')\;\;\;\text{for all }\mathcal{M}',\,\mathcal{M}''\in\mathscr{U}\;.\label{err0}
\end{equation}
Since $\mathscr{U}$ is a recurrent class of $\widehat{{\bf y}}^{(i+1)}(\cdot)$,
we have
\begin{equation}
\widehat{r}^{(i+1)}(\mathcal{M}',\,\mathcal{M}'')\ =\ 0\;\;\;\text{for all }\mathcal{M}'\in\mathscr{U}\;\text{and }\mathcal{M}''\notin\mathscr{U}\ .\label{eq:err0}
\end{equation}
By \eqref{eq:rate_1}, \eqref{eq:rate_2}, and \eqref{eq:err0}, we
get
\begin{equation}
\widehat{r}^{(i)}(\mathcal{M}',\,\mathcal{M}'')\ =\ 0\;\;\;\text{for all }\mathcal{M}'\in\mathscr{U}\;\text{and }\mathcal{M}''\notin\mathscr{U}\;.\label{err1-1}
\end{equation}
The proof of the claim follows from \eqref{err0} and \eqref{err1-1}.

By the claim and the hypothesis that $\mathscr{U}$ is a recurrent
class of $\widehat{{\bf y}}^{(n)}(\cdot)$, $\mathscr{U}$ is a recurrent
class of $\widehat{{\bf y}}^{(j)}(\cdot)$. Since $\mathbf{y}^{(j)}(\cdot)$
is the trace process of $\widehat{{\bf y}}^{(j)}(\cdot)$ on $\mathscr{V}^{(j)}$,
$\mathscr{U}\cap\mathscr{V}^{(j)}$ is a recurrent class of $\mathbf{y}^{(j)}(\cdot)$.
On the other hand, as $\mathscr{U}\subset\mathscr{N}^{(j+1)}=\mathscr{N}^{(j)}\cup\mathscr{T}^{(j)}$,
$\mathscr{U}\cap\mathscr{V}^{(j)}\subset\mathscr{T}^{(j)}$. Hence,
$\mathscr{U}\cap\mathscr{V}^{(j)}$ consists only of transient states
of $\mathbf{y}^{(j)}(\cdot)$ and therefore cannot be a recurrent
class of $\mathbf{y}^{(j)}(\cdot)$. This yields a contradiction.
\end{proof}
Now we are ready to prove Lemma \ref{lem_escape}.
\begin{proof}[Proof of Lemma \ref{lem_escape}]
 (1) Since $\mathcal{M}\to\mathcal{M}'$,
\begin{equation}
\Theta(\mathcal{M},\,\widetilde{\mathcal{M}})\ =\ \Theta(\mathcal{M},\,\mathcal{M}')\ \ge\ H\;,\label{eq:e4101}
\end{equation}
where the last inequality follows from Lemma \ref{lap01}-(2). Suppose,
by contradiction, that there exists $\boldsymbol{m}\in\mathcal{M}^{*}(\mathcal{H})\setminus\mathcal{M}$
so that $U(\boldsymbol{m})\le U(\mathcal{M})$. Then, since $\boldsymbol{m}\in\widetilde{\mathcal{M}}$,
by Lemmata \ref{lem_not}-(1) and \ref{lap01}-(1),
\[
\Theta(\mathcal{M},\,\widetilde{\mathcal{M}})\ \le\ \Theta(\mathcal{M},\,\boldsymbol{m})\ <\ H\;,
\]
in contradiction with \eqref{eq:e4101}. Therefore, $\mathcal{M}=\mathcal{M}^{*}(\mathcal{H})$.
\smallskip{}

\noindent (2) Let $\mathcal{M}=\mathcal{M}^{*}(\mathcal{H})$. Then,
since $\widetilde{\mathcal{M}}\subset\mathcal{H}^{c}$ (it can be
$\widetilde{\mathcal{M}}=\varnothing$), by Lemma \ref{lap01}-(2),
we have
\begin{equation}
\Theta(\mathcal{M},\,\widetilde{\mathcal{M}})\ \ge\ H\;.\label{eq:css1}
\end{equation}
Let $\mathcal{M}''\in\mathscr{S}^{(n)}(\mathcal{H})$, $\mathcal{M}''\neq\mathcal{M}$.
Since $U(\mathcal{M})<U(\mathcal{M}'')$, we have $\mathcal{M}\subset\widetilde{\mathcal{M}''}$
so that by Lemmata \ref{lem_not}-(1) and \ref{lap01}-(1),
\begin{equation}
\Theta(\mathcal{M}'',\,\widetilde{\mathcal{M}''})\ \le\ \Theta(\mathcal{M}'',\,\mathcal{M})\ <\ H\;.\label{eq:css2}
\end{equation}
Thus, by \eqref{eq:css1} and \eqref{eq:css2}, we get $\Theta(\mathcal{M},\,\widetilde{\mathcal{M}})>\Theta(\mathcal{M}'',\,\widetilde{\mathcal{M}''})$.
Combining this fact with $U(\mathcal{M})<U(\mathcal{M}'')$, by assumption
$\Xi(\mathcal{M})\le d^{(n)}$, we get
\[
d^{(n)}\ \ge\ \Xi(\mathcal{M})\ =\ \Theta(\mathcal{M},\,\widetilde{\mathcal{M}})-U(\mathcal{M})\ >\ \Theta(\mathcal{M}'',\,\widetilde{\mathcal{M}''})-U(\mathcal{M}'')\ =\ \Xi(\mathcal{M}'')\;.
\]
Hence, $\Xi(\mathcal{M}'')<d^{(n)}$, and by Lemma \ref{lem_charDelta}
we get $\mathcal{M}''\in\mathscr{N}^{(n)}$. Therefore, in $\mathcal{H}$,
only $\mathcal{M}$ can be an element of $\mathscr{V}^{(n)}$. In
other words, either $\mathscr{V}^{(n)}(\mathcal{H})=\{\mathcal{M}\}$
or $\mathscr{V}^{(n)}=\varnothing$. In the former case, by \eqref{eq:depth_p+1},
$\Xi(\mathcal{M})\ge d^{(n)}$, and therefore combining with the hypothesis
$\Xi(\mathcal{M})\le d^{(n)}$, we get $\Xi(\mathcal{M})=d^{(n)}$.

\noindent \smallskip{}
(3) By the definition of $\mathcal{H}$, there exists a local minimum
of $U$ belonging to $\mathcal{H}$. Since $\mathscr{S}^{(n)}$ is
a partition of $\mathcal{M}_{0}$ and $\mathcal{H}$ does not separate
$(n)$-states, the set $\mathscr{S}^{(n)}(\mathcal{H})$ cannot be
an empty set. In particular, by the hypothesis $\mathscr{V}^{(n)}(\mathcal{H})=\varnothing$,
we have $\mathscr{N}^{(n)}(\mathcal{H})=\mathscr{S}^{(n)}(\mathcal{H})\neq\varnothing$.

Now by (1), it suffices to prove that there exists $\mathcal{N\in\mathscr{N}}^{(n)}(\mathcal{H}),\,\mathcal{N}'\in\mathscr{S}^{(n)}(\mathcal{H}^{c})$
such that $\mathcal{N}\rightarrow\mathcal{N}'$. Suppose this is false,
i.e., for each $\mathcal{N\in\mathscr{N}}^{(n)}(\mathcal{H})$ and
$\mathcal{N}'\in\mathscr{S}^{(n)}(\mathcal{H}^{c})$, we cannot have
$\mathcal{N}\rightarrow\mathcal{N}'$ and thus $\widehat{r}^{(n)}(\mathcal{N},\,\mathcal{N}'')=0$
by $\mathfrak{P}_{3}(n)$. This implies the process $\widehat{{\bf y}}^{(n)}(\cdot)$
starting from $\mathscr{N}^{(n)}(\mathcal{H})$ cannot escape from
there. Therefore, there exists a recurrent class of the process $\widehat{{\bf y}}^{(n)}(\cdot)$
in $\mathscr{N}^{(n)}(\mathcal{H})$. This contradicts to Lemma \ref{l_irred_negli}.
\end{proof}

\subsection{\label{sec5.2}Analysis of closed level sets}

\begin{figure}
\includegraphics[scale=0.2]{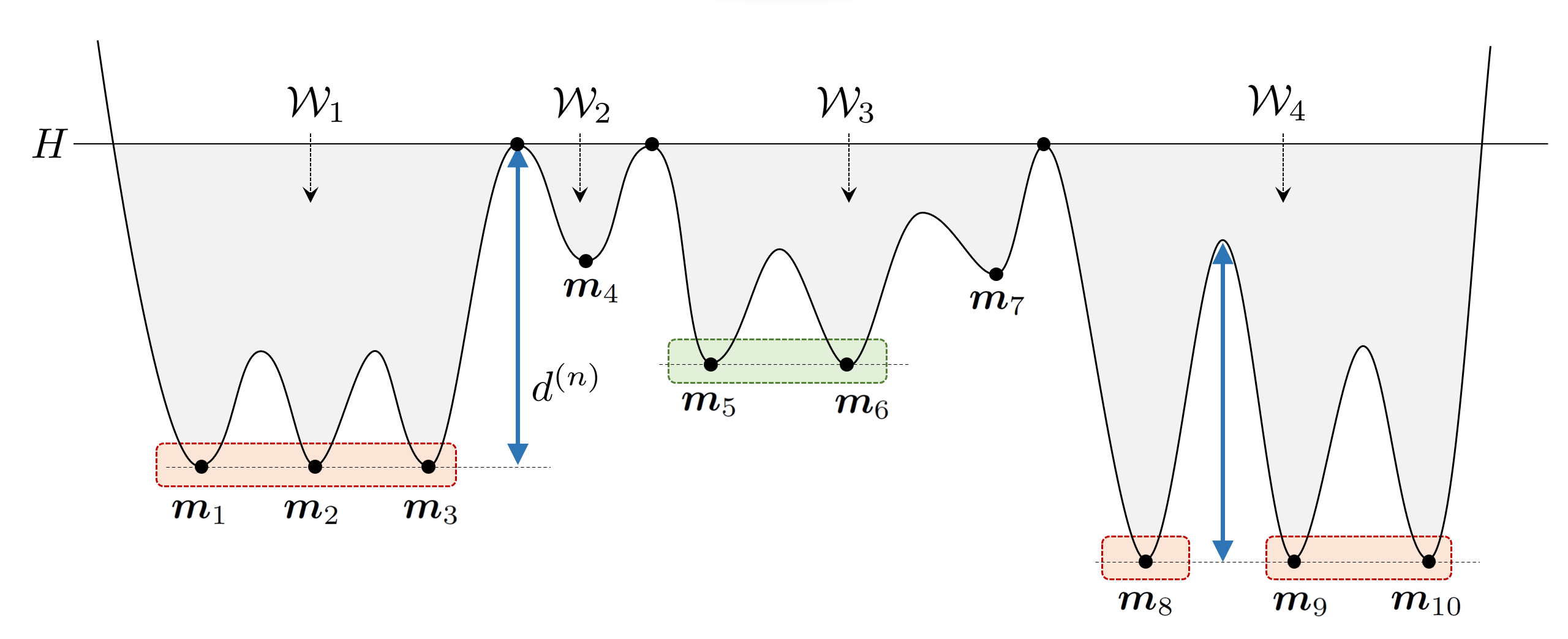}

\caption{\label{fig_l_wells} An example of level set investigated in Section
\ref{sec5.2}. In this figure, $\mathcal{K}$ is the connected component
of level set $\{U\le H\}$ containing $\boldsymbol{m}_{1},\,\dots,\,\boldsymbol{m}_{10}$
and $\{\mathcal{W}_{1},\,\mathcal{W}_{2},\,\mathcal{W}_{3},\,\mathcal{W}_{4}\}$
is a level set decomposition of $\mathcal{K}$.}
\end{figure}

In this subsection, we fix a saddle point $\boldsymbol{\sigma}\in\mathbb{R}^{d}$
denote by \textcolor{blue}{$\mathcal{K}$ }the connected component
of the closed level set $\{U\le U(\boldsymbol{\sigma})\}$ containing
$\boldsymbol{\sigma}$.Mind that since $\bm{\sigma}$
is a saddle point, $\mathcal{K}$ is not a singleton.

We denote by $\mathcal{W}_{1},\,\dots,\,\mathcal{W}_{a}$ all the
connected components of the open level set $\{U<U(\bm{\sigma})\}$
intersecting with $\mathcal{K}$. Then, it is proven in Lemma \ref{2la1-0}
that
\[
\overline{\bigcup_{i\in\llbracket1,\,a\rrbracket}\mathcal{W}_{i}}=\mathcal{K}\;.
\]
We say that $\{\mathcal{W}_{1},\,\dots,\,\mathcal{W}_{a}\}$ is a\textcolor{blue}{{}
}\textit{\textcolor{blue}{level set decomposition}} of $\mathcal{K}$.

This subsection is devoted to analyze the landscape of $\mathcal{K}$
in terms of $\mathcal{W}_{1},\,\dots,\,\mathcal{W}_{a}$. The first
result aims to characterize the elements of set $\mathscr{S}^{(n)}$
included in each $\mathcal{W}_{i}$. We refer to Figure \ref{fig_l_wells}
an illustration of the lemmata below.
\begin{lem}
\label{lem_wells1}Suppose that $a\ge2$ and that
\begin{equation}
\min_{\boldsymbol{x}\in\mathcal{\mathcal{K}}}U(\boldsymbol{x})\ \le\ U(\boldsymbol{\sigma})-d^{(n)}\ .\label{emin}
\end{equation}
Then, the sets $\mathcal{W}_{i}$, $i\in\llbracket1,\,a\rrbracket$,
do not separate $(n)$-states. Moreover, for $i\in\llbracket1,\,a\rrbracket$,
the following hold.
\begin{enumerate}
\item If $\min_{\boldsymbol{x}\in\mathcal{W}_{i}}U(\boldsymbol{x})>U(\boldsymbol{\sigma})-d^{(n)}$,
then $\mathscr{V}^{(n)}(\mathcal{W}_{i})=\varnothing$ and $\mathcal{M}^{*}(\mathcal{W}_{i})\in\mathscr{N}^{(n)}$.
\item If $\min_{\boldsymbol{x}\in\mathcal{W}_{i}}U(\boldsymbol{x})=U(\boldsymbol{\sigma})-d^{(n)}$,
then $\mathscr{V}^{(n)}(\mathcal{W}_{i})=\{\mathcal{M}^{*}(\mathcal{W}_{i})\}$.
\item If $\min_{\boldsymbol{x}\in\mathcal{W}_{i}}U(\boldsymbol{x})<U(\boldsymbol{\sigma})-d^{(n)}$,
then $\mathscr{V}^{(n)}(\mathcal{W}_{i})\neq\varnothing$ and furthermore
there exists a recurrent class $\mathscr{R}$ of chain $\mathbf{y}^{(n)}(\cdot)$
such that $\mathscr{R}\subset\mathscr{V}^{(n)}(\mathcal{W}_{i})$.
\end{enumerate}
\end{lem}

\begin{rem}
In Figure \ref{fig_l_wells}, the sets $\mathcal{W}_{2},\,\mathcal{W}_{3}$
correspond to the assertion (1) and hence $\{\boldsymbol{m}_{5},\,\boldsymbol{m}_{6}\}\in\mathscr{N}^{(n)}$.
The sets $\mathcal{W}_{1}$ and $\mathcal{W}_{4}$ correspond to the
assertions (2) and (3), respectively. In particular, $\{\boldsymbol{m}_{1},\,\boldsymbol{m}_{2},\,\boldsymbol{m}_{3}\}\in\mathscr{V}^{(n)}$.
There is a recurrent class $\{\{\boldsymbol{m}_{8}\},\,\{\boldsymbol{m}_{9},\,\boldsymbol{m}_{10}\}\}\subset\mathscr{V}^{(n)}$
of the chain $\mathbf{y}^{(n)}(\cdot)$ contained in $\mathcal{W}_{4}$.
\end{rem}

\begin{proof}
\noindent We fix $i\in\llbracket1,\,a\rrbracket$. We first prove
that $\mathcal{W}_{i}$ does not separate $(n)$-states. Suppose,
by contradiction, that there exists $\mathcal{M}\in\mathscr{S}^{(n)}$
with two elements $\boldsymbol{m},\,\boldsymbol{m}'\in\mathcal{M}$
such that $\boldsymbol{m}\in\mathcal{W}_{i}$ and $\boldsymbol{m}'\notin\mathcal{W}_{i}$.
By $\mathfrak{P}_{1}(n)$, $U(\boldsymbol{m})=U(\boldsymbol{m}')$.
Then, we have
\[
U(\boldsymbol{\sigma})\ \le\ \Theta(\boldsymbol{m},\,\boldsymbol{m}')\ <\ U(\mathcal{M})+d^{(n)}\ =\ U(\boldsymbol{m})+d^{(n)}
\]
where the first inequality follows from Lemma \ref{lap01}-(1), and
the second inequality and the last equality follows from $\mathfrak{P}_{1}(n)$,
and therefore we have $U(\boldsymbol{m})>U(\boldsymbol{\sigma})-d^{(n)}$.
Hence, by \eqref{emin}, there exists $\boldsymbol{m}''\in\mathcal{K}\cap\mathcal{M}_{0}$
such that $U(\boldsymbol{m}'')<U(\boldsymbol{m})$. Thus, we get a
contradiction since
\begin{equation}
U(\boldsymbol{\sigma})\ \ge\ \Theta(\boldsymbol{m},\,\boldsymbol{m}'')\ \ge\ \Theta(\mathcal{M},\,\widetilde{\mathcal{M}})\ >\ \Theta(\boldsymbol{m},\,\boldsymbol{m}')\ \ge\ U(\boldsymbol{\sigma})\ .\label{epp3}
\end{equation}
where four inequalities follow from path-connectedness of $\mathcal{K}$,
\ref{lem_not}-(1), $\mathfrak{P}_{1}(n)$ (so that $\mathcal{M}$
is a bound set), and Lemma \ref{lap01}-(2) respectively.

Since there exists a local minimum in each set $\mathcal{W}_{i}$
and $\mathcal{W}_{i}$ does not separate $(n)$-states, the set $\mathscr{S}^{(n)}(\mathcal{W}_{i})$
is non-empty.

Next we turn to the second assertion of the lemma. Suppose that
\begin{equation}
\min_{\boldsymbol{x}\in\mathcal{W}_{i}}U(\boldsymbol{x})\ \ge\ U(\boldsymbol{\sigma})-d^{(n)}\;.\label{eq:conWi}
\end{equation}
We prove that if $\mathcal{M}\in\mathscr{S}^{(n)}(\mathcal{W}_{i})$
satisfies $\mathcal{M}\neq\mathcal{M}^{*}(\mathcal{W}_{i}),$ then
$\mathcal{M}\in\mathscr{N}^{(n)}(\mathcal{W}_{i})$. To this end,
let us take $\boldsymbol{m}\in\mathcal{M}^{*}(\mathcal{W}_{i})\setminus\mathcal{M}$.
Since $U(\boldsymbol{m})\le U(\mathcal{M})$, we have $\boldsymbol{m}\in\widetilde{\mathcal{M}}$.
Therefore, as $\{\boldsymbol{m}\},\,\mathcal{M}\subset\mathcal{W}_{i}$,
we get
\[
\Theta(\mathcal{M},\,\widetilde{\mathcal{M}})\ \le\ \Theta(\mathcal{M},\,\boldsymbol{m})\ <\ U(\boldsymbol{\sigma})\ ,
\]
where two inequalities follow from Lemmata \ref{lem_not}-(1) and
\ref{lap01}-(1), respectively. On the other hand, by \eqref{eq:conWi},
\[
\Theta(\mathcal{M},\,\widetilde{\mathcal{M}})\ =\ U(\mathcal{M})+\Xi(\mathcal{M})\ \ge\ U(\boldsymbol{\sigma})-d^{(n)}+\Xi(\mathcal{M})\;.
\]
Combining the two previous estimates and then apply Lemma \ref{lem_charDelta},
we get $\mathcal{M}\in\mathscr{N}^{(n)}$. Hence, if $\mathcal{M}^{*}(\mathcal{W}_{i})\notin\mathscr{S}^{(n)}$,
we must have $\mathscr{V}^{(n)}(\mathcal{W}_{i})=\varnothing$, and
therefore by Lemma \ref{lem_escape}-(3), it holds that $\mathcal{M}^{*}(\mathcal{W}_{i})\in\mathscr{N}^{(n)}$
which contradicts to the assumption $\mathcal{M}^{*}(\mathcal{W}_{i})\notin\mathscr{S}^{(n)}$.
Thus, $\mathcal{M}^{*}(\mathcal{W}_{i})\in\mathscr{S}^{(n)}$. From
now on, we write $\mathcal{M}':=\mathcal{M}^{*}(\mathcal{W}_{i})\in\mathscr{S}^{(n)}$.

\smallskip{}

\noindent \textbf{Case 1. }Suppose first that the equality in \eqref{eq:conWi}
holds. Then, since $\widetilde{\mathcal{M}'}\cap\mathcal{W}_{i}=\emptyset$,
by Lemma \ref{lap01}-(2),
\[
\Xi(\mathcal{M}')\ =\ \Theta(\mathcal{M}',\,\widetilde{\mathcal{M}}')-U(\mathcal{M}')\ \ge\ U(\boldsymbol{\sigma})-U(\mathcal{M}')\ =\ d^{(n)}\;.
\]
Thus, by Lemma \ref{lem_charDelta}, $\mathcal{M}'\in\mathscr{V}^{(n)}(\mathcal{W}_{i})$.

\smallskip{}

\noindent \textbf{Case 2. }On the other hand, if the strict inequality
in \eqref{eq:conWi} holds, since $\widetilde{\mathcal{M}'}\cap\mathcal{W}_{i}=\emptyset$
and $\widetilde{\mathcal{M}'}\cap\mathcal{W}_{j}\neq\emptyset$ for
some $j\neq i$ by \eqref{emin}, we can apply Lemmata \ref{lap01}-(2)
and \ref{2-la1}-(1) to get
\[
\Xi(\mathcal{M}')\ =\ \Theta(\mathcal{M}',\,\widetilde{\mathcal{M}}')-U(\mathcal{M}')\ =\ U(\boldsymbol{\sigma})-U(\mathcal{M}')\ <\ d^{(n)}\;.
\]
Hence, by Lemma \ref{lem_charDelta}, $\mathcal{M}'\in\mathscr{N}^{(n)}(\mathcal{W}_{i})$.

\smallskip{}

\noindent \textbf{Case 3. }Finally, we consider the case
\begin{equation}
\min_{\boldsymbol{x}\in\mathcal{W}_{i}}\,U(\boldsymbol{x})\ <\ U(\boldsymbol{\sigma})-d^{(n)}\;.\label{eq:conWi-1}
\end{equation}
Suppose to the contrary that there is no recurrent class $\mathscr{R}$
of $\mathbf{y}^{(n)}(\cdot)$ such that $\mathscr{R}\subset\mathscr{V}^{(n)}(\mathcal{W}_{i})$.
Take $\boldsymbol{m}\in\mathcal{W}_{i}$ and let $\mathcal{M}=\mathcal{M}(n,\,\bm{m})$.
By the first assertion of the lemma, we have $\mathcal{M}\in\mathscr{S}^{(n)}(\mathcal{W}_{i})$.
Since there is no recurrent class which is a subset of $\mathcal{W}_{i}$,
the process $\widehat{\mathbf{y}}^{(n)}(\cdot)$ starting at $\mathcal{M}$
should escape from $\mathscr{S}^{(n)}(\mathcal{W}_{i})$. Then, by
Lemma \ref{lem_escape}, the last state should be $\mathcal{M}'=\mathcal{M}^{*}(\mathcal{W}_{i})$,
and hence $\mathcal{M}'\in\mathscr{S}^{(n)}$. As $\widetilde{\mathcal{M}'}\cap\mathcal{W}_{i}=\varnothing$,
by Lemma \ref{lap01}-(2), we have $\text{\ensuremath{\Theta(\mathcal{M}',\,}}\widetilde{\mathcal{M}'})\ge U(\boldsymbol{\sigma})$,
and therefore by \eqref{eq:conWi-1}, we get
\[
\Xi(\mathcal{M}')\ =\ \text{\ensuremath{\Theta(\mathcal{M}',\,}}\widetilde{\mathcal{M}'})-U(\mathcal{M}')\ >\ U(\boldsymbol{\sigma})-(U(\boldsymbol{\sigma})-d^{(n)})\ =\ d^{(n)}\;.
\]
Hence, $\mathcal{M}'$ is an absorbing point of the chain $\widehat{\mathbf{y}}^{(n)}(\cdot)$
by $\mathfrak{P}_{3}(n)$. This contradicts to the fact that the process
$\widehat{\mathbf{y}}^{(n)}(\cdot)$ leaves $\mathscr{S}^{(n)}(\mathcal{W}_{i})$
at $\mathcal{M}'$.
\end{proof}
For $\mathcal{M},\,\mathcal{M}'\in\mathscr{S}^{(n)}$, we write \textcolor{blue}{$\mathcal{M}\Rightarrow^{(n)}\mathcal{M}'$}
if the Markov chain $\widehat{\mathbf{y}}^{(n)}(\cdot)$ starting
at $\mathcal{M}$ can reach $\mathcal{M}'$ with positive probability.
\begin{lem}
\label{lem_wells2}Suppose that $a\ge2$ and that
\begin{equation}
\min_{\boldsymbol{x}\in\mathcal{W}_{1}}U(\boldsymbol{x})\ =\ U(\boldsymbol{\sigma})-d^{(n)}\;.\label{twowells}
\end{equation}
Let $b\in\llbracket2,\,a\rrbracket$ and suppose that $\overline{\mathcal{W}_{i}}\cap\overline{\mathcal{W}_{i+1}}\neq\varnothing$
for all $i\in\llbracket1,\,b-1\rrbracket$, and that
\begin{equation}
\begin{cases}
\min_{\boldsymbol{x}\in\mathcal{W}_{i}}U(\boldsymbol{x})\ >\ U(\boldsymbol{\sigma})-d^{(n)}\;\;\text{for}\;i\in\llbracket2,\,b-1\rrbracket\;,\\
\min_{\boldsymbol{x}\in\mathcal{W}_{b}}U(\boldsymbol{x})\ \le\ U(\boldsymbol{\sigma})-d^{(n)}\;.
\end{cases}\label{econdU}
\end{equation}
Then, $\mathcal{M}^{*}(\mathcal{W}_{1})\Rightarrow^{(n)}\mathcal{M}'$
for some $\mathcal{M}'\in\mathscr{V}^{(n)}(\mathcal{W}_{b})$ (Note
that $\mathcal{M}^{*}(\mathcal{W}_{1})\in\mathscr{S}^{(n)}$ by \eqref{twowells}
and Lemma \ref{lem_wells1}).
\end{lem}

\begin{rem}
Figure \ref{fig_l_wells} illustrates the situation described in Lemma
\ref{lem_wells2} with $a=b=4$. In this situation, we can readily
check that
\[
\mathcal{M}^{*}(\mathcal{W}_{1})\ =\ \{\boldsymbol{m}_{1},\,\boldsymbol{m}_{2},\,\boldsymbol{m}_{3}\}\ \Rightarrow^{(n)}\ \{\boldsymbol{m}_{8}\}\in\mathscr{V}^{(n)}(\mathcal{W}_{4})\;.
\]
\end{rem}

\begin{proof}
We simply write $\mathcal{M}_{i}=\mathcal{M}^{*}(\mathcal{W}_{i})$,
$i\in\llbracket2,\,a\rrbracket$. By \eqref{twowells}, \eqref{econdU},
and Lemma \ref{lem_wells1}, we have $\mathcal{M}_{1}\in\mathscr{V}^{(n)}$
and for $i\in\llbracket2,\,b-1\rrbracket$, $\mathcal{M}_{i}\in\mathscr{N}^{(n)}$
and $\mathscr{V}^{(n)}(\mathcal{W}_{i})=\varnothing$.

\smallskip{}
\textbf{\textit{Claim A}}: For each $i\in\llbracket1,\,b-1\rrbracket$,
there exists $\mathcal{M}_{i+1}'\in\mathscr{S}^{(n)}(\mathcal{W}_{i+1})$
such that $\mathcal{M}_{i}\rightarrow\mathcal{M}_{i+1}'$. \smallskip{}

\noindent Let $\boldsymbol{m}_{i}\in\mathcal{M}_{i}$ and let $\boldsymbol{\sigma}'\in\overline{\mathcal{W}_{i}}\cap\overline{\mathcal{W}_{i+1}}$.
By Lemma \ref{l_assu_saddle}-(1), there exist $\boldsymbol{m}'\in\mathcal{W}_{i}$
and $\boldsymbol{m}''\in\mathcal{W}_{i+1}$ such that $\boldsymbol{m}'\curvearrowleft\boldsymbol{\sigma}'\curvearrowright\boldsymbol{m}''$.
By Lemma \ref{l_assu_saddle}-(2), there exist saddle points $\boldsymbol{\sigma}_{1},\,\dots,\,\boldsymbol{\sigma}_{k}$
in $\mathcal{W}_{i}$ and local minima $\boldsymbol{m}_{1}',\,\dots,\,\boldsymbol{m}_{k-1}'$
in $\mathcal{W}_{i}$ such that
\[
\boldsymbol{m}_{i}\curvearrowleft\boldsymbol{\sigma}_{1}\curvearrowright\boldsymbol{m}_{1}'\curvearrowleft\cdots\curvearrowleft\boldsymbol{\sigma}_{k}\curvearrowright\boldsymbol{m}'\;.
\]
Since $U(\boldsymbol{\sigma}_{j})<U(\boldsymbol{\sigma})=U(\boldsymbol{\sigma}')$
for $j\in\llbracket1,\,k\rrbracket$ as $\boldsymbol{\sigma}_{1},\,\dots\,,\ \boldsymbol{\sigma}_{k}\in\mathcal{W}_{i}$,
we can conclude that $\boldsymbol{\sigma}'\rightsquigarrow\boldsymbol{m}_{i}$,
so that $\boldsymbol{\sigma}'\rightsquigarrow\mathcal{M}_{i}$. Hence,
if we denote by $\mathcal{M}_{i+1}'=\mathcal{M}(n,\,\bm{m}'')$, we
have $\mathcal{M}_{i+1}'\in\mathscr{S}^{(n)}(\mathcal{W}_{i+1})$
by Lemma \ref{lem_wells1} and $\mathcal{M}_{i}\leftsquigarrow\boldsymbol{\sigma}'\curvearrowright\mathcal{M}_{i+1}'$
by the observations above. Since $\Theta(\mathcal{M}_{i},\,\mathcal{M}_{i+1}')=\Theta(\mathcal{M}_{i},\,\widetilde{\mathcal{M}_{i}})=U(\boldsymbol{\sigma}')$
by Lemma \ref{2-la1}-(1), we get $\mathcal{M}_{i}\rightarrow\mathcal{M}_{i+1}'$.

\smallskip{}
\textbf{\textit{Claim B}}: For each $i\in\llbracket2,\,b-1\rrbracket$,
either $\mathcal{M}_{i}'=\mathcal{M}_{i}$ or $\mathcal{M}_{i}'\Rightarrow^{(n)}\mathcal{M}_{i}$.
\smallskip{}

\noindent By Lemma \ref{lem_wells1}, $\mathcal{M}'_{i}\in\mathscr{N}^{(n)}$.
Therefore, by Lemma \ref{l_irred_negli}, there exist
\begin{equation}
p_{k}\ge1\;\;\;\text{and\;\;\;\ensuremath{\mathcal{N}_{1}},\,\ensuremath{\dots},\,\ensuremath{\mathcal{N}_{p_{k}-1}\in\mathscr{N}^{(n)}},\,}\mathcal{N}_{p_{k}}\in\mathscr{V}^{(n)}\label{eq:seq}
\end{equation}
such that
\begin{equation}
\widehat{r}^{(n)}(\mathcal{M}'_{i},\,\mathcal{N}_{1})\ >\ 0\;\;\text{and}\;\;\widehat{r}^{(n)}(\mathcal{N}_{i},\,\mathcal{N}_{i+1})\ >\ 0\;\text{for all}\ j\in\llbracket1,\,p_{k}-1\rrbracket\;.\label{eq:seq2}
\end{equation}
Since $\mathscr{V}^{(n)}(\mathcal{W}_{i})=\varnothing$ by Lemma \ref{lem_wells1},
we have $\mathcal{N}_{p_{k}}\notin\mathcal{W}_{i}$. Thus, writing
$\mathcal{N}_{0}=\mathcal{M}'_{i}$, we can find
\begin{equation}
\ell\ =\ \max\left\{ \,j\in\llbracket0,\,p_{k}-1\rrbracket:\mathcal{N}_{j}\subset\mathcal{W}_{i}\,\right\} \;.\label{eq:maxc}
\end{equation}
By Lemma \ref{lem_escape}-(1), we have $\mathcal{N}_{\ell}=\mathcal{M}^{*}(\mathcal{W}_{i})=\mathcal{M}_{i}$.
If $\ell=0$, we get $\mathcal{M}_{i}'=\mathcal{M}_{i}$, and if $\ell\ge1$,
by \eqref{eq:seq2}, we get $\mathcal{M}_{i}'\Rightarrow^{(n)}\mathcal{M}_{i}$.

\noindent \smallskip{}
\textbf{\textit{Claim C}}: Either $\mathcal{M}_{b}'\in\mathscr{V}^{(n)}(\mathcal{W}_{b})$
or $\mathcal{M}_{b}'\Rightarrow^{(n)}\mathcal{M}'$ for some $\mathcal{M}'\in\mathscr{V}^{(n)}(\mathcal{W}_{b})$.
\smallskip{}

\noindent Suppose that $\mathcal{M}_{b}'\in\mathscr{N}^{(n)}(\mathcal{W}_{b})$.
As in the proof of Claim B, we can find \eqref{eq:seq} satisfying
\eqref{eq:seq2}. It remains to show that $\mathcal{N}_{p_{k}}\subset\mathcal{W}_{b}$.
Suppose this is not the case. Then, as in \eqref{eq:maxc}, we can
find $\ell\in\llbracket0,\,p_{k}-1\rrbracket$ such that $\mathcal{N}_{\ell}=\mathcal{M}^{*}(\mathcal{W}_{b})=\mathcal{M}_{b}$.
Then, $\mathcal{M}_{b}=\mathcal{N}_{\ell}\in\mathscr{N}^{(n)}$ by
\eqref{eq:seq}. On the other hand, since $\widetilde{\mathcal{M}_{b}}\cap\mathcal{W}_{b}=\varnothing$,
by Lemma \ref{lap01}-(2), we have $\Theta(\mathcal{M}_{b},\,\widetilde{\mathcal{M}_{b}})\ge U(\boldsymbol{\sigma})$.
Thus,
\[
\Xi(\mathcal{M}_{b})\,=\,\Theta(\mathcal{M}_{b},\,\widetilde{\mathcal{M}_{b}})-U(\mathcal{M}_{b})\,\ge\,U(\boldsymbol{\sigma})-\min_{\boldsymbol{x}\in\mathcal{W}_{b}}\,U(\boldsymbol{x})\,>\,d^{(n)}\;.
\]
Hence, by Lemma \ref{lem_charDelta}, $\mathcal{M}_{b}\in\mathscr{V}^{(n)}$,
in contradiction with the fact that $\mathcal{M}_{b}\in\mathscr{N}^{(n)}$
derived in the previous paragraph.

\smallskip{}
Now, combining Claims A, B and C proves the assertion of the lemma.
\end{proof}
\begin{rem}
\label{rem:p4}We emphasize here that, in this section, the property
$\mathfrak{P}_{4}(n)$ is not used. Hence, lemmata in this section
hold under $\mathfrak{P}\llbracket n-1\rrbracket$ and $\mathfrak{P}_{1,\,2,\,3}(n)$
only. This observation is crucial in the proof of $\mathfrak{P}_{4}(n+1)$
given in Section \ref{sec6}.
\end{rem}

\section{\label{sec6}Proof of Theorem \ref{t:tree}}

In this section, we prove Theorem \ref{t:tree}. Thus, we suppose
that $\Lambda^{(1)},\,\dots,\,\Lambda^{(n)}$ are constructed, the
property $\mathfrak{P}\llbracket n\rrbracket$ holds throughout the
section. We remark that Proposition \ref{prop_global_min} will also
be proven in this section.

\subsection{\label{sec6.1}Proof of $\mathfrak{P}_{1}(n+1)$ and Proposition
\ref{prop_global_min}}

Note that, in the current subsection, we do not assume $n_{\mathfrak{n}}\ge2$
but instead allows $n_{\mathfrak{n}}=1$. We start from a lemma on
the recurrent class of the chain $\mathbf{y}^{(n)}(\cdot)$.
\begin{lem}
\label{lem:rec}Let $\{\mathcal{M}_{1},\,\dots,\,\mathcal{M}_{a}\}\subset\mathscr{V}^{(n)}$
for some $a\ge2$ be a recurrent class of $\mathbf{y}^{(n)}(\cdot)$
and let $\mathcal{M}=\cup_{i=1}^{a}\mathcal{M}_{i}\in\mathscr{V}^{(n+1)}$.
\begin{enumerate}
\item There exists a connected component $\mathcal{K}$ of level set $\{U\le\Theta(\mathcal{M}_{1},\,\widetilde{\mathcal{M}_{1}})\}$
containing $\mathcal{M}$.
\item Write $\{\mathcal{W}_{1},\,\dots,\,\mathcal{W}_{b}\}$ be a level
set decomposition of $\mathcal{K}$. Then, for each $i\in\llbracket1,\,a\rrbracket$,
there exists $r(i)\in\llbracket1,\,b\rrbracket$ such that $\mathcal{M}_{i}=\mathcal{M}^{*}(\mathcal{W}_{r(i)})$.
\item $\widetilde{\mathcal{M}}\cap\mathcal{K}=\emptyset$.
\end{enumerate}
\end{lem}

\begin{proof}
By Lemma \ref{lem_stddec}-(3), $\Theta(\mathcal{M}_{1},\,\widetilde{\mathcal{M}_{1}})=\Theta(\mathcal{M}_{i},\,\widetilde{\mathcal{M}_{i}})=\Theta(\mathcal{M}_{i},\,\mathcal{M}_{j})$
for all $i\ne j\in\llbracket1,\,a\rrbracket$.

\noindent (1) For each $i\in\llbracket1,\,a\rrbracket$, denote by
$\mathcal{K}_{i}$ the connected components of $\{U\le\Theta(\mathcal{M}_{1},\,\widetilde{\mathcal{M}_{1}})\}$
containing $\mathcal{M}_{i}$, whose existences are guaranteed by
Lemma \ref{lem_level}-(1). Fix $j\in\llbracket2,\,a\rrbracket$ and
pick $\bm{m}_{1,\,j}\in\mathcal{M}_{1}$ and $\bm{m}_{j}\in\mathcal{M}_{j}$
such that $\Theta(\bm{m}_{1,\,j},\,\bm{m}_{j})=\Theta(\mathcal{M}_{1},\,\mathcal{M}_{j})=\Theta(\mathcal{M}_{1},\,\widetilde{\mathcal{M}_{1}})$.
By Lemma \ref{l_105a-2}, there exists a connected component of $\{U\le\Theta(\mathcal{M}_{1},\,\widetilde{\mathcal{M}_{1}})\}$
containing both $\bm{m}_{1,\,j}$ and $\bm{m}_{j}$. Since $\bm{m}_{1,\,j}\in\mathcal{K}_{1}$,
this connected component is $\mathcal{K}_{1}$. On the other hand,
since $\bm{m}_{j}\in\mathcal{K}_{1}$, we obtain $\mathcal{K}_{1}=\mathcal{K}_{j}$.
Hence, $\mathcal{K}_{1}=\mathcal{K}_{i}$ for all $i\in\llbracket2,\,a\rrbracket$
so that $\mathcal{K}:=\mathcal{K}_{1}$ contains $\mathcal{M}$.

\smallskip{}

\noindent (2) By Lemma \ref{lem_level}-(3), there exists $\boldsymbol{\sigma}\in\mathcal{S}_{0}\cap\mathcal{K}$
such that $U(\boldsymbol{\sigma})=\Theta(\mathcal{M}_{1},\,\widetilde{\mathcal{M}_{1}})$.
Since $\Xi(\mathcal{M}_{1})=d^{(n)}$ by Remark \ref{rem:nab}, we
have
\begin{equation}
\min_{\boldsymbol{x}\in\mathcal{\mathcal{K}}}U(\boldsymbol{x})\ \le\ U(\mathcal{M}_{1})\ =\ \Theta(\mathcal{M}_{1},\,\widetilde{\mathcal{M}_{1}})-\Xi(\mathcal{M}_{1})\ =\ U(\boldsymbol{\sigma})-d^{(n)}\;.\label{eq:minK}
\end{equation}
Hence, by Lemma \ref{lem_wells1}, each $\mathcal{W}_{i}$, $i\in\llbracket1,\,a\rrbracket$,
does not separate $(n)$-states and therefore each $\mathcal{M}_{i}$
is contained in one of the sets $\mathcal{W}_{1},\,\dots,\,\mathcal{W}_{b}$.
Denote this set by $\mathcal{W}_{r(i)}$. If there exists $\boldsymbol{m}\in\mathcal{W}_{r(i)}\setminus\mathcal{M}_{i}$
such that $U(\boldsymbol{m})\le U(\mathcal{M}_{i})$, then by Lemma
\ref{lem_not}-(1), Lemma \ref{lap01}-(1), and Lemma \ref{lem_stddec}-(3),
we get
\[
\Theta(\mathcal{M}_{i},\,\widetilde{\mathcal{M}_{i}})\ \le\ \Theta(\mathcal{M}_{i},\,\boldsymbol{m})\ <\ \Theta(\mathcal{M}_{1},\,\widetilde{\mathcal{M}_{1}})\ =\ \Theta(\mathcal{M}_{i},\,\widetilde{\mathcal{M}_{i}})\ ,
\]
which is a contradiction. Hence, $\mathcal{M}_{i}=\mathcal{M}^{*}(\mathcal{W}_{r(i)})$.

\smallskip{}
\textit{\emph{(3)}} Suppose, by contradiction, that $\widetilde{\mathcal{M}}\cap\mathcal{K}\ne\varnothing$.
Then, thanks to Lemma \ref{2-la1}, by reindexing the sets $\mathcal{W}_{i}$
if necessary, there exists $k_{0}\in\llbracket2,\,b\rrbracket$ such
that $\overline{\mathcal{W}_{i}}\cap\overline{\mathcal{W}_{i+1}}\neq\varnothing$
for $i\in\llbracket1,\,k_{0}-1\rrbracket$ and
\[
\mathcal{M}_{\ell}\in\mathcal{W}_{1}\;,\;\;\;\mathcal{W}_{k_{0}}\cap\widetilde{\mathcal{M}}\neq\varnothing\;,\;\;\;\min_{i\in\llbracket2,\,k_{0}-1\rrbracket}\,\min_{\boldsymbol{x}\in\mathcal{W}_{i}}\,U(\boldsymbol{x})\ >\ U(\mathcal{M})
\]
for some $\ell\in\llbracket1,\,a\rrbracket$. Then, by Lemma \ref{lem_wells2},
there exist $\mathcal{M}'\in\mathscr{V}^{(n)}(\mathcal{W}_{k_{0}})$
such that $\mathcal{M}_{\ell}\Rightarrow^{(n)}\mathcal{M}'$. Therefore,
$\mathcal{M}_{\ell}$ and $\mathcal{M}'$ belong to the same recurrent
class of $\mathbf{y}^{(n)}(\cdot)$, i.e., $\mathcal{M}'=\mathcal{M}_{i'}$
for some $i'$. By the second assertion, $\mathcal{M}_{i'}=\mathcal{M}^{*}(\mathcal{W}_{k_{0}})$.
This contradicts the fact that $\mathcal{W}_{k_{0}}\cap\widetilde{\mathcal{M}}\ne\varnothing$,
and completes the proof.
\end{proof}
We next prove that each $\mathcal{M}\in\mathscr{V}^{(n+1)}$ has depth
larger than $d^{(n)}$. Note that we can measure the depth of $\mathcal{M}\in\mathscr{V}^{(n+1)}$
in the next lemma since we proved that each $\mathcal{M}\in\mathscr{V}^{(n+1)}$
is simple in Lemma \ref{lem_stddec}-(2).
\begin{lem}
\label{lem_p2pre} For $\mathcal{M}\in\mathscr{V}^{(n+1)}$, we have
$\Xi(\mathcal{M})>d^{(n)}$.
\end{lem}

\begin{proof}
Fix $\mathcal{M}\in\mathscr{V}^{(n+1)}$ and write $\mathscr{V}^{(n)}(\mathcal{M})=\{\mathcal{M}_{1},\,\dots,\,\mathcal{M}_{a}\}$
(cf. Remark \ref{rem:stddec}).

If $a=1$, then $\mathcal{M}=\mathcal{M}_{1}\in\mathscr{V}^{(n)}$
is an absorbing state of the chain ${\bf y}^{(n)}(\cdot)$ so that,
by $\mathfrak{P}_{4}(n)$, we get $\Xi(\mathcal{M})=\Xi(\mathcal{M}_{1})>d^{(n)}$.

Next we consider the case $a\ge2$. Write $H=\Theta(\mathcal{M}_{1},\,\widetilde{\mathcal{M}_{1})}$
and let $\mathcal{K}$ be the connected components of the set $\{U\le H\}$
containing $\mathcal{M}$ (cf. Lemma \ref{lem:rec}). Let $\mathcal{G}_{1},\,\dots.,\,\mathcal{G}_{m}$
be the connected components of $\{U\le\Theta(\mathcal{M}_{1},\,\widetilde{\mathcal{M}_{1}})\}$
which intersect $\widetilde{\mathcal{M}}$. By Lemma \ref{lem:rec}-(3),
we have $\mathcal{G}_{i}\neq\mathcal{K}$ for all $i$. In particular,
the compact sets $\mathcal{K},\,\mathcal{G}_{1},\,\dots.,\,\mathcal{G}_{m}$
are disjoint, and there exist pairwise disjoint open sets $\mathcal{U}_{0},\,\dots,\,\mathcal{U}_{m}$
such that
\[
\mathcal{K}\subset\mathcal{U}_{0}\ \;\;\text{and\;\;}\ \mathcal{G}_{i}\subset\mathcal{U}_{i}\;\quad i\in\llbracket1,\,m\rrbracket\;.
\]
By \cite[Lemma A.14]{LLS-1st}, there exists small enough $\kappa>0$
such that a connected components of the $\{U\le H+\kappa\}$ containing
$\mathcal{\mathcal{K}}$ (resp. $\mathcal{G}_{i}$) is contained in
$\mathcal{U}_{0}$ (resp. $\mathcal{U}_{i}$). Hence, by Lemma \ref{lap01}-(2),
\[
\Theta(\mathcal{M},\,\widetilde{\mathcal{M}})\ \ge\ H+\kappa\ >\ \Theta(\mathcal{M}_{1},\,\widetilde{\mathcal{M}_{1}})\ .
\]
Therefore,
\[
\Xi(\mathcal{M})\ =\ \Theta(\mathcal{M},\,\widetilde{\mathcal{M}})-U(\mathcal{M})\ >\ \Theta(\mathcal{M}_{1},\,\widetilde{\mathcal{M}_{1}})-U(\mathcal{M}_{1})\ =\ d^{(n)}\ .
\]
\end{proof}
We are now ready to prove $\mathfrak{P}_{1}(n+1)$, i.e., part (1)
of Theorem \ref{t:tree}. Notice here that the proof below does not
assume $\mathfrak{n}_{n}\ge2$.
\begin{proof}[Proof of Theorem \ref{t:tree}-(1)]
In view of Remark \ref{rem:N_p+1} and $\mathfrak{P}\llbracket n\rrbracket$,
it suffices to prove that elements in $\mathscr{V}^{(n+1)}$ are simple
and bound.

Fix $\mathcal{M}\in\mathscr{V}^{(n+1)}$ and write $\mathscr{V}^{(n)}(\mathcal{M})=\{\mathcal{M}_{1},\,\dots,\,\mathcal{M}_{a}\}$.
The set $\mathcal{M}$ is simple because of Lemma \ref{lem_stddec}-(2).
It suffices to prove that $\mathcal{M}$ is bound. If $a=1$, then
$\mathcal{M}=\mathcal{M}_{1}$ is bound by $\mathfrak{P}_{1}(n)$.
If $a\ge2$, the same argument as in Lemma \ref{lem:bound_dn} yields
that
\[
\max_{\boldsymbol{m},\boldsymbol{m}'\in\mathcal{M}}\,\Theta(\boldsymbol{m},\,\boldsymbol{m}')\ \le\ U(\mathcal{M})+d^{(n)}\;.
\]
Inserting $U(\mathcal{M})=\Theta(\mathcal{M},\,\widetilde{\mathcal{M}})-\Xi(\mathcal{M})$
to the previous bound, we get
\[
\max_{\boldsymbol{m},\boldsymbol{m}'\in\mathcal{M}}\,\Theta(\boldsymbol{m},\,\boldsymbol{m}')\ \le\ \Theta(\mathcal{M},\,\widetilde{\mathcal{M}})+d^{(n)}-\Xi(\mathcal{M})\ <\ \Theta(\mathcal{M},\,\widetilde{\mathcal{M}})
\]
where the last inequality follows from Lemma \ref{lem_p2pre}. This
proves that $\mathcal{M}$ is bound.
\end{proof}
We can also prove Proposition \ref{prop_global_min} with the results
suggested above.
\begin{proof}[Proof of Proposition \ref{prop_global_min}]
Note that, in the proof of Theorem \ref{t:tree}-(1) above, the condition
$\mathfrak{n}_{n}\ge2$ is not assumed, and therefore the proof therein
also demonstrates that $\mathcal{M}\in\mathscr{V}^{(\mathfrak{q}+1)}$
is simple and bound. This proves assertion (1).

Now we write $\mathscr{V}^{(\mathfrak{q}+1)}=\{\mathcal{M}\}$ and
prove that $\mathcal{M}=\mathcal{M}_{\star}$. Suppose not, so that
$\widetilde{\mathcal{M}}\neq\varnothing$. By Lemma \ref{lem_level},
there exists a connected component $\mathcal{K}$ of $\{U\le\Theta(\mathcal{M},\,\widetilde{\mathcal{M}})\}$
containing $\mathcal{M}$ such that $\mathcal{K}\cap\widetilde{\mathcal{M}}\neq\varnothing$
and there exists $\bm{\sigma}\in\mathcal{S}_{0}\cap\mathcal{K}$ such
that $U(\bm{\sigma})=\Theta(\mathcal{M},\,\widetilde{\mathcal{M}})$.
Morover, let $\{\mathcal{W}_{1},\,\dots,\,\mathcal{W}_{b}\}$ be a
level set decomposition of $\mathcal{K}$. Here, $b\ge2$ by Lemma
\ref{lem_level}.

Let $\mathcal{V}$ be the connected component of $\{U<\Theta(\mathcal{M},\,\widetilde{\mathcal{M}})\}$
containing $\mathcal{M}$, whose existence is guaranteed by Lemma
\ref{l_bound_conn}. Without loss of generality, suppose that $\mathcal{W}_{1}=\mathcal{V}$
and $\mathcal{W}_{2}\cap\widetilde{\mathcal{M}}\ne\varnothing$.

By Lemma \ref{lem_p2pre}, we have $\Xi(\mathcal{M})>d^{(\mathfrak{q})}$
so that
\begin{equation}
\min_{\bm{x}\in\mathcal{K}}\,U(\bm{x})\ \le\ U(\mathcal{M})\ =\ \Theta(\mathcal{M},\,\widetilde{\mathcal{M}})-\Xi(\mathcal{M})\ <\ \Theta(\mathcal{M},\,\widetilde{\mathcal{M}})-d^{(\mathfrak{q})}\ .\label{pf_global_min}
\end{equation}
Therefore, by Lemma \ref{lem_wells1} each $\mathcal{W}_{i}$ does
not separate $(\mathfrak{q})$-states. Since $\mathscr{V}^{(\mathfrak{q})}(\mathcal{M})\subset\mathcal{W}_{1}$
is the unique recurrent class of the chain ${\bf y}^{(\mathfrak{q})}(\cdot)$,
there is no recurrent class in $\mathcal{W}_{2}$ so that by Lemma
\ref{lem_wells1}-(3),
\[
\min_{\bm{x}\in\mathcal{W}_{2}}\,U(\bm{x})\ \ge\ U(\bm{\sigma})-d^{(\mathfrak{q})}\ =\ \Theta(\mathcal{M},\,\widetilde{\mathcal{M}})-d^{(\mathfrak{q})}\;.
\]
On the other hand, since $\widetilde{\mathcal{M}}\cap\mathcal{W}_{2}\ne\varnothing$,
by \eqref{pf_global_min},
\[
\min_{\bm{x}\in\mathcal{W}_{2}}\,U(\bm{x})\ \le\ U(\mathcal{M})\ <\ \Theta(\mathcal{M},\,\widetilde{\mathcal{M}})-d^{(\mathfrak{q})}\ .
\]
Two bounds obtained above contradict each other, and thus we get $\widetilde{\mathcal{M}}=\varnothing$.
Therefore, we can conclude that $\mathcal{M}=\mathcal{M}_{\star}$.
\end{proof}

\subsection{\label{sec6.2}Proof of $\mathfrak{P}_{2}(n+1)$, $\mathfrak{P}_{3}(n+1)$,
and $\mathfrak{P}_{4}(n+1)$}

We from now on assume $\mathfrak{n}_{n}\ge2$ so that we constructed
$\Lambda^{(n+1)}$ according to the procedure given in Section \ref{sec4.2}
thanks to the results obtained in the previous section. In particular,
we will implicitly assume that $\mathfrak{P}_{1}(n+1)$ holds, i.e.,
$\mathcal{M}\in\mathscr{S}^{(n+1)}$ is a simple bound set, since
we already proved it. We first establish several preliminary lemmata.
\begin{lem}
\label{l_adjacent0}Suppose that $\mathcal{M}\in\mathscr{N}^{(n+1)}$
and $\mathcal{M}'\in\mathscr{V}^{(n)}\cup\mathscr{V}^{(n+1)}$ satisfy
$\mathcal{M}\to\mathcal{M}'$. Then,
\[
U(\mathcal{M})\ \ge\ U(\mathcal{M}')\ .
\]
\end{lem}

\begin{proof}
Suppose, by contradiction, that $U(\mathcal{M})<U(\mathcal{M}')$.
Since $\mathcal{M}\to\mathcal{M}'$, by \eqref{eq:con_gate}, we can
find $\boldsymbol{\sigma}\in\mathcal{S}(\mathcal{M},\,\mathcal{M}')$
such that
\[
U(\boldsymbol{\sigma})\ =\ \Theta(\mathcal{M},\,\widetilde{\mathcal{M}})\ =\ \Theta(\mathcal{M},\,\mathcal{M}')\;.
\]
On the other hand, by Lemma \ref{lem_not}, we have
\[
\Theta(\mathcal{M}',\,\widetilde{\mathcal{M}'})\ \le\ \Theta(\mathcal{M}',\,\mathcal{M})
\]
and therefore, by summing up, we get
\[
\Xi(\mathcal{M})\ =\ U(\boldsymbol{\sigma})-U(\mathcal{M})\ >\ \Theta(\mathcal{M}',\,\widetilde{\mathcal{M}'})-U(\mathcal{M}')\ =\ \Xi(\mathcal{M}')\ .
\]
Since $\Xi(\mathcal{M}')\ge d^{(n)}$ as $\mathcal{M}'\in\mathscr{V}^{(n)}\cup\mathscr{V}^{(n+1)}$,
we have $\Xi(\mathcal{M})>d^{(n)}$. Since $\mathscr{N}^{(n+1)}=\mathscr{T}^{(n)}\cup\mathscr{N}^{(n)}$,
this yields a contradiction to $\mathfrak{P}_{4}(n)$ if $\mathcal{M}\in\mathscr{T}^{(n)}$
and to Lemma \ref{lem_charDelta} if $\mathcal{M}\in\mathscr{N}^{(n)}$.
\end{proof}
\begin{lem}
\label{l_adjacent}Let $\mathcal{M}\in\mathscr{N}^{(n+1)}$ and $\mathcal{M}'\in\mathscr{V}^{(n+1)}$.
Then, $\mathcal{M}\to\mathcal{M}'$ if and only if $\mathcal{M}\to\mathcal{M}''$
for some $\mathcal{M}''\in\mathscr{V}^{(n)}(\mathcal{M}')$.
\end{lem}

\begin{proof}
Suppose that $\mathcal{M}\to\mathcal{M}'$. By Lemma \ref{lem_not}-(4),
there exists $\boldsymbol{m}'\in\mathcal{M}'$ such that $\mathcal{M}\to\boldsymbol{m}'$.
Recall the definition of $\mathcal{M}(n,\,\bm{m}')$ from the beginning
of Section \ref{sec2.3}. Since $\mathcal{M}\to\mathcal{M}'$ and
$\mathcal{M}(n,\,\bm{m}')\subset\mathcal{M}'$, by \eqref{eq:con_gate}
and Lemma \ref{lem_not}-(1), we have
\[
\Theta(\mathcal{M},\,\widetilde{\mathcal{M}})\ =\ \Theta(\mathcal{M},\,\mathcal{M}')\ \le\ \Theta\left(\mathcal{M},\,\mathcal{M}(n,\,\bm{m}')\right)\;.
\]
Thus, by Lemma \ref{lem_not}-(5), we have $\mathcal{M}\to\mathcal{M}(n,\,\bm{m}')$.

Conversely, suppose that $\mathcal{M}\to\mathcal{M}''$ for some $\mathcal{M}''\in\mathscr{V}^{(n)}(\mathcal{M}')$.
By Lemma \ref{l_adjacent0}, we have $U(\mathcal{M})\ge U(\mathcal{M}'')=U(\mathcal{M}')$.
Hence, $\mathcal{M}'\subset\widetilde{\mathcal{M}}$ so that we have
$\Theta(\mathcal{M},\,\widetilde{\mathcal{M}})\le\Theta(\mathcal{M},\,\mathcal{M}')$
by Lemma \ref{lem_not}-(1). By Lemma \ref{lem_not}-(4), there is
$\boldsymbol{m}'\in\mathcal{M}''$ such that $\mathcal{M}\to\boldsymbol{m}'$.
Since $\Theta(\mathcal{M},\,\widetilde{\mathcal{M}})\le\Theta(\mathcal{M}\,,\mathcal{M}')$,
by Lemma \ref{lem_not}-(5), $\mathcal{M}\to\mathcal{M}'$, as claimed.
\end{proof}
Now we are ready to complete the proof of Theorem \ref{t:tree}.
\begin{proof}[Proof Theorem \ref{t:tree}-(2) and (3)]

\noindent 1. \emph{Proof of $\mathfrak{P}_{2}(n+1)$}. By \eqref{eq:depth_p+1}
and Lemma \ref{lem_p2pre}, we have
\[
d^{(n+1)}\ =\ \min_{\mathcal{M}\in\mathscr{V}^{(n+1)}}\,\Xi(\mathcal{M})\ >\ d^{(n)}\;.
\]
\smallskip{}

\noindent 2. \emph{Proof of $\mathfrak{P}_{3}(n+1)$}. By \eqref{eq:rate_30}
and \eqref{eq:rate_3}, property $\mathfrak{P}_{3}(n+1)$ immediately
holds for $\mathcal{M}\in\mathscr{V}^{(n+1)}$ and $\mathcal{M}'\in\mathscr{S}^{(n+1)}$.
Hence, it suffices to consider the case $\mathcal{M}\in\mathscr{N}^{(n+1)}$
and $\mathcal{M}'\in\mathscr{S}^{(n+1)}$. By Lemma \ref{lem_charDelta}
and $\mathfrak{P}_{2}(n+1)$ proven above, we have $\Xi(\mathcal{M})\le d^{(n)}$
for this case.

Assume first that $\widehat{r}^{(n+1)}(\mathcal{M},\,\mathcal{M}')>0$.
If $\mathcal{M}'\in\mathscr{N}^{(n+1)}$, we get $\mathcal{M}\to\mathcal{M}'$
by \eqref{eq:rate_1} and $\mathfrak{P}_{3}(n)$. On the other hand,
if $\mathcal{M}'\in\mathscr{V}^{(n+1)}$, by \eqref{eq:rate_2}, we
have $\widehat{r}^{(n)}(\mathcal{M},\,\mathcal{M}'')>0$ for some
$\mathcal{M}''\in\mathscr{V}^{(n)}(\mathcal{M}')$ and thus by $\mathfrak{P}_{3}(n)$
we have $\mathcal{M}\to\mathcal{M}''$. Hence, by Lemma \ref{l_adjacent},
$\mathcal{M}\to\mathcal{M}'$.

Conversely, assume that $\mathcal{M}\rightarrow\mathcal{M}'$ (Mind
that $\Xi(\mathcal{M})\le d^{(n)}$). If $\mathcal{M}'\in\mathscr{N}^{(n+1)}$,
we get $\widehat{r}^{(n+1)}(\mathcal{M},\,\mathcal{M}')>0$ by \eqref{eq:rate_1}
and $\mathfrak{P}_{3}(n)$. On the other hand, if $\mathcal{M}'\in\mathscr{V}^{(n+1)}$,
by Lemma \ref{l_adjacent}, we have $\mathcal{M}\rightarrow\mathcal{M}''$
for some $\mathcal{M}''\in\mathscr{V}^{(n)}(\mathcal{M}')$. Since
$\widehat{r}^{(n)}(\mathcal{M},\,\mathcal{M}'')>0$ by $\mathfrak{P}_{3}(n)$,
we get $\widehat{r}^{(n+1)}(\mathcal{M},\,\mathcal{M}')>0$ by \eqref{eq:rate_2}.

\smallskip{}

\noindent 3. \emph{Proof of $\mathfrak{P}_{4}(n+1)$}.\textbf{ }Fix
$\mathcal{M}\in\mathscr{V}^{(n+1)}$ such that $\Xi(\mathcal{M})>d^{(n+1)}$.
By \eqref{eq:rate_30} and \eqref{eq:rate_3}, $\mathcal{M}$ is an
absorbing state of the process $\widehat{{\bf y}}^{(n+1)}(\cdot)$,
and therefore of the Markov chain $\mathbf{y}^{(n+1)}(\cdot)$.

Conversely, fix $\mathcal{M}\in\mathscr{V}^{(n+1)}$ such that $\Xi(\mathcal{M})\le d^{(n+1)}$
and hence $\Xi(\mathcal{M})=d^{(n+1)}$ by \eqref{eq:depth_p+1}.
We claim that $\mathcal{M}\in\mathscr{V}^{(n+1)}$ is not an absorbing
state of the Markov chain $\mathbf{y}^{(n+1)}(\cdot)$.

By Lemma \ref{lem_level}-(1), there exists a connected component
$\mathcal{K}$ of the level set $\{U\le\Theta(\mathcal{M},\,\widetilde{\mathcal{M}})\}$
containing $\mathcal{M}$ and intersecting with $\widetilde{\mathcal{M}}.$
Then, by Lemma \ref{lem_level}-(3), we can find
$\boldsymbol{\sigma}\in\mathcal{S}_{0}\cap\mathcal{K}$ such that
$U(\boldsymbol{\sigma})=\Theta(\mathcal{M},\,\widetilde{\mathcal{M}})$.
As in Section \eqref{sec5.2}, let $\{\mathcal{W}_{1},\,\dots,\,\mathcal{W}_{b}\}$
be a level set decomposition of $\mathcal{K}$. By Lemma \ref{lem_level}-(2),
we have $b\ge2$, and by Lemma \ref{2-la1}, relablelling the indices
suitably, we can find $k_{0}\in\llbracket2,\,b\rrbracket$ such that
$\overline{\mathcal{W}_{i}}\cap\overline{\mathcal{W}_{i+1}}\neq\varnothing$
for all $i\in\llbracket1,\,k_{0}-1\rrbracket$, and that
\[
\mathcal{M}\subset\mathcal{W}_{1}\;,\;\;\;\widetilde{\mathcal{M}}\cap\mathcal{W}_{k_{0}}\neq\varnothing\;,\text{ and \;\;\;}\min_{i\in\llbracket2,\,k_{0}-1\rrbracket}\,\min_{\boldsymbol{x}\in\mathcal{W}_{i}}\,U(\boldsymbol{x})\ >\ U(\mathcal{M})\;.
\]
By Lemma \ref{lem_wells2}-(2) (with $n+1$ instead of $n$), which
can be used here since we established $\mathfrak{P}_{1,2,3}(n+1)$
(cf. Remark \ref{rem:p4}), there exists $\mathcal{M}'\in\mathscr{V}^{(n+1)}(\mathcal{W}_{k_{0}})$
such that $\mathcal{M}\Rightarrow^{(n+1)}\mathcal{M}'$. Therefore,
the set $\mathcal{M}$ is not an absorbing state of $\mathbf{y}^{(n+1)}(\cdot)$.

\smallskip{}

\noindent 4. \emph{Proof of assertion (3)}. By \eqref{eq:depth_p+1},
there exists $\mathcal{M}\in\mathscr{V}^{(n+1)}$ such that $\Xi(\mathcal{M})=d^{(n+1)}$.
By $\mathfrak{P}_{4}(n+1)$, the set $\mathcal{M}$ is not an absorbing
state of the Markov chain $\mathbf{y}^{(n+1)}(\cdot)$. It is therefore,
either transient or it belongs to a recurrent class with at least
two elements. In any case, $\mathfrak{n}_{n+1}$, the number of recurrent
classes of $\mathbf{y}^{(n+1)}(\cdot)$, should be strictly smaller
than $|\mathscr{V}^{(n+1)}|=\mathfrak{n}_{n}$.
\end{proof}

\section{\label{sec7}Proof of Proposition \ref{prop_H}}

In this section, we fix $p\in\llbracket1,\,\mathfrak{q}\rrbracket$
and suppose that $\mathfrak{C}_{{\rm fdd}}^{(1)},\dots,\mathfrak{C}_{{\rm fdd}}^{(p)}$
are in force and he purpose is to prove $\mathfrak{H}^{(p+1)}$, namely
\eqref{eq:Hp-1} and \eqref{eq:Hp-2} for $p+1$ instead of $p$.
We emphasize that, we will from now on assume that the tree structure
is constructed up to the last scale $\mathfrak{q}$ and $\mathfrak{P}\llbracket\mathfrak{q}\rrbracket$
has been verified as in Corollary \ref{cor:tree}.

\subsection{Proof of \eqref{eq:Hp-1} of condition $\mathfrak{H}^{(p+1)}$}

We start from a Freidlin-Wentzell type result which is independent
of $\mathfrak{C}_{{\rm fdd}}^{(1)},\dots,\mathfrak{C}_{{\rm fdd}}^{(p)}$.
\begin{lem}
\label{lem:FW-type} For any $\mathcal{M}\in\mathscr{V}^{(p+1)}$
and $\mathcal{M}'\in\mathscr{V}^{(p)}$ such that $\mathcal{M}'\not\subset\mathcal{M}$,
\[
\Theta(\mathcal{M},\,\mathcal{M}')-U(\mathcal{M})\ >\ d^{(p)}\;.
\]
In particular, for all $t>0$,
\[
\sup_{\boldsymbol{x}\in\mathcal{E}(\mathcal{M})}\,\mathbb{P}_{\boldsymbol{x}}^{\epsilon}\left[\,\tau_{\mathcal{E}(\mathcal{M}')}\le\theta_{\epsilon}^{(p)}t\,\right]\ =\ 0\;.
\]
\end{lem}

\begin{proof}
Suppose, first, that $U(\mathcal{M})<U(\mathcal{M}')$. In this case,
by Lemma \ref{lem_not}-(1),
\[
\Theta(\mathcal{M},\,\mathcal{M}')-U(\mathcal{M})\ \ge\ \Theta(\widetilde{\mathcal{M}'},\,\mathcal{M}')-U(\mathcal{M})\ >\ \Theta(\widetilde{\mathcal{M}'},\,\mathcal{M}')-U(\mathcal{M}')\ =\ \Xi(\mathcal{M}')\;.
\]
Since $\mathcal{M}'$ belongs to $\mathscr{V}^{(p)}$, by Proposition
\ref{prop:depth}, $\Xi(\mathcal{M}')\ge d^{(p)}$, so that $\Theta(\mathcal{M},\,\mathcal{M}')-U(\mathcal{M})>d^{(p)}$,
as claimed.

Suppose that $U(\mathcal{M})\ge U(\mathcal{M}')$. In this case, $\mathcal{M}'\subset\widetilde{\mathcal{M}}$.
Hence, by Lemma \ref{lem_p2pre},
\[
\Theta(\mathcal{M},\,\mathcal{M}')-U(\mathcal{M})\ \ge\ \Theta(\mathcal{M},\,\widetilde{\mathcal{M}})-U(\mathcal{M})\ =\ \Xi(\mathcal{M})\ >\ d^{(p)}\;.
\]
This proves the first assertion of the lemma.

The second assertion follows from the first one and standard Freidlin-Wentzell
estimate. See, e.g., Proposition \ref{p_FW}.
\end{proof}
Denote by $\mathscr{R}^{(p)}\subset\mathscr{V}^{(p)}$ the union of
all the recurrent classes of the Markov chain $\mathbf{y}^{(p)}(\cdot)$,
\begin{equation}
{\color{blue}\mathscr{R}^{(p)}}\ :=\;\mathscr{R}_{1}^{(p)}\cup\cdots\cup\mathscr{R}_{\mathfrak{n}_{p}}^{(p)}\;,\label{eq:Rp}
\end{equation}
so that $\mathscr{V}^{(p)}=\mathscr{T}^{(p)}\cup\mathscr{R}^{(p)}$.
Note that $\mathcal{E}^{(p+1)}=\bigcup_{\mathcal{M}\in\mathscr{R}^{(p)}}\mathcal{E}(\mathcal{M})$
and therefore the next lemma is a natural result.
\begin{lem}
\label{l2-03} Suppose that $\mathfrak{C}_{{\rm fdd}}^{(p)}$ holds.
Then, for all $\mathcal{M}\in\mathscr{V}^{(p)}$ and $t>0$,
\[
\lim_{\epsilon\to0}\,\sup_{\boldsymbol{x}\in\mathcal{E}(\mathcal{M})}\,\Big|\,\mathbb{P}_{\boldsymbol{x}}^{\epsilon}\left[\,\tau_{\mathcal{E}^{(p+1)}}>\theta_{\epsilon}^{(p)}t\,\right]-\mathcal{Q}_{\mathcal{M}}^{(p)}\left[\,\tau_{\mathscr{R}^{(p)}}>t\,\right]\,\Big|\ =\ 0
\]
\end{lem}

\begin{proof}
By $\mathfrak{C}_{{\rm fdd}}^{(p)}$, it suffices to prove
\begin{equation}
\lim_{\epsilon\to0}\,\sup_{\boldsymbol{x}\in\mathcal{E}(\mathcal{M})}\,\left|\,\mathbb{P}_{\boldsymbol{x}}^{\epsilon}\left[\,\tau_{\mathcal{E}^{(p+1)}}>\theta_{\epsilon}^{(p)}t\,,\ \boldsymbol{x}_{\epsilon}(\theta_{\epsilon}^{(p)}t)\in\mathcal{E}^{(p)}\setminus\mathcal{E}^{(p+1)}\,\right]-\mathcal{Q}_{\mathcal{M}}^{(p)}\left[\,\tau_{\mathscr{V}^{(p+1)}}>t\,\right]\,\right|\ =\ 0\ .\label{e_tar0}
\end{equation}
The first probability can be written as
\[
\mathbb{P}_{\boldsymbol{x}}^{\epsilon}\left[\,\boldsymbol{x}_{\epsilon}(\theta_{\epsilon}^{(p)}t)\in\mathcal{E}^{(p)}\setminus\mathcal{E}^{(p+1)}\,\right]\,-\,\mathbb{P}_{\boldsymbol{x}}^{\epsilon}\left[\,\tau_{\mathcal{E}^{(p+1)}}\le\theta_{\epsilon}^{(p)}t\,,\ \boldsymbol{x}_{\epsilon}(\theta_{\epsilon}^{(p)}t)\in\mathcal{E}^{(p)}\setminus\mathcal{E}^{(p+1)}\,\right]\;.
\]
By the strong Markov property, the second term is bounded by
\[
\sup_{\boldsymbol{y}\in\mathcal{E}^{(p+1)}}\,\mathbb{P}_{\boldsymbol{y}}^{\epsilon}\left[\,\tau_{\mathcal{E}^{(p)}\setminus\mathcal{E}^{(p+1)}}\le\theta_{\epsilon}^{(p)}t\,\right]\;\le\;\sum_{\mathcal{M}''\in\mathscr{T}^{(p)}}\,\max_{\mathcal{M}'\in\mathscr{V}^{(p+1)}}\,\sup_{\boldsymbol{y}\in\mathcal{E}(\mathcal{M}')}\mathbb{P}_{\boldsymbol{y}}^{\epsilon}\left[\,\tau_{\mathcal{E}(\mathcal{M}'')}\le\theta_{\epsilon}^{(p)}t\,\right]\;.
\]
By Lemma \ref{lem:FW-type}, this expression vanishes as $\epsilon\to0$.

It remains to prove that
\[
\lim_{\epsilon\to0}\,\sup_{\boldsymbol{x}\in\mathcal{E}(\mathcal{M})}\,\left|\,\mathbb{P}_{\boldsymbol{x}}^{\epsilon}\left[\,\boldsymbol{x}_{\epsilon}(\theta_{\epsilon}^{(p)}t)\in\mathcal{E}^{(p)}\setminus\mathcal{E}^{(p+1)}\,\right]\,-\,\mathcal{Q}_{\mathcal{M}}^{(p)}\left[\,\tau_{\mathscr{V}^{(p+1)}}>t\,\right]\,\right|\ =\ 0\ .
\]
By $\mathfrak{C}_{{\rm fdd}}^{(p)}$,
\[
\lim_{\epsilon\to0}\,\sup_{\boldsymbol{x}\in\mathcal{E}(\mathcal{M})}\,\left|\,\mathbb{P}_{\boldsymbol{x}}^{\epsilon}\left[\,\boldsymbol{x}_{\epsilon}(\theta_{\epsilon}^{(p)}t)\in\mathcal{E}^{(p)}\setminus\mathcal{E}^{(p+1)}\,\right]\,-\,\mathcal{Q}_{\mathcal{M}}^{(p)}\left[\,{\bf y}^{(p)}(t)\in\mathscr{T}^{(p)}\,\right]\,\right|\ =\ 0\ .
\]
Since $\mathscr{T}^{(p)}$ and $\mathscr{R}^{(p)}$ denote the set
of transient and recurrent states of the Markov chain $\mathbf{y}^{(p)}(\cdot)$,
respectively,
\[
\mathcal{Q}_{\mathcal{M}}^{(p)}\left[\,{\bf y}^{(p)}(t)\in\mathscr{T}^{(p)}\,\right]\ =\ \mathcal{Q}_{\mathcal{M}}^{(p)}\left[\,\tau_{\mathscr{R}^{(p)}}>t\,\right]\;,
\]
which completes the proof.
\end{proof}
Now let us turn to $\mathfrak{H}^{(p+1)}$-(1). Note that we only
need $\mathfrak{C}_{\textup{fdd}}^{(p)}$ (instead of $\mathfrak{C}_{\textup{fdd}}^{(1)}$,
..., $\mathfrak{C}_{\textup{fdd}}^{(p)}$) to prove $\mathfrak{H}^{(p+1)}$-(1).
\begin{proof}[Proof of $\mathfrak{H}^{(p+1)}$-(1) by assuming $\mathfrak{C}_{\textup{fdd}}^{(p)}$]
Let us fix a sequence $(\alpha_{\epsilon})_{\epsilon>0}$ such that
$\alpha_{\epsilon}\succ\theta_{\epsilon}^{(p)}$. Since (cf. Remark
\ref{rem:N_p+1})
\[
\mathscr{N}^{(p+1)}\ =\ \bigcup_{n=1}^{p}\mathscr{T}^{(n)}\;,
\]
it is enough to prove that, for all $r\in\llbracket1,\,p\rrbracket$
and $\mathcal{M}\in\mathscr{T}^{(r)}$,
\begin{equation}
\lim_{\epsilon\to0}\,\sup_{\boldsymbol{x}\in\mathcal{E}(\mathcal{M})}\,\mathbb{P}_{\boldsymbol{x}}^{\epsilon}\left[\,\tau_{\mathcal{E}^{(p+1)}}>\alpha_{\epsilon}\,\right]\ =\ 0\ .\label{eq:Hp-1-1}
\end{equation}
This is proved by a reversed induction on $r$.

Suppose first that $r=p$ and $\mathcal{M}\in\mathscr{T}^{(p)}$.
Since $\alpha_{\epsilon}\succ\theta_{\epsilon}^{(p)}$, for all $t>0$
and $\epsilon>0$ small enough,
\[
\sup_{\boldsymbol{x}\in\mathcal{E}(\mathcal{M})}\,\mathbb{P}_{\boldsymbol{x}}^{\epsilon}\left[\,\tau_{\mathcal{E}^{(p+1)}}>\alpha_{\epsilon}\,\right]\;\le\;\sup_{\boldsymbol{x}\in\mathcal{E}(\mathcal{M})}\mathbb{P}_{\boldsymbol{x}}^{\epsilon}\left[\,\tau_{\mathcal{E}^{(p+1)}}>\theta_{\epsilon}^{(p)}t\,\right]
\]
and therefore by Lemma \ref{l2-03},
\[
\limsup_{\epsilon\to0}\,\sup_{\boldsymbol{x}\in\mathcal{E}(\mathcal{M})}\,\mathbb{P}_{\boldsymbol{x}}^{\epsilon}\left[\,\tau_{\mathcal{E}^{(p+1)}}>\alpha_{\epsilon}\,\right]\ \le\ \mathcal{Q}_{\mathcal{M}}^{(p)}\left[\,\tau_{\mathscr{R}^{(p)}}>t\,\right]\ .
\]
Letting $t\rightarrow\infty$, we get \eqref{eq:Hp-1-1} because $\mathscr{R}^{(p)}$
is the union of all recurrent classes of the Markov chain $\mathbf{y}^{(p)}(\cdot)$.

Next we let $r<p$, and suppose that \eqref{eq:Hp-1-1} has been proved
for all $m\in\llbracket r+1,\,p\rrbracket$ and $\mathcal{M}'\in\mathscr{T}^{(m)}$.

Fix $\mathcal{M}\in\mathscr{T}^{(r)}$ and $\boldsymbol{x}\in\mathcal{M}$.
For all $t>0$, we can write
\begin{equation}
\mathbb{P}_{\boldsymbol{x}}^{\epsilon}\left[\,\tau_{\mathcal{E}^{(p+1)}}>\alpha_{\epsilon}\,\right]\ \le\ \mathbb{P}_{\boldsymbol{x}}^{\epsilon}\left[\,\tau_{\mathcal{E}^{(r+1)}}>\theta_{\epsilon}^{(r)}t\,\right]+\mathbb{P}_{\boldsymbol{x}}^{\epsilon}\left[\,\tau_{\mathcal{E}^{(p+1)}}>\alpha_{\epsilon}\,,\,\tau_{\mathcal{E}^{(r+1)}}\le\theta_{\epsilon}^{(r)}t\,\right]\;.\label{eq:pgn3}
\end{equation}
By Lemma \ref{l2-03}, as $\epsilon\rightarrow0$, the first term
at the right-hand side converges to $\mathcal{Q}_{\mathcal{M}}^{(r)}[\tau_{\mathscr{R}^{(r)}}>t]$.
One the other hand, since $\mathcal{E}^{(p+1)}\subset\mathcal{E}^{(r+1)}$,
we have $\tau_{\mathcal{E}^{(r+1)}}\le\tau_{\mathcal{E}^{(p+1)}}$
and thus by the strong Markov property, the second term is bounded
by
\[
\sup_{\boldsymbol{y}\in\mathcal{E}^{(r+1)}}\,\mathbb{P}_{\boldsymbol{y}}^{\epsilon}\left[\,\tau_{\mathcal{E}^{(p+1)}}>\alpha_{\epsilon}-\theta_{\epsilon}^{(r)}t\,\right]\;=\;\sup_{\mathcal{M}\in\mathscr{T}^{(r+1)}\cup\cdots\cup\mathscr{T}^{(p)}\cup\mathscr{V}^{(p+1)}}\,\sup_{\boldsymbol{y}\in\mathcal{E}(\mathcal{M})}\mathbb{P}_{\boldsymbol{y}}^{\epsilon}\left[\,\tau_{\mathcal{E}^{(p+1)}}>\alpha_{\epsilon}-\theta_{\epsilon}^{(r)}t\,\right]\;.
\]
The last line vanishes as $\epsilon\rightarrow0$ since we have assume
that \eqref{eq:Hp-1-1} holds for $m\in\llbracket r+1,\,p\rrbracket$
and since the last probability is trivially $0$ if $\mathcal{M}\in\mathscr{V}^{(p+1)}$
as in that case $\tau_{\mathcal{E}^{(p+1)}}=0$. Summing up, we have
\[
\sup_{\boldsymbol{x}\in\mathcal{E}(\mathcal{M})}\,\mathbb{P}_{\boldsymbol{x}}^{\epsilon}\left[\,\tau_{\mathcal{E}^{(p+1)}}>\alpha_{\epsilon}\,\right]\ \le\ \mathcal{Q}_{\mathcal{M}}^{(r)}\left[\,\tau_{\mathscr{R}^{(r)}}>t\,\right]\;.
\]
The proof is completed by letting $t\to\infty$ as in the previous
case.
\end{proof}

\subsection{Proof of \eqref{eq:Hp-2} of condition $\mathfrak{H}^{(p+1)}$}

The proof of $\mathfrak{H}^{(p+1)}$-(2) requires further estimates.
The next lemma which is independent of $\mathfrak{C}_{\textup{fdd}}^{(1)},\,\dots,\,\mathfrak{C}_{\textup{fdd}}^{(p)}$,
presents a relation between the processes $\widehat{{\bf y}}^{(p)}(\cdot)$
and $\widehat{{\bf y}}^{(p+1)}(\cdot)$. Recall from \eqref{2-04}
the definition of the set $\mathscr{V}^{(p)}(\mathcal{M}')$ for $\mathcal{M}'$$\in\mathscr{V}^{(p+1)}$
and from \eqref{eq:Rp} the set $\mathscr{R}^{(p)}$.
\begin{lem}
\label{l: tau_V}For all $\mathcal{M}\in\mathscr{N}^{(p+1)}$ and
$\mathcal{M}'\in\mathscr{V}^{(p+1)}$,
\begin{equation}
\widehat{\mathcal{Q}}_{\mathcal{M}}^{(p)}\left[\,\tau_{\mathscr{R}^{(p)}}=\tau_{\mathscr{V}^{(p)}(\mathcal{M}')}\,\right]\ =\ \widehat{\mathcal{Q}}_{\mathcal{M}}^{(p+1)}\left[\,\tau_{\mathscr{V}^{(p+1)}}=\tau_{\mathcal{M}'}\,\right]\;.\label{eq:esc}
\end{equation}
\end{lem}

\begin{rem}
Note that the probability at the left-hand side of \eqref{eq:esc}
makes sense since
\begin{equation}
\mathscr{N}^{(p+1)}\ =\ \mathscr{N}^{(p)}\cup\mathscr{T}^{(p)}\ \subset\ \mathscr{N}^{(p)}\cup\mathscr{V}^{(p)}\ =\ \mathscr{S}^{(p)}\;.\label{eq:np+1}
\end{equation}
\end{rem}

\begin{proof}
By the definition \eqref{eq:rate_1} of the jump rates,
\[
\widehat{r}^{(p+1)}(\mathcal{M},\,\mathcal{M}')\ =\ \widehat{r}^{(p)}(\mathcal{M},\,\mathcal{M}')\quad\text{for}\;\;\mathcal{M},\,\mathcal{M}'\in\mathscr{N}^{(p+1)}\;.
\]
Therefore, starting from a state in $\mathscr{N}^{(p+1)}$, the processes
$\widehat{\mathbf{y}}^{(p)}(\cdot)$ and $\mathbf{\widehat{\mathbf{y}}}^{(p+1)}(\cdot)$
can be coupled to jump together until they leave the set $\mathscr{N}^{(p+1)}$.
Furthermore, by \eqref{eq:rate_2}, for all $\mathcal{M}\in\mathscr{N}^{(p+1)}$,
$\mathcal{M}'\in(\mathscr{N}^{(p+1)})^{c}=\mathscr{V}^{(p+1)}$,
\[
\widehat{r}^{(p+1)}(\mathcal{M},\,\mathcal{M}')\ =\ \widehat{r}^{(p)}(\mathcal{M},\,\mathscr{V}^{(p)}(\mathcal{M}'))\;,
\]
and therefore we can extend the coupling up to the time the processes
leave the set $\mathscr{N}^{(p+1)}$ in such a way that $\mathbf{\widehat{\mathbf{y}}}^{(p+1)}(\cdot)$
exit $\mathscr{N}^{(p+1)}$ by hitting a state $\mathcal{M}'\in\mathscr{V}^{(p+1)}$
if and only if $\mathbf{\widehat{\mathbf{y}}}^{(p)}(\cdot)$ exit
$\mathscr{N}^{(p+1)}$ by hitting a set in $\mathscr{V}^{(p)}(\mathcal{M}')$.
The assertion of the lemma follows since $\mathscr{S}^{(p)}\setminus\mathscr{N}^{(p+1)}=\mathscr{R}^{(p)}$.
\end{proof}
Next lemma establishes a relation between  $\mathfrak{C}_{{\rm fdd}}^{(p)}$
and the hitting time of recurrent classes.
\begin{lem}
\label{lem_Qcom}For all $\mathcal{M}\in\mathscr{V}^{(p)}$, $\mathcal{M}'\in\mathscr{V}^{(p+1)}$,
and $t\ge0$,
\[
\mathcal{Q}_{\mathcal{M}}^{(p)}\left[\,{\bf y}^{(p)}(t)\in\mathscr{V}^{(p)}(\mathcal{M}')\,\right]\ =\ \mathcal{Q}_{\mathcal{M}}^{(p)}\left[\,\tau_{\mathscr{R}^{(p+1)}}\le t,\,{\bf y}^{(p)}(\tau_{\mathscr{R}^{(p+1)}})\in\mathscr{V}^{(p)}(\mathcal{M}')\,\right]
\]
\end{lem}

\begin{proof}
For $\mathcal{M}'\in\mathscr{V}^{(p+1)}$, since $\mathscr{V}^{(p)}(\mathcal{M}')$
is a recurrent class of the chain $\mathbf{y}^{(p)}(\cdot)$ we immediately
have that
\begin{equation}
\mathcal{Q}_{\mathcal{M}''}^{(p)}\left[\,\tau_{(\mathscr{V}^{(p)}(\mathcal{M}'))^{c}}=\infty\,\right]\ =\ 1\;\;\;\text{for all }\mathcal{M}''\in\mathscr{V}^{(p)}(\mathcal{M}')\;.\label{eq:esc_rec}
\end{equation}
Thus, for $\mathcal{M}\in\mathscr{V}^{(p)}$ and $t\ge0$, $\mathcal{Q}_{\mathcal{M}}^{(p)}$-almost
surely, the event $\left\{ {\bf y}^{(p)}(t)\in\mathscr{V}^{(p)}(\mathcal{M}')\right\} $
corresponds to the event
\[
\left\{ \,\tau_{\mathscr{R}^{(p+1)}}\le t\;\;\text{and}\;\;{\bf y}^{(p)}(\tau_{\mathscr{R}^{(p+1)}})\in\mathscr{V}^{(p)}(\mathcal{M}')\,\right\}
\]
and the proof is completed.
\end{proof}
The following two lemmata is a key ingredient to prove $\mathfrak{H}^{(p+1)}$-(2).
\begin{lem}
\label{l: tau_E}Suppose that $\mathfrak{C}_{{\rm fdd}}^{(p)}$ holds.
Then, for all $\mathcal{M}\in\mathscr{V}^{(p)}$ and $\mathcal{M}'\in\mathscr{V}^{(p+1)}$,
\begin{equation}
\lim_{\epsilon\to0}\sup_{\boldsymbol{x}\in\mathcal{E}(\mathcal{M})}\left|\,\mathbb{P}_{\boldsymbol{x}}^{\epsilon}\left[\,\tau_{\mathcal{E}^{(p+1)}}=\tau_{\mathcal{E}(\mathcal{M}')}\,\right]-\widehat{\mathcal{Q}}_{\mathcal{M}}^{(p)}\left[\,\tau_{\mathscr{R}^{(p)}}=\tau_{\mathscr{V}^{(p)}(\mathcal{M}')}\,\right]\,\right|\ =\ 0\ .\label{e_tt}
\end{equation}
\end{lem}

\begin{proof}
The statement of lemma is immediate if $\mathcal{M}\in\mathscr{R}^{(p)}$
since in that case both probabilities are equal to $\mathbf{1}\left\{ \mathcal{M}\subset\mathcal{M}'\right\} $.

We turn to the case $\mathcal{M}\in\mathscr{T}^{(p)}$. Fix $\mathcal{M}'\in\mathscr{V}^{(p+1)}$
and $t>0$. Since $\mathcal{E}(\mathcal{M}')\subset\mathcal{E}^{(p+1)}$,
for all $\boldsymbol{x}\in\mathcal{E}(\mathcal{M})$,
\[
\begin{aligned}\mathbb{P}_{\boldsymbol{x}}^{\epsilon}\left[\,\boldsymbol{x}_{\epsilon}(\theta_{\epsilon}^{(p)}t)\in\mathcal{E}(\mathcal{M}')\,\right]\;= & \;\mathbb{P}_{\boldsymbol{x}}^{\epsilon}\left[\,\boldsymbol{x}_{\epsilon}(\theta_{\epsilon}^{(p)}t)\in\mathcal{E}(\mathcal{M}'),\,\tau_{\mathcal{E}^{(p+1)}}<\theta_{\epsilon}^{(p)}t,\,\boldsymbol{x}_{\epsilon}(\tau_{\mathcal{E}^{(p+1)}})\in\mathcal{E}(\mathcal{M}')\,\right]\\
 & \;+\;\mathbb{P}_{\boldsymbol{x}}^{\epsilon}\left[\,\boldsymbol{x}_{\epsilon}(\theta_{\epsilon}^{(p)}t)\in\mathcal{E}(\mathcal{M}'),\,\tau_{\mathcal{E}^{(p+1)}}<\theta_{\epsilon}^{(p)}t,\,\boldsymbol{x}_{\epsilon}(\tau_{\mathcal{E}^{(p+1)}})\notin\mathcal{E}(\mathcal{M}')\,\right]\;.
\end{aligned}
\]
By Lemma \ref{lem:FW-type} and the strong Markov property at time
$\tau_{\mathcal{E}^{(p+1)}}$, the second probability at the right-hand
side is negligible. On the other hand, the first probability can be
written as
\begin{align*}
 & \mathbb{P}_{\boldsymbol{x}}^{\epsilon}\left[\,\boldsymbol{x}_{\epsilon}(\tau_{\mathcal{E}^{(p+1)}})\in\mathcal{E}(\mathcal{M}')\,\right]\ -\ \mathbb{P}_{\boldsymbol{x}}^{\epsilon}\left[\,\boldsymbol{x}_{\epsilon}(\tau_{\mathcal{E}^{(p+1)}})\in\mathcal{E}(\mathcal{M}'),\,\tau_{\mathcal{E}^{(p+1)}}>\theta_{\epsilon}^{(p)}t\,\right]\\
 & \qquad-\mathbb{P}_{\boldsymbol{x}}^{\epsilon}\left[\,\tau_{\mathcal{E}^{(p+1)}}<\theta_{\epsilon}^{(p)}t,\,\boldsymbol{x}_{\epsilon}(\tau_{\mathcal{E}^{(p+1)}})\in\mathcal{E}(\mathcal{M}'),\,\boldsymbol{x}_{\epsilon}(\theta_{\epsilon}^{(p)}t)\notin\mathcal{E}(\mathcal{M}')\,\right]\ .
\end{align*}
We claim that the last term is negligible. Indeed, by $\mathfrak{C}_{{\rm fdd}}^{(p)}$,
we have
\begin{align*}
 & \limsup_{\epsilon\rightarrow0}\,\sup_{\boldsymbol{x}\in\mathcal{E}(\mathcal{M})}\,\mathbb{P}_{\boldsymbol{x}}^{\epsilon}\left[\,\tau_{\mathcal{E}^{(p+1)}}<\theta_{\epsilon}^{(p)}t,\,\boldsymbol{x}_{\epsilon}(\tau_{\mathcal{E}^{(p+1)}})\in\mathcal{E}(\mathcal{M}'),\,\boldsymbol{x}_{\epsilon}(\theta_{\epsilon}^{(p)}t)\notin\mathcal{E}(\mathcal{M}')\,\right]\\
 & \quad=\ \limsup_{\epsilon\rightarrow0}\,\sup_{\boldsymbol{x}\in\mathcal{E}(\mathcal{M})}\,\mathbb{P}_{\boldsymbol{x}}^{\epsilon}\left[\,\tau_{\mathcal{E}^{(p+1)}}<\theta_{\epsilon}^{(p)}t,\,\boldsymbol{x}_{\epsilon}(\tau_{\mathcal{E}^{(p+1)}})\in\mathcal{E}(\mathcal{M}'),\,\boldsymbol{x}_{\epsilon}(\theta_{\epsilon}^{(p)}t)\in\mathcal{E}^{(p)}\setminus\mathcal{E}(\mathcal{M}')\,\right]\\
 & \quad\le\ |\mathscr{V}^{(p)}|\,\limsup_{\epsilon\rightarrow0}\max_{\mathcal{M}''\in\mathscr{V}^{(p)}\setminus\mathscr{V}^{(p)}(\mathcal{M}')}\,\sup_{\boldsymbol{y}\in\mathcal{E}(\mathcal{M}')}\,\mathbb{P}_{\boldsymbol{y}}^{\epsilon}\left[\,\tau_{\mathcal{E}(\mathcal{M}'')}<\theta_{\epsilon}^{(p)}t\,\right]
\end{align*}
where we applied the strong Markov property at time $\tau_{\mathcal{E}^{(p+1)}}$
at the last inequality. The last line is equal to $0$ by Lemma \ref{lem:FW-type}.
In conclusion, we get
\begin{align*}
 & \limsup_{\epsilon\rightarrow0}\,\sup_{\boldsymbol{x}\in\mathcal{E}(\mathcal{M})}\,\mathbb{P}_{\boldsymbol{x}}^{\epsilon}\left[\boldsymbol{x}_{\epsilon}(\theta_{\epsilon}^{(p)}t)\in\mathcal{E}(\mathcal{M}')\right]\\
 & =\limsup_{\epsilon\rightarrow0}\,\sup_{\boldsymbol{x}\in\mathcal{E}(\mathcal{M})}\,\left\{ \,\mathbb{P}_{\boldsymbol{x}}^{\epsilon}\left[\boldsymbol{x}_{\epsilon}(\tau_{\mathcal{E}^{(p+1)}})\in\mathcal{E}(\mathcal{M}')\right]\,-\,\mathbb{P}_{\boldsymbol{x}}^{\epsilon}\left[\boldsymbol{x}_{\epsilon}(\tau_{\mathcal{E}^{(p+1)}})\in\mathcal{E}(\mathcal{M}'),\,\tau_{\mathcal{E}^{(p+1)}}>\theta_{\epsilon}^{(p)}t\right]\,\right\} \;.
\end{align*}

We turn to term at \eqref{e_tt} involving the probability $\mathcal{Q}_{\mathcal{M}}^{(p)}$.
By Lemma \ref{lem_Qcom}, we can write
\begin{align*}
 & \mathcal{Q}_{\mathcal{M}}^{(p)}\left[\,{\bf y}^{(p)}(t)\in\mathscr{V}^{(p)}(\mathcal{M}')\,\right]\\
 & =\ \mathcal{Q}_{\mathcal{M}}^{(p)}\left[\,{\bf y}^{(p)}(\tau_{\mathscr{R}^{(p)}})\in\mathscr{V}^{(p)}(\mathcal{M}')\,\right]\,-\,\mathcal{Q}_{\mathcal{M}}^{(p)}\left[\,\tau_{\mathscr{R}^{(p)}}>t,\,{\bf y}^{(p)}(\tau_{\mathscr{R}^{(p)}})\in\mathscr{V}^{(p)}(\mathcal{M}')\,\right]\;.
\end{align*}
By the two previous displayed equations,  $\mathfrak{C}_{{\rm fdd}}^{(p)}$
and Lemma \ref{l2-03},
\begin{equation}
\limsup_{\epsilon\to0}\,\sup_{\boldsymbol{x}\in\mathcal{E}(\mathcal{M})}\,\left|\,\mathbb{P}_{\boldsymbol{x}}^{\epsilon}\left[\,\tau_{\mathcal{E}^{(p+1)}}=\tau_{\mathcal{E}(\mathcal{M}')}\,\right]\,-\,\mathcal{Q}_{\mathcal{M}}^{(p)}\left[\,\tau_{\mathscr{R}^{(p)}}=\tau_{\mathscr{V}^{(p)}(\mathcal{M}')}\,\right]\,\right|\ \le\ 2\mathcal{Q}_{\mathcal{M}}^{(p)}\left[\,\tau_{\mathscr{R}^{(p)}}>t\,\right]\;.\label{eqPQ}
\end{equation}
Since $\mathbf{y}^{(p)}(\cdot)$ is the trace of the process $\widehat{\mathbf{y}}^{(p)}(\cdot)$
on $\mathscr{V}^{(p)}$, we have
\[
\mathcal{Q}_{\mathcal{M}}^{(p)}\left[\,\tau_{\mathscr{R}^{(p)}}=\tau_{\mathscr{V}^{(p)}(\mathcal{M}')}\,\right]\ =\ \widehat{\mathcal{Q}}_{\mathcal{M}}^{(p)}\left[\,\tau_{\mathscr{R}^{(p)}}=\tau_{\mathscr{V}^{(p)}(\mathcal{M}')}\,\right]\;,
\]
and therefore by letting $t\rightarrow\infty$ in \eqref{eqPQ} completes
the proof as $\mathscr{R}^{(p)}$ is the union of all recurrent class
of the chain $\mathbf{y}^{(p)}(\cdot)$.
\end{proof}
The next lemma is a generalization of the previous one.
\begin{lem}
\label{l: pf_Cond_H} Fix $r,\,s\in\llbracket1,\,\mathfrak{q}+1\rrbracket$
such that $r<s$ and suppose that $\mathfrak{C}_{{\rm fdd}}^{(r)},\,\dots,\,\mathfrak{C}_{{\rm fdd}}^{(s-1)}$
hold. Then, for all $\mathcal{M}\in\mathscr{V}^{(r)}$ and $\mathcal{M}'\in\mathscr{V}^{(s)}$,
\[
\lim_{\epsilon\to0}\,\sup_{\boldsymbol{x}\in\mathcal{E}(\mathcal{M})}\,\left|\,\mathbb{P}_{\boldsymbol{x}}^{\epsilon}\left[\,\tau_{\mathcal{E}^{(s)}}=\tau_{\mathcal{E}(\mathcal{M}')}\,\right]\,-\,\widehat{\mathcal{Q}}_{\mathcal{M}}^{(s)}\left[\,\tau_{\mathscr{V}^{(s)}}=\tau_{\mathcal{M}'}\,\right]\,\right|\ =\ 0\;.
\]
\end{lem}

\begin{proof}
We fix $r$ and the the proof is carried out by induction on $s$.
For $s=r+1$, the assertion of lemma is a direct consequence of Lemmata
\ref{l: tau_V} and \ref{l: tau_E}. Now we assume that the lemma
has been proven for $s=r+1,,\,\dots,\,r+k$ and investigate the case
$s=r+k+1$.

Fix $\mathcal{M}\in\mathscr{V}^{(r)}$ and $\mathcal{M}'\in\mathscr{V}^{(r+k+1)}$.
Then, as $\mathcal{M}'$ is a disjoint union of elements of $\mathscr{V}^{(r+k)}$,
by the induction hypothesis, we have
\[
\lim_{\epsilon\to0}\,\sup_{\boldsymbol{x}\in\mathcal{E}(\mathcal{M})}\,\left|\,\mathbb{P}_{\boldsymbol{x}}^{\epsilon}\left[\,\tau_{\mathcal{E}^{(r+k)}}=\tau_{\mathcal{E}(\mathcal{M}')}\,\right]\,-\,\widehat{\mathcal{Q}}_{\mathcal{M}}^{(r+k)}\left[\,\tau_{\mathscr{V}^{(r+k)}}=\tau_{\mathcal{M}'}\,\right]\,\right|\ =\ 0\ .
\]
Since $\mathcal{E}^{(r+k+1)}\subset\mathcal{E}^{(r+k)}$, by the strong
Markov property,
\[
\begin{aligned} & \mathbb{P}_{\boldsymbol{x}}^{\epsilon}\left[\,\tau_{\mathcal{E}^{(r+k+1)}}=\tau_{\mathcal{E}(\mathcal{M}')}\,\right]\,=\,\mathbb{P}_{\boldsymbol{x}}^{\epsilon}\left[\,\mathbb{P}_{\boldsymbol{x}_{\epsilon}(\tau_{\mathcal{E}^{(r+k)}})}^{\epsilon}\left[\,\tau_{\mathcal{E}^{(r+k+1)}}=\tau_{\mathcal{E}(\mathcal{M}')}\,\right]\,\right]\\
 & =\ \sum_{\mathcal{M}''\in\mathscr{V}^{(r+k)}}\mathbb{E}_{\boldsymbol{x}}^{\epsilon}\left[\,\mathbb{P}_{\boldsymbol{x}_{\epsilon}(\tau_{\mathcal{E}^{(r+k)}})}\left[\,\tau_{\mathcal{E}^{(r+k+1)}}=\tau_{\mathcal{E}(\mathcal{M}')}\,\right]\,{\bf 1}\left\{ \,\tau_{\mathcal{E}^{(r+k)}}=\tau_{\mathcal{E}(\mathcal{M}'')}\,\right\} \,\right]\;.
\end{aligned}
\]
For $\mathcal{M}''\in\mathscr{V}^{(r+k)}$, by Lemmata \ref{l: tau_V}
and \ref{l: tau_E} for $p=r+k$, we have
\begin{align*}
\limsup_{\epsilon\rightarrow0}\,\sup_{\boldsymbol{x}\in\mathcal{E}(\mathcal{M})}\,\Big|\, & \mathbb{E}_{\boldsymbol{x}}^{\epsilon}\left[\,\mathbb{P}_{\boldsymbol{x}_{\epsilon}(\tau_{\mathcal{E}^{(r+k)}})}\left[\,\tau_{\mathcal{E}^{(r+k+1)}}=\tau_{\mathcal{E}(\mathcal{M}')}\,\right]\,{\bf 1}\left\{ \,\tau_{\mathcal{E}^{(r+k)}}=\tau_{\mathcal{E}(\mathcal{M}'')}\,\right\} \,\right]\\
 & \qquad-\widehat{\mathcal{Q}}_{\mathcal{M}''}^{(r+k)}\left[\,\tau_{\mathscr{R}^{(r+k)}}=\tau_{\mathscr{V}^{(r+k)}(\mathcal{M}')}\,\right]\,\mathbb{P}_{\boldsymbol{x}}^{\epsilon}\left[\,\tau_{\mathcal{E}^{(r+k)}}=\tau_{\mathcal{E}(\mathcal{M}'')}\,\right]\,\Big|\ =\ 0\ .
\end{align*}
Combining the previous estimates, and applying the induction hypothesis
to estimate the term $\mathbb{P}_{\boldsymbol{x}}^{\epsilon}[\,\tau_{\mathcal{E}^{(r+k)}}=\tau_{\mathcal{E}(\mathcal{M}'')}\,]$
yields that
\[
\lim_{\epsilon\to0}\,\sup_{\boldsymbol{x}\in\mathcal{E}(\mathcal{M})}\,\Big|\,\mathbb{P}_{\boldsymbol{x}}^{\epsilon}\left[\,\tau_{\mathcal{E}^{(r+k+1)}}=\tau_{\mathcal{E}(\mathcal{M}')}\,\right]\,-\,\sum_{\mathcal{M}''\in\mathscr{V}^{(r+k)}}\,\widehat{\mathcal{Q}}_{\mathcal{M}''}^{(r+k)}\left[\,\tau_{\mathscr{R}^{(r+k)}}=\tau_{\mathscr{V}^{(r+k)}(\mathcal{M}')}\,\right]\,\widehat{\mathcal{Q}}_{\mathcal{M}}^{(r+k)}\left[\,\tau_{\mathscr{V}^{(r+k)}}=\tau_{\mathcal{M}''}\,\right]\,\Big|\ =\ 0\ .
\]
By the strong Markov property, as $\mathscr{R}^{(r+k)}\subset\mathscr{V}^{(r+k)}$,
we have
\[
\sum_{\mathcal{M}''\in\mathscr{V}^{(r+k)}}\,\widehat{\mathcal{Q}}_{\mathcal{M}}^{(r+k)}\left[\,\tau_{\mathscr{V}^{(r+k)}}=\tau_{\mathcal{M}''}\,\right]\,\widehat{\mathcal{Q}}_{\mathcal{M}''}^{(r+k)}\left[\,\tau_{\mathscr{R}^{(r+k)}}=\tau_{\mathscr{V}^{(r+k)}(\mathcal{M}')}\,\right]\ =\ \widehat{\mathcal{Q}}_{\mathcal{M}}^{(r+k)}\left[\,\tau_{\mathscr{R}^{(r+k)}}=\tau_{\mathscr{V}^{(r+k)}(\mathcal{M}')}\,\right]\;.
\]
By Lemma \ref{l: tau_V}, the last expression is equal to $\widehat{\mathcal{Q}}_{\mathcal{M}''}^{(r+k+1)}\left[\,\tau_{\mathscr{V}^{(r+k+1)}}=\tau_{\mathcal{M}'}\,\right]$.
This completes the proof of the induction step, and we are done.
\end{proof}
Now we are ready to prove $\mathfrak{H}^{(p+1)}$-(2).
\begin{proof}[Proof of $\mathfrak{H}^{(p+1)}$-(2) under $\mathfrak{C}_{{\rm fdd}}^{(1)},\dots,\mathfrak{C}_{{\rm fdd}}^{(p)}$]
We need to prove, for $\mathcal{M}\in\mathscr{N}^{(p+1)}$ and $\mathcal{M}'\in\mathscr{V}^{(p+1)}$,
\begin{equation}
\lim_{\epsilon\to0}\,\sup_{\boldsymbol{x}\in\mathcal{E}(\mathcal{M})}\,\left|\,\mathbb{P}_{\boldsymbol{x}}^{\epsilon}\left[\,\tau_{\mathcal{E}^{(p+1)}}=\tau_{\mathcal{E}(\mathcal{M}')}\,\right]-\widehat{\mathcal{Q}}_{\mathcal{M}}^{(p+1)}\left[\,\tau_{\mathscr{V}^{(p+1)}}=\tau_{\mathcal{M}'}\,\right]\,\right|\ =\ 0\;.\label{eq:Hp-2-1}
\end{equation}
By Remark \ref{rem:N_p+1}, there exists $n\in\llbracket1,\,p\rrbracket$
such that $\mathcal{M}\in\mathscr{V}^{(n)}$. Thus, the estimate \eqref{eq:Hp-2-1}
follows from Lemma \ref{l: pf_Cond_H} with $r=n$ and $s=p+1$.
\end{proof}

\section{Analysis of resolvent equation: flatness of solution\label{sec8}}

Fix $p\in\llbracket1,\,\mathfrak{q}\rrbracket$ and $\mathbf{g}:\mathscr{V}^{(p)}\rightarrow\mathbb{R}$.
Denote by $G=G^{p,{\bf \,g}}\colon\mathbb{R}^{d}\to\mathbb{R}$ the
function given by
\[
{\color{blue}G}\ :=\ \sum_{\mathcal{M}\in\mathscr{V}^{(p)}}{\bf g}(\mathcal{M})\,\chi_{_{\mathcal{E}(\mathcal{M})}}\ .
\]
For $\lambda>0$, denote by ${\color{blue}\phi_{\epsilon}=\phi_{\epsilon}^{p,{\bf \,g},\,\lambda}}\colon\mathbb{R}^{d}\to\mathbb{R}$
the solution of the resolvent equation
\begin{equation}
\left(\lambda-\theta_{\epsilon}^{(p)}\mathscr{L}_{\epsilon}\right)\,\phi_{\epsilon}\,=\,G\;.\label{res}
\end{equation}
It is well-known that $\phi_{\epsilon}$ can be represented as
\begin{equation}
\phi_{\epsilon}(\boldsymbol{x})\ =\ \mathbb{E}_{\boldsymbol{x}}^{\epsilon}\left[\,\int_{0}^{\infty}e^{-\lambda s}\,G(\boldsymbol{x}_{\epsilon}(\theta_{\epsilon}^{(p)}s))\,ds\,\right]\ .\label{eq:probex}
\end{equation}
In particular, if ${\color{blue}\|f\|_{\infty}}$ represents the $L^{\infty}$-norm
of a function $f$ defined on $\mathbb{R}^{d}$ or on $\mathscr{V}^{(p)}$,
\begin{equation}
\|\phi_{\epsilon}\|_{\infty}\ \le\ \frac{\Vert G\Vert_{\infty}}{\lambda}\ =\ \frac{\Vert\mathbf{g}\Vert_{\infty}}{\lambda}\ .\label{e: bound_Feps}
\end{equation}

The main result of this sections establishes that the solution of
the resolvent equation is asymptotically constant in a neighborhood
of each set $\mathcal{M}\in\mathscr{S}^{(p)}$. Mind that we include
the sets in $\mathscr{N}^{(p)}$. It reads as follows. Let $\mathbf{f}_{\epsilon}\colon\mathscr{S}^{(p)}\to\mathbb{R}$
be given by
\begin{equation}
{\color{blue}\mathbf{f}_{\epsilon}(\mathcal{M})}\ :=\ \frac{1}{|\mathcal{M}|}\,\sum_{\boldsymbol{m}\in\mathcal{M}}\phi_{\epsilon}(\boldsymbol{m})\;.\label{05}
\end{equation}

\begin{prop}
\label{p_flat} For each $p\in\llbracket1,\,\mathfrak{q}\rrbracket$,
$\mathbf{g}:\mathscr{V}^{(p)}\rightarrow\mathbb{R}$, and $\mathcal{M}\in\mathscr{S}^{(p)}$,
\[
\lim_{\epsilon\to0}\,\sup_{\boldsymbol{x}\in\mathcal{E}(\mathcal{M})}\,\left|\,\phi_{\epsilon}(\boldsymbol{x})-\mathbf{f}_{\epsilon}(\mathcal{M})\,\right|\ =\ 0\;.
\]
\end{prop}

We prove this proposition in the remainder of the section.

The following definition will be frequently used in the argument.
\begin{defn}
\label{def61} Let $\mathcal{A}\subset\mathbb{R}^{d}$ be a bounded
connected domain with smooth boundary $\partial\mathcal{A}$. Denote
by \textcolor{blue}{$\boldsymbol{x}_{\epsilon}(\cdot\,;\mathcal{A})$}
the diffusion process $\boldsymbol{x}_{\epsilon}(\cdot)$ reflected
at the boundary $\partial\mathcal{A}$.
\begin{enumerate}
\item Let \textcolor{blue}{$\mathbb{P}_{\boldsymbol{z}}^{\epsilon,\,\mathcal{A}}$}
be the law of the process $\boldsymbol{x}_{\epsilon}(\cdot\,;\mathcal{A})$
starting from $\boldsymbol{z}\in\ensuremath{\mathcal{A}}$, and
\item \textcolor{blue}{$p_{\epsilon}^{\mathcal{A}}(\boldsymbol{z},\,\boldsymbol{y};t)$}
the transition kernel of the process $\boldsymbol{x}_{\epsilon}(\cdot\,;\mathcal{A})$.
Thus, for all $\boldsymbol{z}\in\mathcal{A}$ and $\mathcal{B}\subseteq\mathcal{A}$,
\[
\int_{\mathcal{B}}\,p_{\epsilon}^{\mathcal{A}}(\boldsymbol{z},\,\boldsymbol{y};t)\,d\boldsymbol{y}\ =\ \mathbb{P}_{\boldsymbol{z}}^{\epsilon,\,\mathcal{A}}\left[\,\boldsymbol{x}_{\epsilon}(t\,;\mathcal{A})\in\mathcal{B}\,\right]\;.
\]
\end{enumerate}
\end{defn}

\subsection{Consequences of Freidlin--Wentzell theory}

Denote by \textcolor{blue}{$\tau_{\mathcal{A}}$} the time the process
$\boldsymbol{x}_{\epsilon}(\cdot)$ hits the set $\mathcal{A}\subset\mathbb{R}^{d}$.
Next result is \cite[Proposition 4.2]{LLS-1st} and a direct consequence
of Fredlin--Wentzell theory \cite{FW}.
\begin{prop}
\label{p_FW}Fix $h<H$, and denote by $\mathcal{A}$, $\mathcal{B}$
connected components of the set $\{\boldsymbol{x}\in\mathbb{R}^{d}:U(\boldsymbol{x})<h\}$,
$\{\boldsymbol{x}\in\mathbb{R}^{d}:U(\boldsymbol{x})<H\}$, respectively.
Assume that $\mathcal{A}\subset\mathcal{B}$. Suppose that all critical
points $\boldsymbol{c}$ of $U$ in $\mathcal{A}$ are such that $U(\boldsymbol{c})\le h_{0}$
for some $h_{0}<h$. Then, for all $\eta>0$,
\[
\limsup_{\epsilon\rightarrow0}\,\sup_{\boldsymbol{x}\in\mathcal{A}}\,\mathbb{P}_{\boldsymbol{x}}^{\epsilon}\left[\,\tau_{\partial\mathcal{B}}<e^{(H-h_{0}-\eta)/\epsilon}\,\right]\ =\ 0\;.
\]
In particular, for all $\boldsymbol{m}\in\mathcal{M}_{0}$ and $\eta>0$,
\[
\limsup_{\epsilon\rightarrow0}\,\sup_{\boldsymbol{x}\in\mathcal{E}(\boldsymbol{m})}\,\mathbb{P}_{\boldsymbol{x}}^{\epsilon}\left[\,\tau_{\partial\mathcal{W}^{2r_{0}}(\boldsymbol{m})}<e^{(r_{0}-\eta)/\epsilon}\,\right]\ =\ 0\;.
\]
\end{prop}

\subsection{Proof of Proposition \ref{p_flat}\label{subsec_pf_p_flat}}

We start by proving that the solution is flat on a neighborhood of
each minimum.
\begin{prop}
\label{p_flat-1} For each $\boldsymbol{m}\in\mathcal{M}_{0}$,
\[
\lim_{\epsilon\to0}\,\sup_{\boldsymbol{x},\,\boldsymbol{y}\in\mathcal{E}(\boldsymbol{m})}\,\left|\,\phi_{\epsilon}(\boldsymbol{x})-\phi_{\epsilon}(\boldsymbol{y})\right|\ =\ 0\;.
\]
\end{prop}

\begin{proof}
The proof is identitcal to that of \cite[Theorem 4.1]{LLS-1st}. The
key ingredient is \cite[Theorem 3.1]{LLS-1st}.
\end{proof}
Denote by \textcolor{blue}{$B(\boldsymbol{x},\,r)$} the open ball
of radius $r$ centered at $\boldsymbol{x}$.
\begin{lem}
\label{l_flat-1} Fix $p\in\llbracket1,\,\mathfrak{q}\rrbracket$,
and $\mathcal{M}=\{\boldsymbol{m}_{1},\,\dots,\,\boldsymbol{m}_{a}\}\in\mathscr{S}^{(p)}$.
Then, there exists $\alpha_{\epsilon}=\alpha_{\epsilon}(\mathcal{M})>0$
such that $\alpha_{\epsilon}\prec\theta_{\epsilon}^{(p)}$ and for
all $i,\,j\in\llbracket1,\,a\rrbracket$,

\[
\lim_{\epsilon\to0}\,\mathbb{P}_{\boldsymbol{m}_{i}}^{\epsilon}\left[\,\tau_{B(\boldsymbol{m}_{j},\,\epsilon)}>\alpha_{\epsilon}\,\right]\ =\ 0\ .
\]
\end{lem}

The proof is postponed to the next subsection. Write $a_{\epsilon}={\color{blue}o_{\epsilon}(1)}$
if $\lim_{\epsilon\to0}|a_{\epsilon}|=0$ uniformly.
\begin{cor}
\label{l_flat-2} Fix $p\in\llbracket1,\,\mathfrak{q}\rrbracket$,
and $\mathcal{M}=\{\boldsymbol{m}_{1},\,\dots,\,\boldsymbol{m}_{a}\}\in\mathscr{S}^{(p)}$.
Then, for all $i,\,j\in\llbracket1,\,a\rrbracket$,
\[
\lim_{\epsilon\to0}\,\left|\,\phi_{\epsilon}(\boldsymbol{m}_{i})-\phi_{\epsilon}(\boldsymbol{m}_{j})\,\right|\ =\ 0\;.
\]
\end{cor}

\begin{proof}
Let $\alpha_{\epsilon}$ be given by Lemma \ref{l_flat-1}. Since
$\alpha_{\epsilon}\prec\theta_{\epsilon}^{(p)}$ and $G$ is bounded,
\[
\phi_{\epsilon}(\boldsymbol{m}_{i})\ =\ \mathbb{E}_{\boldsymbol{m}_{i}}^{\epsilon}\left[\,\int_{\tau_{B(\boldsymbol{m}_{j},\,\epsilon)}/\theta_{\epsilon}^{(k)}}^{\infty}e^{-\lambda s}\,G(\boldsymbol{x}_{\epsilon}(\theta_{\epsilon}^{(p)}s))\,ds{\bf 1}\{\,\tau_{B(\boldsymbol{m}_{j},\,\epsilon)}\le\alpha_{\epsilon}\,\}\,\right]+o_{\epsilon}(1)\;.
\]
Changing variables, as $\alpha_{\epsilon}\prec\theta_{\epsilon}^{(p)}$
and $G$ is bounded, the previous expectation is equal to
\begin{align*}
 & \mathbb{E}_{\boldsymbol{m}_{i}}^{\epsilon}\left[\,\int_{0}^{\infty}e^{-\lambda s}\,G(\boldsymbol{x}_{\epsilon}(\theta_{\epsilon}^{(p)}s+\tau_{B(\boldsymbol{m}_{j},\,\epsilon)}))\,ds{\bf 1}\{\,\tau_{B(\boldsymbol{m}_{j},\,\epsilon)}\le\alpha_{\epsilon}\,\}\,\right]+o_{\epsilon}(1)\\
 & \quad=\;\mathbb{E}_{\boldsymbol{m}_{i}}^{\epsilon}\left[\,\int_{0}^{\infty}e^{-\lambda s}\,G(\boldsymbol{x}_{\epsilon}(\theta_{\epsilon}^{(p)}s+\tau_{B(\boldsymbol{m}_{j},\,\epsilon)}))\,ds\,\right]+o_{\epsilon}(1)\;.
\end{align*}
By strong Markov property, this expression is equal to
\[
\mathbb{E}_{\boldsymbol{m}_{i}}^{\epsilon}\bigg[\,\mathbb{E}_{\boldsymbol{x}_{\epsilon}(\tau_{B(\boldsymbol{m}_{j},\,\epsilon)})}\left[\,\int_{0}^{\infty}e^{-\lambda s}\,G(\boldsymbol{x}_{\epsilon}(\theta_{\epsilon}^{(p)}s))\,ds\,\right]\,\bigg]+o_{\epsilon}(1)\;=\;\mathbb{E}_{\boldsymbol{m}_{i}}^{\epsilon}\left[\,\phi_{\epsilon}(\boldsymbol{x}_{\epsilon}(\tau_{B(\boldsymbol{m}_{j},\,\epsilon)}))\,\right]+o_{\epsilon}(1)\ .
\]
Hence, by Proposition \ref{p_flat-1},
\[
\phi_{\epsilon}(\boldsymbol{m}_{i})\;=\;\phi_{\epsilon}(\boldsymbol{m}_{j})+o_{\epsilon}(1)\ .
\]
\end{proof}
\begin{proof}[Proof of Proposition \ref{p_flat}]
Fix $p\in\llbracket1,\,\mathfrak{q}\rrbracket$, and $\mathcal{M}\in\mathscr{S}^{(p)}$.
By definition \eqref{05} and Corollary \ref{l_flat-2}, for all $\boldsymbol{m}\in\mathcal{M}$,
\[
\left|\,\mathbf{f}_{\epsilon}(\mathcal{M})-\phi_{\epsilon}(\boldsymbol{m})\,\right|\ \le\ \frac{1}{|\mathcal{M}|}\,\sum_{\boldsymbol{m}'\in\mathcal{M}}\left|\,\phi_{\epsilon}(\boldsymbol{m}')-\phi_{\epsilon}(\boldsymbol{m})\,\right|\,=\,o_{\epsilon}(1)\ .
\]
Hence,
\begin{align*}
\sup_{\boldsymbol{x}\in\mathcal{E}(\mathcal{M})}\,\left|\,\phi_{\epsilon}(\boldsymbol{x})-{\bf f}_{\epsilon}(\mathcal{M})\,\right|\  & =\ \max_{\boldsymbol{m}'\in\mathcal{M}}\,\sup_{\boldsymbol{x}\in\mathcal{E}(\boldsymbol{m}')}\,\left|\,\phi_{\epsilon}(\boldsymbol{x})-{\bf f}_{\epsilon}(\mathcal{M})\,\right|\\
 & \le\ \max_{\boldsymbol{m}'\in\mathcal{M}}\,\sup_{\boldsymbol{x}\in\mathcal{E}(\boldsymbol{m}')}\,\left|\,\phi_{\epsilon}(\boldsymbol{x})-\phi_{\epsilon}(\boldsymbol{m}')\,\right|\,+\,o_{\epsilon}(1)\ .
\end{align*}
By Proposition \ref{p_flat-1}, this expression vanishes as $\epsilon\to0$.
\end{proof}

\subsection{Proof of Lemma \ref{l_flat-1}}

We are now ready to prove Lemma \ref{l_flat-1}.
\begin{proof}[Proof of Lemma \ref{l_flat-1}]
Fix $p\in\llbracket1,\,\mathfrak{q}\rrbracket$ and $\mathcal{M}=\{\boldsymbol{m}_{1},\,\dots,\,\boldsymbol{m}_{a}\}\in\mathscr{S}^{(p)}$.
By $\mathfrak{P}_{1}(p)$,
\[
\max_{\boldsymbol{m},\,\boldsymbol{m}'\in\mathcal{M}}\,\Theta(\boldsymbol{m},\,\boldsymbol{m}')-U(\mathcal{M})\ <\ \Theta(\mathcal{M},\,\widetilde{\mathcal{M}})-U(\mathcal{M})\ =\ \Xi(\mathcal{M})\ .
\]
On the other hand, by Lemma \ref{lem:bound_dn} for $n=p$,
\[
\max_{\boldsymbol{m},\,\boldsymbol{m}'\in\mathcal{M}}\,\Theta(\boldsymbol{m},\,\boldsymbol{m}')-U(\mathcal{M})\;<\;d^{(p)}\ .
\]
Choose $A$ is such that
\begin{equation}
\max_{\boldsymbol{m},\,\boldsymbol{m}'\in\mathcal{M}}\,\Theta(\boldsymbol{m},\,\boldsymbol{m}')-U(\mathcal{M})\ <\ A\ <\ \min\left\{ \,\Xi(\mathcal{M}),\,d^{(p)}\,\right\} \;.\label{e_def_alpha}
\end{equation}
and let $\alpha_{\epsilon}:=e^{A/\epsilon}$.

Note that $\alpha_{\epsilon}\prec\theta_{\epsilon}^{(p)}$ since $A<d^{(p)}$.
Let $0<\xi<\min\{\Xi(\mathcal{M}),\,d^{(p)}\}-A$ be a small number
so that there is no critical point $\boldsymbol{c}$ such that $U(\boldsymbol{c})\in(U(\mathcal{M}),\,U(\mathcal{M})+\xi)$.
By Lemma \ref{l_bound_conn}, there exists a connected component of
$\{U<U(\mathcal{M})+A+\xi\}$ which contains $\mathcal{M}$. Denote
it by $\mathcal{W}=\mathcal{W}^{U(\mathcal{M})+A+\xi}(\mathcal{M})$.
Since $A+\xi\le\Xi(\mathcal{M})$, by Lemma \ref{l_bound_conn}, $\mathcal{M}=\mathcal{M}^{*}(\mathcal{W})$.
Fix $\boldsymbol{m}_{i},\,\boldsymbol{m}_{j}\in\mathcal{M}$. By Proposition
\ref{p_FW} with $H=U(\mathcal{M})+A+\xi$, $h_{0}=U(\mathcal{M})+\xi/2$
and $\eta=\xi/2$,
\[
\mathbb{P}_{\boldsymbol{m}_{i}}^{\epsilon}[\,\tau_{B(\boldsymbol{m}_{j},\,\epsilon)}>\alpha_{\epsilon}\,]\ =\ \mathbb{P}_{\boldsymbol{m}_{i}}^{\epsilon}[\,\tau_{B(\boldsymbol{m}_{j},\,\epsilon)}>\alpha_{\epsilon},\,\tau_{\partial\mathcal{W}}>\alpha_{\epsilon}\,]+o_{\epsilon}(1)\ .
\]
On the event where the boundary of $\mathcal{W}$ is not attained
before time $\alpha_{\epsilon}$, we may couple the original diffusion
with the one reflected at the boundary of $\mathcal{W}$ in such a
way that they coincide up to time $\alpha_{\epsilon}$. The previous
probability is thus equal to
\[
\mathbb{P}_{\boldsymbol{m}_{i}}^{\epsilon,\,\mathcal{W}}[\,\tau_{B(\boldsymbol{m}_{j},\,\epsilon)}>\alpha_{\epsilon},\,\tau_{\partial\mathcal{W}}>\alpha_{\epsilon}\,]\ =\ \mathbb{P}_{\boldsymbol{m}_{i}}^{\epsilon,\,\mathcal{W}}[\,\tau_{B(\boldsymbol{m}_{j},\,\epsilon)}>\alpha_{\epsilon}\,]\,+\,o_{\epsilon}(1)\ \le\ \frac{1}{\alpha_{\epsilon}}\,\mathbb{E}_{\boldsymbol{m}_{i}}^{\epsilon,\,\mathcal{W}}[\,\tau_{B(\boldsymbol{m}_{j},\,\epsilon)}\,]\,+\,o_{\epsilon}(1)\;,
\]
where we adopted the notation introduced at the beginning of this
section, and $\mathbb{E}_{\boldsymbol{m}_{i}}^{\epsilon,\,\mathcal{W}}$
represents the expectation with respect to $\mathbb{P}_{\boldsymbol{m}_{i}}^{\epsilon,\,\mathcal{W}}$.

Recall from \cite[Section C]{LLS-1st}, the definition of the capacity
between two sets, and denote by ${\rm cap}$ the capacity between
$B(\boldsymbol{m}_{i},\,\epsilon)$ and $B(\boldsymbol{m}_{j},\,\epsilon)$:
\begin{gather*}
{\rm cap}\ :=\ {\rm cap}_{\epsilon}^{\mathcal{W}}(B(\boldsymbol{m}_{i},\,\epsilon),\,B(\boldsymbol{m}_{j},\,\epsilon))\;.
\end{gather*}
By \cite[Proposition 7.1]{LS-22}, which clearly holds for reflected
processes, and since the equilibrium potential is bounded by $1$,
\[
\mathbb{E}_{\boldsymbol{m}_{i}}^{\epsilon,\,\mathcal{W}}[\,\tau_{B(\boldsymbol{m}_{j},\,\epsilon)}\,]\ \le\ [\,1+o_{\epsilon}(1)\,]\,\frac{1}{{\rm cap}}\ .
\]
Let
\[
Z_{\epsilon}^{\mathcal{W}}\ =\ \int_{\mathcal{W}}e^{-U(\boldsymbol{x})/\epsilon}\,d\boldsymbol{x}\;\le\;C_{0}\,\epsilon^{d/2}\,e^{-U(\mathcal{M})/\epsilon}\ ,
\]
for some finite constant $C_{0}$ independent of $\epsilon$ and whose
value may change from line to line. To derive this bound we expanded
$U$ around $\boldsymbol{m}_{i}$. By \cite[Lemma 9.2]{LS-22},
\[
{\rm cap}\ \ge\ \frac{\epsilon^{d}}{C_{0}}\,\frac{1}{Z_{\epsilon}^{\mathcal{W}}}\,e^{-\Theta(\boldsymbol{m}_{i},\,B(\boldsymbol{m}_{j},\,\epsilon))/\epsilon}\ \ge\ \frac{\epsilon^{d/2}}{C_{0}}\,e^{-\,[\,\Theta(\boldsymbol{m}_{i},\,B(\boldsymbol{m}_{j},\,\epsilon))\,-\,U(\mathcal{M})\,]/\epsilon}
\]
so that
\[
\mathbb{E}_{\boldsymbol{m}_{i}}^{\epsilon,\,\mathcal{W}}[\,\tau_{B(\boldsymbol{m}_{j},\,\epsilon)}\,]\ \le\ [\,1+o_{\epsilon}(1)\,]\,C_{0}\,\epsilon^{-d/2}\,e^{[\,\Theta(\boldsymbol{m}_{i},\,B(\boldsymbol{m}_{j},\,\epsilon))\,-\,U(\mathcal{M})\,]/\epsilon}\ .
\]
Therefore, by definition \eqref{e_def_alpha} of $\alpha_{\epsilon}$,
and since $\Theta(\boldsymbol{m}_{i},\,B(\boldsymbol{m}_{j},\,\epsilon))=\Theta(\boldsymbol{m}_{i},\,\boldsymbol{m}_{j})$
for $\epsilon$ small enough,
\[
\frac{1}{\alpha_{\epsilon}}\,\mathbb{E}_{\boldsymbol{m}_{i}}^{\epsilon,\,\mathcal{W}}[\,\tau_{B(\boldsymbol{m}_{j},\,\epsilon)}\,]\ \le\ C_{0}\,e^{-c_{1}/\epsilon}
\]
for some $c_{1}>0$, which completes the proof of the lemma.
\end{proof}

\section{\label{sec9}Analysis of resolvent equation: test function}

For $p\in\llbracket1,\,\mathfrak{q}\rrbracket$, we define $\mathscr{V}_{{\rm nab}}^{(p)}$
and $\mathscr{V}_{{\rm ab}}^{(p)}$ as
\[
{\color{blue}\mathscr{V}_{{\rm nab}}^{(p)}}\ :=\ \{\mathcal{M}\in\mathscr{V}^{(p)}\,:\,\Xi(\mathcal{M})=d^{(p)}\}\ \ {\color{blue}\mathscr{V}_{{\rm ab}}^{(p)}}\ :=\ \{\mathcal{M}\in\mathscr{V}^{(p)}\,:\,\Xi(\mathcal{M})>d^{(p)}\}
\]
so that by Proposition \ref{prop:depth}, $\mathscr{V}_{{\rm ab}}^{(p)}$,
$\mathscr{V}_{{\rm nab}}^{(p)}$ are the absorbing and non-absorbing
states of ${\bf y}^{(p)}(\cdot)$ (and of $\widehat{\mathbf{y}}^{(p)}(\cdot)$),
respectively.

In this section, for each $p\in\llbracket1,\,\mathfrak{q}\rrbracket$
and $\mathcal{M}\in\mathscr{V}_{{\rm nab}}^{(p)}$, we construct the
test functions which will be used in Section \ref{sec10} to estimate
the solution of the resolvent equation. These test
functions are similar to the ones appeared before in \cite{BEGK,LMS2,LS-22,LS-22b,RS}
as an approximation of suitable harmonic function, but our application
of this function is quite different to these previous articles.  For
this reason we just present its definition and main properties, and
refer the reader to \cite{LS-22b} for proofs.

\subsection{Neighborhoods of saddle points}

Fix a saddle point $\boldsymbol{\sigma}$ of $U$ such that $\boldsymbol{m}\curvearrowleft\boldsymbol{\sigma}\curvearrowright\boldsymbol{m}'$
for distinct local minima $\boldsymbol{m}$, $\boldsymbol{m}'$ of
$U$. Let ${\color{blue}\mathbb{H}^{\boldsymbol{\sigma}}=\nabla^{2}U(\boldsymbol{\sigma})}$,
${\color{blue}\mathbb{L}^{\boldsymbol{\sigma}}=(D\boldsymbol{\ell})(\boldsymbol{\sigma})}$.
By the paragraph above \eqref{hyp2}, $\mathbb{H}^{\boldsymbol{\sigma}}$
has a unique negative eigenvalue. Denote by ${\color{blue}-\lambda_{1},\,\lambda_{2},\,\dots,\,\lambda_{d}}$
the eigenvalues of $\mathbb{H}^{\boldsymbol{\sigma}}$, where $-\lambda_{1}$
represents the unique negative eigenvalue. Mind that we omit the dependence
on $\boldsymbol{\sigma}$ which is fixed. Let ${\color{blue}\boldsymbol{e}_{1}}$,
${\color{blue}\boldsymbol{e}_{k}}$, $k\ge2$, be the unit eigenvector
associated with the eigenvalue $-\lambda_{1}$, $\lambda_{k}$, respectively.
Choose $\boldsymbol{e}_{1}$ pointing towards $\boldsymbol{m}$: for
all sufficiently small $a>0$, $\boldsymbol{\sigma}+a\boldsymbol{e}_{1}$
belongs to the domain of attraction of $\boldsymbol{m}$. For $\boldsymbol{x}\in\mathbb{R}^{d}$
and $k=1,\,\dots,\,d$, write ${\color{blue}\hat{x}_{k}=(\boldsymbol{x}-\boldsymbol{\sigma})\cdot\boldsymbol{e}_{k}}$,
so that $\boldsymbol{x}=\boldsymbol{\sigma}+\sum_{m=1}^{d}\hat{x}_{m}\boldsymbol{e}_{m}$.

Let
\[
{\color{blue}\delta}\ =\ \delta(\epsilon)\ :=\ (\epsilon\log\frac{1}{\epsilon})^{1/2}\ .
\]
Fix a large constant $J>0$ to be chosen later, and denote by $\mathcal{A}_{\epsilon}^{\pm}$,
$\mathcal{C}_{\epsilon}$ the $d$-dimensional rectangles defined
by
\begin{gather*}
{\color{blue}\mathcal{A}_{\epsilon}^{-}}\ :=\ \Big\{\,\boldsymbol{x}\in\mathbb{R}^{d}\,:\,\hat{x}_{1}\in\Big[\,-\frac{J\delta}{\sqrt{\lambda_{1}}}-\epsilon^{2},\,-\frac{J\delta}{\sqrt{\lambda_{1}}}\,\Big]\,,\,\hat{x}_{k}\in\Big[\,-\frac{2J\delta}{\sqrt{\lambda_{k}}},\,\frac{2J\delta}{\sqrt{\lambda_{k}}}\,\Big]\,,\,2\leq k\leq d\,\Big\}\\
{\color{blue}\mathcal{C}_{\epsilon}}\ :=\ \Big\{\,\boldsymbol{x}\in\mathbb{R}^{d}\,:\,\hat{x}_{1}\in\Big[\,-\frac{J\delta}{\sqrt{\lambda_{1}}},\,\frac{J\delta}{\sqrt{\lambda_{1}}}\,\Big]\,,\,\hat{x}_{k}\in\Big[\,-\frac{2J\delta}{\sqrt{\lambda_{k}}},\,\frac{2J\delta}{\sqrt{\lambda_{k}}}\,\Big]\,,\,2\leq k\leq d\,\Big\}\\
{\color{blue}\mathcal{A}_{\epsilon}^{+}}\ :=\ \Big\{\,\boldsymbol{x}\in\mathbb{R}^{d}\,:\,\hat{x}_{1}\in\Big[\,\frac{J\delta}{\sqrt{\lambda_{1}}},\,\frac{J\delta}{\sqrt{\lambda_{1}}}+\epsilon^{2}\,\Big]\,,\,\hat{x}_{k}\in\Big[\,-\frac{2J\delta}{\sqrt{\lambda_{k}}},\,\frac{2J\delta}{\sqrt{\lambda_{k}}}\,\Big]\,,\,2\leq k\leq d\,\Big\}\ .
\end{gather*}
Figure \ref{fig1} illustrates these definitions and the next ones.

\begin{figure}
\includegraphics[scale=0.2]{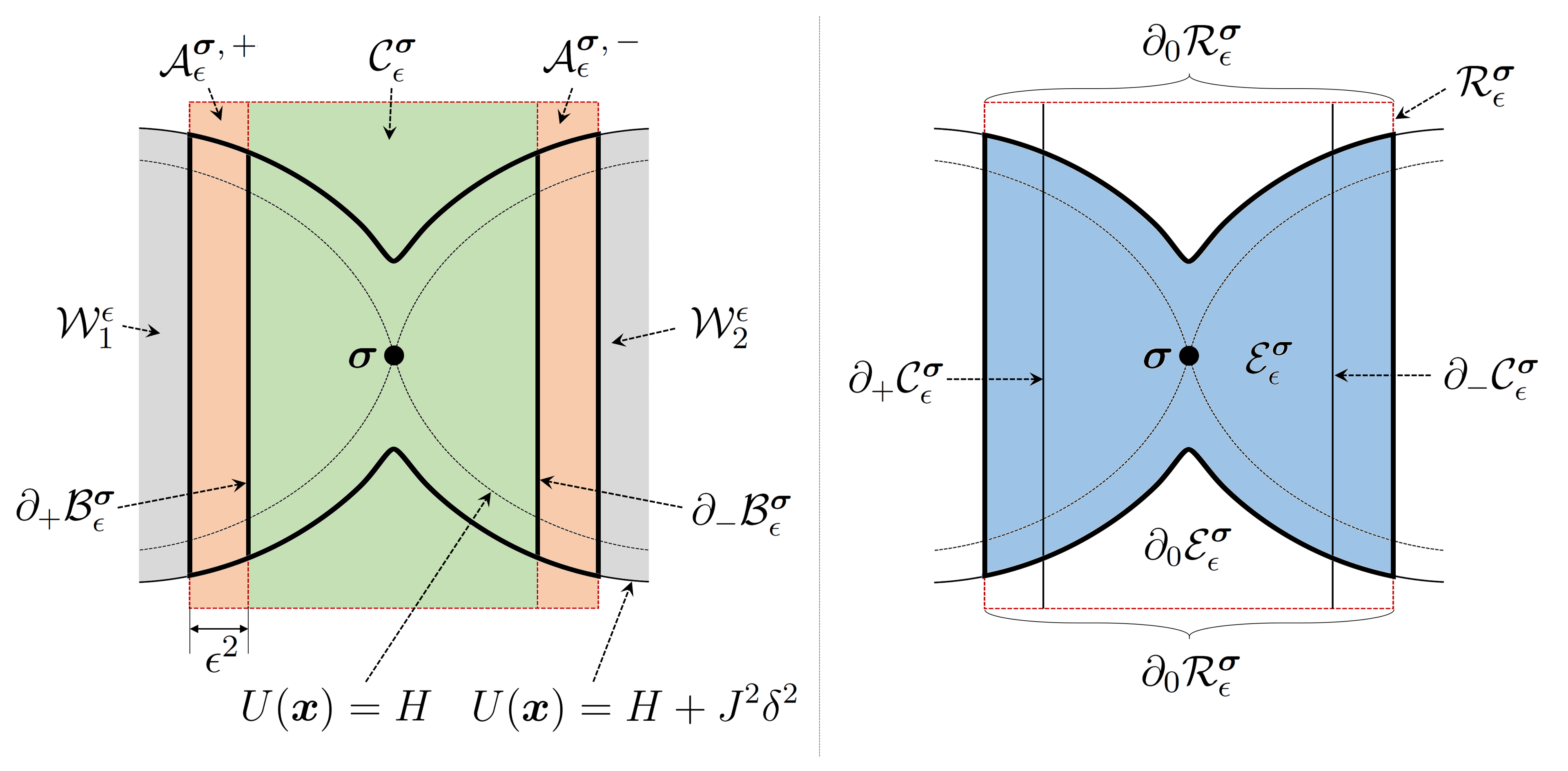} \caption{The sets around a saddle point $\boldsymbol{\sigma}$ appearing in
the definition of the test function.}
\label{fig1}
\end{figure}

Recall from \eqref{eq:omega} that $\mathbb{H}^{\boldsymbol{\sigma}}+\mathbb{L^{\boldsymbol{\sigma}}}$
has a unique negative eigenvalue, denoted by ${\color{blue}-\mu}$.
Denote by $\mathbb{A}^{\dagger}$ the transpose of a matrix $\mathbb{A}$.
By \cite[display (8.1)]{LS-22}, the matrix $\mathbb{H}^{\boldsymbol{\sigma}}-(\mathbb{L^{\boldsymbol{\sigma}}})^{\dagger}$
also has a unique negative eigenvalue equal to $-\mu$. Denote by
\textcolor{blue}{$\boldsymbol{v}$} the unit eigenvector of $\mathbb{H}^{\boldsymbol{\sigma}}-(\mathbb{L^{\boldsymbol{\sigma}}})^{\dagger}$
associated with $-\mu$. By \cite[Lemma 8.1]{LS-22}, $\boldsymbol{v}\cdot\boldsymbol{e}_{1}\neq0$.
We assume that $\boldsymbol{v}\cdot\boldsymbol{e}_{1}>0$, as we can
take $-\boldsymbol{v}$ instead of $\boldsymbol{v}$ if this inner
product is negative.

Let $p_{\epsilon}\colon\mathcal{C}_{\epsilon}\to\mathbb{R}$ be given
by
\begin{equation}
{\color{blue}p_{\epsilon}(\boldsymbol{x})}\ :=\ \frac{1}{M_{\epsilon}}\int_{-\infty}^{(\boldsymbol{x}-\boldsymbol{\sigma})\cdot\boldsymbol{v}}\,e^{-\frac{\mu}{2\epsilon}t^{2}}\,dt\;,\label{e_pesB}
\end{equation}
where the normalizing constant $M_{\epsilon}$ is given by
\begin{equation}
{\color{blue}M_{\epsilon}}\ =\ \int_{-\infty}^{\infty}\,e^{-\frac{\mu}{2\epsilon}t^{2}}\,dt\,=\,\sqrt{\frac{2\pi\epsilon}{\mu}}\;.\label{e_Ces}
\end{equation}

We extend continuously the function $p_{\epsilon}$ to the $d$-dimensional
rectangle ${\color{blue}\mathcal{R}_{\epsilon}=\mathcal{A}_{\epsilon}^{-}\cup\mathcal{C}_{\epsilon}\cup\mathcal{A}_{\epsilon}^{+}}$
as follows. For $\bm{x}=\boldsymbol{\sigma}+\sum_{k=1}^{d}\widehat{x}_{k}\boldsymbol{e}_{k}\in\mathcal{A}_{\epsilon}^{+}$,
let
\begin{equation}
\overline{\boldsymbol{x}}_{r}\ =\ \boldsymbol{\sigma}\,+\,\frac{J\delta}{\sqrt{\lambda_{1}}}\,\boldsymbol{e}_{1}\,+\,\sum_{k=2}^{d}\widehat{x}_{k}\,\boldsymbol{e}_{k}\;.\label{33}
\end{equation}
We define $\overline{\boldsymbol{x}}_{l}$ similarly for $\boldsymbol{x}\in\mathcal{A}_{\epsilon}^{-}$,
replacing on the right-hand side of the previous formula the first
plus sign by a minus sign. Clearly, $\overline{\boldsymbol{x}}_{r}$
and $\overline{\boldsymbol{x}}_{l}$ belong to $\mathcal{C}_{\epsilon}$.
We extend the definition of $p_{\epsilon}$ to $\mathcal{R}_{\epsilon}$
by setting $p_{\epsilon}\colon\mathcal{A}_{\epsilon}^{-}\cup\mathcal{A}_{\epsilon}^{+}\to\mathbb{R}$
as
\begin{equation}
\begin{gathered}p_{\epsilon}^{\boldsymbol{\sigma}}(\boldsymbol{x})\ =\ 1\,+\,\epsilon^{-2}\,\Big[\,\hat{x}_{1}-\frac{J\delta}{\sqrt{\lambda_{1}}}-\epsilon^{2}\,\Big]\,[\,1-p_{\epsilon}(\overline{\boldsymbol{x}}_{r})\,]\;,\quad\boldsymbol{x}\in\mathcal{A}_{\epsilon}^{+}\;,\\
p_{\epsilon}^{\boldsymbol{\sigma}}(\boldsymbol{x})\ =\ \epsilon^{-2}\,\Big[\,\hat{x}_{1}+\frac{J\delta}{\sqrt{\lambda_{1}}}+\epsilon^{2}\,\Big]\,p_{\epsilon}(\overline{\boldsymbol{x}}_{l})\;,\quad\boldsymbol{x}\in\mathcal{A}_{\epsilon}^{-}\;.
\end{gathered}
\label{e_pes}
\end{equation}

The function $p_{\epsilon}$ is an approximating solution of the Dirichlet
problem $\mathscr{L}_{\epsilon}^{*}f=0$, where \textcolor{blue}{$\mathscr{L}_{\epsilon}^{*}$}
is an adjoint of $\mathscr{L}_{\epsilon}$ with respect to $\mu_{\epsilon}(d\bm{x})$.
in $\mathcal{R}_{\epsilon}$ with boundary conditions $f=1$ on the
points of $\mathcal{R}_{\epsilon}$ where $\hat{x}_{1}=J\delta/\sqrt{\lambda_{1}}+\epsilon^{2}$
and $f=0$ on the ones such that $\hat{x}_{1}=-J\delta/\sqrt{\lambda_{1}}-\epsilon^{2}$.
This is the content of \cite[Proposition 6.2]{LS-22b}, which states
that the integral of $\,|\mathscr{L}_{\epsilon}^{*}f|$ on a set slightly
smaller than $\mathcal{R}_{\epsilon}$ divided by the measure of the
set where it is integrated vanishes as $\epsilon\to0$. This result
also justifies the definition of the test function $p_{\epsilon}$.

The test function $p_{\epsilon}(\cdot)$ constructed above depends
on $\boldsymbol{\sigma}$ and $\boldsymbol{m}$. To stress this fact,
it is sometimes represented by ${\color{blue}p_{\epsilon}^{\boldsymbol{\sigma},\boldsymbol{m}}(\cdot)}$.

\subsection{Wells and saddle gates}

In this subsection, for each $\mathcal{M}\in\mathscr{V}_{\textrm{nab}}^{(p)}$,
or more generally for each simple and bound $\mathcal{M}\subset\mathcal{M}_{0}$,
we rigorously define the well and the saddle gate associated with
$\mathcal{M}$.

Let $\mathcal{M}\subset\mathcal{M}_{0}$ be a simple
and bound set such that $\widetilde{\mathcal{M}}\ne\varnothing$.
By Lemma \ref{l_bound_conn}, there exists a connected component of
$\{U<\Theta(\mathcal{M},\,\widetilde{\mathcal{M}})\}$ containing
$\mathcal{M}$. Denote this set by $\mathcal{W}(\mathcal{M})$:
\begin{equation}
{\color{blue}\mathcal{W}(\mathcal{M})}\ \text{is the connected component of}\ \{U<\Theta(\mathcal{M},\,\widetilde{\mathcal{M}})\}\;\text{containing \ensuremath{\mathcal{M}}}\;.\label{e_W(M)}
\end{equation}
By Lemma \ref{l_bound_conn}, $\mathcal{M}^{*}\left(\mathcal{W}(\mathcal{M})\right)=\mathcal{M}$.
Let ${\color{blue}\mathcal{S}(\mathcal{M})}$ be the set of saddle
points $\boldsymbol{\sigma}\in\partial\mathcal{W}(\mathcal{M})\cap\mathcal{S}_{0}$
for which there exists $\delta_{0}>0$ such that
\begin{equation}
\mathcal{B}(\boldsymbol{\sigma},\,\delta)\cap\{\,U(\boldsymbol{x})<\Theta(\mathcal{M},\,\widetilde{\mathcal{M}})\,\}\;\text{is not contained in}\;\mathcal{W}(\mathcal{M})\;\text{for all}\;\delta\in(0,\delta_{0})\ .\label{eq:SM}
\end{equation}
By Lemma \ref{l_level_boundary}, as $\boldsymbol{\sigma}\in\partial\mathcal{W}(\mathcal{M})$,
$U(\boldsymbol{\sigma})=\Theta(\mathcal{M},\,\widetilde{\mathcal{M}})$
for all $\boldsymbol{\sigma}\in\mathcal{S}(\mathcal{M})$.
\begin{lem}
\label{l_gate}Let $\mathcal{M}\subset\mathcal{M}_{0}$ be a simple
bound set satisfying $\widetilde{\mathcal{M}}\ne\varnothing$. Then,
\begin{equation}
\mathcal{S}(\mathcal{M})\ =\ \left\{ \,\boldsymbol{\sigma}\in\partial\mathcal{W}(\mathcal{M})\cap\mathcal{S}_{0}:\boldsymbol{\sigma}\curvearrowright\boldsymbol{m}\;\;\text{for some}\;\;\boldsymbol{m}\notin\mathcal{W}(\mathcal{M})\,\right\} \ .\label{e_gate}
\end{equation}
\end{lem}

\begin{proof}
Denote by $\widetilde{\mathcal{S}}(\mathcal{M})$ the right-hand side
of \eqref{e_gate}. Fix $\boldsymbol{\sigma}\in\mathcal{S}(\mathcal{M})$.
By Lemma \ref{l_level_boundary}, as $\boldsymbol{\sigma}\in\partial\mathcal{W}(\mathcal{M})$,
$U(\boldsymbol{\sigma})=\Theta(\mathcal{M},\,\widetilde{\mathcal{M}})$.
Moreover, by Lemma \ref{l_path_saddle}-(1), there exists small $r'>0$
such that
\[
B(\boldsymbol{\sigma},\,r')\cap\{U<U(\boldsymbol{\sigma})\}
\]
has two connected components. Since $\boldsymbol{\sigma}$ satisfies
\eqref{eq:SM}, one of such components, denoted by $\mathcal{A}$,
is not contained in $\mathcal{W}(\mathcal{M})$. Thus, by Lemma \ref{l_level_connected},
$\mathcal{A}\cap\mathcal{W}(\mathcal{M})=\varnothing$.

By Lemma \ref{l_path_saddle}-(2), (3), there exist a local minimum
$\boldsymbol{m}$, a connected component $\mathcal{H}$ of $\{U<U(\boldsymbol{\sigma})\}$
containing $\boldsymbol{m}$ and $\mathcal{A}$, and a continuous
path $\phi:\mathbb{R}\to\mathbb{R}^{d}$ satisfying
\[
\dot{\phi}(t)\,=\,\boldsymbol{b}(\phi(t))\ \ ,\ \ \phi(\mathbb{R})\,\subset\,\mathcal{H}\ \ ,\ \ \lim_{t\to-\infty}\phi(t)\,=\,\boldsymbol{\sigma}\ ,\ \ \lim_{t\to\infty}\phi(t)\,=\,\boldsymbol{m}\ .
\]
In particular, $\boldsymbol{\sigma}\curvearrowright\boldsymbol{m}$.
Since $\mathcal{A}\subset\mathcal{H}$ and $\mathcal{A}\cap\mathcal{W}(\mathcal{M})=\varnothing$,
by Lemma \ref{l_level_connected}, $\mathcal{H}\cap\mathcal{W}(\mathcal{M})=\varnothing$
so that $\boldsymbol{m}\notin\mathcal{W}(\mathcal{M})$. Therefore,
$\boldsymbol{\sigma}\in\widetilde{\mathcal{S}}(\mathcal{M})$.

To prove the converse, fix $\boldsymbol{\sigma}\in\widetilde{\mathcal{S}}(\mathcal{M})$.
By Lemma \ref{l_level_boundary}, $U(\boldsymbol{\sigma})=\Theta(\mathcal{M},\,\widetilde{\mathcal{M}})$.
Since $\boldsymbol{\sigma}\in\widetilde{\mathcal{S}}(\mathcal{M})$,
there exists $\boldsymbol{m}_{0}\in\mathcal{M}_{0}\cap\mathcal{W}(\mathcal{M})^{c}$
such that $\boldsymbol{\sigma}\curvearrowright\boldsymbol{m}_{0}$.
On the other hand, by Lemma \ref{l_path_saddle}, there exist two
connected components $\mathcal{H}^{\pm}$ of $\{U<U(\boldsymbol{\sigma})\}$,
two local minima $\boldsymbol{m}^{\pm}$, and two continuous path
$\phi^{\pm}:\mathbb{R}\to\mathbb{R}^{d}$ such that
\begin{equation}
\dot{\phi}^{\pm}(t)\,=\,\boldsymbol{b}(\phi^{\pm}(t))\ \ ,\ \ \phi^{\pm}(\mathbb{R})\,\subset\,\mathcal{H}^{\pm}\ \ ,\ \ \lim_{t\to-\infty}\phi^{\pm}(t)\,=\,\boldsymbol{\sigma}\ \ ,\ \ \lim_{t\to\infty}\phi^{\pm}(t)\,=\,\boldsymbol{m}^{\pm}\ .\label{eapn4}
\end{equation}
Moreover, the set $B(\boldsymbol{\sigma},\,r')\cap\{U<U(\boldsymbol{\sigma})\}$
has two connected components for some $r'>0$, denoted by $\mathcal{A}^{-}$
and $\mathcal{A}^{+}$, and $\mathcal{A}^{\pm}\subset\mathcal{H}^{\pm}$.

By \eqref{eapn4} and definition \eqref{2-13}, $\bm{\sigma}\curvearrowright\bm{m}^{\pm}$.
Assume without loss of generality, that $\boldsymbol{m}_{0}=\boldsymbol{m}^{+}$.
Since $\boldsymbol{m}_{0}=\boldsymbol{m}^{+}\in\mathcal{H}^{+}\setminus\mathcal{W}(\mathcal{M})$,
$\mathcal{H}^{+}$ and $\mathcal{W}(\mathcal{M})$ are two distinct
connected component of the set $\{U<U(\boldsymbol{\sigma})\}$. Since
$\mathcal{A}^{+}\subset\mathcal{H}^{+}$, $\mathcal{A}^{+}\not\subset\mathcal{W}(\mathcal{M})$
so that $\boldsymbol{\sigma}\in\mathcal{S}(\mathcal{M})$, as claimed.
\end{proof}

\subsection{Construction of test function }

In this subsection, we fix $p\in\llbracket1,\,\mathfrak{q}\rrbracket$
and $\mathcal{M}\in\mathscr{V}_{{\rm nab}}^{(p)}$. Then, we shall
define a test function $Q_{\epsilon}=Q_{\epsilon}^{\mathcal{M}}$
on $\mathbb{R}^{d}$ associated with $\mathcal{M}$ with the help
of the test functions $p_{\epsilon}^{\boldsymbol{\sigma},\,\boldsymbol{m}}$
introduced in the previous subsection.

Write \textcolor{blue}{$H=\Theta(\mathcal{M},\,\widetilde{\mathcal{M}})=U(\mathcal{M})+d^{(p)}$}.
By Lemma \ref{l_level_boundary}, $U(\bm{\sigma})=H$ for $\bm{\sigma}\in\mathcal{S}(\mathcal{M})$
and Lemma \ref{l_exist_saddle}, $\mathcal{S}(\mathcal{M})\ne\varnothing$.

Fix $\eta>0$ small enough so that there is no critical point $\boldsymbol{c}$
of $U$ with height in the interval $(H,\,H+2\eta)$. Let ${\color{blue}\Omega(\mathcal{M})}$,
${\color{blue}\mathcal{K}_{\epsilon}(\mathcal{M})}$ be the connected
component of the sets $\left\{ U\le H+\eta\right\} $, $\left\{ U<H+J^{2}\delta^{2}\right\} $
which contains $\mathcal{W}(\mathcal{M})$, respectively. Denote by
\textcolor{blue}{$\partial_{0}\mathcal{R}_{\epsilon}^{\boldsymbol{\sigma}}$},
$\boldsymbol{\sigma}\in\mathcal{S}(\mathcal{M})$, the boundary of
the set $\mathcal{R}_{\epsilon}^{\boldsymbol{\sigma}}$, introduced
in the previous subsection, given by
\begin{gather*}
\partial_{0}\mathcal{R}_{\epsilon}^{\boldsymbol{\sigma}}\ =\ \Big\{\,\boldsymbol{x}\in\mathcal{C}_{\epsilon}^{\boldsymbol{\sigma}}\,:\,\hat{x}_{k}=\pm\frac{2J\delta}{\sqrt{\lambda_{k}}}\ \ \text{for some}\ 2\leq k\leq d\,\Big\}\;.
\end{gather*}
By the proof of \cite[Lemma 8.3]{LS-22},
\begin{equation}
U(\boldsymbol{x})\ \geq\ U(\boldsymbol{\sigma})+\frac{3}{2}\,J^{2}\,\delta^{2}\,[\,1+o_{\epsilon}(1)\,]\label{23}
\end{equation}
for all $\boldsymbol{x}\in\partial_{0}\mathcal{R}_{\epsilon}^{\boldsymbol{\sigma}}$.
In particular, $\partial_{0}\mathcal{R}_{\epsilon}^{\boldsymbol{\sigma}}$
is contained in the complement of $\mathcal{K}_{\epsilon}(\mathcal{M})$
provided that $\epsilon$ is sufficiently small.

Let ${\color{blue}\mathcal{E}_{\epsilon}^{\boldsymbol{\sigma}}=\mathcal{E}_{\epsilon}^{\boldsymbol{\sigma}}(\mathcal{M})\,:=\,\mathcal{R}_{\epsilon}^{\boldsymbol{\sigma}}\cap\mathcal{K}_{\epsilon}(\mathcal{M})}$,
$\boldsymbol{\sigma}\in\mathcal{S}(\mathcal{M})$. Denote by \textcolor{blue}{$\mathcal{W}_{1}^{\epsilon}(\mathcal{M})$}
the connected component of $\mathcal{K}_{\epsilon}(\mathcal{M})\setminus(\bigcup_{\boldsymbol{\sigma}\in\mathcal{S}(\mathcal{M})}\mathcal{E}_{\epsilon}^{\boldsymbol{\sigma}})$
which intersects $\mathcal{W}(\mathcal{M})$, and let $${\color{blue}\mathcal{W}_{2}^{\epsilon}(\mathcal{M})}:=\mathcal{K}_{\epsilon}(\mathcal{M})\setminus(\mathcal{W}_{1}^{\epsilon}\cup\bigcup_{\boldsymbol{\sigma}\in\mathcal{S}(\mathcal{M})}\mathcal{E}_{\epsilon}^{\boldsymbol{\sigma}})\;.
$$
With this notation,
\begin{equation}
\Omega(\mathcal{M})\;=\;\bigcup_{\boldsymbol{\sigma}\in\mathcal{S}(\mathcal{M})}\mathcal{E}_{\epsilon}^{\boldsymbol{\sigma}}\,\cup\,\mathcal{W}_{1}^{\epsilon}(\mathcal{M})\,\cup\,\mathcal{W}_{2}^{\epsilon}(\mathcal{M})\,\cup\,\big(\Omega(\mathcal{M})\setminus\mathcal{K}_{\epsilon}(\mathcal{M})\,\big)\;.\label{36}
\end{equation}
From now on, we omit $\mathcal{M}$ from the notation as much as possible.

Fix $\boldsymbol{\sigma}\in\mathcal{S}(\mathcal{M})\subset\partial\mathcal{W}(\mathcal{M})\cap\mathcal{S}_{0}$.
By Lemma \ref{l_assu_saddle}, there exists a local minimum in $\mathcal{W}(\mathcal{M})$,
denoted by ${\color{blue}\boldsymbol{m}_{\boldsymbol{\sigma}}}$,
such that $\boldsymbol{\sigma}\curvearrowright\boldsymbol{m}_{\boldsymbol{\sigma}}$.
Recall the notation introduced at the end of the previous subsection,
and let ${\color{blue}q^{\boldsymbol{\sigma}}=p^{\boldsymbol{\sigma},\boldsymbol{m}_{\boldsymbol{\sigma}}}}$.
Consider the test function ${\color{blue}Q_{\epsilon}\colon\mathcal{K}_{\epsilon}\to\mathbb{R}}$
given by
\begin{equation}
\begin{gathered}Q_{\epsilon}(\boldsymbol{x})\,=\,1\;,\;\;\boldsymbol{x}\in\mathcal{W}_{1}^{\epsilon}\;;\quad Q_{\epsilon}(\boldsymbol{y})\,=\,0\;,\;\;\boldsymbol{y}\in\mathcal{W}_{2}^{\epsilon}\;;\\
Q_{\epsilon}(\boldsymbol{x})\,=\,q_{\epsilon}^{\boldsymbol{\sigma}}(\boldsymbol{x})\;,\;\;\boldsymbol{x}\in\mathcal{E}_{\epsilon}^{\boldsymbol{\sigma}},\,\boldsymbol{\sigma}\in\mathcal{S}(\mathcal{M})\;.
\end{gathered}
\label{2-02}
\end{equation}

By \eqref{e_pes}. the function $Q_{\epsilon}$ is continuous on $\mathcal{K}_{\epsilon}$.
Moreover, if $\mathcal{G}_{\epsilon}$ represents the open set formed
by the union of the interiors of the set $\mathcal{E}_{\epsilon}^{\boldsymbol{\sigma}}$,
$\boldsymbol{\sigma}\in\mathcal{S}(\mathcal{M})$, and the interior
of the sets $\mathcal{W}_{i}^{\epsilon}$, $i=1$, $2$,
\[
\lVert\nabla Q_{\epsilon}\rVert_{L^{\infty}(\mathcal{G}_{\epsilon})}\,=\,O(\epsilon^{-1/2})\;\ \text{and}\;\ \|\Delta Q_{\epsilon}\|_{L^{\infty}(\mathcal{G}_{\epsilon})}\,=\,O(\epsilon^{-3/2})\ .
\]
We can extend $Q_{\epsilon}$ to $\Omega$ keeping these bounds out
of a $(d-1)$ dimensional manifold:
\begin{equation}
\lVert Q_{\epsilon}\rVert_{L^{\infty}(\Omega_{0})}\,\le\,1\;,\ \ \lVert\nabla Q_{\epsilon}\rVert_{L^{\infty}(\Omega_{0})}\,=\,O(\epsilon^{-1/2})\;,\ \;\text{and\;\;}\|\Delta Q_{\epsilon}\|_{L^{\infty}(\Omega_{0})}\,=\,O(\epsilon^{-3/2})\ ,\label{18}
\end{equation}
where $\Omega_{0}=\Omega\setminus\mathfrak{M}$, and $\mathfrak{M}$
is $(d-1)$ dimensional manifold at which the gradient of $Q_{\epsilon}$
is discontinuous. We further impose the condition that $Q_{\epsilon}$
vanishes away from $\Omega(\mathcal{M})$:
\begin{equation}
Q_{\epsilon}\ \equiv\ 0\ \text{ on }\ \bigg\{\,\boldsymbol{x}\in\mathbb{R}^{d}:U(\boldsymbol{x})>H+\frac{\eta}{2}\,\bigg\}\;,\label{2-03}
\end{equation}
respecting the previous bounds. The function $Q_{\epsilon}$ is the
test function associated to the well $\mathcal{W}(\mathcal{M})$.

Now we turn to the main estimate regarding the test function $Q_{\epsilon}$.

Let $\boldsymbol{\sigma}\in\mathcal{S}(\mathcal{M})$.
Since $B(\boldsymbol{\sigma},\,\delta)\cap\{U<H\}$ has two connected
components, there is a connected component $\mathcal{V}\ne\mathcal{W}(\mathcal{M})$
of $\{U<H\}$ intersecting with the ball $B(\boldsymbol{\sigma},\,\delta)$.
Hence, as $\boldsymbol{\sigma}\in\partial\mathcal{V}\cap\partial\mathcal{W}(\mathcal{M})$,
by Lemma \ref{l_assu_saddle}, there exist ${\color{blue}\boldsymbol{m}_{\boldsymbol{\sigma}}^{+},\,\boldsymbol{m}_{\boldsymbol{\sigma}}^{-}}\in\mathcal{M}_{0}$
such that $\boldsymbol{m}_{\boldsymbol{\sigma}}^{+}\in\mathcal{W}(\mathcal{M})$,
$\boldsymbol{m}_{\boldsymbol{\sigma}}^{-}\in\mathcal{V}$, and $\boldsymbol{m}_{\boldsymbol{\sigma}}^{-}\curvearrowleft\boldsymbol{\sigma}\curvearrowright\boldsymbol{m}_{\boldsymbol{\sigma}}^{+}$.

From now on, without loss of generality, we assume
that $\min_{\boldsymbol{x}\in\mathbb{R}^{d}}U(\boldsymbol{x})=0$,
since adding a constant to $U$ does not change the result.
\begin{lem}
\label{l2-01} For all ${\bf g}:\mathscr{V}^{(p)}\to\mathbb{R}$,
we have
\begin{align*}
 & \int_{\mathbb{R}^{d}}Q_{\epsilon}\phi_{\epsilon}\,d\mu_{\epsilon}\ =\ \big\{\,{\bf f}_{\epsilon}(\mathcal{M})\,+\,o_{\epsilon}(1)\,\big\}\,\mu_{\epsilon}(\mathcal{E}(\mathcal{M}))\;,\;\;\text{and}\\
 & \theta_{\epsilon}^{(p)}\int_{\mathbb{R}^{d}}Q_{\epsilon}(-\mathscr{L}_{\epsilon}\phi_{\epsilon})\,d\mu_{\epsilon}\ =\ \left(\frac{1}{2\pi\nu_{\star}}\,\sum_{\boldsymbol{\sigma}\in\mathcal{S}(\mathcal{M})}\,[\phi_{\epsilon}(\boldsymbol{m}_{\boldsymbol{\sigma}}^{+})-\phi_{\epsilon}(\boldsymbol{m}_{\boldsymbol{\sigma}}^{-})]\,\frac{\mu^{\boldsymbol{\sigma}}}{\sqrt{-\det\mathbb{H}^{\boldsymbol{\sigma}}}}+o_{\epsilon}(1)\right)\,e^{-U(\mathcal{M})/\epsilon}\ ,
\end{align*}
where ${\bf f}_{\epsilon}(\cdot)$ is defined in \eqref{05}.
\end{lem}

\begin{proof}
The proof is identical to that of \cite[display (9.3) and Lemma 8.5]{LLS-1st}.
To perceive that the second identity corresponds to \cite[Lemma 8.5]{LLS-1st},
observe that we have $e^{H/\epsilon}=\theta_{\epsilon}^{(1)}\,e^{U(\mathcal{M})/\epsilon}$
in the situation of \cite[Lemma 8.5]{LLS-1st} because $H$ is the
height of the saddle point so that $\theta_{\epsilon}^{(1)}=e^{[H-U(\mathcal{M})]/\epsilon}$.
\end{proof}

\section{\label{sec10}Analysis of resolvent equation: characterization}

In this section, we prove Proposition \ref{prop_R}.

Recall from Section \ref{sec3} the definition of $\mathfrak{H}^{(p)}$,
and from \eqref{05} the definition of the function ${\bf f}_{\epsilon}\colon\mathscr{S}^{(p)}\to\mathbb{R}$.
We prove in this section the following proposition.
\begin{prop}
\label{p_res} Fix $p\in\llbracket1,\,\mathfrak{q}\rrbracket$, and
assume that $\mathfrak{H}^{(p)}$ holds. Then, for all $\lambda>0$
and ${\bf g}:\mathscr{V}^{(p)}\to\mathbb{R}$,
\begin{equation}
\left(\lambda-\mathfrak{L}^{(p)}\right){\bf f}_{\epsilon}\ =\ {\bf g}+o_{\epsilon}(1)\ .\label{06}
\end{equation}
\end{prop}

Now we assume this proposition and prove Proposition \ref{prop_R}.
\begin{proof}[Proof of Proposition \ref{prop_R}]
The assertion follows from two observations. The sequence ${\bf f}_{\epsilon}$
is uniformly bounded and the equation $(\lambda\,-\,\mathfrak{L}^{(p)})\,{\bf f}\;=\;{\bf g}$
has a unique solution.
\end{proof}
In the remainder of the section, we focus on the proof of Proposition
\ref{p_res}. \emph{Thus, we fix $p\in\llbracket1,\,\mathfrak{q}\rrbracket$and
assume that $\mathfrak{H}^{(p)}$ is in force throughout the remainder
of the section. }

\subsection{Characterization on absorbing states}

In this subsection, we first show that \eqref{06} holds on the set
$\mathscr{V}_{{\rm ab}}^{(p)}$.
\begin{lem}
\label{l09} For all $\mathcal{M}\in\mathscr{V}_{{\rm ab}}^{(p)}$.
\[
\lim_{\epsilon\to0}\,{\bf f}_{\epsilon}(\mathcal{M})\ =\ {\bf f}(\mathcal{M})\ =\ \frac{1}{\lambda}\,{\bf g}(\mathcal{M})\ .
\]
\end{lem}

\begin{proof}
Fix $\mathcal{M}\in\mathscr{V}_{{\rm ab}}^{(p)}$. Denote by $\mathcal{A}_{b}$,
$b>0$, the connected component of the set $\{U\le U(\mathcal{M})+d^{(p)}+b\}$
which contains the set $\mathcal{M}$. Let $a>0$ be a small enough
constant satisfying
\[
4a\ <\ \Xi(\mathcal{M})-d^{(p)}\ .
\]
Then, by Lemma \ref{l_bound_conn}, $\mathcal{A}_{4a}$ is well-defined
and $\mathcal{M}=\mathcal{M}^{*}(\mathcal{A}_{4a})$. Also since $\mathcal{M}\in\mathscr{V}^{(p+1)}$,
by Lemma \ref{lem:FW-type} and \ref{lap01}, $\mathcal{A}_{4a}\cap\mathcal{E}^{(p)}=\mathcal{E}(\mathcal{M})$.

Let $Q\colon\mathbb{R}^{d}\to\mathbb{R}_{+}$be a smooth function
such that
\[
Q(\boldsymbol{x})\,=\,1\;,\;\;Q(\boldsymbol{y})\,=\,0\;,\quad\boldsymbol{x}\in\mathscr{A}_{2a}\;,\;\;\boldsymbol{y}\in\mathcal{A}_{4a}^{c}\;.
\]
By \eqref{res},
\[
\int_{\mathcal{A}_{4a}}\,Q\,\left(\lambda-\theta_{\epsilon}^{(p)}\,\mathscr{L}_{\epsilon}\right)\,\phi_{\epsilon}\,d\mu_{\epsilon}\;=\;\sum_{\mathcal{\mathcal{M}'}\in\mathscr{V}^{(p)}}{\bf g}(\mathcal{M}')\,\int_{\mathcal{A}_{4a}\cap\mathcal{E}(\mathcal{M}')}\,Q\,d\mu_{\epsilon}\;.
\]
Since $\mathcal{A}_{4a}\cap\mathcal{E}^{(p)}=\mathcal{E}(\mathcal{M})$,
the right-hand side is equal to
\[
{\bf g}(\mathcal{M})\,\mu_{\epsilon}(\mathcal{E}(\mathcal{M}))\;=\;\lambda\,{\bf f}(\mathcal{M})\,\mu_{\epsilon}(\mathcal{E}(\mathcal{M}))\;,
\]
where the last identity follows from the definition of $\boldsymbol{g}$
and the fact that $(\mathfrak{L}^{(p)}{\bf f})(\mathcal{M})=0$ because
$\mathcal{M}$ is an absorbing point.

We turn to the left-hand side. Since $Q$ vanishes at the boundary
of $\mathcal{A}_{4a}$, by the divergence theorem and the properties
\eqref{eq:decb} of the divergence free field $\boldsymbol{\ell}$,
it is equal to
\[
\lambda\int_{\mathcal{A}_{4a}}Q\,\phi_{\epsilon}\,d\mu_{\epsilon}\,-\,\epsilon\,\theta_{\epsilon}^{(p)}\,\int_{\mathscr{A}_{4a}}\nabla Q\cdot\nabla\phi_{\epsilon}\;d\mu_{\epsilon}\,+\,\theta_{\epsilon}^{(p)}\,\int_{\mathscr{A}_{4a}}\nabla Q\cdot\boldsymbol{\ell}\;\phi_{\epsilon}\;d\mu_{\epsilon}\ .
\]

Since $\mathcal{M}=\mathcal{M}^{*}(\mathcal{A}_{4a})$, it is clear
from Laplace asymptotics that $\mu_{\epsilon}(\mathcal{A}_{4a}\setminus\mathcal{E}(\mathcal{M}))\prec\mu_{\epsilon}(\mathcal{E}(\mathcal{M}))$.
Therefore by definition of $Q$ and Proposition \ref{p_flat}, the
first term is equal to
\[
\mu_{\epsilon}(\mathcal{E}(\mathcal{M}))\,[\,\lambda\,{\bf f}_{\epsilon}(\mathcal{M})+o_{\epsilon}(1)\,]\;.
\]
As $Q$ is smooth and is constant on $\mathcal{A}_{2a}$, the absolute
value of the third term is bounded by
\begin{equation}
C_{0}\,\theta_{\epsilon}^{(p)}\,\mu_{\epsilon}(\mathcal{A}_{4a}\setminus\mathcal{A}_{2a})\label{12}
\end{equation}
for some finite constant $C_{0}$ independent of $\epsilon$.

We turn to the second term. By the divergence theorem and since $\nabla Q$
vanishes at the boundary of $\mathcal{A}_{a}$, it is equal to
\[
\theta_{\epsilon}^{(p)}\,\int_{(\mathcal{A}_{a})^{c}}\phi_{\epsilon}\,\big\{\,\epsilon\,\Delta Q\,-\,\nabla Q\cdot\nabla U\,\big\}\;d\mu_{\epsilon}\;.
\]
This expression is bounded by $C_{0}\,\theta_{\epsilon}^{(p)}\,\mu_{\epsilon}(\mathcal{A}_{4a}\setminus\mathcal{A}_{a})$.
By definition of $\theta_{\epsilon}^{(p)}$ and the sets $\mathcal{A}_{b}$,
this expression, as well as \eqref{12}, is equal to $o_{\epsilon}(1)\,\mu_{\epsilon}(\mathcal{E}(\mathcal{M}))$.
Putting together the previous estimates yields that
\[
\int_{\mathscr{A}_{4a}}\,Q\,\left(\lambda-\theta_{\epsilon}^{(p)}\,\mathscr{L}_{\epsilon}\right)\,\phi_{\epsilon}\,d\mu_{\epsilon}\;=\;\mu_{\epsilon}(\mathcal{E}(\mathcal{M}))\,[\,\lambda\,{\bf f}_{\epsilon}(\mathcal{M})+o_{\epsilon}(1)\,]\;.
\]

In conclusion, we proved that
\[
{\bf g}(\mathcal{M})\,\mu_{\epsilon}(\mathcal{E}(\mathcal{M}))\;=\;\lambda\,{\bf f}(\mathcal{M})\,\mu_{\epsilon}(\mathcal{E}(\mathcal{M}))\;=\;\mu_{\epsilon}(\mathcal{E}(\mathcal{M}))\,[\,\lambda\,{\bf f}_{\epsilon}(\mathcal{M})+o_{\epsilon}(1)\,]\;.
\]
The assertion of the lemma follows from this identity.
\end{proof}

\subsection{Characterization on negligible valleys }

Note that since $\mathscr{N}^{(1)}=\varnothing$, the following proposition
is immediate for $p=1$.
\begin{prop}
\label{p_char_neg} For all $\mathcal{M}\in\mathscr{N}^{(p)}$,
\[
{\bf f}_{\epsilon}(\mathcal{M})\;=\;\sum_{\mathcal{M}'\in\mathscr{V}^{(p)}}\,\widehat{\mathcal{Q}}_{\mathcal{M}}^{(p)}\left[\,\tau_{\mathscr{V}^{(p)}}=\tau_{\mathcal{M}'}\,\right]\,{\bf f}_{\epsilon}(\mathcal{M}')+o_{\epsilon}(1)\ .
\]
\end{prop}

\begin{proof}
Fix $\mathcal{M}\in\mathscr{N}^{(p)}$. Let $\alpha_{\epsilon}$ be
a sequence such that $\theta_{\epsilon}^{(p-1)}\prec\alpha_{\epsilon}\prec\theta_{\epsilon}^{(p)}$.
Fix $\boldsymbol{x}\in\mathcal{E}(\mathcal{M})$. By Proposition \ref{p_flat},
\[
\phi_{\epsilon}(\boldsymbol{x})\;=\;{\bf f}_{\epsilon}(\mathcal{M})\,+\,o_{\epsilon}(1)\;.
\]
On the other hand, since $G$ is bounded and $\theta_{\epsilon}^{(p-1)}\prec\alpha_{\epsilon}\prec\theta_{\epsilon}^{(p)}$,
by $\mathfrak{H}^{(p)}$-(1) and by the strong Markov property,
\begin{align*}
\phi_{\epsilon}(\boldsymbol{x}) & \;=\;\mathbb{E}_{\boldsymbol{x}}^{\epsilon}\left[\,\int_{0}^{\infty}e^{-\lambda s}\,G(\boldsymbol{x}_{\epsilon}(\theta_{\epsilon}^{(p)}s))\,ds\,{\bf 1}\{\,\tau_{\mathcal{E}^{(p)}}\le\alpha_{\epsilon}\,\}\,\right]\,+\,o_{\epsilon}(1)\\
 & \;=\;\mathbb{E}_{\boldsymbol{x}}^{\epsilon}\left[\,\mathbb{E}_{\boldsymbol{x}_{\epsilon}(\tau_{\mathcal{E}^{(p)}})}^{\epsilon}\Big[\,\int_{0}^{\infty}e^{-\lambda s}\,G(\boldsymbol{x}_{\epsilon}(\theta_{\epsilon}^{(p)}s))\,ds\,\Big]\,{\bf 1}\{\,\tau_{\mathcal{E}^{(p)}}\le\alpha_{\epsilon}\,\}\,\right]\,+\,o_{\epsilon}(1)\ .
\end{align*}
By $\mathfrak{H}^{(p)}$-(1), we may remove the indicator in the previous
expectation and rewrite the right-hand side as
\[
\sum_{\mathcal{M}'\in\mathcal{E}^{(p)}}\,\mathbb{E}_{\boldsymbol{x}}^{\epsilon}\left[\,\phi_{\epsilon}(\boldsymbol{x}_{\epsilon}(\tau_{\mathcal{E}(\mathcal{M}')}))\,{\bf 1}\{\,\tau_{\mathcal{E}^{(p)}}=\tau_{\mathcal{E}(\mathcal{M}')}\,\}\,\right]\,+\,o_{\epsilon}(1)\ .
\]
By Proposition \ref{p_flat} for $\boldsymbol{y}\in\mathcal{M}'$
instead of $\boldsymbol{x}\in\mathcal{M}$, this expression is equal
to

\[
\sum_{\mathcal{M}'\in\mathcal{E}^{(p)}}\,{\bf f}_{\epsilon}(\mathcal{M}')\,\mathbb{P}_{\boldsymbol{x}}^{\epsilon}\left[\,\tau_{\mathcal{E}^{(p)}}=\tau_{\mathcal{E}(\mathcal{M}')}\,\right]\,+\,o_{\epsilon}(1)\ .
\]
Applying $\mathfrak{H}^{(p)}$-(2) completes the proof of the proposition.
\end{proof}

\subsection{Characterization on non-absorbing states}

Fix $\mathcal{M}\in\mathscr{V}_{{\rm nab}}^{(p)}$ in this subsection.
Recall the definition of the sets $\mathcal{W}(\mathcal{M})$, $\Omega(\mathcal{M})$
and $\mathcal{S}(\mathcal{M})$ introduced in Section \ref{sec9},
and the definition of the test function $Q_{\epsilon}$.
\begin{lem}
\label{l_p_res-0} For all $\boldsymbol{m}\in\mathcal{W}(\mathcal{M})\cap\mathcal{M}_{0}$,
\[
\lim_{\epsilon\to0}\,\left|\,\phi_{\epsilon}(\boldsymbol{m})-{\bf f}_{\epsilon}(\mathcal{M})\,\right|\;=\;0\ .
\]
\end{lem}

\begin{proof}
Since $\mathcal{M}\in\mathscr{V}_{{\rm nab}}^{(p)}$, $\Xi(\mathcal{M})=d^{(p)}$.
By Lemma \ref{l_bound_conn} and definition, $\mathcal{W}(\mathcal{M})$
is the connected component of the set $\{U<U(\mathcal{M})+d^{(p)}\}$
which contains $\mathcal{M}$ and $\mathcal{M}=\mathcal{M}^{*}\left(\mathcal{W}(\mathcal{M})\right)$.
By Lemma \ref{lem_separate}, the set $\mathcal{W}(\mathcal{M})$
does not separate $(p)$-states. Hence, by Lemma \ref{lem_escape}-(2),
$\mathscr{V}^{(p)}\left(\mathcal{W}(\mathcal{M})\right)=\{\mathcal{M}\}$.

Fix $\boldsymbol{m}\in\mathcal{W}(\mathcal{M})\cap\mathcal{M}_{0}$.
If $\boldsymbol{m}\in\mathcal{M}$, the assertion follows from Proposition
\ref{p_flat}. Suppose that $\boldsymbol{m}\in(\mathcal{W}(\mathcal{M})\cap\mathcal{M}_{0})\setminus\mathcal{M}$.
Let $\mathcal{M}'=\mathcal{M}(p,\,\bm{m})$. Since $\mathcal{W}(\mathcal{M})$
does not separate $(p)$-states, $\mathcal{M}'\in\mathscr{S}^{(p)}\left(\mathcal{W}(\mathcal{M})\right)$.
Since $\boldsymbol{m}\not\in\mathcal{M}$ and $\mathscr{V}^{(p)}\left(\mathcal{W}(\mathcal{M})\right)=\{\mathcal{M}\}$,
$\mathcal{M}'\in\mathscr{N}^{(p)}\left(\mathcal{W}(\mathcal{M})\right)$.

We claim that
\begin{equation}
\widehat{\mathcal{Q}}_{\mathcal{M}'}^{(p)}[\,\tau_{\mathscr{V}^{(p)}}=\tau_{\mathcal{M}}\,]\;=\;1\;.\label{2-01}
\end{equation}
Indeed, by Lemma \ref{l_irred_negli} for $\ell=p-1$, the Markov
chain $\widehat{{\bf y}}^{(p)}(\cdot)$ starting from $\mathcal{M}'\in\mathscr{N}^{(p)}$
must visit an element of $\mathscr{V}^{(p)}$. Suppose that $\widehat{{\bf y}}^{(p)}$
can hit $\mathcal{M}_{1}\in\mathscr{V}^{(p)}$, $\mathcal{M}_{1}\neq\mathcal{M}$,
before visiting $\mathcal{M}$. Since $\mathcal{M}_{1}\subset\left(\mathcal{W}(\mathcal{M})\right)^{c}$,
$\widehat{{\bf y}}^{(p)}$ escapes from $\mathscr{S}^{(p)}\left(\mathcal{W}(\mathcal{M})\right)$
before visiting $\mathcal{M}$. Let $\mathcal{M}_{2}$ be the last
state in $\mathscr{S}^{(p)}\left(\mathcal{W}(\mathcal{M})\right)$
and $\mathcal{M}_{3}$ be the first state in $\mathscr{S}^{(p)}\left(\left(\mathcal{W}(\mathcal{M})\right)^{c}\right)$
visited by $\widehat{{\bf y}}^{(p)}(\cdot)$ in its path from $\mathcal{M}'$
to $\mathcal{M}_{1}$ without hitting $\mathcal{M}$. In particular,
$\widehat{r}^{(p)}(\mathcal{M}_{2},\,\mathcal{M}_{3})>0$ so that,
by $\mathfrak{P}_{3}(p)$, $\mathcal{M}_{2}\to\mathcal{M}_{3}$. By
Lemma \ref{lem_escape}-(1) and since $\mathcal{M}=\mathcal{M}^{*}\left(\mathcal{W}(\mathcal{M})\right)$,
$\mathcal{M}_{2}=\mathcal{M}$ which is a contradiction.

By \eqref{2-01} and Lemma \ref{p_char_neg},
\[
{\bf f}_{\epsilon}(\mathcal{M}')\;=\;{\bf f}_{\epsilon}(\mathcal{M})+o_{\epsilon}(1)\ .
\]
At this point, the assertion of the lemma follows from Proposition
\ref{p_flat}.
\end{proof}
By the previous result, Proposition \ref{p_flat} and Lemma \ref{l2-01},
\begin{equation}
\theta_{\epsilon}^{(p)}\int_{\mathbb{R}^{d}}\,Q_{\epsilon}\,(-\mathscr{L}_{\epsilon}\phi_{\epsilon})\,d\mu_{\epsilon}\;=\;\left(\,\frac{1}{2\pi\nu_{\star}}\sum_{\boldsymbol{\sigma}\in\mathcal{S}(\mathcal{M})}\,[{\bf f}_{\epsilon}(\mathcal{M})-{\bf f}_{\epsilon}(\mathcal{M}_{\boldsymbol{\sigma}})]\,\frac{\mu^{\boldsymbol{\sigma}}}{\sqrt{-\det\mathbb{H}^{\boldsymbol{\sigma}}}}+o_{\epsilon}(1)\,\right)e^{-U(\mathcal{M})/\epsilon}\ ,\label{l_p_res-1}
\end{equation}
where $\mathcal{M}_{\boldsymbol{\sigma}}=\mathcal{M}(p,\,\boldsymbol{m}_{\boldsymbol{\sigma}}^{-})$.

Recall from \eqref{eq:chain_p+1} that we denote by $\widehat{\mathfrak{L}}^{(p)}$
the generator of the Markov chain $\widehat{{\bf y}}^{(p)}$.
\begin{lem}
\label{l2-02} For all ${\bf g}:\mathscr{V}^{(p)}\to\mathbb{R}$ and
$\mathcal{M}\in\mathscr{V}_{{\rm nab}}^{(p)}$,
\[
\theta_{\epsilon}^{(p)}\int_{\mathbb{R}^{d}}\,Q_{\epsilon}(-\mathscr{L}_{\epsilon}\phi_{\epsilon})\,d\mu_{\epsilon}\;=\;-\,\left[\,\frac{\nu(\mathcal{M})}{\nu_{\star}}\,\big(\widehat{\mathfrak{L}}^{(p)}{\bf f}_{\epsilon}\big)(\mathcal{M})\,+\,o_{\epsilon}(1)\,\right]\,e^{-U(\mathcal{M})/\epsilon}\ .
\]
\end{lem}

\begin{proof}
By \eqref{l_p_res-1} and the definition of $\omega(\boldsymbol{\sigma})$
given in \eqref{eq:omega}, it suffices to show that
\[
\sum_{\boldsymbol{\sigma}\in\mathcal{S}(\mathcal{M})}\,[\,{\bf f}_{\epsilon}(\mathcal{M}_{\boldsymbol{\sigma}})-{\bf f}_{\epsilon}(\mathcal{M})\,]\,\omega(\boldsymbol{\sigma})\;=\;\nu(\mathcal{M})\,(\widehat{\mathfrak{L}}^{(p)}{\bf f}_{\epsilon})\,(\mathcal{M})\ .
\]

By the definition \eqref{eq:rate_3} of the generator $\widehat{\mathfrak{L}}^{(p)}$
, since $\Xi(\mathcal{M})=d^{(p)}$,
\[
\nu(\mathcal{M})\,(\widehat{\mathfrak{L}}^{(p)}{\bf f}_{\epsilon})\,(\mathcal{M})\;=\;\sum_{\mathcal{M}'\in\mathscr{S}^{(p)}}\,\omega_{p}(\mathcal{M},\,\mathcal{M}')\,\left[\,{\bf f}_{\epsilon}(\mathcal{M}')-{\bf f}_{\epsilon}(\mathcal{M})\,\right]\ .
\]
By \eqref{eq:rate_30}, we may restrict the sum to sets $\mathcal{M}'$
such that $\mathcal{M}\to\mathcal{M}'$ and rewrite the previous expression
as
\[
\sum_{\mathcal{M}'\in\mathscr{S}^{(p)}}\,\sum_{\boldsymbol{\sigma}\in\mathcal{S}(\mathcal{M},\,\mathcal{M}')}\,\omega(\boldsymbol{\sigma})\,\left[\,{\bf f}_{\epsilon}(\mathcal{M}')-{\bf f}_{\epsilon}(\mathcal{M})\,\right]\ .
\]
For $\boldsymbol{\sigma}\in\mathcal{S}(\mathcal{M},\,\mathcal{M}')$,
the set $\mathcal{M}_{\boldsymbol{\sigma}}$ appearing in equation
\eqref{l_p_res-1} is equal to $\mathcal{M}'$. To complete the proof,
it remains to prove that
\[
\mathcal{S}(\mathcal{M})\;=\;\bigcup_{\mathcal{M}'\in\mathscr{S}^{(p)}}\mathcal{S}(\mathcal{M},\,\mathcal{M}')\ ,
\]
which is a content of the next lemma.
\end{proof}
\begin{lem}
\label{l_SM}For each $\mathcal{M}\in\mathscr{V}_{{\rm nab}}^{(p)}$,
\begin{equation}
\mathcal{S}(\mathcal{M})\;=\;\bigcup_{\mathcal{M}'\in\mathscr{S}^{(p)}}\mathcal{S}(\mathcal{M},\,\mathcal{M}')\ ,\label{eq:l_SM}
\end{equation}
where the right-hand side represents a disjoint union.
\end{lem}

\begin{proof}
Since the disjointness of the union at the right-hand side of \eqref{eq:l_SM}
is immediate as a saddle cannot be connected with three different
minima, we focus only on the proof of \eqref{eq:l_SM}.

Fix $\boldsymbol{\sigma}\in\mathcal{S}(\mathcal{M})$. By \eqref{e_gate},
$\boldsymbol{\sigma}\in\partial\mathcal{W}(\mathcal{M})$ and there
exists $\boldsymbol{m}_{1}\notin\mathcal{W}(\mathcal{M})$ such that
$\boldsymbol{\sigma}\curvearrowright\boldsymbol{m}_{1}$. Therefore,
by Lemma \ref{l_level_boundary},
\[
U(\boldsymbol{\sigma})\;=\;\Theta(\mathcal{M},\,\widetilde{\mathcal{M}})\;=\;U(\mathcal{M})+\Xi(\mathcal{M})\;.
\]
As $\boldsymbol{m}_{1}\notin\mathcal{W}(\mathcal{M})$ and $\boldsymbol{\sigma}\curvearrowright\boldsymbol{m}_{1}$,
$\Theta(\mathcal{M},\,\boldsymbol{m}_{1})=U(\boldsymbol{\sigma})$
so that
\[
\Theta(\mathcal{M},\,\boldsymbol{m}_{1})\;=\;\Theta(\mathcal{M},\,\widetilde{\mathcal{M}})\ .
\]

Let $\mathcal{M}'=\mathcal{M}(p,\,\bm{m}_{1})$. We claim that
\begin{equation}
\Theta(\mathcal{M},\,\mathcal{M}')\,=\,U(\boldsymbol{\sigma})\quad\text{and that}\quad\mathcal{M}\leftsquigarrow\boldsymbol{\sigma}\curvearrowright\mathcal{M}'\ .\label{09}
\end{equation}
By Lemma \ref{l_assu_saddle}-(3), $\mathcal{M}\leftsquigarrow\boldsymbol{\sigma}$.
Hence, as $\boldsymbol{\sigma}\curvearrowright\boldsymbol{m}_{1}\in\mathcal{M}'$,
it remains to show that $\Theta(\mathcal{M},\,\mathcal{M}')=U(\boldsymbol{\sigma})$.
As $\boldsymbol{m}_{1}\in\mathcal{M}'$,
\[
\Theta(\mathcal{M},\,\mathcal{M}')\,\le\,\Theta(\mathcal{M},\,\boldsymbol{m}_{1})\,=\,\Theta(\mathcal{M},\,\widetilde{\mathcal{M}})\ .
\]
Since $\boldsymbol{m}_{1}\notin\mathcal{W}(\mathcal{M})$ and $\mathcal{M}'\in\mathscr{S}^{(p)}$,
by Lemma \ref{lem_separate}, $\mathcal{M}'\subset\mathcal{W}(\mathcal{M})^{c}$.
Hence, by Lemma \ref{lap01}-(2), $\Theta(\mathcal{M},\,\widetilde{\mathcal{M}})\le\Theta(\mathcal{M},\,\mathcal{M}')$.
This completes the proof of \eqref{09} and shows that $\boldsymbol{\sigma}\in\mathcal{S}(\mathcal{M},\,\mathcal{M}')$.
Therefore, left-hand side of \eqref{eq:l_SM} is a subset of the right-hand
side.

To prove the reversed inclusion, fix $\boldsymbol{\sigma}\in\mathcal{S}(\mathcal{M},\,\mathcal{M}')$
for some $\mathcal{M}'\in\mathscr{S}^{(p)}$. By definition of $\mathcal{S}(\mathcal{M},\,\mathcal{M}')$
and by Lemma \ref{lem_not}-(4), there exists $\boldsymbol{m}_{2}\in\mathcal{M}'$
such that $\mathcal{M}\leftsquigarrow\boldsymbol{\sigma}\curvearrowright\boldsymbol{m}_{2}$
and $U(\boldsymbol{\sigma})=\Theta(\mathcal{M},\,\widetilde{\mathcal{M}})=\Theta(\mathcal{M},\,\boldsymbol{m}_{2})$.
By Lemma \ref{lap01}-(1), $\boldsymbol{m}_{2}\notin\mathcal{W}(\mathcal{M})$.
It remains to prove $\bm{\sigma}\in\partial\mathcal{W}(\mathcal{M})$.
Let $\boldsymbol{m}_{3}\in\mathcal{M}$ be such that $\boldsymbol{\sigma}\rightsquigarrow\boldsymbol{m}_{3}$.
Since $\mathcal{W}(\mathcal{M})$ is a connected component of $\{U<U(\boldsymbol{\sigma})\}$
containing $\boldsymbol{m}_{3}$ and $U(\boldsymbol{\sigma})=\Theta(\mathcal{M},\,\widetilde{\mathcal{M}})$,
by Lemma \ref{l_squig_saddle}, we have $\boldsymbol{\sigma}\in\partial\mathcal{W}(\mathcal{M})$
so that $\boldsymbol{\sigma}\in\mathcal{S}(\mathcal{M})$ and thus
the proof is completed.
\end{proof}
\begin{rem}
\label{rem_assu_saddle} A careful reading, shows that assumption
\eqref{hyp2} is not needed for all saddle points, but only for those
in $\bigcup_{p=1}^{\mathfrak{q}}\bigcup_{\mathcal{M}\in\mathscr{V}_{{\rm nab}}^{(p)}}\mathcal{S}(\mathcal{M})$.
A rigorous proof of this claim is, however, much longer.
\end{rem}

\subsection{Proof of Proposition \ref{p_res}\label{subsec_pf_t_res-gen}}

The proof of Proposition \ref{p_res} relies on the following result.
\begin{lem}
\label{l_p_res-3} For all ${\bf g}:\mathscr{V}^{(p)}\to\mathbb{R}$
and $\mathcal{M}\in\mathscr{V}^{(p)}$,
\[
(\widehat{\mathfrak{L}}^{(p)}{\bf f}_{\epsilon})\,(\mathcal{M})\;=\;(\mathfrak{L}^{(p)}{\bf f}_{\epsilon})\,(\mathcal{M})\,+\,o_{\epsilon}(1)\ .
\]
\end{lem}

\begin{proof}
Fix $\mathcal{M}\in\mathscr{V}^{(p)}$. By definition,
\[
(\widehat{\mathfrak{L}}^{(p)}{\bf f}_{\epsilon})\,(\mathcal{M})\,=\,\sum_{\mathcal{M}'\in\mathscr{S}^{(p)}}\widehat{r}^{(p)}(\mathcal{M},\,\mathcal{M}')\,\left(\,{\bf f}_{\epsilon}(\mathcal{M}')-{\bf f}_{\epsilon}(\mathcal{M})\,\right)\ .
\]
By Proposition \ref{p_char_neg}, for all $\mathcal{M}'\in\mathscr{N}^{(p)}$,
\[
{\bf f}_{\epsilon}(\mathcal{M}')\;=\;\sum_{\mathcal{M}''\in\mathscr{V}^{(p)}}\,\widehat{\mathcal{Q}}_{\mathcal{M}'}^{(p)}\left[\,\tau_{\mathscr{V}^{(p)}}=\tau_{\mathcal{M}''}\,\right]\,{\bf f}_{\epsilon}(\mathcal{M}'')+o_{\epsilon}(1)\ ,
\]
For $\mathcal{M}'\in\mathscr{V}^{(p)}$ this identity also holds.
It is therefore valid for all $\mathcal{M}'\in\mathscr{S}^{(p)}$.
Therefore,
\[
(\widehat{\mathfrak{L}}^{(p)}{\bf f}_{\epsilon})\,(\mathcal{M})\;=\;\sum_{\mathcal{M}'\in\mathscr{S}^{(p)}}\,\sum_{\mathcal{M}''\in\mathscr{V}^{(p)}}\,\widehat{r}^{(p)}(\mathcal{M},\,\mathcal{M}')\,\widehat{\mathcal{Q}}_{\mathcal{M}'}^{(p)}\left[\,\tau_{\mathscr{V}^{(p)}}=\tau_{\mathcal{M}''}\,\right]\left(\,{\bf f}_{\epsilon}(\mathcal{M}'')-{\bf f}_{\epsilon}(\mathcal{M})\,\right)+o_{\epsilon}(1)\;.
\]
Interchanging the sums, by \cite[Corollary 6.2]{BL1},
\[
(\widehat{\mathfrak{L}}^{(p)}{\bf f}_{\epsilon})\,(\mathcal{M})\;=\;\sum_{\mathcal{M}''\in\mathscr{V}^{(p)}}\,r^{(p)}(\mathcal{M},\,\mathcal{M}'')\,\left(\,{\bf f}_{\epsilon}(\mathcal{M}'')-{\bf f}_{\epsilon}(\mathcal{M})\,\right)+o_{\epsilon}(1)\;=\;(\mathfrak{L}^{(p)}{\bf f}_{\epsilon})\,(\mathcal{M})\,+\,o_{\epsilon}(1)\;,
\]
as claimed.
\end{proof}
Now we are ready to prove Proposition \ref{p_res}
\begin{proof}[Proof of Proposition \ref{p_res}]
Since $(\mathfrak{L}^{(p)}{\bf h})(\mathcal{M})=0$ for all $\mathcal{M}\in\mathscr{V}_{{\rm ab}}^{(p)}$,
${\bf h}\colon\mathscr{V}^{(p)}\to\mathbb{R}$, by Lemma \ref{l09},
the assertion of the proposition holds for $\mathcal{M}\in\mathscr{V}_{{\rm ab}}^{(p)}$.

Fix $\mathcal{M}\in\mathscr{V}_{{\rm nab}}^{(p)}$. Multiply both
sides of the equation \eqref{res} by the test function $Q_{\epsilon}$
introduced in \eqref{2-02} and integrate with respect to $\mu_{\epsilon}$.
By Lemmata \ref{l2-01}, \ref{l2-02} and \ref{l_p_res-3},
\begin{align*}
 & \big\{\,\lambda\,{\bf f}_{\epsilon}(\mathcal{M})\,+\,o_{\epsilon}(1)\,\big\}\,\mu_{\epsilon}(\mathcal{E}(\mathcal{M}))\;-\;\left[\,\frac{\nu(\mathcal{M})}{\nu_{\star}}\,\big(\mathfrak{L}^{(p)}{\bf f}_{\epsilon}\big)\,(\mathcal{M})\,+\,o_{\epsilon}(1)\,\right]\,e^{-U(\mathcal{M})/\epsilon}\\
 & \quad\;=\;\sum_{\mathcal{M}'\in\mathscr{V}^{(p)}}\,{\bf g}(\mathcal{M}')\int_{\mathbb{R}^{d}}\,Q_{\epsilon}\,d\mu_{\epsilon}\;.
\end{align*}
By definition of $Q_{\epsilon}$ and since $\mathcal{M}=\mathcal{M}^{*}\left(\mathcal{W}(\mathcal{M})\right)$,
the right-hand side is equal to $(1+o_{\epsilon}(1))\,{\bf g}(\mathcal{M})\,\mu_{\epsilon}\left(\mathcal{E}(\mathcal{M})\right)$.
On the other hand, by definition of $\nu(\mathcal{M})$, $\nu_{\star}$,
$[\nu(\mathcal{M})/\nu_{\star}]\,e^{-U(\mathcal{M})/\epsilon}=(1+o_{\epsilon}(1))\,\mu_{\epsilon}\left(\mathcal{E}(\mathcal{M})\right)$.
Hence, dividing the previous equation by $\mu_{\epsilon}\left(\mathcal{E}(\mathcal{M})\right)$
yields that
\[
\lambda\,{\bf f}_{\epsilon}(\mathcal{M})\,-\,\big(\mathfrak{L}^{(p)}{\bf f}_{\epsilon}\big)\,(\mathcal{M})\;=\;{\bf g}(\mathcal{M})\,+\,o_{\epsilon}(1)\;,
\]
as claimed.
\end{proof}

\section{\label{sec11}Proof of Proposition \ref{prop_NE}}

In this section, we fix $p\in\llbracket1,\,\mathfrak{q}\rrbracket$
and show that $\mathfrak{R}^{(p)}$ and $\mathfrak{H}^{(p)}$ imply
$\mathfrak{G}^{(p)}$, as claimed in Proposition \ref{prop_NE}. The
proof is similar to that of \cite[Proposition 10.4]{LLS-1st}, but
we present a complete proof in sake of completeness as the setting
and notation are slightly different.

\subsection{Reduction to hitting time estimate}

Denote by $\boldsymbol{y}_{\epsilon}^{(p)}(\cdot)$ the process $\boldsymbol{x}_{\epsilon}(\cdot)$
speeded-up by $\theta_{\epsilon}^{(p)}$:
\[
{\color{blue}\boldsymbol{y}_{\epsilon}^{(p)}(t)}\;:=\;\boldsymbol{x}_{\epsilon}(\theta_{\epsilon}^{(p)}t)\;\;\;\;;\;t\ge0\;,
\]
and denote by ${\color{blue}\mathbb{Q}_{\boldsymbol{x}}^{\epsilon,p}}$
the law of the process $\boldsymbol{y}_{\epsilon}^{(p)}(\cdot)$ starting
from $\boldsymbol{x}$. With these notations, we can restate $\mathfrak{G}^{(p)}$
as,
\begin{equation}
\limsup_{b\rightarrow0}\,\limsup_{\epsilon\rightarrow0}\,\sup_{\boldsymbol{x}\in\mathcal{E}(\mathcal{M})}\,\sup_{t\in[2b,\,4b]}\,\mathbb{Q}_{\boldsymbol{x}}^{\epsilon,p}\big[\,\boldsymbol{y}_{\epsilon}^{(p)}(t)\notin\mathcal{E}^{(p)}\,\big]\ =\ 0\;\;\;\text{for all}\ \mathcal{M}\in\mathscr{V}^{(p)}\;,\label{conEp}
\end{equation}
where the set $\mathcal{E}^{(p)}$ has been introduced in \eqref{07}.

Since our purpose is to prove \eqref{conEp} for all $\mathcal{M}\in\mathscr{V}^{(p)}$,
in the remainder of the section, we fix a set $\mathcal{M}\in\mathscr{V}^{(p)}$.
In addition, we fix a constant $\eta$ such that
\begin{equation}
\eta\in\left(\,0,\,\min\{r_{0}/2,\,d^{(p)}/2,\,(d^{(p)}-d^{(p-1)})\}\,\right)\label{etac}
\end{equation}
so that there is no critical point $\boldsymbol{c}$ of $U$ satisfying
\begin{equation}
U(\boldsymbol{c})\,\in\,\bigcup_{\boldsymbol{m}\in\mathcal{M}_{0}}\left(\,U(\boldsymbol{m}),\,U(\boldsymbol{m})+2\eta\,\right]\,\cup\,\left[\,U(\mathcal{M})+d^{(p)}-\eta,\,U(\mathcal{M})+d^{(p)}\,\right)\;.\label{etac2}
\end{equation}

\begin{figure}
\includegraphics[scale=0.12]{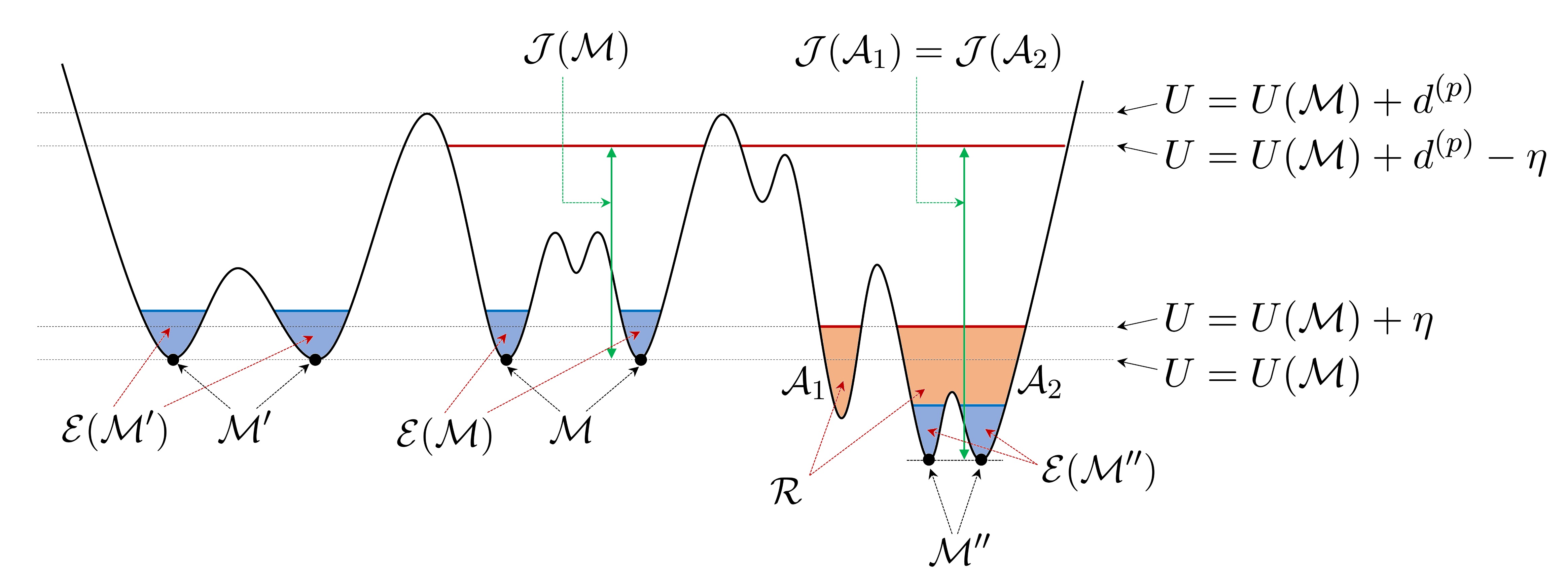} \caption{Illustration of the sets used in Section \ref{sec11}. In this figure,
we illustrated the situation when $\mathscr{V}^{(p)}=\{\mathcal{M},\,\mathcal{M}',\,\mathcal{M}''\}$
where $\mathcal{M}''$ is the only absorbing state. The set $\mathcal{R}$
is defined by $\{U<U(\mathcal{M})+\eta\}\setminus\mathcal{E}^{(p)}$
and thus in this figure $\mathcal{R}$ has two connected components
$\mathcal{A}_{1}$ and $\mathcal{A}_{2}$. We also illustrated the
sets $\mathcal{J}(\mathcal{M})$, $\mathcal{J}(\mathcal{A}_{1})$,
and $\mathcal{J}(\mathcal{A}_{2})$ as well. We note that in this
situation it holds that $\mathcal{J}(\mathcal{A}_{1})=\mathcal{J}(\mathcal{A}_{2})$.}
\label{fig_W_A}
\end{figure}

Define the set $\mathcal{R}=\mathcal{R}(\mathcal{M})$ as (cf. Figure
\ref{fig_W_A})
\begin{equation}
{\color{blue}\mathcal{R}}\ :=\ \big\{\,U<U(\mathcal{M})+\eta\,\big\}\,\setminus\,\mathcal{E}^{(p)}\;.\label{2-04b}
\end{equation}
The set $\mathcal{R}$ is disjoint with $\mathcal{M}$ since $\mathcal{M}\subset\mathcal{E}^{(p)}$.
When $\mathcal{M}\subset\mathcal{M}_{\star}$, $\mathcal{R}$ is empty
because $\eta<r_{0}$. The main ingredient of the proof of \eqref{conEp}
is the following lemma.
\begin{prop}
\label{prop: nojump2} Assume $\mathfrak{R}^{(p)}$ and $\mathfrak{H}^{(p)}$.
Then,
\[
\limsup_{a\rightarrow0}\,\limsup_{\epsilon\rightarrow0}\,\sup_{\boldsymbol{x}\in\mathcal{E}(\mathcal{M})}\,\mathbb{Q}_{\boldsymbol{x}}^{\epsilon,\,p}\left[\,\tau_{\mathcal{R}}\le a\,\right]\,=\,0\;.
\]
\end{prop}

It is proven in \cite[Proposition 10.4]{LLS-1st} that Proposition
\ref{prop: nojump2} implies \eqref{conEp}. The same proof applies
here without any modification. We turn to the proof of Proposition
\ref{prop: nojump2}.
\begin{rem}
\label{rem_empty} When , as observed above, $\mathcal{R}=\varnothing$,
Proposition \ref{prop: nojump2} trivially holds. Therefore, in the
remainder of the section, we not only fix $\mathcal{M}\in\mathscr{V}^{(p)}$
but also assume $\mathcal{M}\ne\mathcal{M}_{\star}$ does not contain
any global minimum of $U$.
\end{rem}

\subsection{Further anaysis on open level sets }

Consider the following sets illustrated in Figure \ref{fig_W_A}.
\begin{enumerate}
\item The set $\mathcal{J}$ is the level set
\begin{equation}
{\color{blue}\mathcal{J}}\ :=\ \big\{\,U<U(\mathcal{M})+d^{(p)}-\eta\,\big\}\;.\label{levset}
\end{equation}
\item By $\mathfrak{P}_{1}(p)$, Lemma \ref{lem:bound_dn}, and the definition
of $\eta$, the set $\mathcal{M}$ is contained in a connected component
of $\mathcal{J}$, denoted by \textcolor{blue}{$\mathcal{J}(\mathcal{M})$}.
This set can also be written as $\mathcal{W}^{U(\mathcal{M})+d^{(p)}-\eta}(\mathcal{M})$
according to previous notation.
\item Let ${\color{blue}\mathcal{A}}$ be a connected component of $\mathcal{R}$.
By Lemma \ref{l_level_connected} and since $\eta<d^{(p)}-\eta$ (cf.
\eqref{etac}), the set $\mathcal{A}$ is contained in a connected
component of $\mathcal{J}$, denoted by \textcolor{blue}{$\mathcal{J}(\mathcal{A})$}.
\end{enumerate}
The next lemma provides a characterization of the landscape of $U$.
\begin{lem}
\label{l: nojump2-2} Let $\mathcal{A}$ be a connected component
of $\mathcal{R}$. Then,
\begin{enumerate}
\item $\mathcal{M}=\mathcal{M}^{*}\left(\mathcal{J}(\mathcal{M})\right)$.
\item $\mathcal{J}(\mathcal{A})$ contains a minimum $\boldsymbol{m}$ of
$U$ such that $U(\boldsymbol{m})\le U(\mathcal{M})$.
\item $\mathcal{J}(\mathcal{M})\cap\mathcal{J}(\mathcal{A})=\varnothing$.
\item $\mathcal{J}(\mathcal{M})$ and $\mathcal{J}(\mathcal{A})$ do not
separate $(p)$-states.
\item $\mathscr{V}^{(p)}(\mathcal{J}(\mathcal{A}))\ne\varnothing$.
\end{enumerate}
\end{lem}

\begin{proof}
(1) The assertion follows from Lemma \ref{l_bound_conn} since we
have
\begin{equation}
U(\mathcal{M})+d^{(p)}-\eta\ <\ U(\mathcal{M})+\Xi(\mathcal{M})\ =\ \Theta(\mathcal{M},\,\widetilde{\mathcal{M}})\ .\label{ctdc}
\end{equation}
\smallskip{}

\noindent (2) By Lemma \ref{l_level_connected}, the set $\mathcal{A}$
is contained in a connected component of the set $\{U<U(\mathcal{M})+\eta\}$,
denoted by $\widetilde{\mathcal{A}}$. As $\eta<d^{(p)}-\eta$ (cf.
\eqref{etac}), we have $\widetilde{\mathcal{A}}\subset\mathcal{J}(\mathcal{A})$.
Since $\widetilde{\mathcal{A}}$ is a level set, it contains a local
minimum, say $\boldsymbol{m}$, and by definition of $\widetilde{\mathcal{A}}$,
$U(\boldsymbol{m})<U(\mathcal{M})+\eta$. As there is no critical
point $\boldsymbol{c}$ of $U$ such that $U(\boldsymbol{c})\in(U(\mathcal{M}),\,U(\mathcal{M})+2\eta)$
by the definition of $\eta$, $U(\boldsymbol{m})\le U(\mathcal{M})$.
This completes the proof because $\widetilde{\mathcal{A}}\subset\mathcal{J}(\mathcal{A})$.

\smallskip{}
(3) Suppose, by contradiction, that
\[
\boldsymbol{z}_{0}\in\ \mathcal{J}(\mathcal{M})\cap\mathcal{J}(\mathcal{A})\;.
\]
Since $\mathcal{J}(\mathcal{M})$ and $\mathcal{J}(\mathcal{A})$
are connected component of the level set $\mathcal{J}$ defined in
\eqref{levset},
\[
\Theta(\boldsymbol{x},\,\boldsymbol{y})\,<\,U(\mathcal{M})+d^{(p)}-\eta\ \text{for all}\ \boldsymbol{x},\,\boldsymbol{y}\in\mathcal{J}(\mathcal{A})\;\;\text{or}\;\;\boldsymbol{x},\,\boldsymbol{y}\in\mathcal{J}(\mathcal{M})\;.
\]
Therefore,
\[
\Theta(\mathcal{M},\,\boldsymbol{z}_{0}),\,\Theta(\boldsymbol{m},\,\boldsymbol{z}_{0})\ <\ U(\mathcal{M})+d^{(p)}-\eta\;,
\]
where $\boldsymbol{m}$ is the local minimum of $U$ in $\mathcal{J}(\mathcal{A})$
such that $U(\boldsymbol{m})\le U(\mathcal{M})$. Therefore,
\[
\Theta(\mathcal{M},\,\widetilde{\mathcal{M}})\,\le\,\Theta(\mathcal{M},\,\boldsymbol{m})\,<\,\max\big\{\,\Theta(\mathcal{M},\,\boldsymbol{z}_{0}),\,\Theta(\boldsymbol{z}_{0},\,\boldsymbol{m})\,\big\}\,<\,U(\mathcal{M})+d^{(p)}-\eta\ ,
\]
which is not possible, as seen in \eqref{ctdc}.

\noindent \smallskip{}
(4) Let $\widehat{\mathcal{J}}(\mathcal{A})$ be the connected component
of $\{U<U(\mathcal{M})+d^{(p)}\}$ containing $\mathcal{J}(\mathcal{A})$.
By Lemma \ref{lem_H,H+a} and \eqref{etac2},
\begin{equation}
\mathcal{M}_{0}\cap\widehat{\mathcal{J}}(\mathcal{A})\ =\ \mathcal{M}_{0}\cap\mathcal{J}(\mathcal{A})\;.\label{jhatj}
\end{equation}
By part (2), there exists a local minimum $\boldsymbol{m}$ in $\widehat{\mathcal{J}}(\mathcal{A})$
such that $U(\boldsymbol{m})\le U(\mathcal{M})$. By Lemma \ref{lem_separate},
$\widehat{\mathcal{J}}(\mathcal{A})$ does not separate $(p)$-states.
Therefore, by \eqref{jhatj}, the set $\mathcal{J}(\mathcal{A})$
does not separate $(p)$-states as well. The proof for $\mathcal{J}(\mathcal{M})$
is the same.

\noindent \smallskip{}
(5) Suppose, by contradiction, that $\mathscr{V}^{(p)}(\mathcal{J}(\mathcal{A}))=\varnothing$.
Let $\mathcal{M}'=\mathcal{M}^{*}\left(\mathcal{J}(\mathcal{A})\right)$.
By Lemma \ref{lem_escape}-(3), $\mathcal{M}'\in\mathscr{N}^{(p)}$,
so that, by Lemma \ref{lem_charDelta}, $\Xi(\mathcal{M}')\le d^{(p-1)}$.

\noindent By the definition of $\mathcal{J}(\mathcal{A})$, Lemma
\ref{lap01}-(2), and part (2),
\[
\Theta(\mathcal{M}',\widetilde{\mathcal{M}'})\ \ge\ U(\mathcal{M})+d^{(p)}-\eta\ \ge\ U(\mathcal{M}')+d^{(p)}-\eta\;.
\]
By \eqref{etac}, the previous estimate, and Lemma \ref{lap01}-(1),
\[
\Xi(\mathcal{M}')\ =\ \Theta(\mathcal{M}',\,\widetilde{\mathcal{M}'})-U(\mathcal{M}')\ \ge\ d^{(p)}-\eta\ >\ d^{(p-1)}\ ,
\]
in contradiction with the estimate obtained in the previous paragraph.
\end{proof}

\subsection{Analysis of hitting times}

In this subsection, we derive a refined version of the hitting time
estimate \cite[Corollary 6.2]{LLS-1st} obtained in our companion
paper.

In this section, we fix\textcolor{blue}{{} $H\in\mathbb{R}$} and such
that there is no critical point $\boldsymbol{c}$ of $U$ satisfying
$U(\boldsymbol{c})=H$, and fix a connected component \textcolor{blue}{$\mathcal{H}$
}of the level set $\{U<H\}$. Assume, in addition, that
\begin{equation}
\mathscr{V}^{(p)}(\mathcal{H})\,\neq\,\varnothing\;\;\text{and}\;\;\mathscr{V}^{(p)}(\mathcal{H}^{c})\,\neq\,\varnothing\;.\label{nempt}
\end{equation}
Then, by Lemma \ref{lem_separate} and \eqref{nempt}, the set $\mathcal{H}$
does not separate $(p)$-states.
\begin{lem}
\label{l:nesc}For all $\mathcal{M}\in\mathscr{S}^{(p)}(\mathcal{H})$
and $\mathcal{M}'\in\mathscr{V}^{(p)}(\mathcal{H}^{c})$,
\[
\widehat{\mathcal{Q}}_{\mathcal{M}}^{(p)}\left[\,\tau_{\mathscr{V}^{(p)}}=\tau_{\mathcal{M}'}\,\right]\ =\ 0\;.
\]
\end{lem}

\begin{proof}
Suppose, by contradiction, that $\widehat{\mathcal{Q}}_{\mathcal{M}}^{(p)}\left[\tau_{\mathscr{V}^{(p)}}=\tau_{\mathcal{M}'}\right]>0$.
Then, the process $\widehat{\mathbf{y}}^{(p)}(\cdot)$ can jump from
$\mathcal{H}$ to $\mathcal{H}^{c}$. By Lemma \ref{lem_escape}-(1),
this jump is from $\mathcal{M}_{1}=\mathcal{M}^{*}(\mathcal{H})$
to $\mathcal{M}_{2}\in\mathscr{S}^{(p)}(\mathcal{H}^{c})$. Note here
that, if $\mathcal{M}_{1}\in\mathscr{V}^{(p)}$, we cannot have $\tau_{\mathscr{V}^{(p)}}=\tau_{\mathcal{M}'}$.
Thus, $\mathcal{M}_{1}\in\mathscr{N}^{(p)}$.

By \eqref{nempt}, there exists $\mathcal{M}_{3}\in\mathscr{V}^{(p)}(\mathcal{H})$,
and, by Lemma \ref{lap01}-(1),
\[
\Theta(\mathcal{M}_{3},\,\widetilde{\mathcal{M}_{3}})\ \le\ \Theta(\mathcal{M}_{3},\,\mathcal{M}_{1})\ <\ H\;.
\]
By Proposition \ref{prop:depth}, $\Xi(\mathcal{M}_{3})\ge d^{(p)}$.
Thus, as $U(\mathcal{M}_{3})>U(\mathcal{M}_{1})$,
\[
U(\mathcal{M}_{1})+d^{(p)}\ <\ U(\mathcal{M}_{3})+\Xi(\mathcal{M}_{3})\ =\ \Theta(\mathcal{M}_{3},\,\widetilde{\mathcal{M}_{3}})\ <\ H\;.
\]
Since $\mathcal{M}_{1}=\mathcal{M}^{*}(\mathcal{H})$, by Lemma \ref{lap01}-(2),
\[
\Xi(\mathcal{M}_{1})\ =\ \Theta(\mathcal{M}_{1},\,\widetilde{\mathcal{M}_{1}})-U(\mathcal{M}_{1})\ \ge\ H-U(\mathcal{M}_{1})\ >\ d^{(p)}\;,
\]
which contradicts Proposition \ref{prop:depth} since $\mathcal{M}_{1}\in\mathscr{N}^{(p)}$.
\end{proof}
We next show that the process starting from a point in $\mathcal{H}$
attains the set $\mathcal{E}^{(p)}$ at a point in $\mathcal{H}$.
\begin{lem}
\label{lem:hit001} Assume $\mathfrak{H}^{(p)}$. Then,
\[
\lim_{\epsilon\to0}\,\sup_{\boldsymbol{x}\in\mathcal{H}}\,\mathbb{P}_{\boldsymbol{x}}^{\epsilon}\left[\,\tau_{\mathcal{E}^{(p)}}=\tau_{\mathcal{E}^{(p)}\cap\mathcal{H}^{c}}\,\right]\ =\ 0\;.
\]
\end{lem}

\begin{proof}
We first claim that
\begin{equation}
\lim_{\epsilon\to0}\,\sup_{\mathcal{M}\in\mathscr{N}^{(p)}(\mathcal{H})}\,\sup_{\boldsymbol{x}\in\mathcal{E}(\mathcal{M})}\,\mathbb{P}_{\boldsymbol{x}}^{\epsilon}\left[\,\tau_{\mathcal{E}^{(p)}}=\tau_{\mathcal{E}^{(p)}\cap\mathcal{H}^{c}}\,\right]\ =\ 0\ .\label{hit001}
\end{equation}
By \eqref{eq:Hp-2}, for $\boldsymbol{x}\in\mathcal{E}(\mathcal{M})$,
$\mathcal{M}\in\mathscr{N}^{(p)}(\mathcal{H})$,
\[
\lim_{\epsilon\to0}\,\sup_{\mathcal{M}\in\mathscr{N}^{(p)}(\mathcal{H})}\,\sup_{\boldsymbol{x}\in\mathcal{E}(\mathcal{M})}\,\Big|\,\mathbb{P}_{\boldsymbol{x}}^{\epsilon}\left[\,\tau_{\mathcal{E}^{(p)}}=\tau_{\mathcal{E}^{(p)}\cap\mathcal{H}^{c}}\,\right]-\sum_{\mathcal{M}'\in\mathscr{V}^{(p)}(\mathcal{H}^{c})}\,\widehat{\mathcal{Q}}_{\mathcal{M}}^{(p)}\left[\,\tau_{\mathscr{V}^{(p)}}=\tau_{\mathcal{M}'}\,\right]\,\Big|\ =\ 0\;.
\]
To complete the proof of the claim it is enough to recall the assertion
of Lemma \ref{l:nesc}.

Since $\mathcal{H}$ does not separate $(p)$-states, for all $\bm{m}\in\mathcal{M}_{0}\cap\mathcal{H}$,
$\mathcal{M}(p,\,\bm{m})$ belongs to $\mathscr{N}^{(p)}(\mathcal{H})$
or to $\mathscr{V}^{(p)}(\mathcal{H})$. Therefore, by the strong
Markov property, it suffices to show that
\[
\lim_{\epsilon\to0}\,\sup_{\boldsymbol{x}\in\mathcal{H}}\,\mathbb{P}_{\boldsymbol{x}}^{\epsilon}\left[\,\tau_{\mathcal{E}(\mathcal{M}_{0})}=\tau_{\mathcal{E}(\mathcal{M}_{0})\cap\mathcal{H}^{c}}\,\right]\ =\ 0\;.
\]
By \cite[Corollary 6.2]{LLS-1st}, we have
\begin{align*}
\limsup_{\epsilon\to0}\,\sup_{\boldsymbol{x}\in\mathcal{H}}\,\mathbb{P}_{\boldsymbol{x}}^{\epsilon}\left[\tau_{\mathcal{E}(\mathcal{M}_{0})}=\tau_{\mathcal{E}(\mathcal{M}_{0})\cap\mathcal{H}^{c}}\right]\  & =\ \limsup_{\epsilon\to0}\,\sup_{\boldsymbol{x}\in\mathcal{H}}\,\mathbb{P}_{\boldsymbol{x}}^{\epsilon}\left[\,\tau_{\mathcal{E}(\mathcal{M}_{0})}=\tau_{\mathcal{E}(\mathcal{M}_{0})\cap\mathcal{H}^{c}},\,\tau_{\mathcal{E}(\mathcal{M}_{0})}<\epsilon^{-1}\,\right]\\
 & \le\ \limsup_{\epsilon\to0}\,\sup_{\boldsymbol{x}\in\mathcal{H}}\,\mathbb{P}_{\boldsymbol{x}}^{\epsilon}\left[\,\tau_{\mathcal{E}(\mathcal{M}_{0})\cap\mathcal{H}^{c}}<\epsilon^{-1}\,\right]\;.
\end{align*}
Let $a>0$ be a small number such that there is no critical point
$\boldsymbol{c}$ satisfying $U(\boldsymbol{c})\in[H,\,H+a)$. Let
$\mathcal{H}'$ be a connected component of $\{U<H+a\}$ containing
$\mathcal{H}$. Then, by Lemma \ref{lem_H,H+a}, $\mathcal{M}_{0}\cap\mathcal{H}=\mathcal{M}_{0}\cap\mathcal{H}'$.
Hence, by Proposition \ref{p_FW},
\[
\limsup_{\epsilon\to0}\,\sup_{\boldsymbol{x}\in\mathcal{H}}\,\mathbb{P}_{\boldsymbol{x}}^{\epsilon}\left[\,\tau_{\mathcal{E}(\mathcal{M}_{0})\cap\mathcal{H}^{c}}<\epsilon^{-1}\,\right]\ \le\ \limsup_{\epsilon\to0}\,\sup_{\boldsymbol{x}\in\mathcal{H}}\,\mathbb{P}_{\boldsymbol{x}}^{\epsilon}\left[\,\tau_{\partial\mathcal{H}'}<\epsilon^{-1}\,\right]\ =\ 0\ .
\]
This completes the proof.
\end{proof}
The next result is a refinement of \cite[Theorem 6.1]{LLS-1st}.
\begin{lem}
\label{lem_hitting}Assume $\mathfrak{H}^{(p)}$. Then, for all sequences
$(\alpha_{\epsilon})_{\epsilon>0}$ such that $\alpha_{\epsilon}\succ\theta_{\epsilon}^{(p-1)}$
if $p\ge2$ and $\alpha_{\epsilon}\succ\frac{1}{\epsilon}$ if $p=1$,
\[
\lim_{\epsilon\to0}\,\sup_{\boldsymbol{x}\in\mathcal{H}}\,\mathbb{P}_{\boldsymbol{x}}^{\epsilon}\left[\,\tau_{\mathcal{E}^{(p)}\cap\mathcal{H}}>\alpha_{\epsilon}\,\right]\ =\ 0\;.
\]
\end{lem}

\begin{proof}
The case $p=1$ is the contents of \cite[Corollary 6.2]{LLS-1st}.

Suppose now that $p\ge2$. By the strong Markov property and by \cite[Theorem 6.1]{LLS-1st},
which asserts that the process starting from $\mathcal{H}$ hits the
set $\mathcal{E}(\mathcal{M}_{0})\cap\mathcal{H}$ within time $1/\epsilon$
with overwhelming probability, it suffices to prove
\[
\lim_{\epsilon\to0}\,\sup_{\boldsymbol{m}\in\mathcal{M}_{0}\cap\mathcal{H}}\,\sup_{\boldsymbol{y}\in\mathcal{E}(\boldsymbol{m})}\,\mathbb{P}_{\boldsymbol{y}}\left[\,\tau_{\mathcal{E}^{(p)}\cap\mathcal{H}}>\alpha_{\epsilon}\,\right]\ =\ 0\;.
\]
If $\mathcal{M}(p,\,\bm{m})\in\mathscr{V}^{(p)}(\mathcal{H})$, the
assertion is clear. It remains to show that
\[
\lim_{\epsilon\to0}\,\sup_{\mathcal{M}\in\mathscr{N}^{(p)}(\mathcal{H})}\,\sup_{\boldsymbol{y}\in\mathcal{E}(\mathcal{M})}\,\mathbb{P}_{\boldsymbol{y}}\left[\,\tau_{\mathcal{E}^{(p)}\cap\mathcal{H}}>\alpha_{\epsilon}\,\right]\ =\ 0\;.
\]
By \eqref{eq:Hp-1} of $\mathfrak{H}^{(p)}$ and Lemma \ref{lem:hit001},
\begin{align*}
 & \limsup_{\epsilon\to0}\,\sup_{\mathcal{M}\in\mathscr{N}^{(p)}(\mathcal{H})}\,\sup_{\boldsymbol{y}\in\mathcal{E}(\mathcal{M})}\,\mathbb{P}_{\boldsymbol{y}}\left[\,\tau_{\mathcal{E}^{(p)}\cap\mathcal{H}}>\alpha_{\epsilon}\,\right]\\
 & =\limsup_{\epsilon\to0}\,\sup_{\mathcal{M}\in\mathscr{N}^{(p)}(\mathcal{H})}\,\sup_{\boldsymbol{y}\in\mathcal{E}(\mathcal{M})}\,\mathbb{P}_{\boldsymbol{y}}\left[\,\tau_{\mathcal{E}^{(p)}\cap\mathcal{H}}>\alpha_{\epsilon},\,\tau_{\mathcal{E}^{(p)}}=\tau_{\mathcal{E}^{(p)}\cap\mathcal{H}}\,\right]\\
 & =\limsup_{\epsilon\to0}\,\sup_{\mathcal{M}\in\mathscr{N}^{(p)}(\mathcal{H})}\,\sup_{\boldsymbol{y}\in\mathcal{E}(\mathcal{M})}\,\mathbb{P}_{\boldsymbol{y}}\left[\,\tau_{\mathcal{E}^{(p)}}>\alpha_{\epsilon}\,\right]\ =\ 0
\end{align*}
which completes the proof.
\end{proof}

\subsection{Proof of Proposition \ref{prop: nojump2}}

The next result guarantees that jumps between valleys $\mathcal{E}(\mathcal{M})$,
$\mathcal{M}\in\mathscr{V}^{(p)}$, cannot happen in time-scales shorter
than $\theta_{\epsilon}^{(p)}$. It is proven in \cite[Lemma 4.2]{LMS}
in a general set-up.
\begin{lem}
\label{l: nojump} Assume $\mathfrak{R}^{(p)}$. Then, for all $\mathcal{M}\in\mathscr{V}^{(p)}$,
\begin{equation}
\limsup_{a\rightarrow0}\,\limsup_{\epsilon\rightarrow0}\,\sup_{\boldsymbol{x}\in\mathcal{E}(\mathcal{M})}\,\mathbb{Q}_{\boldsymbol{x}}^{\epsilon,\,p}\left[\,\tau_{\mathcal{E}^{(p)}\setminus\mathcal{E}(\mathcal{M})}\le a\,\right]\ =\ 0\;.\label{e:nojump}
\end{equation}
\end{lem}

\begin{proof}[Proof of Proposition \ref{prop: nojump2}]
The proof is similar to, but slightly more complicate than, that
of \cite[Lemma 10.5]{LLS-1st}. Clearly,
\[
\mathbb{Q}_{\boldsymbol{x}}^{\epsilon,\,p}\left[\,\tau_{\mathcal{R}}\le a\,\right]\ \le\ \max_{\mathcal{A}\subset\mathcal{R}}\,\mathbb{Q}_{\boldsymbol{x}}^{\epsilon,\,p}\left[\,\tau_{\mathcal{J}(\mathcal{A})}\le a\,\right]\;,
\]
where the maximum is carried out over all connected components of
$\mathcal{R}$ and $\mathcal{J}(\mathcal{A})$ is defined below Remark
\ref{rem_empty}. Thus, by Lemma \ref{l: nojump}, it suffices to
show
\[
\limsup_{\epsilon\rightarrow0}\,\sup_{\boldsymbol{x}\in\mathcal{E}(\mathcal{M})}\,\mathbb{Q}_{\boldsymbol{x}}^{\epsilon,\,p}\left[\,\tau_{\mathcal{J}(\mathcal{A})}\le a\,\right]\ \le\ \limsup_{\epsilon\rightarrow0}\,\sup_{\boldsymbol{x}\in\mathcal{E}(\mathcal{M})}\,\mathbb{Q}_{\boldsymbol{x}}^{\epsilon,\,p}\left[\,\tau_{\mathcal{E}^{(p)}\setminus\mathcal{E}(\mathcal{M})}\le2a\,\right]
\]
for all connected components $\mathcal{A}$ of $\mathcal{R}$.

Fix a connected component $\mathcal{A}$ of $\mathcal{R}$. By Lemma
\ref{l: nojump2-2}, the set $\mathcal{J}(\mathcal{A})$ does not
separate $(p)$-states and $\mathscr{V}^{(p)}(\mathcal{J}(\mathcal{A}))\ne\varnothing$.
Let $\theta_{\epsilon}^{(p-1)}\prec\alpha_{\epsilon}\prec\theta_{\epsilon}^{(p)}$,
set $\iota_{\epsilon}=a+\alpha_{\epsilon}/\theta_{\epsilon}^{(p)}$,
and write
\begin{align}
\mathbb{Q}_{\boldsymbol{x}}^{\epsilon,\,p}\left[\,\tau_{\mathcal{J}(\mathcal{A})}\le a\,\right]\ =\; & \mathbb{Q}_{\boldsymbol{x}}^{\epsilon,\,p}\left[\,\tau_{\mathcal{J}(\mathcal{A})}\le a,\,\sigma_{\mathcal{E}^{(p)}\cap\mathcal{J}(\mathcal{A})}<\iota_{\epsilon}\,\right]\nonumber \\
 & \qquad+\mathbb{Q}_{\boldsymbol{x}}^{\epsilon,\,p}\left[\,\tau_{\mathcal{J}(\mathcal{A})}\le a,\,\sigma_{\mathcal{E}^{(p)}\cap\mathcal{J}(\mathcal{A})}\ge\iota_{\epsilon}\,\right]\ ,\label{e:hit003}
\end{align}
where $\sigma_{\mathcal{Z}}$, $\mathcal{Z}\subset\mathbb{R}^{d}$,
is the first time after $\tau_{\mathcal{J}(\mathcal{A})}$ that the
process visits $\mathcal{Z}$:
\[
\sigma_{\mathcal{Z}}\ :=\ \inf\Big\{\,t>\tau_{\mathcal{J}(\mathcal{A})}:\boldsymbol{y}_{\epsilon}^{(p)}(t)\in\mathcal{Z}\,\Big\}\ .
\]

Since $\iota_{\epsilon}<2a$ for sufficiently small $\epsilon>0$
and,, $\mathcal{J}(\mathcal{A})\cap\mathcal{E}(\mathcal{M})=\varnothing$
by Lemma \ref{l: nojump2-2}-(3), the first term at the right-hand
side of \eqref{e:hit003} can be bounded by
\[
\mathbb{Q}_{\boldsymbol{x}}^{\epsilon,\,p}\left[\,\tau_{\mathcal{E}^{(p)}\cap\mathcal{J}(\mathcal{A})}<\iota_{\epsilon}\,\right]\ \le\ \mathbb{Q}_{\boldsymbol{x}}^{\epsilon,\,p}\left[\,\tau_{\mathcal{E}^{(p)}\setminus\mathcal{E}(\mathcal{M})}\le2a\,\right]\;.
\]
By the strong Markov property, the second term of the right-hand is
bounded by
\[
\sup_{\boldsymbol{y}\in\mathcal{J}(\mathcal{A})}\,\mathbb{P}_{\boldsymbol{y}}^{\epsilon}\left[\,\tau_{\mathcal{E}^{(p)}\cap\mathcal{J}(\mathcal{A})}>\alpha_{\epsilon}\,\right]\ .
\]
By Lemma \ref{lem_hitting}, this term converges to $0$ as $\epsilon\to0$.
Combining the previous estimates completes the proof.
\end{proof}

\section{\label{sec12}Proof of Theorem \ref{t00}}

In this section, we finally prove Theorem \ref{t00}. The proof of
Theorem \ref{t00} is based on three propositions proved in this section.
\begin{prop}
\label{p_fdd01} Fix $p\in\llbracket1,\,\mathfrak{q}\rrbracket$ and
$\mathcal{M}\in\mathscr{S}^{(p)}$. Then, for all $\boldsymbol{m}\in\mathcal{M}$,
$\boldsymbol{x}\in\mathcal{D}(\boldsymbol{m})$, $t>0$ and bounded
continuous function $F\colon\mathbb{R}^{d}\to\mathbb{R}$,
\begin{align*}
 & \lim_{\epsilon\rightarrow0}\,\mathbb{E}_{\boldsymbol{x}}^{\epsilon}\bigg[\,\,F(\boldsymbol{x}_{\epsilon}(\theta_{\epsilon}^{(p)}t))\,\bigg]\\
 & =\ \sum_{\mathcal{M}'\in\mathscr{V}^{(p)}}\,\sum_{\mathcal{M}''\in\mathscr{V}^{(p)}}\,\sum_{\boldsymbol{m}'\in\mathcal{M}''}\,\widehat{\mathcal{Q}}_{\mathcal{M}}^{(p)}\left[\,\tau_{\mathscr{V}^{(p)}}=\tau_{\mathcal{M}'}\,\right]\,\mathcal{Q}_{\mathcal{M}'}^{(p)}\left[{\bf y}^{(p)}(t)=\mathcal{M}''\right]\,\frac{\nu(\boldsymbol{m}')}{\nu(\mathcal{M}'')}\,F(\boldsymbol{m}')\ .
\end{align*}
\end{prop}

\begin{prop}
\label{p_fdd02} Fix $p\in\llbracket1,\,\mathfrak{q}\rrbracket$ and
$\mathcal{M}\in\mathscr{S}^{(p)}$. Then, for all $\boldsymbol{m}\in\mathcal{M}$,
$\boldsymbol{x}\in\mathcal{D}(\boldsymbol{m})$, sequence $(\varrho_{\epsilon})_{\epsilon>0}$
satisfying $\theta_{\epsilon}^{(p-1)}\prec\varrho_{\epsilon}\prec\theta_{\epsilon}^{(p)},$
and bounded continuous function $F\colon\mathbb{R}^{d}\to\mathbb{R}$,
\begin{equation}
\lim_{\epsilon\rightarrow0}\,\mathbb{E}_{\boldsymbol{x}}^{\epsilon}\left[\,F(\boldsymbol{x}_{\epsilon}(\varrho_{\epsilon}))\,\right]\ =\ \sum_{\mathcal{M}'\in\mathscr{V}^{(p)}}\,\sum_{\boldsymbol{m}'\in\mathcal{M}'}\,\widehat{\mathcal{Q}}_{\mathcal{M}}^{(p)}\left[\,\tau_{\mathscr{V}^{(p)}}=\tau_{\mathcal{M}'}\,\right]\,\frac{\nu(\boldsymbol{m}')}{\nu(\mathcal{M}')}\,F(\boldsymbol{m}')\ .\label{e:fdd02}
\end{equation}
\end{prop}

\begin{prop}
\label{p_fdd03}
For all $\boldsymbol{x}\in\mathbb{R}^{d}$,
sequence $(\varrho_{\epsilon})_{\epsilon>0}$ satisfying $\varrho_{\epsilon}\succ\theta_{\epsilon}^{(\mathfrak{q})}$,
and bounded continuous function $F\colon\mathbb{R}^{d}\to\mathbb{R}$,
\begin{equation}
\lim_{\epsilon\rightarrow0}\,\mathbb{E}_{\boldsymbol{x}}^{\epsilon}\left[\,F(\boldsymbol{x}_{\epsilon}(\varrho_{\epsilon}))\,\right]\ =\ \sum_{\,\boldsymbol{m}\in\mathcal{M}_{\star}}\,\frac{\nu(\boldsymbol{m})}{\nu(\mathcal{M}_{\star})}\,F(\boldsymbol{m})\ .\label{e:fdd02-1}
\end{equation}
\end{prop}

We now complete the proof of Theorem \ref{t00} by assuming three
propositions above.
\begin{proof}[Proof of Theorem \ref{t00}]
Fix $\boldsymbol{m}\in\mathcal{M}_{0}$ and write $\mathcal{M}=\mathcal{M}(\boldsymbol{m},\,p)$.
Then, for $\boldsymbol{x}\in\mathcal{D}(\boldsymbol{m})$, by \eqref{eq:probexp}
and Proposition \ref{p_fdd01},
\begin{align*}
 & \lim_{\epsilon\to0}\,u_{\epsilon}(\theta_{\epsilon}^{(p)}t,\boldsymbol{x})\ =\ \lim_{\epsilon\rightarrow0}\,\mathbb{E}_{\boldsymbol{x}}^{\epsilon}\left[\,u_{0}\,(\boldsymbol{x}_{\epsilon}(\theta_{\epsilon}^{(p)}t))\,\right]\\
 & =\ \sum_{\mathcal{M}'\in\mathscr{V}^{(p)}}\,\sum_{\mathcal{M}''\in\mathscr{V}^{(p)}}\,\sum_{\boldsymbol{m}'\in\mathcal{M}''}\,\widehat{\mathcal{Q}}_{\mathcal{M}}^{(p)}\left[\,\tau_{\mathscr{V}^{(p)}}=\tau_{\mathcal{M}'}\,\right]\,\mathcal{Q}_{\mathcal{M}'}^{(p)}\left[\,{\bf y}^{(p)}(t)=\mathcal{M}''\,\right]\,\frac{\nu(\boldsymbol{m}')}{\nu(\mathcal{M}'')}\,u_{0}(\boldsymbol{m}')\;.
\end{align*}
This proves (a) by recalling \eqref{noteq1} and \eqref{noteq2}.

We now turn to (b) and (c). By \eqref{eq:probexp} and Proposition
\ref{p_fdd02}, we have
\[
\lim_{\epsilon\to0}u_{\epsilon}(\varrho_{\epsilon},\,\boldsymbol{x})\ =\ \lim_{\epsilon\rightarrow0}\mathbb{E}_{\boldsymbol{x}}^{\epsilon}\left[\,u_{0}(\boldsymbol{x}_{\epsilon}(\varrho_{\epsilon}))\,\right]\ =\ \sum_{\mathcal{M}'\in\mathscr{V}^{(p)}}\,\sum_{\boldsymbol{m}'\in\mathcal{M}'}\,\widehat{\mathcal{Q}}_{\mathcal{M}}^{(p)}\left[\,\tau_{\mathscr{V}^{(p)}}=\tau_{\mathcal{M}'}\,\right]\,\frac{\nu(\boldsymbol{m}')}{\nu(\mathcal{M}')}\,u_{0}(\boldsymbol{m}')\ .
\]
Hence, recalling \eqref{noteq1} completes the proof for (b). Similarly,
(c) is an immediate consequence of Proposition \ref{p_fdd03} and
\eqref{eq:probexp}.
\end{proof}
Now we turn to the proof of Propositions \ref{p_fdd01}, \ref{p_fdd02}
and \ref{p_fdd03}.

We emphasize that $\mathfrak{H}^{(p)}$, $\mathfrak{R}^{(p)}$, $\mathfrak{C}^{(p)}$,
$\mathfrak{C}_{{\rm fdd}}^{(p)}$, $\mathfrak{G}^{(p)}$, $p\in\llbracket1,\,\mathfrak{q}\rrbracket$,
have now been proven and hence can be used without further comments.
In some statements, however, to stress which of these conditions are
necessary, we include them in the statement.

\subsection{\label{sec14.1}Starting from domain of attraction of negligible
valley}

The following propositions are the main consequences of $\mathfrak{H}^{(p)}$.
We first look at the metastable scale.
\begin{prop}
\label{prop:fd1}Fix $p\in\llbracket1,\,\mathfrak{q}\rrbracket$ and
suppose that $\mathfrak{H}^{(p)}$ is in force. Suppose that for all
$\mathcal{M}\in\mathscr{V}^{(p)}$, $\mathcal{M}'\in\mathscr{S}^{(p)}$,
$\mathcal{A}\subset\mathcal{M}'$ , $t>0$, and a sequence $(\gamma_{\epsilon})_{\epsilon>0}$
such that $\gamma_{\epsilon}\prec\theta_{\epsilon}^{(p)}$, we have
\begin{equation}
\lim_{\epsilon\rightarrow0}\,\sup_{\boldsymbol{x}\in\mathcal{E}(\mathcal{M})}\,\sup_{|s|\le\gamma_{\epsilon}}\,\left|\,\mathbb{P}_{\boldsymbol{x}}^{\epsilon}\Big[\,\bm{x}_{\epsilon}(\theta_{\epsilon}^{(p)}t+s)\in\mathcal{E}(\mathcal{A})\,\Big]-\Phi_{t}^{(p)}(\mathcal{M},\,\mathcal{A})\,\right|\ =\ 0\label{eq:fd0}
\end{equation}
for some $\Phi_{t}^{(p)}(\mathcal{M},\,\mathcal{A})\in[0,\,1]$. Then,
for all $\mathcal{M},\,\mathcal{M}'\in\mathscr{S}^{(p)}$, $\boldsymbol{m}\in\mathcal{M}$,
$\mathcal{A}\subset\mathcal{M}'$, $t>0$, a sequence $(\gamma_{\epsilon})_{\epsilon>0}$
such that $\gamma_{\epsilon}\prec\theta_{\epsilon}^{(p)}$, and $\boldsymbol{x}\in\mathcal{D}(\boldsymbol{m})$,
we have
\begin{equation}
\lim_{\epsilon\rightarrow0}\,\sup_{|s|\le\gamma_{\epsilon}}\,\left|\,\mathbb{P}_{\boldsymbol{x}}^{\epsilon}\Big[\,\bm{x}_{\epsilon}(\theta_{\epsilon}^{(p)}t+s)\in\mathcal{E}(\mathcal{A})\,\Big]-\Psi_{t}^{(p)}(\mathcal{M},\,\mathcal{A})\,\right|\ =\ 0\;,\label{eq:fd1}
\end{equation}
where
\[
\Psi_{t}^{(p)}(\mathcal{M},\,\mathcal{A})\ =\ \sum_{\mathcal{M}''\in\mathscr{V}^{(p)}}\,\widehat{\mathcal{Q}}_{\mathcal{M}}^{(p)}\left[\,\tau_{\mathscr{V}^{(p)}}=\tau_{\mathcal{M}''}\,\right]\,\Phi_{t}^{(p)}(\mathcal{M}'',\,\mathcal{A})\;.
\]
\end{prop}

\begin{proof}
Fix $\mathcal{M},\,\mathcal{M}'\in\mathscr{S}^{(p)}$, $\boldsymbol{m}\in\mathcal{M}$,
$\mathcal{A}\subset\mathcal{M}'$, $t>0$ and sequence $(\gamma_{\epsilon})_{\epsilon>0}$
such that $\gamma_{\epsilon}\prec\theta_{\epsilon}^{(p)}$. In addition,
we take an auxiliary sequence $(\widehat{\gamma}_{\epsilon})_{\epsilon>0}$
such that $\theta_{\epsilon}^{(p-1)}\prec\widehat{\gamma}_{\epsilon}\prec\theta_{\epsilon}^{(p)}$.

We first let $\boldsymbol{x}\in\mathcal{E}(\mathcal{M})$ and $s\in[-\gamma_{\epsilon},\,\gamma_{\epsilon}]$.
Then, by $\mathfrak{H}^{(p)}$, we can replace the first probability
at the left-hand side of \eqref{eq:fd1} with
\[
\mathbb{P}_{\boldsymbol{x}}^{\epsilon}\left[\,\boldsymbol{x}_{\epsilon}(\theta_{\epsilon}^{(p)}t+s)\in\mathcal{E}(\mathcal{A}),\,\tau_{\mathcal{E}^{(p)}}\le\widehat{\gamma}_{\epsilon}\,\right]\;.
\]
Decompose this probability into
\begin{equation}
\sum_{\mathcal{M}''\in\mathscr{V}^{(p)}}\,\mathbb{P}_{\boldsymbol{x}}^{\epsilon}\left[\,\boldsymbol{x}_{\epsilon}(\theta_{\epsilon}^{(p)}t+s)\in\mathcal{E}(\mathcal{A})\,\Big|\,\tau_{\mathcal{E}^{(p)}}=\tau_{\mathcal{E}(\mathcal{M}'')}\le\widehat{\gamma}_{\epsilon}\,\right]\,\mathbb{P}_{\boldsymbol{x}}^{\epsilon}\left[\,\tau_{\mathcal{E}^{(p)}}=\tau_{\mathcal{E}(\mathcal{M}'')}\le\widehat{\gamma}_{\epsilon}\,\right]\;.\label{eq:fdddec-1}
\end{equation}
By the strong Markov property and \eqref{eq:fd0} (with $\gamma_{\epsilon}+\widehat{\gamma}_{\epsilon}$
instead of $\gamma_{\epsilon}$ since $\widehat{\gamma}_{\epsilon}\prec\theta_{\epsilon}^{(p)}$),
\begin{equation}
\lim_{\epsilon\to0}\,\sup_{\boldsymbol{x}\in\mathcal{E}(\mathcal{M})}\,\sup_{|s|\le\gamma_{\epsilon}}\,\left|\,\mathbb{P}_{\boldsymbol{x}}^{\epsilon}\left[\boldsymbol{x}_{\epsilon}(\theta_{\epsilon}^{(p)}t+s)\in\mathcal{E}(\mathcal{A})\,\Big|\,\tau_{\mathcal{E}^{(p)}}=\tau_{\mathcal{E}(\mathcal{M}'')}\le\widehat{\gamma}_{\epsilon}\right]-\Phi_{t}(\mathcal{M}'',\,\mathcal{A})\,\right|\ =\ 0\;.\label{eq:fdi1}
\end{equation}
On the other hand, by  $\mathfrak{H}^{(p)}$,
\begin{equation}
\lim_{\epsilon\to0}\,\sup_{\boldsymbol{x}\in\mathcal{E}(\mathcal{M})}\,\left|\,\mathbb{P}_{\boldsymbol{x}}^{\epsilon}\left[\,\tau_{\mathcal{E}^{(p)}}=\tau_{\mathcal{E}(\mathcal{M}'')}\le\widehat{\gamma}_{\epsilon}\,\right]-\widehat{\mathcal{Q}}_{\mathcal{M}}^{(p)}\left[\,\tau_{\mathscr{V}^{(p)}}=\tau_{\mathcal{M}''}\,\right]\,\right|\ =\ 0\;.\label{eq:fdi2}
\end{equation}
Inserting last two estimates to \eqref{eq:fdddec-1}, we get
\begin{equation}
\lim_{\epsilon\rightarrow0}\,\sup_{\boldsymbol{x}\in\mathcal{E}(\boldsymbol{m})}\,\sup_{|s|\le\gamma_{\epsilon}}\,\bigg|\,\mathbb{P}_{\boldsymbol{x}}^{\epsilon}\Big[\,\bm{x}_{\epsilon}(\theta_{\epsilon}^{(p)}t+s)\in\mathcal{E}(\mathcal{A})\,\Big]-\Psi_{t}(\mathcal{M},\,\mathcal{A})\,\bigg|\ =\ 0\label{eq: fd2}
\end{equation}

Next we fix $\boldsymbol{x}\in\mathcal{D}(\boldsymbol{m})$. Then,
by \cite[Corollary 6.2, Lemma 6.7]{LLS-1st},
\[
\lim_{\epsilon\rightarrow0}\,\sup_{|s|\le\gamma_{\epsilon}}\,\bigg|\,\mathbb{P}_{\boldsymbol{x}}^{\epsilon}\left[\,\boldsymbol{x}_{\epsilon}(\theta_{\epsilon}^{(p)}t+s)\in\mathcal{E}(\mathcal{A})\,\right]-\mathbb{P}_{\boldsymbol{x}}^{\epsilon}\left[\,\boldsymbol{x}_{\epsilon}(\theta_{\epsilon}^{(p)}t+s)\in\mathcal{E}(\mathcal{A})\,\Big|\,\tau_{\mathcal{E}(\boldsymbol{m})}\le\epsilon^{-1}\,\right]\,\bigg|\ =\ 0\ .
\]
By the strong Markov property and \eqref{eq: fd2},
\begin{align*}
 & \lim_{\epsilon\to0}\,\sup_{|s|\le\gamma_{\epsilon}}\,\bigg|\,\mathbb{P}_{\boldsymbol{x}}^{\epsilon}\left[\,\boldsymbol{x}_{\epsilon}(\theta_{\epsilon}^{(p)}t+s)\in\mathcal{E}(\mathcal{A})\,\Big|\,\tau_{\mathcal{E}(\boldsymbol{m})}\le\epsilon^{-1}\,\right]-\Psi_{t}(\mathcal{M},\,\mathcal{A})\,\bigg|\\
 & \le\ \lim_{\epsilon\to0}\,\sup_{\boldsymbol{y}\in\mathcal{E}(\boldsymbol{m})}\,\sup_{|s|\le\gamma_{\epsilon}+\epsilon^{-1}}\,\bigg|\,\mathbb{P}_{\boldsymbol{y}}^{\epsilon}\left[\,\boldsymbol{x}_{\epsilon}(\theta_{\epsilon}^{(p)}t+s)\in\mathcal{E}(\mathcal{A})\,\right]-\Psi_{t}(\mathcal{M},\,\mathcal{A})\,\bigg|\ =\ 0\;.
\end{align*}
Combining last two computations complete the proof.
\end{proof}
Next we have a similar result for the intermediate time scale.
\begin{prop}
\label{prop:fd2_int}Fix $p\in\llbracket2,\,\mathfrak{q}\rrbracket$
and suppose that $\mathfrak{H}^{(p)}$ is in force. Suppose that for
all $\mathcal{M}\in\mathscr{V}^{(p)}$, $\boldsymbol{m}'\in\mathcal{M}$,
and sequences $(\varrho_{\epsilon})_{\epsilon>0}$ and $(\gamma_{\epsilon})_{\epsilon>0}$
such that $\theta_{\epsilon}^{(p-1)}\prec\varrho_{\epsilon}\prec\theta_{\epsilon}^{(p)}$
and $\gamma_{\epsilon}\prec\varrho_{\epsilon}$, we have
\begin{equation}
\lim_{\epsilon\rightarrow0}\,\sup_{\boldsymbol{x}\in\mathcal{E}(\mathcal{M})}\,\sup_{|s|\le\gamma_{\epsilon}}\,\left|\,\mathbb{P}_{\boldsymbol{x}}^{\epsilon}\Big[\,\bm{x}_{\epsilon}(\varrho_{\epsilon}+s)\in\mathcal{E}(\boldsymbol{m}')\,\Big]-\Phi^{(p)}(\mathcal{M},\,\boldsymbol{m}')\,\right|\ =\ 0\label{fd3}
\end{equation}
for some $\Phi^{(p)}(\mathcal{M},\,\bm{m}')\in[0,\,1]$. Then, for
all $\mathcal{M}\in\mathscr{S}^{(p)}$, $\mathcal{M}'\in\mathscr{V}^{(p)}$,
$\boldsymbol{m}\in\mathcal{M}$, $\boldsymbol{m}'\in\mathcal{M}'$,
a sequence $(\gamma_{\epsilon})_{\epsilon>0}$ such that $\gamma_{\epsilon}\prec\theta_{\epsilon}^{(p)}$,
and $\boldsymbol{x}\in\mathcal{D}(\boldsymbol{m})$, we have
\begin{equation}
\lim_{\epsilon\rightarrow0}\,\mathbb{P}_{\boldsymbol{x}}^{\epsilon}\Big[\,\bm{x}_{\epsilon}(\varrho_{\epsilon})\in\mathcal{E}(\boldsymbol{m}')\,\Big]\ =\ \Psi^{(p)}(\mathcal{M},\,\boldsymbol{m}')\;.\label{fd4}
\end{equation}
where
\[
\Psi^{(p)}(\mathcal{M},\,\boldsymbol{m}')\ =\ \widehat{\mathcal{Q}}_{\mathcal{M}}^{(p)}\left[\,\tau_{\mathscr{V}^{(p)}}=\tau_{\mathcal{M}'}\,\right]\Phi^{(p)}(\mathcal{M}',\,\boldsymbol{m}')\;.
\]
\end{prop}

\begin{proof}
We fix $\mathcal{M}\in\mathscr{S}^{(p)}$, $\mathcal{M}'\in\mathscr{V}^{(p)}$
and $\boldsymbol{m}'\in\mathcal{M}'$. We take an auxiliary sequence
$(\widehat{\gamma}_{\epsilon})_{\epsilon>0}$ such that $\theta_{\epsilon}^{(p-1)}\prec\widehat{\gamma}_{\epsilon}\prec\varrho_{\epsilon}$.
Then, as in the proof of Proposition \ref{prop:fd1}, by  $\mathfrak{H}^{(p)}$,
for all $\boldsymbol{x}\in\mathcal{E}(\mathcal{M})$, it is enough
to estimate
\begin{align*}
 & \mathbb{P}_{\boldsymbol{x}}^{\epsilon}\left[\,\boldsymbol{x}_{\epsilon}(\varrho_{\epsilon})\in\mathcal{E}(\boldsymbol{m}')\,,\,\tau_{\mathcal{E}^{(p)}}\le\widehat{\gamma}_{\epsilon}\,\right]\\
 & =\ \sum_{\mathcal{M}''\in\mathscr{V}^{(p)}}\,\mathbb{P}_{\boldsymbol{x}}^{\epsilon}\left[\,\boldsymbol{x}_{\epsilon}(\varrho_{\epsilon})\in\mathcal{E}(\boldsymbol{m}')\,\Big|\,\tau_{\mathcal{E}^{(p)}}=\tau_{\mathcal{E}(\mathcal{M}'')}\le\widehat{\gamma}_{\epsilon}\,\right]\,\mathbb{P}_{\boldsymbol{x}}^{\epsilon}\left[\,\tau_{\mathcal{E}^{(p)}}=\tau_{\mathcal{E}(\mathcal{M}'')}\le\widehat{\gamma}_{\epsilon}\,\right]\;.
\end{align*}
By the strong Markov property and Lemma \ref{l: nojump}, for $\mathcal{M}''\ne\mathcal{M}'$,
\[
\lim_{\epsilon\rightarrow0}\,\sup_{\boldsymbol{x}\in\mathcal{E}(\mathcal{M})}\,\mathbb{P}_{\boldsymbol{x}}^{\epsilon}\left[\,\boldsymbol{x}_{\epsilon}(\varrho_{\epsilon})\in\mathcal{E}(\boldsymbol{m}')\,\Big|\,\tau_{\mathcal{E}^{(p)}}=\tau_{\mathcal{E}(\mathcal{M}'')}\le\gamma_{\epsilon}\,\right]\ =\ 0\;.
\]
On the other hand, a similar consideration with \eqref{eq:fdi1} and
\eqref{eq:fdi2} based on \eqref{fd3}, we have that
\begin{align*}
 & \lim_{\epsilon\to0}\,\sup_{\boldsymbol{x}\in\mathcal{E}(\mathcal{M})}\,\left|\,\mathbb{P}_{\boldsymbol{x}}^{\epsilon}\left[\,\boldsymbol{x}_{\epsilon}(\varrho_{\epsilon})\in\mathcal{E}(\boldsymbol{m}')\,\Big|\,\tau_{\mathcal{E}^{(p)}}=\tau_{\mathcal{E}(\mathcal{M}')}\le\gamma_{\epsilon}\,\right]-\Phi^{(p)}(\mathcal{M}',\,\boldsymbol{m}')\,\right|\ =\ 0\;,\\
 & \lim_{\epsilon\to0}\,\sup_{\boldsymbol{x}\in\mathcal{E}(\mathcal{M})}\,\left|\,\mathbb{P}_{\boldsymbol{x}}^{\epsilon}\left[\,\tau_{\mathcal{E}^{(p)}}=\tau_{\mathcal{E}(\mathcal{M}')}\le\gamma_{\epsilon}\,\right]-\widehat{\mathcal{Q}}_{\mathcal{M}}^{(p)}\left[\,\tau_{\mathscr{V}^{(p)}}=\tau_{\mathcal{M}'}\,\right]\,\right|\ =\ 0\;.
\end{align*}
Combining these estimates, we get
\begin{equation}
\lim_{\epsilon\rightarrow0}\,\sup_{\boldsymbol{x}\in\mathcal{E}(\mathcal{M})}\,\left|\,\mathbb{P}_{\boldsymbol{x}}^{\epsilon}\Big[\,\bm{x}_{\epsilon}(\varrho_{\epsilon})\in\mathcal{E}(\boldsymbol{m}')\,\Big]-\Psi^{(p)}(\mathcal{M},\,\boldsymbol{m}')\,\right|\ =\ 0\;.\label{eq:fd4}
\end{equation}
Then, as in the proof of Proposition \ref{prop:fd1}, we can conclude
that, for all $\bm{m}\in\mathcal{M}$ and $\boldsymbol{x}\in\mathcal{D}(\bm{m})$,
\[
\lim_{\epsilon\rightarrow0}\,\mathbb{P}_{\boldsymbol{x}}^{\epsilon}\Big[\,\bm{x}_{\epsilon}(\varrho_{\epsilon})\in\mathcal{E}(\boldsymbol{m}')\,\Big]\ =\ \Psi^{(p)}(\mathcal{M},\,\boldsymbol{m}')\;.
\]
\end{proof}
Before investigating further, we mention a direct consequence of Proposition
\ref{prop:fd1} that allows us to fill the missing step in the proof
of Theorem \ref{t_meta}.
\begin{prop}
\label{p_meta}Suppose that $\mathfrak{C}_{\textup{fdd}}^{(p)}$ and
$\mathfrak{H}^{(p)}$ holds for all $p\in\llbracket1,\,\mathfrak{q}\rrbracket$.
Then, Theorem \ref{t_meta} holds.
\end{prop}

\begin{proof}
It suffices to prove Theorem \ref{t_meta} for $n=1$, since the case
$n\ge2$ follows from $n=1$ and the Markov property. For the case
$n=1$, let us fix $\mathcal{M}\in\mathscr{S}^{(p)}$, $\mathcal{M}'\in\mathscr{V}^{(p)}$,
$\boldsymbol{m}\in\mathcal{M}$, $t>0$, and a sequence $(\gamma_{\epsilon})_{\epsilon>0}$
such that $\gamma_{\epsilon}\prec\theta_{\epsilon}^{(p)}$. Then,
it suffices to prove that, for all $\boldsymbol{x}\in\mathcal{D}(\boldsymbol{m})$,
\[
\lim_{\epsilon\rightarrow0}\,\sup_{|s|\le\gamma_{\epsilon}}\,\left|\,\mathbb{P}_{\boldsymbol{x}}^{\epsilon}\Big[\,\bm{x}_{\epsilon}(\theta_{\epsilon}^{(p)}t+s)\in\mathcal{E}(\mathcal{M}')\,\Big]-\mathcal{Q}_{\mathfrak{a}^{(p-1)}(\boldsymbol{m},\,\cdot)}^{(p)}\Big[\,{\bf y}^{(p)}(t)=\mathcal{M}'\,\Big]\,\right|\ =\ 0\;.
\]
This is a direct consequence of $\mathfrak{C}_{\textup{fdd}}^{(p)}$
and Proposition \ref{prop:fd1} (by recalling the definition \eqref{noteq1}
of $\mathfrak{a}^{(p-1)}(\boldsymbol{m},\,\cdot)$) which is based
on $\mathfrak{H}^{(p)}$. This completes the proof.
\end{proof}

\subsection{Marginal distribution at metastable scales}

By Remark \ref{rem:fdd} and $\mathfrak{C}_{{\rm fdd}}^{(p)}$, for
all $p\in\llbracket1,\,\mathfrak{q}\rrbracket$, $\mathcal{M},\,\mathcal{M}'\in\mathscr{V}^{(p)}$
and sequence $(\gamma_{\epsilon})_{\epsilon>0}$ such that $\gamma_{\epsilon}\prec\theta_{\epsilon}^{(p)}$,
we have
\begin{equation}
\lim_{\epsilon\rightarrow0}\,\sup_{\boldsymbol{x}\in\mathcal{E}(\mathcal{M})}\,\sup_{|s|\le\gamma_{\epsilon}}\,\left|\,\mathbb{P}_{\boldsymbol{x}}^{\epsilon}\left[\,\boldsymbol{x}_{\epsilon}(\theta_{\epsilon}^{(p)}t+s)\in\mathcal{E}(\mathcal{M}')\,\right]-\mathcal{Q}_{\mathcal{M}}^{(p)}\left[\,\mathbf{y}^{(p)}(t)=\mathcal{M}'\,\right]\,\right|\ =\ 0\ .\label{eq:fdd2}
\end{equation}
We start from a refinement of this result. To that end, we need the
following result which indicates that the process ${\bf y}^{(p)}(\cdot)$
is reversible when restricted to a recurrence class.
\begin{prop}
\label{p_rev_y} Fix $p\in\llbracket1,\,\mathfrak{q}\rrbracket$ and
a recurrent class $\mathscr{R}$ of $\mathbf{y}^{(p)}(\cdot)$. Then,
the Markov chain $\mathbf{y}^{(p)}(\cdot)$ restricted to $\mathscr{R}$
is reversible with respect to the measure $\pi_{\mathscr{R}}(\cdot)$
on $\mathscr{R}$ defined by $\pi_{\mathscr{R}}(\cdot)=\nu(\cdot)/\nu(\mathcal{M}(\mathscr{R}))$
where $\mathcal{M}(\mathscr{R})=\cup_{\mathcal{M}\in\mathscr{R}}\mathcal{M}$.
\end{prop}

The proof of this proposition is postponed to Section \ref{sec13}.
Now let us consider the refinement of \eqref{eq:fdd2}.
\begin{lem}
\label{lem:margin1} Fix $p\in\llbracket1,\,\mathfrak{q}\rrbracket$,
$\mathcal{M},\,\mathcal{M}'\in\mathscr{V}^{(p)}$ and $\boldsymbol{m}'\in\mathcal{M}'$.
Then, for any sequence $(\gamma_{\epsilon})_{\epsilon>0}$ such that
$\gamma_{\epsilon}\prec\theta_{\epsilon}^{(p)}$,
\[
\lim_{\epsilon\rightarrow0}\,\sup_{\boldsymbol{x}\in\mathcal{E}(\mathcal{M})}\,\sup_{|s|\le\gamma_{\epsilon}}\,\left|\,\mathbb{P}_{\boldsymbol{x}}^{\epsilon}\left[\,\boldsymbol{x}_{\epsilon}(\theta_{\epsilon}^{(p)}t+s)\in\mathcal{E}(\boldsymbol{m}')\,\right]-\frac{\nu(\boldsymbol{m}')}{\nu(\mathcal{M}')}\,\mathcal{Q}_{\mathcal{M}}^{(p)}\left[\,\mathbf{y}^{(p)}(t)=\mathcal{M}'\,\right]\,\right|\ =\ 0\ .
\]
\end{lem}

\begin{proof}
We prove by induction. For $p=1$, as $\mathcal{M}\in\mathscr{V}^{(1)}$
is singleton, the assertion of lemma is direct from \eqref{eq:fdd2}.

Fix $k\ge2$ and assume that the lemma holds for $p\in\llbracket1,\,k-1\rrbracket$.
Fix $\mathcal{M},\,\mathcal{M}'\in\mathscr{V}^{(k)}$, $\boldsymbol{m}'\in\mathcal{M}'$
and a sequence $(\gamma_{\epsilon})_{\epsilon>0}$ such that $\gamma_{\epsilon}\prec\theta_{\epsilon}^{(k)}$.
Let $\mathscr{V}^{(k-1)}(\mathcal{M}')=\{\mathcal{M}_{1}',\,\dots,\,\mathcal{M}_{n}'\}$
be the standard decomposition of $\mathcal{M}'$, and suppose, without
loss of generality, that $\boldsymbol{m}'\in\mathcal{M}_{1}'$.

Fix $\eta>0$. Since $\mathscr{V}^{(k-1)}(\mathcal{M}')$ forms a
recurrent class of the chain $\mathbf{y}^{(k-1)}(\cdot)$, by Proposition
\ref{p_rev_y}, there exists $s_{0}=s_{0}(\eta)>0$ such that
\begin{equation}
\max_{j\in\llbracket1,\,n\rrbracket}\,\left|\,\mathcal{Q}_{\mathcal{M}_{j}'}^{(k-1)}\left[\,{\bf y}^{(k-1)}(s_{0})=\mathcal{M}_{1}'\,\right]-\frac{\nu(\mathcal{M}_{1}')}{\nu(\mathcal{M}')}\,\right|\ \le\ \eta\ .\label{e_l_t01-1}
\end{equation}
Fix $\boldsymbol{x}\in\mathcal{E}(\mathcal{M})$, $|s|\le\gamma_{\epsilon}$,
and write
\[
\mathbb{P}_{\boldsymbol{x}}^{\epsilon}\left[\,\boldsymbol{x}_{\epsilon}(\theta_{\epsilon}^{(k)}t+s)\in\mathcal{E}(\boldsymbol{m}')\,\right]\ =\ \sum_{j=1}^{3}\mathbb{P}_{\boldsymbol{x}}^{\epsilon}\big[\,\boldsymbol{x}_{\epsilon}(\theta_{\epsilon}^{(k)}t+s)\in\mathcal{E}(\boldsymbol{m}')\,\Big|\,\mathscr{A}_{j}\,\big]\,\mathbb{P}_{\boldsymbol{x}}^{\epsilon}\big[\,\mathscr{A}_{j}\,\big]\;,
\]
where
\begin{gather*}
\mathscr{A}_{1}\ =\ \{\,\boldsymbol{x}_{\epsilon}(t_{\epsilon})\in\mathcal{E}(\mathcal{M}')\,\}\;,\quad\mathscr{A}_{2}\ =\ \{\,\boldsymbol{x}_{\epsilon}(t_{\epsilon})\in\mathcal{E}^{(k)}\setminus\mathcal{E}(\mathcal{M}')\,\}\;,\quad\mathscr{A}_{3}\ =\ \{\,\boldsymbol{x}_{\epsilon}(t_{\epsilon})\not\in\mathcal{E}^{(k)}\,\}\;,
\end{gather*}
and $t_{\epsilon}=\theta_{\epsilon}^{(k)}t+s-\theta_{\epsilon}^{(k-1)}s_{0}$.

We claim that the term corresponding to $\mathscr{A}_{2}$ is negligible.
Bound the second probability by $1$ and apply the strong Markov property
to get that this term is less than or equal to
\[
\sup_{\mathcal{M}''\in\mathscr{V}^{(k)}\setminus\{\mathcal{M}'\}}\,\sup_{\boldsymbol{y}\in\mathcal{E}(\mathcal{M}'')}\,\mathbb{P}_{\boldsymbol{y}}^{\epsilon}\left[\,\tau_{\mathcal{E}(\mathcal{M}')}<\theta_{\epsilon}^{(k-1)}s_{0}\,\right]\;.
\]
By Lemma \ref{l: nojump} this expression vanishes as $\epsilon\to0$.
To see that the term corresponding to $\mathscr{A}_{3}$ is also negligible,
bound the first probability by $1$ and apply Remark \ref{rem:fddneg},
$\mathfrak{C}_{{\rm fdd}}^{(k)}$ to the second probability.

We turn to the term corresponding to $\mathscr{A}_{1}$. We estimate
the two probabilities at $\mathscr{A}_{1}$ separately. For the first
probability, by the strong Markov property,
\begin{align*}
 & \sup_{|s|\le\gamma_{\epsilon}}\,\left|\,\mathbb{P}_{\boldsymbol{x}}^{\epsilon}\left[\boldsymbol{x}_{\epsilon}(\theta_{\epsilon}^{(k)}t+s)\in\mathcal{E}(\boldsymbol{m}')\,\Big|\,\boldsymbol{x}_{\epsilon}(t_{\epsilon})\in\mathcal{E}(\mathcal{M}')\right]-\frac{\nu(\boldsymbol{m}')}{\nu(\mathcal{M}')}\,\right|\\
 & \quad\le\ \sup_{\boldsymbol{y}\in\mathcal{E}(\mathcal{M}')}\,\left|\,\mathbb{P}_{\boldsymbol{y}}^{\epsilon}\left[\boldsymbol{x}_{\epsilon}(\theta_{\epsilon}^{(k-1)}s_{0})\in\mathcal{E}(\boldsymbol{m}')\right]-\frac{\nu(\boldsymbol{m}')}{\nu(\mathcal{M}')}\,\right|\ .
\end{align*}
This expression is less than or equal to
\begin{align*}
 & \sup_{j\in\llbracket1,\,n\rrbracket}\,\sup_{\boldsymbol{y}\in\mathcal{E}(\mathcal{M}_{j}')}\,\left|\,\mathbb{P}_{\boldsymbol{y}}^{\epsilon}\left[\,\boldsymbol{x}_{\epsilon}(\theta_{\epsilon}^{(k-1)}s_{0})\in\mathcal{E}(\boldsymbol{m}')\,\right]-\frac{\nu(\boldsymbol{m}')}{\nu(\mathcal{M}_{1}')}\mathcal{Q}_{\mathcal{M}_{j}'}^{(k-1)}\left[\,{\bf y}^{(k-1)}(s_{0})=\mathcal{M}_{1}'\,\right]\,\right|\\
 & \quad+\sup_{j\in\llbracket1,\,n\rrbracket}\,\sup_{\boldsymbol{y}\in\mathcal{E}(\mathcal{M}_{j}')}\,\left|\,\frac{\nu(\boldsymbol{m}')}{\nu(\mathcal{M}_{1}')}\mathcal{Q}_{\mathcal{M}_{j}'}^{(k-1)}\left[\,{\bf y}^{(k-1)}(s_{0})=\mathcal{M}_{1}'\,\right]-\frac{\nu(\boldsymbol{m}')}{\nu(\mathcal{M}')}\,\right|\ .
\end{align*}
By the induction hypothesis, the first term converges to $0$ as $\epsilon\to0$.
Therefore, by \eqref{e_l_t01-1},
\[
\limsup_{\epsilon\to0}\,\sup_{\boldsymbol{x}\in\mathcal{E}(\mathcal{M})}\,\sup_{|s|\le\gamma_{\epsilon}}\,\left|\,\mathbb{P}_{\boldsymbol{x}}^{\epsilon}\left[\,\boldsymbol{x}_{\epsilon}(\theta_{\epsilon}^{(k)}t+s)\in\mathcal{E}(\boldsymbol{m}')\,\Big|\,\boldsymbol{x}_{\epsilon}(t_{\epsilon})\in\mathcal{E}(\mathcal{M}')\,\right]-\frac{\nu(\boldsymbol{m}')}{\nu(\mathcal{M}')}\,\right|\ \le\ \eta\ .
\]

For the second probability, by \eqref{eq:fdd2},
\[
\lim_{\epsilon\to0}\,\sup_{\boldsymbol{x}\in\mathcal{E}(\mathcal{M})}\,\sup_{|s|\le\gamma_{\epsilon}}\,\left|\,\mathbb{P}_{\boldsymbol{x}}^{\epsilon}\left[\,\boldsymbol{x}_{\epsilon}(t_{\epsilon})\in\mathcal{E}(\mathcal{M}')\,\right]-\mathcal{Q}_{\mathcal{M}}^{(k)}\left[\,{\bf y}^{(k)}(t)=\mathcal{M}'\,\right]\,\right|\ =\ 0\ .
\]
Putting together the previous estimates yields that the left-hand
side of the estimate given in the lemma is bounded by $\eta$. To
complete the proof, it remains to let $\eta\to0$.
\end{proof}
Now we have the following result.
\begin{prop}
\label{p_marginDA} Fix $p\in\llbracket1,\,\mathfrak{q}\rrbracket$,
$\mathcal{M},\,\mathcal{M}'\in\mathscr{S}^{(p)}$, $\boldsymbol{m}\in\mathcal{M}$,
and $\boldsymbol{m}'\in\mathcal{M}'$. Then, for all $\boldsymbol{x}\in\mathcal{D}(\boldsymbol{m})$
and $t>0$,
\[
\lim_{\epsilon\rightarrow0}\,\mathbb{P}_{\boldsymbol{x}}^{\epsilon}\left[\,\boldsymbol{x}_{\epsilon}(\theta_{\epsilon}^{(p)}t)\in\mathcal{E}(\boldsymbol{m}')\,\right]\ =\ \sum_{\mathcal{M}''\in\mathscr{V}^{(p)}}\,\widehat{\mathcal{Q}}_{\mathcal{M}}^{(p)}\left[\,\tau_{\mathscr{V}^{(p)}}=\tau_{\mathcal{M}''}\,\right]\,\mathcal{Q}_{\mathcal{M}''}^{(p)}\left[\,\mathbf{y}^{(p)}(t)=\mathcal{M}'\,\right]\,\frac{\nu(\boldsymbol{m}')}{\nu(\mathcal{M}')}\ .
\]
In particular, we have
\[
\lim_{\epsilon\rightarrow0}\,\mathbb{P}_{\boldsymbol{x}}^{\epsilon}\left[\,\boldsymbol{x}_{\epsilon}(\theta_{\epsilon}^{(p)}t)\notin\mathcal{E}^{(p)}\,\right]\ =\ 0\;.
\]
\end{prop}

\begin{proof}
The first assertion is direct from Lemma \ref{lem:margin1} and Proposition
\ref{prop:fd1}, and the second assertion is immediate from the first
one.
\end{proof}

\subsection{Marginal distribution at intermediate scales}

In this subsection, we establish the results corresponding to Proposition
\ref{p_marginDA} for intermediate scales.

Fix ${\color{blue}R_{0}>0}$ large enough so that the level set $\{U\le R_{0}\}$
is connected and contains all critical points of $U$. Write ${\color{blue}\Lambda}=\{U\le R_{0}\}$
and ${\color{blue}\widetilde{\Lambda}}=\{U\le2R_{0}\}$. The first
lemma asserts that at intermediate time-scale, with overwhelming probability,
the process belongs to valleys at the level $p$.
\begin{lem}
\label{l_4141} Fix $p\in\llbracket2,\,\mathfrak{q}+1\rrbracket$,
and a sequence $(\varrho_{\epsilon})_{\epsilon>0}$ satisfying $\theta_{\epsilon}^{(p-1)}\prec\varrho_{\epsilon}\prec\theta_{\epsilon}^{(p)}$.
Then,
\[
\limsup_{\epsilon\to0}\,\sup_{\boldsymbol{x}\in\Lambda}\,\mathbb{P}_{\boldsymbol{x}}^{\epsilon}\left[\,\boldsymbol{x}_{\epsilon}(\varrho_{\epsilon})\notin\mathcal{E}^{(p)}\,\right]\ =\ 0\ .
\]
\end{lem}

\begin{proof}
First we consider the case $p\in\llbracket2,\,\mathfrak{q}\rrbracket$.
For each $t>0$, by Proposition \ref{p_FW},
\[
\limsup_{\epsilon\to0}\,\sup_{\boldsymbol{x}\in\Lambda}\,\mathbb{P}_{\boldsymbol{x}}^{\epsilon}\left[\,\boldsymbol{x}_{\epsilon}(\varrho_{\epsilon}-\theta_{\epsilon}^{(p-1)}t)\notin\widetilde{\Lambda}\,\right]\ =\ 0\;.
\]
Therefore, by the strong Markov property,
\[
\limsup_{\epsilon\to0}\,\sup_{\boldsymbol{x}\in\Lambda}\,\mathbb{P}_{\boldsymbol{x}}^{\epsilon}\left[\,\boldsymbol{x}_{\epsilon}(\varrho_{\epsilon})\notin\mathcal{E}^{(p)}\,\right]\ \le\ \limsup_{\epsilon\to0}\,\sup_{\boldsymbol{x}\in\widetilde{\Lambda}}\mathbb{P}_{\boldsymbol{x}}^{\epsilon}\left[\,\boldsymbol{x}_{\epsilon}(\theta_{\epsilon}^{(p-1)}t)\notin\mathcal{E}^{(p)}\,\right]\;.
\]
Fix a sequence $(\gamma_{\epsilon})_{\epsilon>0}$ such that $\theta_{\epsilon}^{(p-2)}\prec\gamma_{\epsilon}\prec\theta_{\epsilon}^{(p-1)}$
if $p\ge2$ and $\frac{1}{\epsilon}\prec\gamma_{\epsilon}\prec\theta_{\epsilon}^{(1)}$
if $p=1$. Then, by Lemma \ref{lem_hitting},
\[
\limsup_{\epsilon\to0}\,\sup_{\boldsymbol{x}\in\widetilde{\Lambda}}\,\mathbb{P}_{\boldsymbol{x}}^{\epsilon}\left[\,\tau_{\mathcal{E}^{(p-1)}}>\gamma_{\epsilon}\,\right]\ =\ 0\;.
\]
Therefore, by the strong Markov property
\[
\limsup_{\epsilon\to0}\,\sup_{\boldsymbol{x}\in\widetilde{\Lambda}}\,\mathbb{P}_{\boldsymbol{x}}^{\epsilon}\left[\,\boldsymbol{x}_{\epsilon}(\theta_{\epsilon}^{(p-1)}t)\notin\mathcal{E}^{(p)}\,\right]\ \le\ \limsup_{\epsilon\to0}\,\sup_{\boldsymbol{x}\in\mathcal{E}^{(p-1)}}\,\sup_{s\in[0,\,\gamma_{\epsilon}]}\,\mathbb{P}_{\boldsymbol{x}}^{\epsilon}\left[\,\boldsymbol{x}_{\epsilon}(\theta_{\epsilon}^{(p-1)}t-s)\notin\mathcal{E}^{(p)}\,\right]\;.
\]
By Remark \ref{rem:fdd}and $\mathfrak{C}_{{\rm fdd}}^{(p-1)}$, since
$\mathcal{E}^{(p)}=\bigcup_{\mathcal{M}\in\mathscr{R}^{(p-1)}}\mathcal{E}(\mathcal{M})$,
where $\mathscr{R}^{(p-1)}$ is the union of the recurrence classes
of the process ${\bf y}^{(p-1)}(\cdot)$, the right-hand side is equal
to
\[
\max_{\mathcal{M}\in\mathscr{V}^{(p-1)}}\,\mathcal{Q}_{\mathcal{M}}^{(p-1)}\left[\,{\bf y}^{(p-1)}(t)\notin\mathscr{R}^{(p-1)}\,\right]\;.
\]
Since the last expression converges to $0$ by the definition of $\mathscr{R}^{(p-1)}$
as $t\to\infty$, we can complete the proof by letting $t\rightarrow\infty$.

Next we consider the case $p=\mathfrak{q}+1$. Let
$\varrho_{\epsilon}\succ\theta_{\epsilon}^{(\mathfrak{q})}$. By \cite[Corollary 6.2]{LLS-1st}
and the strong Markov property, for all sequence $(\gamma_{\epsilon})_{\epsilon>0}$
such that $\theta_{\epsilon}^{(\mathfrak{q})}\prec\gamma_{\epsilon}\prec\varrho_{\epsilon}$,
\begin{align*}
\limsup_{\epsilon\to0}\,\sup_{\boldsymbol{x}\in\Lambda}\,\mathbb{P}_{\boldsymbol{x}}^{\epsilon}\left[\,\boldsymbol{x}_{\epsilon}(\varrho_{\epsilon})\notin\mathcal{E}^{(\mathfrak{q}+1)}\,\right]\  & \le\ \limsup_{\epsilon\to0}\,\sup_{\boldsymbol{x}\in\mathcal{E}(\mathcal{M}_{0})}\,\sup_{s\in[0,\,\epsilon^{-1}]}\,\mathbb{P}_{\boldsymbol{x}}^{\epsilon}\left[\,\boldsymbol{x}_{\epsilon}(\varrho_{\epsilon}+s)\notin\mathcal{E}^{(\mathfrak{q}+1)}\,\right]\\
 & =\;\limsup_{\epsilon\to0}\,\sup_{\boldsymbol{x}\in\mathcal{E}^{(\mathfrak{q}+1)}}\,\sup_{s\in[0,\,\gamma_{\epsilon}]}\,\mathbb{P}_{\boldsymbol{x}}^{\epsilon}\left[\,\boldsymbol{x}_{\epsilon}(\varrho_{\epsilon}+s)\notin\mathcal{E}^{(\mathfrak{q}+1)}\,\right]\;,
\end{align*}
where the last equality follows from $\mathfrak{H}^{(\mathfrak{q}+1)}$.
Since $\mathcal{E}^{(\mathfrak{q}+1)}=\mathcal{E}(\mathcal{M}_{\star})$
by Proposition \ref{prop_global_min}, it suffices to prove that for
$\bm{m}\in\mathcal{M}_{\star}$,
\begin{equation}
\limsup_{\epsilon\to0}\,\sup_{\boldsymbol{x}\in\mathcal{E}(\bm{m})}\,\sup_{s\in[0,\,\gamma_{\epsilon}]}\,\mathbb{P}_{\boldsymbol{x}}^{\epsilon}\left[\,\boldsymbol{x}_{\epsilon}(\varrho_{\epsilon}+s)\notin\mathcal{E}^{(\mathfrak{q}+1)}\,\right]=0\ .\label{eq:aa4}
\end{equation}
This follows from the local ergodicity proved in \cite[Section 4]{LLS-1st}.
More precisely, let $\bm{x}_{\epsilon}^{F}(\cdot)$ be a diffusion
process defined in \cite[Section 4]{LLS-1st} (with respect to the
minimum $\bm{m}\in\mathcal{M}_{\star}$) and let $\mu_{\epsilon}^{F}(\cdot)$
be the invariant measure of $\bm{x}_{\epsilon}^{F}(\cdot)$. Then,
by \cite[Theorem 3.1]{LLS-1st}, we have
\[
\mathbb{P}_{\boldsymbol{x}}^{\epsilon}\left[\,\boldsymbol{x}_{\epsilon}(\varrho_{\epsilon}+s)\notin\mathcal{E}^{(\mathfrak{q}+1)}\,\right]\ =\ \mathbb{P}_{\mu_{\epsilon}^{F}}^{\epsilon}\left[\,\boldsymbol{x}_{\epsilon}(\varrho_{\epsilon}+s-\epsilon^{-1/6})\notin\mathcal{E}^{(\mathfrak{q}+1)}\,\right]+o_{\epsilon}(1)\ ,
\]
where the error term is uniform over $\bm{x}\in\mathcal{E}(\bm{m})$.
By \cite[display (4.10)]{LLS-1st}, we can replace the measure $\mu_{\epsilon}^{F}$
at the right-hand side with the invariant measure $\mu_{\epsilon}$
conditioned on $\mathcal{E}(\bm{m})$, and hence
\begin{align*}
\mathbb{P}_{\mu_{\epsilon}^{F}}^{\epsilon}\left[\,\boldsymbol{x}_{\epsilon}(\varrho_{\epsilon}+s-\epsilon^{-1/6})\notin\mathcal{E}^{(\mathfrak{q}+1)}\,\right]\  & \le\ \frac{1}{\mu_{\epsilon}(\mathcal{E}(\bm{m}))} \mathbb{P}_{\mu_\epsilon}^{\epsilon}\left[\,\boldsymbol{x}_{\epsilon}(\varrho_{\epsilon}+s-\epsilon^{-1/6})\notin\mathcal{E}^{(\mathfrak{q}+1)}\,\right]\,+\,o_{\epsilon}(1)\\
 & =\ \frac{1}{\mu_{\epsilon}(\mathcal{E}(\bm{m}))}\,\mu_{\epsilon}(\mathbb{R}^{d}\setminus\mathcal{E}^{(\mathfrak{q}+1)})+o_{\epsilon}(1)\ .
\end{align*}
This completes the proof of \eqref{eq:aa4} since it is well-known
that $\mu_{\epsilon}(\mathbb{R}^{d}\setminus\mathcal{E}^{(\mathfrak{q}+1)})\ll\mu_{\epsilon}(\mathcal{E}(\bm{m}))$
by the Laplace asymptotics.
\end{proof}
Next lemma states that the process cannot escape from a valley around
$\mathscr{V}^{(p)}$ in a scale less than $\theta_{\epsilon}^{(p)}$.
\begin{lem}
\label{l_4242} Fix $p\in\llbracket2,\,\mathfrak{q}+1\rrbracket$,
$\mathcal{M}\in\mathscr{V}^{(p)}$ , and a sequence $(\varrho_{\epsilon})_{\epsilon>0}$
satisfying $\theta_{\epsilon}^{(p-1)}\prec\varrho_{\epsilon}\prec\theta_{\epsilon}^{(p)}$.
Then,
\[
\limsup_{\epsilon\to0}\,\sup_{\boldsymbol{x}\in\mathcal{E}(\mathcal{M})}\,\mathbb{P}_{\boldsymbol{x}}^{\epsilon}\left[\,\boldsymbol{x}_{\epsilon}(\varrho_{\epsilon})\notin\mathcal{E}(\mathcal{M})\,\right]\ =\ 0\ .
\]
\end{lem}

\begin{proof}
By Lemma \ref{l_4141}, it suffices to show that
\begin{equation}
\limsup_{\epsilon\to0}\,\sup_{\boldsymbol{x}\in\mathcal{E}(\mathcal{M})}\,\mathbb{P}_{\boldsymbol{x}}^{\epsilon}\left[\,\boldsymbol{x}_{\epsilon}(\varrho_{\epsilon})\in\mathcal{E}^{(p)}\setminus\mathcal{E}(\mathcal{M})\,\right]\ =\ 0\;.\label{eq:4242}
\end{equation}
This is immediate for $p=\mathfrak{q}+1$ since $\mathcal{M}_{\star}$
is the only element of $\mathscr{V}^{(\mathfrak{q}+1)}$ by Proposition
\ref{prop_global_min} and hence $\mathcal{E}^{(p)}\setminus\mathcal{E}(\mathcal{M})=\varnothing$.
For $p\in\llbracket2,\,\mathfrak{q}\rrbracket$, since
\[
\mathbb{P}_{\boldsymbol{x}}^{\epsilon}\left[\,\boldsymbol{x}_{\epsilon}(\varrho_{\epsilon})\in\mathcal{E}^{(p)}\setminus\mathcal{E}(\mathcal{M})\,\right]\ \le\ \mathbb{P}_{\boldsymbol{x}}^{\epsilon}\left[\,\tau_{\mathcal{E}^{(p)}\setminus\mathcal{E}(\mathcal{M})}\le\varrho_{\epsilon}\,\right]\;,
\]
the assertion \eqref{eq:4242} follows from Lemma \ref{l: nojump}.
\end{proof}
The following result is a refinement of the previous lemma and can
be regarded as an intermediate scale version of Lemma \ref{lem:margin1}.
\begin{lem}
\label{lem:margin_int} Fix $p\in\llbracket2,\,\mathfrak{q}+1\rrbracket$,
$\mathcal{M}\in\mathscr{V}^{(p)}$ and $\boldsymbol{m}\in\mathcal{M}$.
Then, for any sequences $(\varrho_{\epsilon})_{\epsilon>0}$ and $(\gamma_{\epsilon})_{\epsilon>0}$
satisfying $\theta_{\epsilon}^{(p-1)}\prec\varrho_{\epsilon}\prec\theta_{\epsilon}^{(p)}$
and $\gamma_{\epsilon}\prec\varrho_{\epsilon}$,
\[
\lim_{\epsilon\rightarrow0}\,\sup_{\boldsymbol{x}\in\mathcal{E}(\mathcal{M})}\,\sup_{s\in[0,\,\gamma_{\epsilon}]}\,\left|\,\mathbb{P}_{\boldsymbol{x}}^{\epsilon}\left[\,\boldsymbol{x}_{\epsilon}(\varrho_{\epsilon}+s)\in\mathcal{E}(\boldsymbol{m})\,\right]-\frac{\nu(\boldsymbol{m})}{\nu(\mathcal{M})}\,\right|\ =\ 0\ .
\]
\end{lem}

\begin{proof}
Let $\{\mathcal{M}_{1},\,\dots,\,\mathcal{M}_{n}\}$ be the standard
decomposition of $\mathcal{M}$. Fix $\eta>0$. By Proposition \ref{p_rev_y},
there exists $s_{0}=s_{0}(\eta)>0$ such that
\begin{equation}
\max_{1\le i,\,j\le n}\,\left|\,\mathcal{Q}_{\mathcal{M}_{i}}^{(p-1)}\left[{\bf y}^{(p-1)}(s_{0})\in\mathcal{M}_{j}\right]-\frac{\nu(\mathcal{M}_{j})}{\nu(\mathcal{M})}\,\right|\ <\ \eta\ .\label{eq:condt_0}
\end{equation}
Let $\boldsymbol{x}\in\mathcal{E}(\mathcal{M})$. As in the proof
of Lemma \ref{lem:margin1}, we decompose
\[
\mathbb{P}_{\boldsymbol{x}}^{\epsilon}\left[\,\boldsymbol{x}_{\epsilon}(\varrho_{\epsilon}+s)\in\mathcal{E}(\boldsymbol{m})\,\right]\ =\ \sum_{j=1}^{2}\,\mathbb{P}_{\boldsymbol{x}}^{\epsilon}\left[\,\boldsymbol{x}_{\epsilon}(\varrho_{\epsilon}+s)\in\mathcal{E}(\boldsymbol{m})\,\Big|\,\mathscr{A}_{j}\,\right]\,\mathbb{P}_{\boldsymbol{x}}^{\epsilon}\left[\mathscr{A}_{j}\right]\;,
\]
where $t_{\epsilon}=\varrho_{\epsilon}+s-\theta_{\epsilon}^{(p-1)}s_{0}$
and
\[
\mathscr{A}_{1}\ =\ \{\,\boldsymbol{x}_{\epsilon}(t_{\epsilon})\in\mathcal{E}(\mathcal{M})\,\}\;,\quad\mathscr{A}_{2}\ =\ \{\,\boldsymbol{x}_{\epsilon}(t_{\epsilon})\not\in\mathcal{E}(\mathcal{M})\,\}\;.
\]

By Lemma \ref{l_4242}, the expression corresponding to $\mathscr{A}_{2}$
is negligible since $\mathbb{P}_{\boldsymbol{x}}^{\epsilon}\left[\mathscr{A}_{2}\right]$
vanishes as $\epsilon\to0$. It remains to estimate the term corresponding
to $\mathscr{A}_{1}$. We consider each term of the product separately.
For the first one, by the Markov property,
\begin{align*}
 & \sup_{s\in[0,\,\gamma_{\epsilon}]}\,\left|\,\mathbb{P}_{\boldsymbol{x}}^{\epsilon}\left[\,\boldsymbol{x}_{\epsilon}(\varrho_{\epsilon}+s)\in\mathcal{E}(\boldsymbol{m})\,\Big|\,\boldsymbol{x}_{\epsilon}(\varrho_{\epsilon}+s-\theta_{\epsilon}^{(p-1)}s_{0})\in\mathcal{E}(\mathcal{M})\,\right]-\frac{\nu(\boldsymbol{m})}{\nu(\mathcal{M})}\,\right|\\
 & \quad\le\sup_{i\in\llbracket1,\,n\rrbracket}\,\sup_{\boldsymbol{y}\in\mathcal{E}(\mathcal{M}_{i})}\,\left|\,\mathbb{P}_{\boldsymbol{y}}^{\epsilon}\left[\,\boldsymbol{x}_{\epsilon}(\theta_{\epsilon}^{(p-1)}s_{0})\in\mathcal{E}(\boldsymbol{m})\,\right]-\frac{\nu(\boldsymbol{m})}{\nu(\mathcal{M})}\,\right|\;.
\end{align*}
Assume, without loss of generality, that $\boldsymbol{m}\in\mathcal{M}_{1}$.
Bound the right-hand side as
\begin{align*}
 & \sup_{i\in\llbracket1,\,n\rrbracket}\,\sup_{\boldsymbol{y}\in\mathcal{E}(\mathcal{M}_{i})}\,\left|\mathbb{P}_{\boldsymbol{y}}^{\epsilon}\left[\,\boldsymbol{x}_{\epsilon}(\theta_{\epsilon}^{(p-1)}s_{0})\in\mathcal{E}(\boldsymbol{m})\,\right]-\mathcal{Q}_{\mathcal{M}_{i}}^{(p-1)}\left[\,{\bf y}^{(p-1)}(s_{0})=\mathcal{M}_{1}\,\right]\,\frac{\nu(\boldsymbol{m})}{\nu(\mathcal{M}_{1})}\right|\\
 & \ +\sup_{i\in\llbracket1,\,n\rrbracket}\,\left|\,\mathcal{Q}_{\mathcal{M}_{i}}^{(p-1)}\left[\,{\bf y}^{(p-1)}(s_{0})=\mathcal{M}_{1}\,\right]\,\frac{\nu(\boldsymbol{m})}{\nu(\mathcal{M}_{1})}-\frac{\nu(\boldsymbol{m})}{\nu(\mathcal{M})}\,\right|\ .
\end{align*}
By Lemma \ref{lem:margin1}, the first term vanishes as $\epsilon\to0$.
By \eqref{eq:condt_0}, the second one is bounded by $\eta$. Hence,
\[
\limsup_{\epsilon\to0}\,\sup_{\boldsymbol{x}\in\mathcal{E}(\mathcal{M})}\,\sup_{s\in[0,\,\gamma_{\epsilon}]}\,\left|\,\mathbb{P}_{\boldsymbol{x}}^{\epsilon}\left[\,\boldsymbol{x}_{\epsilon}(\varrho_{\epsilon}+s)\in\mathcal{E}(\boldsymbol{m})\,\Big|\,\boldsymbol{x}_{\epsilon}(\varrho_{\epsilon}+s-\theta_{\epsilon}^{(p-1)}s_{0})\in\mathcal{E}(\mathcal{M})\,\right]-\frac{\nu(\boldsymbol{m})}{\nu(\mathcal{M})}\,\right|\ \le\ \eta\ .
\]
On the other hand, by Lemma \ref{l_4242},
\[
\lim_{\epsilon\to0}\,\sup_{\boldsymbol{x}\in\mathcal{E}(\mathcal{M})}\,\mathbb{P}_{\boldsymbol{z}_{\epsilon}}^{\epsilon}\left[\,\boldsymbol{x}_{\epsilon}(\varrho_{\epsilon}+s-\theta_{\epsilon}^{(p-1)}s_{0})\in\mathcal{E}(\mathcal{M})\,\right]\ =\ 1\ .
\]
Putting together the previous estimates yields that
\[
\limsup_{\epsilon\to0}\,\sup_{\boldsymbol{x}\in\mathcal{E}(\mathcal{M})}\,\sup_{s\in[0,\,\gamma_{\epsilon}]}\,\left|\,\mathbb{P}_{\boldsymbol{x}}^{\epsilon}\left[\boldsymbol{x}_{\epsilon}(\varrho_{\epsilon}+s)\in\mathcal{E}(\boldsymbol{m})\right]-\frac{\nu(\boldsymbol{m})}{\nu(\mathcal{M})}\,\right|\ \le\ \eta\ .
\]
To complete the proof of the lemma, it remains to let $\eta\rightarrow0$.
\end{proof}
The next result is an intermediate scale version of Proposition \ref{p_marginDA}.
\begin{prop}
\label{p_marginDA-int} Fix $p\in\llbracket2,\,\mathfrak{q}\rrbracket$,
$\mathcal{M}\in\mathscr{S}^{(p)}$, $\mathcal{M}'\in\mathscr{V}^{(p)}$,
and $\boldsymbol{m}\in\mathcal{M}$, $\boldsymbol{m}'\in\mathcal{M}'$.
Then, for all $\boldsymbol{x}\in\mathcal{D}(\boldsymbol{m})$ and
sequences $(\varrho_{\epsilon})_{\epsilon>0}$ satisfying $\theta_{\epsilon}^{(p-1)}\prec\varrho_{\epsilon}\prec\theta_{\epsilon}^{(p)}$,
\[
\lim_{\epsilon\rightarrow0}\,\mathbb{P}_{\boldsymbol{x}}^{\epsilon}\left[\,\boldsymbol{x}_{\epsilon}(\varrho_{\epsilon})\in\mathcal{E}(\boldsymbol{m}')\,\right]\ =\ \widehat{\mathcal{Q}}_{\mathcal{M}}^{(p)}\left[\,\tau_{\mathscr{V}^{(p)}}=\tau_{\mathcal{M}'}\,\right]\,\frac{\nu(\boldsymbol{m}')}{\nu(\mathcal{M}')}\;.
\]
In particular, we have
\[
\lim_{\epsilon\rightarrow0}\,\mathbb{P}_{\boldsymbol{x}}^{\epsilon}\left[\,\boldsymbol{x}_{\epsilon}(\varrho_{\epsilon})\notin\mathcal{E}^{(p)}\,\right]\ =\ 0\ .
\]
\end{prop}

\begin{proof}
The first assertion is immediate from the first assertion of Proposition
\ref{prop:fd2_int} and Lemma \ref{lem:margin_int}. The second assertion
is direct from the first one.
\end{proof}
By \cite[ Corollary 6.2 ]{LLS-1st}, we can extend the previous estimate
in the case of the largest time scale.
\begin{prop}
\label{p_marginDA-int-1}For all $\boldsymbol{x}\in\mathbb{R}^{d}$,
$\boldsymbol{m}\in\mathcal{M}_{\star}$, and a sequence $(\varrho_{\epsilon})_{\epsilon>0}$
such that $\varrho_{\epsilon}\succ\theta_{\epsilon}^{(\mathfrak{q})}$,
\[
\lim_{\epsilon\rightarrow0}\,\mathbb{P}_{\boldsymbol{x}}^{\epsilon}\left[\,\boldsymbol{x}_{\epsilon}(\varrho_{\epsilon})\in\mathcal{E}(\boldsymbol{m})\,\right]\ =\ \frac{\nu(\boldsymbol{m})}{\nu(\mathcal{M}_{\star})}\;.
\]
\end{prop}

\begin{proof}
We fix $\boldsymbol{x}\in\mathbb{R}^{d}$ and $\boldsymbol{m}\in\mathcal{M}^{\star}$.
By \cite[ Corollary 6.2 ]{LLS-1st},
\[
\lim_{\epsilon\rightarrow0}\,\mathbb{P}_{\boldsymbol{x}}^{\epsilon}\left[\,\boldsymbol{x}_{\epsilon}(\varrho_{\epsilon})\in\mathcal{E}(\boldsymbol{m})\,\right]\ =\ \lim_{\epsilon\rightarrow0}\mathbb{P}_{\boldsymbol{x}}^{\epsilon}\left[\,\boldsymbol{x}_{\epsilon}(\varrho_{\epsilon})\in\mathcal{E}(\boldsymbol{m})\,\Big|\,\tau_{\mathcal{E}(\mathcal{M}_{0})}\le\epsilon^{-1}\,\right]\;.
\]
By the strong Markov property,
\begin{align*}
 & \limsup_{\epsilon\rightarrow0}\,\left|\,\mathbb{P}_{\boldsymbol{x}}^{\epsilon}\left[\,\boldsymbol{x}_{\epsilon}(\varrho_{\epsilon})\in\mathcal{E}(\boldsymbol{m}')\,\Big|\,\tau_{\mathcal{E}(\mathcal{M}_{0})}\le\epsilon^{-1}\,\right]-\frac{\nu(\boldsymbol{m})}{\nu(\mathcal{M}_{\star})}\,\right|\\
 & \ \ \ \le\,\limsup_{\epsilon\rightarrow0}\,\sup_{\boldsymbol{y}\in\mathcal{E}(\mathcal{M}_{0})}\,\sup_{|s|\le\epsilon^{-1}}\,\left|\,\mathbb{P}_{\boldsymbol{y}}^{\epsilon}\left[\,\boldsymbol{x}_{\epsilon}(\varrho_{\epsilon}+s)\in\mathcal{E}(\boldsymbol{m}')\,\right]-\frac{\nu(\boldsymbol{m})}{\nu(\mathcal{M}_{\star})}\,\right|\;.
\end{align*}
The right-hand side of the previous inequality converges to $0$ by
Lemma \ref{lem:margin_int}.
\end{proof}

\subsection{Proof of main propositions }
\begin{proof}[Proof of Proposition \ref{p_fdd01}]
Fix $p\in\llbracket1,\,\mathfrak{q}\rrbracket$, $\mathcal{M}\in\mathscr{S}^{(p)}$,
$\boldsymbol{m}\in\mathcal{M}$, and $\boldsymbol{x}\in\mathcal{D}(\boldsymbol{m})$.

Fix $\eta>0$. Since $F$ is continuous, take the constant $r_{0}=r_{0}(\eta)>0$
appearing in the definition of $\mathcal{E}(\boldsymbol{m})$ small
enough so that
\begin{equation}
\max_{\boldsymbol{m}\in\mathcal{M}_{0}}\,\sup_{\boldsymbol{x}\in\mathcal{E}(\boldsymbol{m})}\,\big|\,F(\boldsymbol{x})-F(\boldsymbol{m})\,\big|\ \le\ \eta\ .\label{eq:fdd01-1}
\end{equation}
Denote by $J$ the right-hand side of the expression appearing in
the statement of the proposition. Since $F$ is bounded, by the second
assertion of Proposition \ref{p_marginDA} and the previous estimate,
\begin{align}
 & \limsup_{\epsilon\rightarrow0}\,\left|\,\mathbb{E}_{\boldsymbol{x}}^{\epsilon}\left[\,F(\boldsymbol{x}_{\epsilon}(\theta_{\epsilon}^{(p)}t))\right]-J\,\right|\nonumber \\
 & \quad=\,\limsup_{\epsilon\rightarrow0}\,\bigg|\,\sum_{\mathcal{M}'\in\mathscr{V}^{(p)}}\,\sum_{\boldsymbol{m}'\in\mathcal{M}'}\,\mathbb{E}_{\boldsymbol{x}}^{\epsilon}\left[\,F(\boldsymbol{x}_{\epsilon}(\theta_{\epsilon}^{(p)}t))\,\mathbf{1}\{\boldsymbol{x}_{\epsilon}(\theta_{\epsilon}^{(p)}t)\in\mathcal{E}(\boldsymbol{m}')\}\,\right]-J\,\bigg|\nonumber \\
 & \quad\le\,\eta+\limsup_{\epsilon\rightarrow0}\,\Big|\,\sum_{\mathcal{M}'\in\mathscr{V}^{(p)}}\,\sum_{\boldsymbol{m}'\in\mathcal{M}'}\,F(\boldsymbol{m}')\,\mathbb{P}_{\boldsymbol{x}}^{\epsilon}\left[\,\boldsymbol{x}_{\epsilon}(\theta_{\epsilon}^{(p)}t)\in\mathcal{E}(\boldsymbol{m}')\,\right]-J\,\Big|\;.\label{eq:fdd01-2}
\end{align}
By Proposition \ref{p_marginDA} and the definition of $J$ the second
term vanishes. To complete the proof of the proposition, it remains
to let $\eta\rightarrow0$.
\end{proof}
\begin{proof}[Proof of Proposition \ref{p_fdd02}]
Note that the case $p=1$ is already handled in \cite[Theorem 2.1]{LLS-1st}.
Thus, let us suppose that $p\in\llbracket2,\,\mathfrak{q}+1\rrbracket$.

Denote the right-hand side of \eqref{e:fdd02} by $J$. Fix $\eta>0$
and take $r_{0}=r_{0}(\eta)>0$ as in \eqref{eq:fdd01-1}. The same
computations with presented in \eqref{eq:fdd01-2} yield that
\begin{align*}
 & \limsup_{\epsilon\rightarrow0}\,\big|\,\mathbb{E}_{\boldsymbol{x}}^{\epsilon}\left[\,F(\boldsymbol{x}_{\epsilon}(\varrho_{\epsilon}))\,\right]-J\,\big|\\
 & \le\eta+\limsup_{\epsilon\rightarrow0}\,\bigg|\,\sum_{\mathcal{M}'\in\mathscr{V}^{(p)}}\,\sum_{\boldsymbol{m}'\in\mathcal{M}'}\,F(\boldsymbol{m}')\,\mathbb{P}_{\boldsymbol{x}}^{\epsilon}\left[\,\boldsymbol{x}_{\epsilon}(\varrho_{\epsilon})\in\mathcal{E}(\boldsymbol{m}')\,\right]-J\,\bigg|\;.
\end{align*}
It remains to recall the assertion of Proposition \ref{p_marginDA-int}
and then let $\eta\rightarrow0$ to complete the proof.
\end{proof}
\begin{proof}[Proof of Proposition \ref{p_fdd03}]
The proof is identical to the one of Proposition \ref{p_fdd02}.
It suffices to apply Proposition \ref{p_marginDA-int-1} instead of
Proposition \ref{p_marginDA}.
\end{proof}

\section{Local Reversibility of ${\bf y}^{(p)}$\label{sec13}}

In this section, we fix $p\in\llbracket1,\,\mathfrak{q}\rrbracket$
and the recurrent class $\mathscr{R}$ of the chain $\mathbf{y}^{(p)}(\cdot)$,
and prove Proposition \ref{p_rev_y} for this $p$ and $\mathscr{R}$.

We note first that there is nothing to prove if $\left|\mathscr{R}\right|=1$,
i.e., when $\mathscr{R}$ consists of a single element which is an
absorbing state of the process $\mathbf{y}^{(p)}(\cdot)$. Hence,
we shall assume in this section that $\left|\mathscr{R}\right|\ge2$.
We note that, in this case, $\cup_{\mathcal{M}\in\mathscr{R}}\mathcal{M}$
is an element of $\mathscr{V}^{(p+1)}$ and $\mathscr{R}$ is of the
form $\mathscr{V}^{(p)}(\mathcal{M}')$ for some $\mathcal{M}'\in\mathscr{V}^{(p+1)}$.

We start by extending $\mathscr{R}$ to a recurrent class of the process
$\widehat{\mathbf{y}}^{(p)}(\cdot)$.
\begin{lem}
\label{lem:extR} There exists a recurrent class $\widehat{\mathscr{R}}$
of $\widehat{{\bf y}}^{(p)}(\cdot)$ containing $\mathscr{R}$. Furthermore,
$\mathscr{R}=\widehat{\mathscr{R}}\cap\mathscr{V}^{(p)}$.
\end{lem}

\begin{proof}
Since $\mathbf{y}^{(p)}(\cdot)$ is the trace of $\widehat{\mathbf{y}}^{(p)}(\cdot)$
on $\mathscr{V}^{(p)}$ and $\mathscr{R}$ is a recurrent class of
$\mathbf{y}^{(p)}(\cdot)$, for all $\mathcal{M},\,\mathcal{M}'\in\mathscr{R}$,
\begin{equation}
\mathcal{\widehat{Q}}_{\mathcal{M}}^{(p)}\left[\,\tau_{\mathcal{M}'}<\infty\,\right]\ =\ \mathcal{Q}_{\mathcal{M}}^{(p)}\left[\,\tau_{\mathcal{M}'}<\infty\,\right]\ =\ 1\;.\label{eq:traceQQ}
\end{equation}
Therefore, $\mathscr{R}$ is contained in a recurrent class $\widehat{\mathscr{R}}$
of the process $\widehat{\mathbf{y}}^{(p)}(\cdot)$.

It remains to prove that $\widehat{\mathscr{R}}\cap\mathscr{V}^{(p)}\subset\mathscr{R}$.
Fix $\mathcal{M}''\in\widehat{\mathscr{R}}\cap\mathscr{V}^{(p)}$.
Then, as above, for each $\mathcal{M}\in\mathscr{R}$, \eqref{eq:traceQQ},
\[
\mathcal{Q}_{\mathcal{M}}^{(p)}\left[\,\tau_{\mathcal{M}''}<\infty\,\right]\ =\ \mathcal{\widehat{Q}}_{\mathcal{M}}^{(p)}\left[\,\tau_{\mathcal{M}''}<\infty\,\right]\ =\ 1\;,\quad\mathcal{Q}_{\mathcal{M}''}^{(p)}\left[\,\tau_{\mathcal{M}}<\infty\,\right]\ =\ \mathcal{\widehat{Q}}_{\mathcal{M}''}^{(p)}\left[\,\tau_{\mathcal{M}}<\infty\,\right]\ =\ 1\;.
\]
Therefore $\mathcal{M}$ and $\mathcal{M}''$ belong to the same recurrent
class of $\mathbf{y}^{(p)}(\cdot)$, which proves the lemma.
\end{proof}

\subsection{Further analysis on landscape}

Denote by $\widehat{\mathscr{R}}$ the recurrent class of $\widehat{{\bf y}}^{(p)}(\cdot)$
containing $\mathscr{R}$ obtained in the previous lemma. Since $\cup_{\mathcal{M}\in\mathscr{R}}\mathcal{M}$
is an element of $\mathscr{V}^{(p+1)}$, by $\mathfrak{P}_{1}(p+1)$,
we have $U(\mathcal{M})=U(\mathcal{M}')$ for all $\mathcal{M}$,
$\mathcal{M}'\in\mathscr{R}$. Let
\begin{equation}
{\color{blue}H\,=\,H(\mathscr{R})}\,=\,U(\mathcal{M})+d^{(p)}\,=\,\Theta(\mathcal{M},\,\widetilde{\mathcal{M}})\ ;\ \mathcal{M}\in\mathscr{R}\;.\label{eq:def_HR}
\end{equation}

\begin{lem}
\label{lem:rev2}The following hold.
\begin{enumerate}
\item $\Xi(\mathcal{M})\le d^{(p)}$ for all $\mathcal{M}\in\widehat{\mathscr{R}}$.
Moreover, $\Xi(\mathcal{M})=d^{(p)}$ if and only if $\mathcal{M}\in\mathscr{R}$.
\item $\Theta(\mathcal{M},\,\widetilde{\mathcal{M}})\le H$ for all $\mathcal{M}\in\widehat{\mathscr{R}}$.
In particular, $U(\mathcal{M})<H$.
\item $\Theta(\mathcal{M},\,\mathcal{M}')\le H$ for all $\mathcal{M},\,\mathcal{M}'\in\widehat{\mathscr{R}}$.
In particular, the set $\widehat{\mathscr{R}}$ is contained in a
connected component of the level set $\{U\le H\}$.
\end{enumerate}
\end{lem}

\begin{proof}
\noindent To prove the first claim, fix $\mathcal{M}\in\widehat{\mathscr{R}}$.
Since $\widehat{\mathscr{R}}$ is a recurrent class of the process
$\widehat{\mathbf{y}}^{(p)}(\cdot)$, whose jump rate is denoted by
$\widehat{r}^{(p)}(\cdot,\,\cdot)$, there exists $\mathcal{M}'\in\widehat{\mathscr{R}}$
such that $\widehat{r}^{(p)}(\mathcal{M},\,\mathcal{M}')>0$. Hence,
by $\mathfrak{P}_{3}(p)$, $\Xi(\mathcal{M})\le d^{(p)}$. The second
assertion of the claim follows from the first assertion and Proposition
\ref{prop:depth}.

\smallskip{}
\textit{Claim A:} For $\mathcal{M}\in\widehat{\mathscr{R}}$, $\mathcal{M}''\in\mathscr{R}$,
\begin{equation}
\Theta(\mathcal{M},\,\widetilde{\mathcal{M}})\ \le\ H\;\;\;\text{and\;\;\;}\Theta(\mathcal{M},\,\mathcal{M}'')\ \le\ H\;.\label{eq:cla1}
\end{equation}

\noindent To prove Claim A, fix $\mathcal{M}\in\widehat{\mathscr{R}}$,
$\mathcal{M}''\in\mathscr{R}$. By the first assertion of this lemma
and \eqref{eq:def_HR}, $\Theta(\mathcal{M}'',\,\widetilde{\mathcal{M}''})=H$.
By $\mathfrak{P}_{3}(p)$ and the fact $\widehat{\mathscr{R}}$ is
a recurrent class, there exist $\mathcal{M}_{1},\,\dots,\,\mathcal{M}_{n}\in\widehat{\mathscr{R}}$
such that $\mathcal{M}''\to\mathcal{M}_{1}\to\cdots\to\mathcal{M}_{n}\to\mathcal{M}$.
By the first assertion of the lemma, $\Xi(\mathcal{M}'')=d^{(p)}\ge\Xi(\mathcal{M}_{1}),\,\dots,\,\Xi(\mathcal{M}_{n}),\,\Xi(\mathcal{M})$.
Therefore, by Lemma \ref{l_M->M},
\[
\Theta(\mathcal{M}'',\,\widetilde{\mathcal{M}''})\ \ge\ \Theta(\mathcal{M},\,\widetilde{\mathcal{M}})\;.
\]
This proves the first assertion of \eqref{eq:cla1} since $\Theta(\mathcal{M}'',\,\widetilde{\mathcal{M}''})=H$.

We turn to the second one. By \eqref{e:bdtheta},
\[
\Theta(\mathcal{M}'',\,\mathcal{M})\ \le\ \max\,\left\{ \,\Theta(\mathcal{M}'',\,\mathcal{M}_{1}),\,\Theta(\mathcal{M}_{1},\,\mathcal{M}_{2}),\,\dots,\,\Theta(\mathcal{M}_{n},\,\mathcal{M})\,\right\} \;.
\]
By \eqref{elmm}, this expression is equal to
\[
\max\,\left\{ \,\Theta(\mathcal{M}'',\,\widetilde{\mathcal{M}''}),\,\Theta(\mathcal{M}_{1},\,\widetilde{\mathcal{M}_{1}}),\,\dots,\,\Theta(\mathcal{M}_{n},\,\widetilde{\mathcal{M}_{n}})\,\right\} \;.
\]
By the first part of the claim and since $\Theta(\mathcal{M}'',\,\widetilde{\mathcal{M}''})=H$,
this expression is equal to $H$ which completes the proof of the
second assertion of Claim A.

Claim (2) of the lemma follows from Claim A. To complete the proof
of Claim (3), observe that for $\mathcal{M},\,\mathcal{M}'\in\widehat{\mathscr{R}}$,
\[
\Theta(\mathcal{M},\,\mathcal{M}')\ \le\ \max\,\Big\{\,\Theta(\mathcal{M},\,\mathcal{M}''),\,\Theta(\mathcal{M}'',\,\mathcal{M}')\,\Big\}\ .
\]
By Claim A, this expression is less than or equal to $H$.
\end{proof}
We denote by\textcolor{blue}{{} $\mathcal{H}=\mathcal{H}(\mathscr{\mathscr{R}})$}
the connected component of $\{U\le H\}$ obtained in Lemma \ref{lem:rev2}-(3),
and let $\{\mathcal{W}_{1},\,\dots,\,\mathcal{W}_{m}\}$ the level
set decomposition of $\mathcal{H}$
\begin{lem}
\label{lem:EL} Assume that a connected component $\mathcal{W}_{k}$,
$k\in\llbracket1,\,m\rrbracket$, contains an element $\mathcal{M}\in\mathscr{R}$.
Then, $\mathcal{M}=\mathcal{M}^{*}(\mathcal{W}_{k})$. In particular,
each set $\mathcal{W}_{k}$ contains at most one element of $\mathscr{R}$.
\end{lem}

\begin{proof}
Suppose that $\mathcal{W}_{k}$ contains a set $\mathcal{M}\in\mathscr{R}$,
and that there exists a minimum of $U$ in $\mathcal{W}_{k}\setminus\mathcal{M}$,
denoted by $\boldsymbol{m}\in\mathcal{M}_{0}$, such that $U(\boldsymbol{m})\le U(\mathcal{M})$.
Since $\mathcal{M}\in\mathscr{R}$ and $\boldsymbol{m}\in\widetilde{\mathcal{M}}$,
by Lemma \ref{lem:rev2} and Lemma \ref{lap01}-(1),
\[
U(\mathcal{M})+d^{(p)}\ =\ U(\mathcal{M})+\Xi(\mathcal{M})\ =\ \Theta(\mathcal{M},\widetilde{\mathcal{M}})\ \le\ \Theta(\mathcal{M},\boldsymbol{m})\ <\ H\;,
\]
in contradiction with \eqref{eq:def_HR}.
\end{proof}
\begin{prop}
\label{prop:EL} Let $\mathcal{M}_{k}=\mathcal{M}^{*}(\mathcal{W}_{k})$,
$k\in\llbracket1,\,m\rrbracket$. Then,
\begin{enumerate}
\item For each $\mathcal{M}\in\mathscr{R}$, there exists $j\in\llbracket1,\,m\rrbracket$
such that $\mathcal{M}=\mathcal{M}_{j}$.
\item $U(\mathcal{M}_{k})\ge H-d^{(p)}$ for all $k\in\llbracket1,\,m\rrbracket$,
\item $\mathcal{M}_{k}\in\mathscr{S}^{(p)}\cap\widehat{\mathscr{R}}$ for
all $k\in\llbracket1,\,m\rrbracket$.
\item $\Theta(\mathcal{M}_{k},\,\widetilde{\mathcal{M}_{k}})=H$ for all
$k\in\llbracket1,\,m\rrbracket$.
\end{enumerate}
\end{prop}

\begin{proof}
(1) Let $\mathcal{M}$ be an element of $\mathscr{R}$. By Lemmata
\ref{lem:rev2}-(3) and Lemma \ref{2la1-0}, $\mathcal{M}\subset\bigcup_{i=1}^{k}\mathcal{W}_{k}$.
Since each $\mathcal{W}_{k}$ does not separate $(p)$-states by Lemma
\ref{lem_wells1}, $\mathcal{M}\subset\mathcal{W}_{j}$ for some $j\in\llbracket1,\,m\rrbracket$.
Now, by Lemma \ref{lem:EL}, $\mathcal{M}=\mathcal{M}_{j}$, as claimed.
\smallskip{}

\noindent (2) Suppose that $U(\mathcal{M}_{k})<H-d^{(p)}$ for some
$k\in\llbracket1,\,m\rrbracket$. Fix $j\in\llbracket1,\,m\rrbracket$
such that $\mathcal{M}_{j}\in\mathscr{R}$. This set exists in view
of the first assertion of the proposition. By Lemma \ref{2-la1}-(3),
there exist $r(1),\,\dots,\,r(b)$ such that
\[
\overline{\mathcal{W}_{j}}\cap\overline{\mathcal{W}_{r(1)}}\,\ne\,\varnothing\ ,\ \overline{\mathcal{W}_{r(1)}}\cap\overline{\mathcal{W}_{r(2)}}\,\ne\,\varnothing\ ,\ \cdots,\ \overline{\mathcal{W}_{r(b-1)}}\cap\overline{\mathcal{W}_{r(b)}}\,\ne\,\varnothing\ ,\ \overline{\mathcal{W}_{r(b)}}\cap\overline{\mathcal{W}_{k}}\,\ne\,\varnothing\ .
\]
Without loss of generality, suppose that
\[
\min_{\boldsymbol{x}\in\mathcal{W}_{r(i)}}U(\boldsymbol{x})\ \ge\ H-d^{(p)}\;\text{for all }i\in\llbracket1,\,b\rrbracket\;.
\]
Let $1\le i_{1}<\dots<i_{a}\le b$ such that $\min_{\boldsymbol{x}\in\mathcal{W}_{r(i_{n})}}U(\boldsymbol{x})=H-d^{(p)}$.
By Lemma \ref{lem_wells2},
\[
\mathcal{M}_{j}\,\Rightarrow^{(p)}\,\mathcal{M}_{r(i_{1})}\,\Rightarrow^{(p)}\,\cdots\,\Rightarrow^{(p)}\,\mathcal{M}_{r(i_{a})}\,\Rightarrow^{(p)}\,\mathcal{M}'
\]
for some $\mathcal{M}'\in\mathscr{V}^{(p)}(\mathcal{W}_{k})$ so that
the process $\widehat{\mathbf{y}}^{(p)}(\cdot)$ starting from $\mathcal{M}_{j}$
reaches $\mathcal{M}'$ with positive probability. Hence, $\mathcal{M}'\in\widehat{\mathscr{R}}$.
Since $\mathcal{M}'\in\mathscr{V}^{(p)}(\mathcal{W}_{k})$, $\mathcal{M}'\in\mathscr{R}$.
By Lemma \ref{lem:EL}, $\mathcal{M}'=\mathcal{M}_{k}$. By Lemma
\ref{lem:rev2} and the definition of $H$, $U(\mathcal{M}_{k})=H-d^{(p)}$
in contradiction with the initial assumption that $U(\mathcal{M}_{k})<H-d^{(p)}$.
\smallskip{}

\noindent (3) In view of Lemma \ref{lem_wells1}-(1), (2), $\mathcal{M}_{k}\in\mathscr{S}^{(p)}$.
Also, for $\mathcal{M}\in\mathscr{R}$, since $U(\mathcal{M})=H-d^{(p)}$,
by the proof of Lemma \ref{lem_wells2}, $\mathcal{M}\Rightarrow^{(p)}\mathcal{M}_{k}$
so that $\mathcal{M}_{k}\in\widehat{\mathscr{R}}$.\smallskip{}

\noindent (4) By Lemma \ref{lem:rev2},
\[
\Theta(\mathcal{M}_{k},\,\widetilde{\mathcal{M}_{k}})\ \le\ H\ .
\]
Also, since $\mathcal{M}_{k}=\mathcal{M}^{*}(\mathcal{W}_{k})$, $\widetilde{\mathcal{M}_{k}}\cap\mathcal{W}_{k}=\varnothing$
so that by Lemma \ref{lap01}-(2),
\[
\Theta(\mathcal{M}_{k},\,\widetilde{\mathcal{M}_{k}})\ \ge\ H\ ,
\]
which completes the proof.
\end{proof}

\subsection{Proof of Proposition \ref{p_rev_y}}

Let us fix $\mathscr{R}$ and keep the notations as in the current
subsection.

The proof of Proposition \ref{p_rev_y} requires the following lemma.
\begin{lem}
\label{l_SM-WM-1}For all $i,\,j\in\llbracket1,\,m\rrbracket$,
\begin{equation}
\overline{\mathcal{W}_{i}}\cap\overline{\mathcal{W}_{j}}\ =\ \bigcup_{\mathcal{M}'\in\mathscr{S}^{(p)}(\mathcal{W}_{j})}\mathcal{S}(\mathcal{M}_{i},\,\mathcal{M}')\ ,\label{eq:decWW-1}
\end{equation}
where the right-hand side represents a disjoint union.
\end{lem}

\begin{proof}[Proof of Proposition \ref{p_rev_y}]
Let us regard $\widehat{\mathbf{y}}^{(p)}(\cdot)$ as a Markov chain
on $\mathscr{S}^{(p)}(\mathcal{H})$. In view of Lemma \ref{2la1-0},
since $\mathcal{W}_{i}$, $i\in\llbracket1,\,m\rrbracket$, do not
separate $(p)$-states, we can decompose $\mathscr{S}^{(p)}(\mathcal{H})$
into
\[
\mathscr{S}^{(p)}(\mathcal{H})\ =\ \bigcup_{i=1}^{m}\mathscr{S}^{(p)}(\mathcal{W}_{i})\;.
\]
By Proposition \ref{prop:EL}-(3), each set $\mathscr{S}^{(p)}(\mathcal{W}_{i})$
is non-empty as it contains the set $\mathcal{M}_{i}=\mathcal{M}^{*}(\mathcal{W}_{i})$.
We write $\mathscr{S}_{0}^{(p)}=\{\mathcal{M}_{1},\,\dots,\,\mathcal{M}_{m}\}$.

We now claim that $\mathbf{z}(\cdot):=\widehat{\mathbf{y}}^{(p)}(\cdot)$,
$S:=\mathscr{S}^{(p)}(\mathcal{H})$, $S_{i}:=\mathscr{S}^{(p)}(\mathcal{W}_{i})$
for $i\in\llbracket1,\,m\rrbracket$, $S_{0}:=\mathscr{S}_{0}^{(p)}$,
and $\rho:=\pi_{\mathcal{M}(\mathscr{R})}$ (defined in Proposition
\ref{p_rev_y}) satisfy all the requirements of Proposition \ref{prop:revpre}.
Firstly, \eqref{eq:revm1} is immediate from Lemma \ref{lem_escape}
and $\mathfrak{P}_{3}(p)$. Hence, it suffices to check \eqref{eq:revm2},
i.e.,
\begin{equation}
\nu(\mathcal{M}_{i})\,\widehat{r}^{(p)}(\mathcal{M}_{i},\,\mathscr{S}^{(p)}(\mathcal{W}_{j}))\ =\ \nu(\mathcal{M}_{j})\,\widehat{r}^{(p)}(\mathcal{M}_{j},\,\mathscr{S}^{(p)}(\mathcal{W}_{i}))\;\;\;\text{for all }i,\,j\in\llbracket1,\,m\rrbracket\;.\label{eq:dbal}
\end{equation}
To that end, we further claim that
\begin{equation}
\nu(\mathcal{M}_{i})\,\widehat{r}^{(p)}(\mathcal{M}_{i},\,\mathscr{S}^{(p)}(\mathcal{W}_{j}))\ =\ \sum_{\boldsymbol{\sigma}\in\overline{\mathcal{W}_{i}}\cap\overline{\mathcal{W}_{j}}}\,\omega(\boldsymbol{\sigma})\label{eq:dbal2}
\end{equation}
so that \eqref{eq:dbal} follows immediately from \eqref{eq:dbal2}.
If $\mathcal{M}_{i}\in\mathscr{V}^{(p)}$, by \eqref{eq:rate_3},
\[
\nu(\mathcal{M}_{i})\,\widehat{r}^{(p)}(\mathcal{M}_{i},\,\mathscr{S}^{(p)}(\mathcal{W}_{j}))\ =\ \sum_{\mathcal{M}'\in\mathscr{S}^{(p)}(\mathcal{W}_{j})}\,\sum_{\boldsymbol{\sigma}\in\mathcal{S}(\mathcal{M}_{i},\,\mathcal{M}')}\,\omega(\boldsymbol{\sigma})\;.
\]
Hence, for this case, \eqref{eq:dbal2} follows from Lemma \ref{l_SM-WM-1}.
On the other hand, if $\mathcal{M}_{i}\in\mathscr{N}^{(p)}$, let
us take $p_{0}<p$ such that $\mathcal{M}_{i}\in\mathscr{V}^{(p_{0})}$.
Since we can deduce from \eqref{eq:rate_2} and \eqref{eq:rate_1}
that
\[
\widehat{r}^{(p)}(\mathcal{M}_{i},\,\mathscr{S}^{(p)}(\mathcal{W}_{j}))\ =\ \widehat{r}^{(p-1)}(\mathcal{M}_{i},\,\mathscr{S}^{(p-1)}(\mathcal{W}_{j}))\;,
\]
we repeat this procedure to get
\begin{equation}
\widehat{r}^{(p)}(\mathcal{M}_{i},\,\mathscr{S}^{(p)}(\mathcal{W}_{j}))\ =\ \widehat{r}^{(p_{0})}(\mathcal{M}_{i},\,\mathscr{S}^{(p_{0})}(\mathcal{W}_{j}))\;.\label{eq:dbal3}
\end{equation}
Since $\mathcal{M}_{i}\in\mathscr{V}^{(p_{0})}$, now by \eqref{eq:rate_3},

\begin{equation}
\nu(\mathcal{M}_{i})\,\widehat{r}^{(p_{0})}(\mathcal{M}_{i},\,\mathscr{S}^{(p_{0})}(\mathcal{W}_{j}))\ =\ \sum_{\mathcal{M}'\in\mathscr{S}^{(p_{0})}(\mathcal{W}_{j})}\,\sum_{\boldsymbol{\sigma}\in\mathcal{S}(\mathcal{M}_{i},\,\mathcal{M}')}\,\omega(\boldsymbol{\sigma})\;.\label{eq:dbal4}
\end{equation}
Combining \eqref{eq:dbal3}, \eqref{eq:dbal4}, and Lemma \ref{l_SM-WM-1}
completes the proof of \eqref{eq:dbal2} for this case as well.

Therefore, the trace $\widetilde{\mathbf{y}}^{(p)}(\cdot)$ of $\widehat{\mathbf{y}}^{(p)}(\cdot)$
on $\mathscr{S}_{0}^{(p)}$ is reversible by Proposition of \ref{prop:revpre}.
Since $\mathscr{R}\subset\mathscr{S}_{0}^{(p)}$ by Proposition \ref{prop:EL}-(1),
the process $\mathbf{y}^{(p)}(\cdot)$ (which was defined as the trace
of $\widehat{\mathbf{y}}^{(p)}(\cdot)$ on $\mathscr{R}$) is the
trace of a reversible chain $\widetilde{\mathbf{y}}^{(p)}(\cdot)$
on $\mathscr{R}$ and hence is reversible as well.
\end{proof}

\subsection{Proof of Lemma \ref{l_SM-WM-1}}

Finally, we prove Lemma \ref{l_SM-WM-1} in this subsection. We need
the following lemma.
\begin{lem}
\label{l_SM-WM}Let $\mathcal{M}\in\mathscr{V}_{{\rm nab}}^{(p)}$
and let $\mathcal{V}$ be a connected component of $\{U<\Theta(\mathcal{M},\,\widetilde{\mathcal{M}})\}$
such that
\[
\mathcal{V}\ \ne\ \mathcal{W}(\mathcal{M})\ \ ,\ \ \overline{\mathcal{W}(\mathcal{M})}\cap\overline{\mathcal{V}}\ \ne\ \varnothing\ .
\]
Then,
\begin{equation}
\overline{\mathcal{W}(\mathcal{M})}\cap\overline{\mathcal{V}}\ =\ \bigcup_{\mathcal{M}'\in\mathscr{S}^{(p)}(\mathcal{V})}\mathcal{S}(\mathcal{M},\,\mathcal{M}')\ ,\label{eq:decWW}
\end{equation}
where the right-hand side represents a disjoint union.
\end{lem}

\begin{proof}
Since the disjointness of the union at the right-hand side of \eqref{eq:decWW}
is immediate as a saddle cannot be connected with three different
minima, we focus only on the proof of \eqref{eq:decWW}.

Fix $\mathcal{M}\in\mathscr{V}_{{\rm nab}}^{(p)}$ and let $\mathcal{V}$
be a connected component of $\{U<\Theta(\mathcal{M},\,\widetilde{\mathcal{M}})\}$
such that
\[
\mathcal{V}\ \ne\ \mathcal{W}(\mathcal{M})\ \ ,\ \ \overline{\mathcal{W}(\mathcal{M})}\cap\overline{\mathcal{V}}\ \ne\ \varnothing\ .
\]
Then, $\mathcal{V}$ is an element of level set decomposition of $\{U\le\Theta(\mathcal{M},\,\widetilde{\mathcal{M}})\}$
containing $\mathcal{M}$. Since $U(\mathcal{M})=\Theta(\mathcal{M},\,\widetilde{\mathcal{M}})-d^{(p)}$,
by Lemma \ref{lem_wells1}, $\mathcal{V}$ does not separate $(p)$-states.

Let $\boldsymbol{\sigma}\in\mathcal{S}(\mathcal{M},\,\mathcal{M}')$
for some $\mathcal{M}'\in\mathscr{S}^{(p)}(\mathcal{V})$ so that
\begin{equation}
\mathcal{M}\leftsquigarrow\boldsymbol{\sigma}\curvearrowright\mathcal{M}'\;\;\;\text{and}\;\;\;U(\boldsymbol{\sigma})\ =\ \Theta(\mathcal{M},\,\mathcal{M}')\ =\ \Theta(\mathcal{M},\,\widetilde{\mathcal{M}})\ .\label{eMMsigma}
\end{equation}
Let $\boldsymbol{m}_{1}\in\mathcal{M}$ be such that $\boldsymbol{\sigma}\rightsquigarrow\boldsymbol{m}_{1}$.
Since $\mathcal{W}(\mathcal{M})$ is a connected component of $\{U<U(\boldsymbol{\sigma})\}$
containing $\boldsymbol{m}_{1}$ and $U(\boldsymbol{\sigma})=\Theta(\mathcal{M},\,\widetilde{\mathcal{M}})$,
by Lemma \ref{l_squig_saddle}, $\boldsymbol{\sigma}\in\partial\mathcal{W}(\mathcal{M})$.
Let $\boldsymbol{m}_{2}\in\mathcal{M}'$ such that $\boldsymbol{\sigma}\curvearrowright\boldsymbol{m}_{2}$
(cf. \eqref{eMMsigma}). Since $\mathcal{V}$ is the connected component
of $\{U<U(\boldsymbol{\sigma})\}$ (cf. \eqref{eMMsigma}) containing
$\boldsymbol{m}_{2}$, by Lemma \ref{l_path_saddle}-(3), $\boldsymbol{\sigma}\in\partial\mathcal{V}$.
Summing up, we get
\[
\boldsymbol{\sigma}\in\partial\mathcal{W}(\mathcal{M})\cap\partial\mathcal{V}\ ,
\]
which implies
\begin{equation}
\bigcup_{\mathcal{M}'\in\mathscr{S}^{(p)}(\mathcal{V})}\mathcal{S}(\mathcal{M},\,\mathcal{M}')\ \subset\ \partial\mathcal{W}(\mathcal{M})\cap\partial\mathcal{V}\ .\label{eqMM2}
\end{equation}

To prove the reversed inclusion, let us take $\boldsymbol{\sigma}\in\partial\mathcal{W}(\mathcal{M})\cap\partial\mathcal{V}$
so that by Lemma \ref{l_assu_saddle}, we can find $\bm{m}_{3}\in\mathcal{V}$
such that $\boldsymbol{\sigma}\curvearrowright\bm{m}_{3}$, and that
$\boldsymbol{\sigma}\rightsquigarrow\mathcal{M}$. Finally, let $\bm{m}_{3}\in\mathcal{V}$
belong to $\mathcal{M}'\in\mathscr{S}^{(p)}(\mathcal{V})$ so that
\begin{equation}
\mathcal{M}\leftsquigarrow\boldsymbol{\sigma}\curvearrowright\mathcal{M}'\ .\label{eqMM1}
\end{equation}
Therefore, by Lemma \ref{lem_not}-(3),
\[
\Theta(\mathcal{M},\,\mathcal{M}')\ \le\ U(\boldsymbol{\sigma})\ .
\]
On the other hand, since $\mathcal{M}'\cap\mathcal{W}(\mathcal{M})=\varnothing$,
\[
\Theta(\mathcal{M},\,\mathcal{M}')\ \ge\ \Theta(\mathcal{M},\,\widetilde{\mathcal{M}})
\]
by Lemma \ref{l_level_boundary}. By Lemma \ref{l_level_boundary},
we have $U(\bm{\sigma})=\Theta(\mathcal{M},\,\widetilde{\mathcal{M}})$
so that we get
\[
\Theta(\mathcal{M},\,\mathcal{M}')\ =\ \Theta(\mathcal{M},\,\widetilde{\mathcal{M}})\ =\ U(\bm{\sigma})\ .
\]
Combining this with \eqref{eqMM1}, we get $\mathcal{M}\to_{\boldsymbol{\sigma}}\mathcal{M}'$.
This proves that $\boldsymbol{\sigma}\in\mathcal{S}(\mathcal{M},\,\mathcal{M}')$
for some $\mathcal{M}'\in\mathscr{S}^{(p)}(\mathcal{V})$, and therefore
the reversed inclusion of \eqref{eqMM2} has been established. Finally,
by Lemma \ref{l_cap_saddle}, we have $\partial\mathcal{W}(\mathcal{M})\cap\partial\mathcal{V}=\overline{\mathcal{W}(\mathcal{M})}\cap\overline{\mathcal{V}}$
which completes the proof.
\end{proof}
Now, we are ready to prove Lemma \ref{l_SM-WM-1}.
\begin{proof}[Proof of Lemma \ref{l_SM-WM-1}]
Fix $k\in\llbracket1,\,m\rrbracket$. By Proposition \ref{prop:EL}-(2),
$\Xi(\mathcal{M}_{k})=\Theta(\mathcal{M}_{k},\,\widetilde{\mathcal{M}_{k}})-U(\mathcal{M}_{k})\le d^{(p)}$
, so that by Proposition \ref{prop:depth}, $\mathcal{M}_{k}\in\mathscr{V}_{{\rm nab}}^{(p)}\cup\mathscr{N}^{(p)}$.
Also, since $\mathcal{W}_{k}$ is a connected component of level set
$\{U<H\}=\{U<\Theta(\mathcal{M}_{k},\,\widetilde{\mathcal{M}_{k}})\}$
by Proposition \ref{prop:EL}-(4) containing $\mathcal{M}_{k}$, $\mathcal{W}_{k}=\mathcal{W}(\mathcal{M}_{k})$.
Therefore, for all $j\in\llbracket1,\,m\rrbracket$, we have
\[
\overline{\mathcal{W}(\mathcal{M}_{k})}\cap\overline{\mathcal{W}_{j}}\ =\ \overline{\mathcal{W}_{k}}\cap\overline{\mathcal{W}_{j}}\ .
\]
First, suppose that $\mathcal{M}_{k}\in\mathscr{V}_{{\rm nab}}^{(p)}$.
By Lemma \ref{l_SM-WM},
\[
\bigcup_{\mathcal{M}'\in\mathscr{S}^{(p)}(\mathcal{W}_{j})}\mathcal{S}(\mathcal{M}_{k},\,\mathcal{M}')\ =\ \overline{\mathcal{W}(\mathcal{M}_{k})}\cap\overline{\mathcal{W}_{j}}\ .
\]
Next, suppose that $\mathcal{M}_{k}\in\mathscr{N}^{(p)}$. Then, there
is $p_{0}\in\llbracket1,\,p-1\rrbracket$ such that $\mathcal{M}_{k}\in\mathscr{V}_{{\rm nab}}^{(p_{0})}\cap\mathscr{N}^{(p_{0}+1)}$.
Then, by Lemma \ref{l_SM-WM},
\[
\bigcup_{\mathcal{M}'\in\mathscr{S}^{(p_{0})}(\mathcal{W}_{j})}\mathcal{S}(\mathcal{M}_{k},\,\mathcal{M}')\ =\ \overline{\mathcal{W}(\mathcal{M}_{i})}\cap\overline{\mathcal{W}_{j}}\ .
\]
Therefore, it suffices to prove
\[
\bigcup_{\mathcal{M}'\in\mathscr{S}^{(p_{0})}(\mathcal{W}_{j})}\mathcal{S}(\mathcal{M}_{k},\,\mathcal{M}')\ =\ \bigcup_{\mathcal{M}''\in\mathscr{S}^{(p)}(\mathcal{W}_{j})}\mathcal{S}(\mathcal{M}_{k},\,\mathcal{M}'')
\]

Let $\bm{\sigma}\in\mathcal{S}(\mathcal{M}_{k},\,\mathcal{M}')$ for
some $\mathcal{M}'\in\mathscr{S}^{(p_{0})}(\mathcal{W}_{j})$. By
Lemma \ref{lem_not}-(4), $\mathcal{M}_{k}\to_{\bm{\sigma}}\bm{m}'$
for some $\bm{m}'\in\mathcal{M}'$. By the construction, there exists
$\mathcal{M}''\in\mathscr{S}^{(p)}$ such that $\mathcal{M}'\subset\mathcal{M}''$
so that $\bm{m}'\in\mathcal{M}''$ and
\[
\mathcal{M}\leftsquigarrow\boldsymbol{\sigma}\curvearrowright\mathcal{M}''\ .
\]
Since $\mathcal{M}_{k}\to\mathcal{M}'$ and $\mathcal{M}'\subset\mathcal{M}''$,
$\Theta(\mathcal{M}_{k},\,\mathcal{M}'')\le\Theta(\mathcal{M}_{k},\,\mathcal{M}')=\Theta(\mathcal{M}_{k},\,\widetilde{\mathcal{M}_{k}})$.
Since $\mathcal{W}_{j}$ does not separate $(p)$-states, $\mathcal{M}''\in\mathscr{S}^{(p)}(\mathcal{W}_{j})$
so that by Lemma \ref{lap01}-(2), $\Theta(\mathcal{M}_{k},\,\widetilde{\mathcal{M}_{k}})\le\Theta(\mathcal{M}_{k},\,\mathcal{M}'')$.
Therefore, $\Theta(\mathcal{M}_{k},\,\widetilde{\mathcal{M}_{k}})=\Theta(\mathcal{M}_{k},\,\mathcal{M}'')$
and $\mathcal{M}\to_{\bm{\sigma}}\mathcal{M}''$ so that $\bm{\sigma}\in\mathcal{S}(\mathcal{M},\,\mathcal{M}'')$
for some $\mathcal{M}''\in\mathscr{S}^{(p)}(\mathcal{W}_{j})$.

Let $\bm{\sigma}\in\mathcal{S}(\mathcal{M}_{k},\,\mathcal{M}'')$
for some $\mathcal{M}''\in\mathscr{S}^{(p)}(\mathcal{W}_{j})$. Then,
by the same argument of the above paragraph, we have $\mathcal{M}\to_{\bm{\sigma}}\mathcal{M}'$
for some $\mathcal{M}'\in\mathscr{S}^{(p_{0})}(\mathcal{W}_{j})$.
In particular, $\mathcal{M}'\subset\mathcal{M}''$.
\end{proof}

\appendix

\section{\label{app:alter} Alternative tree construction}

In this section, we present an alternative construction
of the tree introduced in Section \ref{sec4}. The construction is
carried out, in the tree language, from the leaves to the root. To
facilitate reading, we present some notation introduced in the article
and we postpone to the a subsection all proofs.

The tree constructed below satisfies the conditions
enumerated in Definition \ref{def:tree}, and (a)-(d) in the next
page.

\subsection{The first time-scale}

Recall the relation $\boldsymbol{x}\curvearrowright\boldsymbol{y}$
introduced below \eqref{eq:x(t)}. Denote by $\mathscr{P}(\boldsymbol{m})$,
$\boldsymbol{m}\in\mathcal{M}_{0}$, the set of passes for $\boldsymbol{m}$:
\begin{gather*}
{\color{blue}\mathscr{P}(\boldsymbol{m})}\;:=\;\big\{\,\boldsymbol{\sigma}\in\mathcal{S}_{0}:\exists\,\boldsymbol{m}'\neq\boldsymbol{m}\;\;\text{such that}\;\;\boldsymbol{\sigma}\curvearrowright\boldsymbol{m}\;,\;\;\boldsymbol{\sigma}\curvearrowright\boldsymbol{m}'\,\big\}\;.
\end{gather*}
By Lemma \ref{l_ms_P}, the set $\mathscr{P}(\boldsymbol{m})$ is
not empty. Denote by $\Pi^{(1)}(\boldsymbol{m})$ the depth (or profundity)
of a local minimum $\boldsymbol{m}$:
\[
{\color{blue}\Pi^{(1)}(\boldsymbol{m})}\,:=\,\min_{\boldsymbol{\sigma}\in\mathscr{P}(\boldsymbol{m})}U(\boldsymbol{\sigma})\,-\,U(\boldsymbol{m})\;,
\]
by $\mathfrak{d}^{(1)}$ the depth of the shallowest minimum, and
by $\vartheta_{\epsilon}^{(1)}$ the corresponding time-scale:
\[
{\color{blue}\mathfrak{d}^{(1)}}\;:=\;\min_{\boldsymbol{m}\in\mathcal{M}_{0}}\Pi^{(1)}(\boldsymbol{m})\;,\quad{\color{blue}\vartheta_{\epsilon}^{(1)}}\;:=\;e^{\mathfrak{d}^{(1)}/\epsilon}\;.
\]

Denote by $\mathscr{G}(\boldsymbol{m})$ the set
of gates (passes of lowest height) of $\boldsymbol{m}\in\mathcal{M}_{0}$:
\begin{gather*}
{\color{blue}\mathscr{G}(\boldsymbol{m})}\;:=\;\big\{\,\boldsymbol{\sigma}\in\mathscr{P}(\boldsymbol{m}):U(\boldsymbol{\sigma})=U(\boldsymbol{m})+\Pi^{(1)}(\boldsymbol{m})\,\big\}\;,
\end{gather*}
and by $\mathscr{V}(\boldsymbol{m})$ the set of local minima $\boldsymbol{m}'\in\mathcal{M}_{0}$,
$\boldsymbol{m}'\neq\boldsymbol{m}$, for which there exist a critical
point $\boldsymbol{\sigma}\in\mathscr{G}(\boldsymbol{m})$ and heteroclinic
orbits from $\boldsymbol{\sigma}$ to $\boldsymbol{m}$ and $\boldsymbol{\sigma}$
to $\boldsymbol{m}'$:
\begin{gather*}
{\color{blue}\mathscr{V}(\boldsymbol{m})}\;:=\;\big\{\,\boldsymbol{m}'\in\mathcal{M}_{0}\setminus\{\boldsymbol{m}\}:\exists\,\boldsymbol{\sigma}\in\mathscr{G}(\boldsymbol{m})\;\;\text{such that}\;\;\boldsymbol{\sigma}\curvearrowright\boldsymbol{m}\;,\;\;\boldsymbol{\sigma}\curvearrowright\boldsymbol{m}'\,\big\}\;.
\end{gather*}
Denote by $\mathcal{W}(\boldsymbol{m},\boldsymbol{m}')$, $\boldsymbol{m}'\neq\boldsymbol{m}$,
the set of saddle points which separate $\boldsymbol{m}$ from $\boldsymbol{m}'$:
\[
{\color{blue}\mathcal{W}(\boldsymbol{m},\boldsymbol{m}')}\;:=\;\big\{\,\boldsymbol{\sigma}\in\mathscr{G}(\boldsymbol{m}):\boldsymbol{\sigma}\curvearrowright\boldsymbol{m}\;,\;\;\boldsymbol{\sigma}\curvearrowright\boldsymbol{m}'\,\big\}\;.
\]
Mind that this set is empty if $\boldsymbol{m}'\not\in\mathscr{V}(\boldsymbol{m})$.

Recall from \eqref{eq:omega} the definition of
the weight $\omega(\boldsymbol{\sigma})$ of a saddle point $\boldsymbol{\sigma}$,
and let $\widehat{\omega}(\boldsymbol{m},\boldsymbol{m}')$, $\boldsymbol{m}\neq\boldsymbol{m}'\in\mathcal{M}_{0}$,
be the one given by
\begin{equation}
{\color{blue}\widehat{\omega}(\boldsymbol{m},\boldsymbol{m}')}\;:=\;\sum_{\boldsymbol{\sigma}\in\mathcal{W}(\boldsymbol{m},\boldsymbol{m}')}\omega(\boldsymbol{\sigma})\;.\label{eq:vartheta}
\end{equation}
Note that neither $\mathcal{W}(\,\cdot\,,\,\cdot\,)$ nor $\widehat{\omega}(\,\cdot\,,\,\cdot\,)$
are symmetric in their arguments. To include the depth of the local
minimum $\boldsymbol{m}$ in the definition of the weight $\widehat{\omega}(\boldsymbol{m},\boldsymbol{m}')$,
set
\begin{equation}
{\color{blue}\widehat{\omega}_{1}(\boldsymbol{m},\boldsymbol{m}')}\;:=\;\widehat{\omega}(\boldsymbol{m},\boldsymbol{m}')\,\boldsymbol{1}\{\Pi^{(1)}(\boldsymbol{m})=\mathfrak{d}^{(1)}\,\}\;,\quad R^{(1)}(\boldsymbol{m},\boldsymbol{m}')\,:=\,\frac{1}{\nu(\boldsymbol{m})}\,\widehat{\omega}_{1}(\boldsymbol{m},\boldsymbol{m}')\;,\label{40a}
\end{equation}
where $\nu(\boldsymbol{m})$ is the weight defined in \eqref{eq:nu}.

By Lemma \ref{l_Gamma-Pi-1}, $\mathfrak{d}^{(1)}=d^{(1)}$,
and for all $\boldsymbol{m}\in\mathcal{M}_{0}$, $\Pi^{(1)}(\boldsymbol{m})=\mathfrak{d}^{(1)}$
if, and only if, $\Xi(\boldsymbol{m})=d^{(1)}$. Moreover, by Lemma
\ref{l_SW}, for all $\mathcal{W}(\boldsymbol{m},\boldsymbol{m}')=\mathcal{S}(\boldsymbol{m},\boldsymbol{m}')$
for all $\boldsymbol{m}'\ne\boldsymbol{m}\in\mathcal{M}_{0}$ such
that $\Xi(\bm{m})=d^{(1)}$. Therefore, the previous construction
coincides with the one presented in Section \ref{sec4}, and $R^{(1)}(\cdot\,,\,\cdot)=r^{(1)}(\cdot\,,\,\cdot)$.

Recall that $\mathscr{V}^{(1)}=\{\,\{\boldsymbol{m}\}:\boldsymbol{m}\in\mathcal{M}_{0}\}$,
and denote by $\mathfrak{L}^{(1)}$ the generator given by
\[
(\mathfrak{L}^{(1)}\boldsymbol{h})(\{\boldsymbol{m}\})\;=\;\sum_{\boldsymbol{m}'\in\mathcal{M}_{0}}R^{(1)}(\boldsymbol{m},\boldsymbol{m}')\,[\,\boldsymbol{h}(\{\boldsymbol{m}'\})\,-\,\boldsymbol{h}(\{\boldsymbol{m}\})\,]
\]
for $\boldsymbol{h}\colon\mathscr{V}^{(1)}\to\mathbb{R}$. Let $\widehat{\mathbf{y}}^{(1)}(\cdot)$,
$\mathbf{y}^{(1)}(\cdot)$ be the $\mathscr{V}^{(1)}=\mathscr{S}^{(1)}$-valued
Markov chains induced by the generator $\mathfrak{A}_{1}$.

Recall that we denote by $\mathfrak{n}_{0}$ the
number of local minima of $U$, $\mathfrak{n}_{0}=|\mathcal{M}_{0}|$,
and by $\mathscr{R}_{1}^{(1)},\dots,\mathscr{R}_{\mathfrak{n}_{1}}^{(1)}$,
$\mathscr{T}^{(1)}$ the closed irreducible classes and the transient
states of the Markov chain ${\bf y}^{(1)}$, respectively. If $\mathfrak{n}_{1}=1$,
the construction is complete, $\mathfrak{q}=1$, and $\mathfrak{d}^{(1)}$,
$\mathscr{V}^{(1)}$, $\mathscr{N}^{(1)}$, $\widehat{\mathbf{y}}^{(1)}(\cdot)$,
$\mathbf{y}^{(1)}(\cdot)$ have been defined. If $\mathfrak{n}_{1}>1$,
we add below a new layer to the construction.

Only property (c) introduced at the beginning of
this section is meaningful for $q=1$, and it is clearly fulfilled.
On the other hand, as this construction coincides with the one presented
in Section \ref{sec4}, by Proposition \ref{prop_init}, $\mathfrak{P}_{1}(1)$
-- $\mathfrak{P}_{4}(1)$ are in force

\subsection{General time-scale}

Assume that the depths $0<\mathfrak{d}^{(1)}<\dots<\mathfrak{d}^{(p)}<\infty$,
the sets $\mathscr{V}^{(1)},\dots,\mathscr{V}^{(p)}$, $\mathscr{N}^{(1)},\mathscr{N}^{(2)}\subset\cdots\subset\mathscr{N}^{(p)}$,
the continuous-time Markov chains $\widehat{\mathbf{y}}^{(q)}(\cdot)$,
$\mathbf{y}^{(q)}(\cdot)$ have been constructed and satisfy $\mathfrak{P}_{1}(q)$
-- $\mathfrak{P}_{4}(q)$ and (a) -- (d) for $q\in\llbracket1,\,p\rrbracket$.

We construct a weighted graph whose vertices are
the sets in $\mathscr{S}^{(p+1)}$. For $\boldsymbol{\sigma}\in\mathcal{S}_{0}$,
$\mathcal{M}\in\mathscr{S}^{(p+1)}$, write $\boldsymbol{\sigma}\curvearrowright\mathcal{M}$
if there exists $\boldsymbol{m}\in\mathcal{M}$, such that $\boldsymbol{\sigma}\curvearrowright\boldsymbol{m}$.
Draw an edge between the sets $\mathcal{M}$ and $\mathcal{M}'$ if
there exists $\boldsymbol{\sigma}\in\mathcal{S}_{0}$ such that $\boldsymbol{\sigma}\curvearrowright\mathcal{M}$,
$\boldsymbol{\sigma}\curvearrowright\mathcal{M}'$. Denote by $\Delta(\mathcal{M},\mathcal{M}')$
the weight of this edge given by
\[
{\color{blue}\Delta(\mathcal{M},\mathcal{M}')}:=\min\big\{\,U(\boldsymbol{\sigma}):\boldsymbol{\sigma}\curvearrowright\mathcal{M}\,,\,\boldsymbol{\sigma}\curvearrowright\mathcal{M}'\,\big\}\;.
\]

Let ${\color{blue}\mathbb{G}_{p+1}=(\mathscr{S}^{(p+1)},\mathbb{B}_{p+1},\Delta)}$
be the graph thus obtained, where $\mathbb{B}_{p+1}$ represents the
set of edges (or bonds), and $\Delta$ the weights attached to each
edge. By Lemma \ref{l_assu_saddle}, this graph is connected.

A \textit{{path $\gamma$}}
between $\mathcal{M}$ and $\mathcal{M}'\in\mathscr{S}^{(p+1)}$ is
a sequence of distinct sets $\mathcal{M}_{j}\in\mathscr{S}^{(p+1)}$,
$j\in\llbracket1,\,n\rrbracket$, such that $\mathcal{M}_{1}=\mathcal{M}$,
$\mathcal{M}_{n}=\mathcal{M}'$, $\{\mathcal{M}_{i},\mathcal{M}_{i+1}\}\in\mathbb{B}_{p+1}$,
$i\in\llbracket1,\,n-1\rrbracket$.

For $\mathcal{M}\in\mathscr{V}^{(p+1)}$, $\mathcal{M}'\in\mathscr{S}^{(p+1)}$,
denote by $U(\mathcal{M},\mathcal{M}')$ the height of the pass between
$\mathcal{M}$ and $\mathcal{M}'$:
\[
{\color{blue}U(\mathcal{M},\mathcal{M}')}:=\min_{\gamma}\max_{1\le j<n}\Delta(\mathcal{M}_{j},\mathcal{M}_{j+1})\;,
\]
where the minimum is carried over all paths $\gamma=(\mathcal{M}=\mathcal{M}_{1},\dots,\mathcal{M}_{n}=\mathcal{M}')$
between $\mathcal{M}$ and $\mathcal{M}'$ such that $\mathcal{M}_{k}\in\mathscr{N}^{(p+1)}$,
$k\in\llbracket2,\,n-1\rrbracket$. Let $\Pi^{(p+1)}(\mathcal{M})$,
be the depth of the set $\mathcal{M}\in\mathscr{V}^{(p+1)}$:
\begin{equation}
{\color{blue}\Pi^{(p+1)}(\mathcal{M})}\;:=\min\big\{\,U(\mathcal{M},\mathcal{M}'):\mathcal{M}'\in\mathscr{V}^{(p+1)}\,,\,\mathcal{M}'\neq\mathcal{M}\,\big\}\,-\,U(\mathcal{M})\;.\label{2-15}
\end{equation}
Denote by $\mathfrak{d}^{(p+1)}$ the depth of the shallowest set
in $\mathscr{V}^{(p+1)}$, and by $\vartheta_{\epsilon}^{(p+1)}$
the corresponding time-scale:
\[
{\color{blue}\mathfrak{d}^{(p+1)}}\;:=\;\min_{\mathcal{M}\in\mathscr{V}^{(p+1)}}\Pi^{(p+1)}(\mathcal{M})\;,\quad{\color{blue}\vartheta_{\epsilon}^{(p+1)}}\;:=\;e^{\mathfrak{d}^{(p+1)}/\epsilon}\;.
\]
By \eqref{2-12} and $\mathfrak{P}_{2}(p+1)$, $\mathfrak{d}^{(p+1)}>\mathfrak{d}^{(p)}$.

For $\mathcal{M}\in\mathscr{V}^{(p+1)}$, $\mathcal{M}'\in\mathscr{S}^{(p+1)}$,
$\mathcal{M}'\neq\mathcal{M}$, denote by ${\color{blue}\mathcal{W}(\mathcal{M},\mathcal{M}')}$
the set of gates from $\mathcal{M}$ to $\mathcal{M}'$. This is the
set of saddles points $\boldsymbol{\sigma}\in\mathcal{S}_{0}$ such
that there exists $\mathcal{M}''\in\mathscr{N}^{(p+1)}$ and a path
$\gamma=(\mathcal{M}=\mathcal{M}_{1},\mathcal{M}_{2},\dots,\mathcal{M}_{n}=\mathcal{M}'')$
from $\mathcal{M}$ to $\mathcal{M}''$ satisfying
\begin{equation}
\begin{gathered}\boldsymbol{\sigma}\curvearrowright\mathcal{M}'\,,\,\;\boldsymbol{\sigma}\curvearrowright\mathcal{M}''\;,\;\;U(\boldsymbol{\sigma})=U(\mathcal{M},\mathcal{M}')=\Pi^{(p+1)}(\mathcal{M})+U(\mathcal{M})\;,\\
\mathcal{M}_{j}\in\mathscr{N}^{(p+1)}\;,\quad\Delta(\mathcal{M}_{j-1},\mathcal{M}_{j})<U(\mathcal{M},\mathcal{M}')\;,\;\;j\in\llbracket2,\,n\rrbracket\;.
\end{gathered}
\label{2-16}
\end{equation}
In words, there is a path $\gamma$ from $\mathcal{M}$ to a set $\mathcal{M}''$
with the following properties. All the sets in the path $\gamma$
but the first (which is $\mathcal{M}$) belong to $\mathscr{N}^{(p+1)}$.
Mind that the set $\mathcal{M}'$ does not belong to the path. All
the barriers between consecutive sets in the path $\gamma$ are strictly
below the barrier between $\mathcal{M}$ and $\mathcal{M}'$. The
height of the barrier between the last element of the path ($\mathcal{M}_{n}=\mathcal{M}''$)
and $\mathcal{M}'$ is equal to $U(\mathcal{M},\mathcal{M}')$, the
barrier between $\mathcal{M}$ and $\mathcal{M}'$. Moreover, $U(\mathcal{M},\mathcal{M}')-U(\mathcal{M})$
is the depth of the set $\mathcal{M}$. Finally, there is saddle point
$\boldsymbol{\sigma}$ at height $U(\mathcal{M},\mathcal{M}')$ whose
heteroclinic orbits lead to $\mathcal{M}''$ and $\mathcal{M}'$.

For $\mathcal{M}\in\mathscr{V}^{(p+1)}$, $\mathcal{M}'\in\mathscr{S}^{(p+1)}$,
let
\begin{equation}
{\color{blue}\widehat{\omega}(\mathcal{M},\,\mathcal{M}')}\;:=\;\sum_{\boldsymbol{\sigma}\in\mathcal{W}(\mathcal{M},\,\mathcal{M}')}\omega(\boldsymbol{\sigma})\;,\;\;{\color{blue}\widehat{\omega}_{p+1}(\mathcal{M},\,\mathcal{M}')}\;:=\;\widehat{\omega}(\mathcal{M},\,\mathcal{M}')\,\boldsymbol{1}\{\Pi^{(p+1)}(\mathcal{M})=\mathfrak{d}^{(p+1)}\}\;.\label{e:omega2}
\end{equation}
It is understood here that $\widehat{\omega}(\mathcal{M},\,\mathcal{M}')=0$
if the set $\mathcal{W}(\mathcal{M},\,\mathcal{M}')$ is empty, that
is if $\mathcal{M}'$ is not adjacent to $\mathcal{M}$. By \eqref{2-12}
and \eqref{2-20}, $\widehat{\omega}_{p+1}(\mathcal{M},\,\mathcal{M}')=\omega_{p+1}(\mathcal{M},\,\mathcal{M}')$.

The jump rates of the auxiliary dynamics $\widehat{\mathbf{y}}^{(p+1)}(\cdot)$
are defined as follows. Recall that we denote by $\widehat{R}^{(p)}$
the jump rates of the $\mathscr{S}^{(p)}$-valued Markov chain $\widehat{\mathbf{y}}^{(p+1)}(\cdot)$.
Fix $\mathcal{M}\in\mathscr{N}^{(p+1)}$ and $\mathcal{M}'\in\mathscr{V}^{(p+1)}$.
By construction, $\mathcal{M}\in\mathscr{N}^{(p)}\cup\mathscr{V}^{(p)}=\mathscr{S}^{(p)}$,
and $\mathcal{M}'$ is the union of elements (may be just one) in
$\mathscr{V}^{(p)}$. Set
\begin{equation}
\widehat{R}^{(p+1)}(\mathcal{M},\,\mathcal{M}')\;=\;\sum_{\mathcal{M}''\in\mathscr{V}^{(p)}:\mathcal{M}''\subset\mathcal{M}'}\widehat{R}^{(p)}(\mathcal{M},\,\mathcal{M}'')\;.\label{04b-1}
\end{equation}
If $\mathcal{M}\in\mathscr{N}^{(p+1)}$, $\mathcal{M}'\in\mathscr{N}^{(p+1)}$,
by construction, $\mathcal{M}'$ is an element of $\mathscr{S}^{(p)}$,
and set
\begin{equation}
\widehat{R}^{(p+1)}(\mathcal{M},\,\mathcal{M}')\;=\;\widehat{R}^{(p)}(\mathcal{M},\,\mathcal{M}')\;.\label{04-1}
\end{equation}
Finally, if $\mathcal{M}\in\mathscr{V}^{(p+1)}$ $\mathcal{M}'\in\mathscr{S}^{(p+1)}$,
set
\begin{equation}
\widehat{R}^{(p+1)}(\mathcal{M}',\,\mathcal{M}'')\;=\;\frac{1}{\nu(\mathcal{M}')}\,\widehat{\omega}_{p+1}(\mathcal{M}',\,\mathcal{M}'')\;,\label{2-10b-1}
\end{equation}
where $\nu(\mathcal{M}')$ is given by \eqref{eq:nu}.

By \eqref{2-12}, \eqref{2-20}, $\mathfrak{d}^{(p+1)}=d^{(p+1)}$,
and for all $\mathcal{M}\in\mathscr{V}^{(p+1)}$, $\Pi^{(p+1)}(\mathcal{M})=\mathfrak{d}^{(p+1)}$
if, and only if, $\Xi(\mathcal{M})=d^{(p+1)}$. Moreover, by Lemma
\ref{2-l4}, $\mathcal{W}(\mathcal{M},\mathcal{M}')=\mathcal{S}(\mathcal{M},\mathcal{M}')$
for all $\mathcal{M}'\neq\mathcal{M}$, $\mathcal{M}\in\mathscr{V}^{(p+1)}$,
$\mathcal{M}'\in\mathscr{S}^{(p+1)}$. Therefore, the previous construction
coincides with the one presented in Section \ref{sec4}, and $\widehat{R}^{(p+1)}(\cdot\,,\,\cdot)=\widehat{r}^{(p+1)}(\cdot\,,\,\cdot)$,

Denote by $\widehat{\mathbf{y}}^{(p+1)}(\cdot)$
the $\mathscr{S}^{(p+1)}$-valued, continuous-time Markov process
with jump rates $\widehat{R}^{(p+1)}(\cdot,\,\cdot)$, and by $\mathbf{y}^{(p+1)}(\cdot)$
the trace of the process $\widehat{\mathbf{y}}^{(p+1)}(\cdot)$ on
$\mathscr{V}^{(p+1)}$. The jump rates of the Markov chain $\mathbf{y}^{(p+1)}(\cdot)$
are represented by $r^{(p+1)}(\cdot,\,\cdot)$. Let $\mathfrak{L}^{(p+1)}$,
$\widehat{\mathfrak{L}}^{(p+1)}$ be the generators of the Markov
chains $\mathbf{y}^{(p+1)}(\cdot)$, $\widehat{\mathbf{y}}^{(p+1)}(\cdot)$,
respectively:
\begin{equation}
\begin{gathered}(\widehat{\mathfrak{L}}^{(p+1)}\boldsymbol{h})(\mathcal{M})\;=\;\sum_{\mathcal{M}'\in\mathscr{S}^{(p+1)}}\widehat{R}^{(p+1)}(\mathcal{M},\mathcal{M}')\,[\,\boldsymbol{h}(\mathcal{M}')\,-\,\boldsymbol{h}(\mathcal{M})\,]\;,\\
(\mathfrak{L}^{(p+1)}\boldsymbol{h})(\mathcal{M})\;=\;\sum_{\mathcal{M}'\in\mathscr{V}^{(p+1)}}R^{(p+1)}(\mathcal{M},\mathcal{M}')\,[\,\boldsymbol{h}(\mathcal{M}')\,-\,\boldsymbol{h}(\mathcal{M})\,]\;.
\end{gathered}
\label{2-17-1}
\end{equation}

Since $\mathfrak{n}_{p}=|\mathscr{V}^{(p+1)}|\ge2$,
by Theorem \ref{t:tree} for $n=p$, $\mathfrak{P}_{1}(p+1)$ --
$\mathfrak{P}_{4}(p+1)$ are satisfied. Moreover, as seen above, conditions
(a) -- (d) for $q=p+1$ are also fulfilled. \smallskip{}

\subsection{Lemmata}

We present in this subsection the results used
in the tree construction. First, we have the following auxiliary lemma.
\begin{lem}
\label{l_ms_P}Fix $\bm{m}\in\mathcal{M}_{0}$.
Then, we have $\mathscr{P}(\bm{m})\ne\varnothing$ and
\begin{equation}
\Pi^{(1)}(\bm{m})\le\min_{\bm{m}'\ne\bm{m}}\Theta(\bm{m},\,\bm{m}')-U(\bm{m})\ .\label{e_Pi-Theta}
\end{equation}
\end{lem}

\begin{proof}
Fix $\bm{m}\in\mathcal{M}_{0}$ and let
\[
H=\min_{\bm{m}'\ne\bm{m}}\Theta(\bm{m},\,\bm{m}')\ .
\]
Let $\mathcal{W}$ be a connected component of $\{U<H\}$ containing
$\bm{m}$. By Lemma \ref{lap01}-(1), $\mathcal{M}_{0}\cap\mathcal{W}=\{\bm{m}\}$.
By \cite[Lemma A.2]{LLS-1st}, there exists a connected component
$\mathcal{W}'$ of $\{U<H\}$ such that
\[
\mathcal{W}\ne\mathcal{W}'\ \ \text{and}\ \ \overline{\mathcal{W}}\cap\overline{\mathcal{W}'}\ne\varnothing\ .
\]
Let $\bm{\sigma}\in\overline{\mathcal{W}}\cap\overline{\mathcal{W}'}$.
By Lemma \ref{l_cap_saddle}, $\bm{\sigma}$ is a saddle point. By
Lemma \ref{l_assu_saddle}-(1), $\bm{\sigma}\curvearrowright\bm{m}_{i}$
for some $\bm{m}_{1}\in\mathcal{M}_{0}\cap\mathcal{W}$ and $\bm{m}_{2}\in\mathcal{M}_{0}\cap\mathcal{W}'$.
Since $\mathcal{M}_{0}\cap\mathcal{W}=\{\bm{m}\}$, $\bm{m}_{1}=\bm{m}$
and $\bm{m}_{2}\ne\bm{m}$. Therefore, $\bm{\sigma}\in\mathscr{P}(\bm{m})$.

Suppose that there exists $\bm{\sigma}'\in\mathscr{P}(\bm{m})$
such that $U(\bm{\sigma})<\min_{\bm{m}'\ne\bm{m}}\Theta(\bm{m},\,\bm{m}')$.
Let $\bm{m}\curvearrowleft\bm{\sigma}'\curvearrowright\bm{m}''\ne\bm{m}$.
Then, by Lemma \ref{lem_not}-(3), we have $\Theta(\bm{m},\,\bm{m}'')\le U(\bm{\sigma}')<\min_{\bm{m}'\ne\bm{m}}\Theta(\bm{m},\,\bm{m}')$
which is a contradction. Hence, \eqref{e_Pi-Theta} holds
\end{proof}
Now, we characterize $\Pi^{(1)}(\boldsymbol{m})$
and $\mathscr{G}(\boldsymbol{m})$.
\begin{lem}
\label{l_Gamma-Pi-1}We have $\mathfrak{d}^{(1)}=d^{(1)}$.
Furthermore, for all $\bm{m}\in\mathcal{M}_{0}$,
\[
\Pi^{(1)}(\bm{m})=\mathfrak{d}^{(1)}\;\;\text{if, and only if,}\;\;\Xi(\bm{m})=d^{(1)}\;.
\]
\end{lem}

\begin{proof}
We claim that
\begin{equation}
\Xi(\bm{m})\ge\Pi^{(1)}(\bm{m})\ \text{for all}\ \bm{m}\in\mathcal{M}_{0}\ .\label{e_Gamma-Pi-1}
\end{equation}
Indeed, since $\widetilde{\bm{m}}\subset\mathcal{M}_{0}\setminus\{\bm{m}\}$,
by \eqref{e_Pi-Theta}, we obtain
\[
\Pi^{(1)}(\bm{m})\le\min_{\bm{m}'\ne\bm{m}}\Theta(\bm{m},\,\bm{m}')-U(\bm{m})\le\Theta(\bm{m},\,\widetilde{\bm{m}})-U(\bm{m})=\Xi(\bm{m})\ .
\]

Next, we claim that
\begin{equation}
\Xi(\bm{m})\le\mathfrak{d}^{(1)}\ \text{if}\ \Pi^{(1)}(\bm{m})=\mathfrak{d}^{(1)}\ .\label{e_Gamma-Pi-2}
\end{equation}
Let $\bm{m}\in\mathcal{M}_{0}$ satisfy $\Pi^{(1)}(\bm{m})=\mathfrak{d}^{(1)}$.
Then, there exists $\bm{\sigma}\in\mathscr{P}(\bm{m})$ such that
$\mathfrak{d}^{(1)}=\Pi^{(1)}(\bm{m})=U(\bm{\sigma})-U(\bm{m})$.
Let $\bm{m}'\ne\bm{m}$ such that $\bm{\sigma}\curvearrowright\bm{m}'$.
It is clear that $U(\bm{m}')\le U(\bm{m})$. Otherwise, we have $\bm{\sigma}\in\mathscr{P}(\bm{m}')$
and $U(\bm{\sigma})-U(\bm{m}')<\mathfrak{d}^{(1)}$, in contradiction
with the definition of $\mathfrak{d}^{(1)}$. Hence, $\bm{m}'\subset\widetilde{\bm{m}}$,
and by Lemma \ref{lem_not}-(3),
\[
U(\bm{m})+\Xi(\bm{m})=\Theta(\bm{m},\,\widetilde{\bm{m}})\le\Theta(\bm{m},\bm{m}')\le U(\bm{\sigma})=\mathfrak{d}^{(1)}+U(\bm{m})\ ,
\]
which proves \eqref{e_Gamma-Pi-2}.

By \eqref{e_Gamma-Pi-1}, we have $\mathfrak{d}^{(1)}\le d^{(1)}$
and on the other hand, by \eqref{e_Gamma-Pi-2}, we obtain $d^{(1)}\le\mathfrak{d}^{(1)}$.
Therefore, we have
\[
\mathfrak{d}^{(1)}=d^{(1)}\ .
\]
Furthermore, if $\Pi^{(1)}>\mathfrak{d}^{(1)}$, we have $\Xi(\bm{m})>\mathfrak{d}^{(1)}$
by \eqref{e_Gamma-Pi-1} and if $\Pi^{(1)}(\bm{m})=\mathfrak{d}^{(1)}$,
we have $\Xi(\bm{m})\le\mathfrak{d}^{(1)}=d^{(1)}$. Since $d^{(1)}\le\Xi(\bm{m})$,
we have $\Xi(\bm{m})=d^{(1)}$ and this proves Lemma \ref{l_Gamma-Pi-1}.
\end{proof}
\begin{lem}
\label{l_SW}Fix $\bm{m}\in\mathscr{V}^{(1)}=\mathcal{M}_{0}$
and $\bm{m}'\in\mathscr{V}^{(1)}$ such that $\Xi(\bm{m})=d^{(1)}$
and $\mathcal{S}(\bm{m},\,\bm{m}')\ne\varnothing$. Then, $\mathcal{S}(\bm{m},\,\bm{m}')=\mathcal{W}(\bm{m},\,\bm{m}')$.
\end{lem}

\begin{proof}
Fix $\bm{m}\in\mathscr{V}^{(1)}$ and $\bm{m}'\in\mathscr{V}^{(1)}$
such that $\Xi(\bm{m})=d^{(1)}$ and $\mathcal{S}(\bm{m},\,\bm{m}')\ne\varnothing$.
We have $U(\bm{m})\ge U(\bm{m}')$. Indeed, if $U(\bm{m})<U(\bm{m}')$,
we have
\[
\Xi(\bm{m}')=\Theta(\bm{m}',\,\widetilde{\bm{m}'})-U(\bm{m}')\le\Theta(\bm{m}',\,\bm{m})-U(\bm{m}')<\Theta(\bm{m},\,\widetilde{\bm{m}})-U(\bm{m})=\Xi(\bm{m})=d^{(1)}\ ,
\]
which contradicts to the definition of $d^{(1)}$. Also, by Lemma
\ref{l_Gamma-Pi-1},
\begin{equation}
\Pi^{(1)}(\bm{m})=d^{(1)}\ .\label{e_SW-1}
\end{equation}

Let $\bm{\sigma}\in\mathcal{S}(\bm{m},\,\bm{m}')$.
By definition,
\[
\bm{m}\curvearrowleft\bm{\sigma}\curvearrowright\bm{m}'\ ,\ U(\bm{\sigma})=U(\bm{m})+\Xi(\bm{m})=U(\bm{m})+d^{(1)}\;.
\]
Therefore, $\bm{\sigma}\in\mathscr{P}(\bm{m})$ so that
\[
\Pi^{(1)}(\bm{m})\le U(\bm{\sigma})-U(\bm{m})=d^{(1)}\ .
\]
By \eqref{e_SW-1}, the inequality is indeed equality so that $\bm{\sigma}\in\mathscr{G}(\bm{m})$.
Therefore, since $\bm{m}\curvearrowleft\bm{\sigma}\curvearrowright\bm{m}'$,
we have $\bm{\sigma}\in\mathcal{W}(\bm{m},\,\bm{m}')$.

Now, let $\bm{\sigma}\in\mathcal{W}(\bm{m},\,\bm{m}')$.
By definition,
\[
\bm{m}\curvearrowleft\bm{\sigma}\curvearrowright\bm{m}'\ ,\ U(\bm{\sigma})=U(\bm{m})+\Pi^{(1)}(\bm{m})\ .
\]
Therefore, by Lemma \ref{lem_not}-(3) and \eqref{e_SW-1}, we have
\[
\Theta(\bm{m},\,\bm{m}')\le U(\bm{\sigma})=U(\bm{m})+\Pi^{(1)}(\bm{m})=U(\bm{m})+d^{(1)}\ .
\]
Since $U(\bm{m})\ge U(\bm{m}')$, we obtain
\[
U(\bm{m})+d^{(1)}=U(\bm{m})+\Xi(\bm{m})=\Theta(\bm{m},\,\widetilde{\bm{m}})\le\Theta(\bm{m},\,\bm{m}')\ .
\]
Hence, the inequalities are indeed equalities so that we have
\[
U(\bm{\sigma})=\Theta(\bm{m},\,\widetilde{\bm{m}})=U(\bm{m})+\Xi(\bm{m})\ .
\]
Therefore, since $\bm{m}\curvearrowleft\bm{\sigma}\curvearrowright\bm{m}'$
and $U(\bm{\sigma})=U(\bm{m})+\Pi^{(1)}(\bm{m})=U(\bm{m})+\Xi(\bm{m})$,
$\bm{\sigma}\in\mathcal{S}(\bm{m},\,\bm{m}')$.
\end{proof}
\begin{lem}
\label{2-l3} For all $\mathcal{M}\in\mathscr{V}^{(p+1)}$,
$\mathcal{M}'\in\mathscr{S}^{(p+1)}$, $\mathcal{M}'\neq\mathcal{M}$,
we have $\Theta(\mathcal{M},\mathcal{M}')\le U(\mathcal{M},\mathcal{M}')$.
\end{lem}

\begin{proof}
By definition of $U(\mathcal{M},\mathcal{M}')$,
there exists a path $\gamma=(\mathcal{M}=\mathcal{M}_{0},\dots,\mathcal{M}_{n}=\mathcal{M}')$
between $\mathcal{M}$ and $\mathcal{M}'$ such that $\mathcal{M}_{k}\in\mathscr{N}^{(p+1)}$,
$k\in\llbracket1,\,n-1\rrbracket$, and $U(\mathcal{M},\mathcal{M}')=\max_{j}\Delta(\mathcal{M}_{j},\mathcal{M}_{j+1})$.
As the sets $\mathcal{M}_{j}$ are bound, it is possible to construct
from this sequence, by gluing paths just described in \eqref{z(t)_exp},
one obtains a continuous path $\boldsymbol{z}:[0,\,1]\to\mathbb{R}^{d}$
such that $\boldsymbol{z}(0)\in\mathcal{M}$, $\boldsymbol{z}(1)=\mathcal{M}'$
such that $\boldsymbol{z}(t)\le U(\mathcal{M},\mathcal{M}')$. Thus,
$\Theta(\mathcal{M},\mathcal{M}')\le U(\mathcal{M},\mathcal{M}')$,
as claimed.
\end{proof}
We claim that
\begin{equation}
d^{(p+1)}\le\mathfrak{d}^{(p+1)}\;.\label{2-09}
\end{equation}
Indeed, by definition, there exist $\mathcal{M}\neq\mathcal{M}'\in\mathscr{V}^{(p+1)}$
such that $\mathfrak{d}^{(p+1)}=U(\mathcal{M},\mathcal{M}')-U(\mathcal{M})$.
It is clear that $U(\mathcal{M}')\le U(\mathcal{M})$. Otherwise,
$U(\mathcal{M},\mathcal{M}')-U(\mathcal{M}')<\mathfrak{d}^{(p+1)}$,
in contradiction with the definition of $\mathfrak{d}^{(p+1)}$. Hence,
$\mathcal{M}'\subset\widetilde{\mathcal{M}}$, and by Lemma \ref{2-l3},
$\Theta(\mathcal{M},\widetilde{\mathcal{M}})\le\Theta(\mathcal{M},\mathcal{M}')\le U(\mathcal{M},\mathcal{M}')=\mathfrak{d}^{(p+1)}+U(\mathcal{M})$,
which proves \eqref{2-09}.
\begin{lem}
\label{2-l2} Fix $\mathcal{M}\in\mathscr{V}^{(p+1)}$
and $\mathcal{M}'\in\mathscr{S}^{(p+1)}$ such that $\Xi(\mathcal{M})=d^{(p+1)}$
and $\mathcal{S}(\mathcal{M},\mathcal{M}')\ne\varnothing$. Then,
$\mathcal{S}(\mathcal{M},\,\mathcal{M}')\subset\mathcal{W}(\mathcal{M},\mathcal{M}')$
and $\Theta(\mathcal{M},\mathcal{M}')=U(\mathcal{M},\mathcal{M}')$.
\end{lem}

\begin{proof}
Fix $\mathcal{M}\in\mathscr{V}^{(p+1)}$ such that
$\Xi(\mathcal{M})=d^{(p+1)}$, $\mathcal{M}'\in\mathscr{S}^{(p+1)}$,
and $\boldsymbol{\sigma}\in\mathcal{S}(\mathcal{M},\mathcal{M}')$.
By definition, there exist $k\ge1$, $\boldsymbol{\sigma}_{1},\,\dots,\,\boldsymbol{\sigma}_{k}\in\mathcal{S}_{0}$,
$\boldsymbol{m}_{1}\,\,\dots,\,\boldsymbol{m}_{k+1}\in\mathcal{M}_{0}$,
and $\boldsymbol{m}'\in\mathcal{M}'$ such that $U(\boldsymbol{\sigma}_{j})<U(\boldsymbol{\sigma})$,
$j\in\llbracket1,\,k\rrbracket$, and
\[
\boldsymbol{\sigma}\curvearrowright\boldsymbol{m}_{1}\curvearrowleft\boldsymbol{\sigma}_{1}\curvearrowright\cdots\curvearrowright\boldsymbol{m}_{k+1}\curvearrowleft\boldsymbol{\sigma}_{k}\curvearrowright\boldsymbol{m}\;,\qquad\boldsymbol{\sigma}\curvearrowright\boldsymbol{m}'\;.
\]

We may assume, without loss of generality that
$\boldsymbol{m}_{j}\not\in\mathcal{M}$, $1\le j\le k$. Fix a point
$\boldsymbol{m}_{j}$. Since $\mathscr{S}^{(p+1)}$ forms a partition
of $\mathcal{M}_{0}$, $\boldsymbol{m}_{j}\in\mathcal{M}''$ for some
$\mathcal{M}''\in\mathscr{S}^{(p+1)}$, $\mathcal{M}''\neq\mathcal{M}$.
Since $U(\boldsymbol{\sigma}_{i})<U(\boldsymbol{\sigma})$, $i\in\llbracket1,\,k\rrbracket$,
$\Theta(\mathcal{M},\mathcal{M}'')<U(\boldsymbol{\sigma})=\Theta(\mathcal{M},\widetilde{\mathcal{M}})$.
Thus, the simple set $\mathcal{M}''$ is not contained in $\widetilde{\mathcal{M}}$
so that $U(\mathcal{M}'')>U(\mathcal{M})$. Therefore, $\widetilde{\mathcal{M}}\subset\widetilde{\mathcal{M}''}$,
and $\Theta(\mathcal{M}'',\widetilde{\mathcal{M}''})\le U(\boldsymbol{\sigma})$.
Since $U(\mathcal{M}'')>U(\mathcal{M})$, this implies that $\Xi(\mathcal{M}'')<d^{(p+1)}$,
so that $\mathcal{M}''\in\mathscr{N}^{(p+1)}$.

From the consecutive sequence of local minima and
saddle points, we can construct a sequence (may be shorter since two
different local minima may belong to the same set $\mathcal{M}''\in\mathscr{N}^{(p+1)}$),
$\mathcal{M}_{2},\dots,\mathcal{M}_{n}\in\mathscr{N}^{(p+1)}$ such
that $0<\Delta(\mathcal{M}_{j},\mathcal{M}_{j+1})\le U(\boldsymbol{\sigma})$,
$j\in\llbracket1,\,n\rrbracket$, where $\mathcal{M}_{1}=\mathcal{M}$,
$\mathcal{M}_{n+1}=\mathcal{M}'$. In particular, by definition, $U(\mathcal{M},\mathcal{M}')\le U(\boldsymbol{\sigma})=\Theta(\mathcal{M},\mathcal{M}')$.

On the other hand, by Lemma \ref{2-l3}, $\Theta(\mathcal{M},\mathcal{M}')\le U(\mathcal{M},\mathcal{M}')$.
Hence, $U(\mathcal{M},\mathcal{M}')=U(\boldsymbol{\sigma})=\Theta(\mathcal{M},\mathcal{M}')$,
and $\boldsymbol{\sigma}\in\mathcal{W}(\mathcal{M},\mathcal{M}')$,
as claimed.
\end{proof}
We push further the reasoning presented in the
proof of the previous lemma. Keep the hypotheses and the notation
introduced. Since $\boldsymbol{\sigma}\in\mathcal{S}(\mathcal{M},\mathcal{M}')$,
$U(\boldsymbol{\sigma})=\Theta(\mathcal{M},\widetilde{\mathcal{M}})=\Xi(\mathcal{M})+U(\mathcal{M})=d^{(p+1)}+U(\mathcal{M})$.
Thus, as $\Theta(\mathcal{M},\widetilde{\mathcal{M}})=U(\mathcal{M},\widetilde{\mathcal{M}})$,
$U(\mathcal{M},\mathcal{M}')-U(\mathcal{M})=d^{(p+1)}$ so that $\Pi^{(p=1)}(\mathcal{M})\le d^{(p+1)}$.
From this we conclude that $d^{(p+1)}\le d^{(p+1)}$. Hence, by \eqref{2-09},
\begin{equation}
d^{(p+1)}=\mathfrak{d}^{(p+1)}\;.\label{2-12}
\end{equation}

Moreover, as $\Pi^{(p+1)}(\mathcal{M})\le d^{(p+1)}=\mathfrak{d}^{(p+1)}$,
we actually have that $\Pi^{(p+1)}(\mathcal{M})=\mathfrak{d}^{(p+1)}$
because $\Pi^{(p+1)}(\mathcal{M}'')\ge\mathfrak{d}^{(p+1)}$ for all
$\mathcal{M}''\in\mathscr{V}^{(p+1)}$. Thus, we proved that
\begin{equation}
\Pi^{(p+1)}(\mathcal{M})=\mathfrak{d}^{(p+1)}\;\;\text{for all}\;\;\mathcal{M}\in\mathscr{V}^{(p+1)}\;\;\text{such that}\;\;\Xi(\mathcal{M})=d^{(p+1)}\;.\label{2-19}
\end{equation}

\begin{lem}
\label{2-l4} Fix $\mathcal{M}\in\mathscr{V}^{(p+1)}$
such that $\Pi^{(p+1)}(\mathcal{M})=\mathfrak{d}^{(p+1)}$. Then,
$\Xi(\mathcal{M})=d^{(p+1)}$. Moreover, $\mathcal{W}(\mathcal{M},\mathcal{M}')=\mathcal{S}(\mathcal{M},\mathcal{M}')$
for all $\mathcal{M}'\in\mathscr{S}^{(p+1)}$.
\end{lem}

\begin{proof}
Fix such $\mathcal{M}$. By \eqref{2-12}, by hypothesis,
and by definition of $\Pi^{(p+1)}(\mathcal{M})$, $\mathfrak{d}^{(p+1)}=d^{(p+1)}=\Pi^{(p+1)}(\mathcal{M})=U(\mathcal{M},\mathcal{M}'')-U(\mathcal{M})$
for some $\mathcal{M}''\in\mathscr{V}^{(p+1)}$. We claim that $U(\mathcal{M}'')\le U(\mathcal{M})$.
Indeed, if this is not true, $U(\mathcal{M},\mathcal{M}'')-U(\mathcal{M}'')<U(\mathcal{M},\mathcal{M}'')-U(\mathcal{M})=\Pi^{(p+1)}(\mathcal{M})=\mathfrak{d}^{(p+1)}$,
in contradiction with the definition of $\mathfrak{d}^{(p+1)}$.

Thus, $\mathcal{M}''\subset\widetilde{\mathcal{M}}$.
Hence, by Lemma \ref{2-l3} and the definition of $\mathfrak{d}^{(p+1)}$,
$U(\mathcal{M},\mathcal{M}'')-U(\mathcal{M})\ge\Theta(\mathcal{M},\mathcal{M}'')-U(\mathcal{M})\ge\Theta(\mathcal{M},\widetilde{\mathcal{M}})-U(\mathcal{M})\ge d^{(p+1)}$.
The previous estimates yield that $\Xi(\mathcal{M})=d^{(p+1)}$, as
claimed,

We turn to the proof of the second claim of the
lemma. By Lemma \ref{2-l2}, $\mathcal{S}(\mathcal{M},\mathcal{M}')\subset\mathcal{W}(\mathcal{M},\mathcal{M}')$.
It remains to show that $\mathcal{W}(\mathcal{M},\mathcal{M}')\subset\mathcal{S}(\mathcal{M},\mathcal{M}')$.
Fix $\boldsymbol{\sigma}\in\mathcal{W}(\mathcal{M},\mathcal{M}')$.
From the definition of this set, and from the fact that the sets in
$\mathscr{S}^{(p+1)}$ are bound, it is easy to construct, with the
help of Lemma \ref{l_assu_saddle}, an alternate sequence of saddle
points and local minima such that $\mathcal{M}\,\leftsquigarrow\,\boldsymbol{\sigma}\,\curvearrowright\,\mathcal{M}'$,
and $U(\boldsymbol{\sigma})=U(\mathcal{M},\mathcal{M}')=\Pi^{(p+1)}(\mathcal{M})+U(\mathcal{M})$.

Since we already established that $\Xi(\mathcal{M})=d^{(p+1)}$,
by Lemma \ref{2-l2}, $U(\mathcal{M},\mathcal{M}')=\Theta(\mathcal{M},\mathcal{M}')$,
so that $U(\boldsymbol{\sigma})=\Theta(\mathcal{M},\mathcal{M}')\ge\Theta(\mathcal{M},\widetilde{\mathcal{M}})\ge d^{(p+1)}+U(\mathcal{M})$.
By \eqref{2-12} and the hypothesis, this quantity is equal to $\mathfrak{d}^{(p+1)}+U(\mathcal{M})=\Pi^{(p+1)}(\mathcal{M})+U(\mathcal{M})=U(\boldsymbol{\sigma})$.
Thus $U(\boldsymbol{\sigma})=\Theta(\mathcal{M},\mathcal{M}')=\Theta(\mathcal{M},\widetilde{\mathcal{M}})$,
so that $\boldsymbol{\sigma}\in\mathcal{S}(\mathcal{M},\mathcal{M}')$.
\end{proof}
By the previous lemma and \eqref{2-19}, for all
$\mathcal{M}\in\mathscr{V}^{(p+1)}$,
\begin{equation}
\Pi^{(p+1)}(\mathcal{M})=\mathfrak{d}^{(p+1)}\;\;\text{if, and only if,}\;\;\Xi(\mathcal{M})=d^{(p+1)}\;.\label{2-20}
\end{equation}

\section{\label{app_b}Elementary Properties of Landscape of $U$}

In this section, we present several elementary results needed to investigate
the multi-scale structure of the energy landscape.

\subsection{\label{sec:paper-1} Frequently referred results from \cite{LLS-1st}}

We start by summarizing several results on level sets of $U$ obtained
in \cite{LLS-1st} which are frequently used in the current article.

The first one is \cite[Lemma A.3]{LLS-1st}.
\begin{lem}
\label{l_cap_saddle} Fix $H\in\mathbb{R}$, and let $\mathcal{V}_{1}$
and $\mathcal{V}_{2}$ be two disjoint connected components of $\{U<H\}$.
If $\overline{\mathcal{V}_{1}}\cap\overline{\mathcal{V}_{2}}\ne\varnothing$,
then $\overline{\mathcal{V}_{1}}\cap\overline{\mathcal{V}_{2}}=\partial\mathcal{V}_{1}\cap\partial\mathcal{V}_{2}$
and any element $\boldsymbol{\sigma}$ of $\overline{\mathcal{V}_{1}}\cap\overline{\mathcal{V}_{2}}$
is a saddle point such that $U(\boldsymbol{\sigma})=H$. Moreover,
for all $r>0$ small enough, $\mathcal{A}(\boldsymbol{\sigma},\,r)$
has two connected components: $\mathcal{A}(\boldsymbol{\sigma},\,r)\cap\mathcal{V}_{1}$
and $\mathcal{A}(\boldsymbol{\sigma},\,r)\cap\mathcal{V}_{2}$
\end{lem}

The followings are \cite[Lemma A.9-A.11]{LLS-1st}.
\begin{lem}
\label{l_level_connected} Fix $H\in\mathbb{R}$. Let $\mathcal{G}$
be a connected component of $\{U<H\}$. Let $\mathcal{H}\subset\{U<H\}$
be a connected set satisfying $\mathcal{G}\cap\mathcal{H}\ne\varnothing$.
Then, $\mathcal{G}\subset\mathcal{H}$. The same assertion holds if
we replace all strict inequalities by inequalities.
\end{lem}

\begin{lem}
\label{l_level_boundary} Fix $H\in\mathbb{R}$. Let $\mathcal{H}$
be a connected component of the set $\{U<H\}$ or one of the set $\{U\le H\}$.
Then, $U(\boldsymbol{x}_{0})=H$ for all $\boldsymbol{x}_{0}\in\partial\mathcal{H}$.
\end{lem}

\begin{lem}
\label{lap01} Fix $H\in\mathbb{R}$. Let $\mathcal{V}\subset\mathbb{R}^{d}$
be a connected component of level set $\{U<H\}$. Let $\mathcal{M}\subset\mathcal{M}_{0}\cap\mathcal{V}$
and $\mathcal{M}'\subset\mathcal{M}_{0}\setminus\mathcal{M}$.
\begin{enumerate}
\item If $\mathcal{M}'\subset\mathcal{V}$, then $\Theta(\mathcal{M},\,\mathcal{M}')<H$.
In other words, if $\Theta(\mathcal{M},\,\bm{m})\ge H$ for all $\bm{m}\in\mathcal{M}'$,
then $\mathcal{M}'\subset\mathbb{R}^{d}\setminus\mathcal{V}$.
\item If $\mathcal{M}'\subset\mathbb{R}^{d}\setminus\mathcal{V}$, then
$\Theta(\mathcal{M},\,\mathcal{M}')\ge H$. In other words, if $\Theta(\mathcal{M},\,\bm{m})<H$
for all $\bm{m}\in\mathcal{M}'$, then $\mathcal{M}'\subset\mathcal{V}$.
\end{enumerate}
\end{lem}

The last one is \cite[Lemma A.17]{LLS-1st}.
\begin{lem}
\label{l_105a-2}Let $\boldsymbol{m}',\,\boldsymbol{m}''\in\mathcal{M}_{0}$
and let $H\ge\Theta(\bm{m}',\,\bm{m}'')$. Then, the connected component
of $\{U\le H\}$ containing $\bm{m}'$ also contains $\bm{m}''$.
\end{lem}

\subsection{Properties of simple and bound sets}

In this subsection we investiate the properties of simple and bound
sets.
\begin{lem}
\label{lem_not} The followings hold. Here, all subsets of $\mathcal{M}_{0}$
are simple and bound.
\begin{enumerate}
\item If $\mathcal{M}_{1}'\subset\mathcal{M}_{1}$ and $\mathcal{M}_{2}'\subset\mathcal{M}_{2}$,
then $\Theta(\mathcal{M}_{1},\,\mathcal{M}_{2})\le\Theta(\mathcal{M}_{1}',\,\mathcal{M}_{2}')$.
\item For any disjoint subsets $\mathcal{M}_{1},\dots,\,\mathcal{M}_{n}$
of $\mathcal{M}_{0}$,
\begin{equation}
\Theta(\mathcal{M}_{1},\,\mathcal{M}_{n})\ \le\ \max_{i=1,\,\dots,\,n-1}\Theta(\mathcal{M}_{i},\,\mathcal{M}_{i+1})\;.\label{e:bdtheta}
\end{equation}
\item Let $\mathcal{M},\,\mathcal{M}'$ be subsets of $\mathcal{M}_{0}$
and $\boldsymbol{\sigma}$ a saddle point in $\mathcal{S}_{0}$ such
that $\mathcal{M}\,\leftsquigarrow\,\boldsymbol{\sigma}\,\curvearrowright\,\mathcal{M}'$.
Then, $\Theta(\mathcal{M},\,\mathcal{M}')\le U(\boldsymbol{\sigma})$.
\item Let $\mathcal{M},\,\mathcal{M}'\subset\mathcal{M}_{0}$ be disjoint
subsets such that $\mathcal{M}\to_{\boldsymbol{\sigma}}\mathcal{M}'$
for some $\boldsymbol{\sigma}\in\mathcal{S}_{0}$. Then, there exists
$\boldsymbol{m}'\in\mathcal{M}'$ such that $\mathcal{M}\to_{\boldsymbol{\sigma}}\boldsymbol{m}'$.
\item Let $\mathcal{M},\,\mathcal{M}'\subset\mathcal{M}_{0}$ satisfy $\Theta(\mathcal{M},\,\widetilde{\mathcal{M}})\le\Theta(\mathcal{M},\,\mathcal{M}')$.
If $\mathcal{M}\to_{\boldsymbol{\sigma}}\boldsymbol{m}'$ for some
$\boldsymbol{m}'\in\mathcal{M}'$ and $\boldsymbol{\sigma}\in\mathcal{S}_{0}$,
then $\mathcal{M}\to_{\boldsymbol{\sigma}}\mathcal{M}'$.
\end{enumerate}
\end{lem}

\begin{proof}
Assertions (1) and (2) are straightforward consequences of definition
\eqref{Theta}.

\noindent (3) By definition, there exists $\boldsymbol{m}\in\mathcal{M}$
and $\boldsymbol{m}'\in\mathcal{M}'$ such that $\boldsymbol{m}\,\leftsquigarrow\,\boldsymbol{\sigma}\,\curvearrowright\,\boldsymbol{m}'$.
Hence, by \eqref{z(t)_exp} and \eqref{eq:approx} we can construct,
by concatenating paths of the form \eqref{z(t)_exp}, a path $\boldsymbol{z}\colon[0,\,1]\rightarrow\mathbb{R}^{d}$
such that
\[
\boldsymbol{z}(0)\ =\ \boldsymbol{m}\;,\;\;\;\boldsymbol{z}(1)\ =\ \boldsymbol{m}'\;,\;\text{and\;\;\;}\max_{t\in[0,\,1]}U(\boldsymbol{z}(t))\ =\ U(\boldsymbol{\sigma})\;.
\]
Therefore, in view of definition \eqref{Theta}, by the first assertion,
$\Theta(\mathcal{M},\,\mathcal{M}')\le\Theta(\bm{m},\,\bm{m}')\le U(\boldsymbol{\sigma})$.

\noindent (4) Suppose that $\mathcal{M}\to_{\boldsymbol{\sigma}}\mathcal{M}'$
for some saddle point $\boldsymbol{\sigma}$. By \eqref{eq:con_gate},
\[
\mathcal{M}\,\leftsquigarrow\,\boldsymbol{\sigma}\,\curvearrowright\,\mathcal{M}'\;,\;\;\text{and}\;\;\ U(\boldsymbol{\sigma})\ =\ \Theta(\mathcal{M},\,\widetilde{\mathcal{M}})\ =\ \Theta(\mathcal{M},\,\mathcal{M}')\ .
\]
By definition below \eqref{eq:approx}, there exists $\boldsymbol{m}'\in\mathcal{M}'$
such that $\boldsymbol{\sigma}\curvearrowright\boldsymbol{m}'$. By
the third assertion,
\[
\Theta(\mathcal{M},\,\boldsymbol{m}')\ \le\ U(\boldsymbol{\sigma})\ =\ \Theta(\mathcal{M},\,\widetilde{\mathcal{M}})\;.
\]
Since $\boldsymbol{m}'\in\mathcal{M}'$, by the first assertion,
\[
\Theta(\mathcal{M},\,\mathcal{M}')\ \le\ \Theta(\mathcal{M},\,\boldsymbol{m}')\;.
\]
Combining the previous inequalities yields that
\[
\Theta(\mathcal{M},\,\boldsymbol{m}')\ =\ \Theta(\mathcal{M},\,\widetilde{\mathcal{M}})\ =\ U(\boldsymbol{\sigma})\;.
\]
Thus, by definition, $\mathcal{M}\to_{\boldsymbol{\sigma}}\{\boldsymbol{m}'\}$.

\noindent (5) Suppose that $\mathcal{M}\to_{\bm{\sigma}}\{\boldsymbol{m}'\}\subset\mathcal{M}'$.
By definition,
\[
\mathcal{M}\,\leftsquigarrow\,\boldsymbol{\sigma}\,\curvearrowright\,\boldsymbol{m}'\;,\;\;\text{and}\;\;U(\boldsymbol{\sigma})\ =\ \Theta(\mathcal{M},\,\boldsymbol{m}')\ =\ \Theta(\mathcal{M},\,\widetilde{\mathcal{M}})\ .
\]
Hence, by definition, $\mathcal{M}\leftsquigarrow\boldsymbol{\sigma}\curvearrowright\mathcal{M}'$.
Thus, to prove that $\mathcal{M}\to_{\sigma}\mathcal{M}'$, it suffices
to verify that
\[
\Theta(\mathcal{M},\,\mathcal{M}')\ =\ U(\boldsymbol{\sigma})\ =\ \Theta(\mathcal{M},\,\widetilde{\mathcal{M}})\;.
\]
Since $\mathcal{M}\leftsquigarrow\boldsymbol{\sigma}\curvearrowright\mathcal{M}'$,
by the third assertion,
\[
\Theta(\mathcal{M},\,\mathcal{M}')\ \le\ U(\boldsymbol{\sigma})\ =\ \Theta(\mathcal{M},\,\widetilde{\mathcal{M}})\;.
\]
On the other hand, by the hypothesis $\Theta(\mathcal{M},\,\widetilde{\mathcal{M}})\le\Theta(\mathcal{M},\,\mathcal{M}')$,
\[
\Theta(\mathcal{M},\,\widetilde{\mathcal{M}})\le\Theta(\mathcal{M},\,\mathcal{M}')\;,
\]
which completes the proof.
\end{proof}
Under the hypothesis $\Theta(\mathcal{M},\,\widetilde{\mathcal{M}})\le\Theta(\mathcal{M},\,\mathcal{M}')$,
by the fourth and fifth assertion provides $\mathcal{M}\to_{\bm{\sigma}}\mathcal{M}'$
if and only if $\mathcal{M}\to_{\bm{\sigma}}\bm{m}'$ for some $\bm{m}'\in\mathcal{M}'$.
\begin{lem}
\label{l_M->M0} Let $\mathcal{M}_{1}$ and $\mathcal{M}_{2}$ be
disjoint simple bound subsets of $\mathcal{M}_{0}$ such that
\begin{equation}
U(\mathcal{M}_{1})\ \le\ U(\mathcal{M}_{2})\text{\;\;\;and}\;\;\;\Xi(\mathcal{M}_{1})\ <\ \Xi(\mathcal{M}_{2})\;.\label{eq:lem21}
\end{equation}
Then, we cannot have $\mathcal{M}_{1}\to\mathcal{M}_{2}$.
\end{lem}

\begin{proof}
Suppose by contradiction that $\mathcal{M}_{1}\to\mathcal{M}_{2}$.
Since $\mathcal{M}_{1}$ is simple, by definition \eqref{e:Gamma(M)}
of $\Xi(\mathcal{M}_{1})$,
\[
U(\mathcal{M}_{1})+\Xi(\mathcal{M}_{1})\ =\ \Theta(\mathcal{M}_{1},\,\widetilde{\mathcal{M}_{1}})
\]
As $\mathcal{M}_{1}\to\mathcal{M}_{2}$, $\Theta(\mathcal{M}_{1},\,\widetilde{\mathcal{M}_{1}})=\Theta(\mathcal{M}_{1},\,\mathcal{M}_{2})$.
Since $U(\mathcal{M}_{1})\le U(\mathcal{M}_{2})$, $\mathcal{M}_{1}\subset\widetilde{\mathcal{M}_{2}}$
so that
\[
\Theta(\mathcal{M}_{1},\,\mathcal{M}_{2})\ \ge\ \Theta(\mathcal{M}_{2},\,\widetilde{\mathcal{M}_{2}})\;.
\]
As $\mathcal{M}_{2}$ is simple, by definition of $\Xi(\mathcal{M}_{2})$,
$\Theta(\mathcal{M}_{2},\,\widetilde{\mathcal{M}_{2}})=U(\mathcal{M}_{2})+\Xi(\mathcal{M}_{2})$.
This proves that $U(\mathcal{M}_{1})+\Xi(\mathcal{M}_{1})=U(\mathcal{M}_{2})+\Xi(\mathcal{M}_{2})$,
in contradiction with \eqref{eq:lem21}.
\end{proof}
\begin{lem}
\label{l_M->M} Let $\mathcal{M}_{1},\,\mathcal{M}_{2},\,\dots,\,\mathcal{M}_{\ell}$,
$\ell\ge2$, be mutually disjoint simple bound subsets of $\mathcal{M}_{0}$
such that
\begin{equation}
\mathcal{M}_{1}\to\mathcal{M}_{2}\to\cdots\to\mathcal{M}_{\ell}\;\;\;\;\text{and\;\;\;\;}\Xi(\mathcal{M}_{1})\ \ge\ \max_{j\in\llbracket2,\,\ell-1\rrbracket}\Xi(\mathcal{M}_{j})\;.\label{e:m->m}
\end{equation}
Then,
\begin{enumerate}
\item $\Theta(\mathcal{M}_{1},\,\widetilde{\mathcal{M}_{1}})\ge\Theta(\mathcal{M}_{i},\,\widetilde{\mathcal{M}_{i}})$
for all $i\in\llbracket2,\,\ell-1\rrbracket$.
\item If $\Xi(\mathcal{M}_{\ell})\ge\Xi(\mathcal{M}_{1})$, then
\[
U(\mathcal{M}_{1})\ \ge\ U(\mathcal{M}_{\ell})\;\;\;\;\text{and\;\;\;\;}\Theta(\mathcal{M}_{1},\,\widetilde{\mathcal{M}_{1}})\ =\ \Theta(\mathcal{M}_{1},\,\mathcal{M}_{\ell})\;.
\]
\end{enumerate}
\end{lem}

\begin{proof}
We start with the first assertion of the lemma. By the definition
\eqref{eq:con_gate} of adjacency and hypothesis \eqref{e:m->m},
\begin{equation}
\Theta(\mathcal{M}_{i},\,\widetilde{\mathcal{M}_{i}})\ =\ \Theta(\mathcal{M}_{i},\,\mathcal{M}_{i+1})\;,\quad i\in\llbracket1,\,\ell-1\rrbracket\;.\label{elmm}
\end{equation}
The proof is carried out by induction in $i$. For $i=2$, consider
two cases separately.

\smallskip{}
{[}\emph{Case 1:} $U(\mathcal{M}_{2})\ge U(\mathcal{M}_{1})${]}.
Since $\mathcal{M}_{1}\subset\widetilde{\mathcal{M}_{2}}$, $\Theta(\mathcal{M}_{2},\widetilde{\mathcal{M}_{2}})\le\Theta(\mathcal{M}_{2},\mathcal{M}_{1})$.
This quantity is equal to $\Theta(\mathcal{M}_{1},\,\mathcal{M}_{2})$,
which, by \eqref{elmm}, is equal to $\Theta(\mathcal{M}_{1},\widetilde{\mathcal{M}_{1}})$.

\noindent \smallskip{}
{[}\emph{Case 2:} $U(\mathcal{M}_{2})<U(\mathcal{M}_{1})${]}. As,
by hypothesis, $\Xi(\mathcal{M}_{1})\ge\Xi(\mathcal{M}_{2})$,
\[
\Theta(\mathcal{M}_{2},\widetilde{\mathcal{M}_{2}})\ =\ U(\mathcal{M}_{2})+\Xi(\mathcal{M}_{2})\ <\ U(\mathcal{M}_{1})+\Xi(\mathcal{M}_{1})\ =\ \Theta(\mathcal{M}_{1},\widetilde{\mathcal{M}_{1}})\;.
\]
This proves assertion (1) for $i=2$.

\noindent We turn to the induction step. Fix $i\in\llbracket2,\,\ell-2\rrbracket$
and assume that
\begin{equation}
\Theta(\mathcal{M}_{1},\widetilde{\mathcal{M}_{1}})\ \ge\ \max_{j\in\llbracket2,\,i\rrbracket}\,\Theta(\mathcal{M}_{j},\widetilde{\mathcal{M}_{j}})\;.\label{elmm2}
\end{equation}

\smallskip{}
{[}\emph{Case 1:} $U(\mathcal{M}_{i+1})\ge U(\mathcal{M}_{1})${]}
Since $\mathcal{M}_{1}\subset\widetilde{\mathcal{M}_{i+1}}$ and by
symmetry of $\Theta$, $\Theta(\mathcal{M}_{i+1},\widetilde{\mathcal{M}_{i+1}})\le\Theta(\mathcal{M}_{i+1},\mathcal{M}_{1})=\Theta(\mathcal{M}_{1},\mathcal{M}_{i+1})$.
By Lemma \ref{lem_not}-(2) and since $\mathcal{M}_{j}\to\mathcal{M}_{j+1}$
for $j\in\llbracket1,\,i\rrbracket$,
\begin{equation}
\Theta(\mathcal{M}_{1},\mathcal{M}_{i+1})\;\le\;\max_{j\in\llbracket1,\,i\rrbracket}\Theta(\mathcal{M}_{j},\mathcal{M}_{j+1})\;=\;\max_{j\in\llbracket1,\,i\rrbracket}\Theta(\mathcal{M}_{j},\widetilde{\mathcal{M}_{j}})\;.\label{05b}
\end{equation}
By the induction hypothesis \eqref{elmm2}, this maximum is bounded
by $\Theta(\mathcal{M}_{1},\widetilde{\mathcal{M}_{1}})$, as claimed.

\noindent \smallskip{}
{[}\emph{Case 2:} $U(\mathcal{M}_{i+1})<U(\mathcal{M}_{1})${]}. The
argument is identical to the case $i=2$ and does not require the
induction assumption. This completes the proof of the first assertion
of the lemma.

\noindent We turn to Assertion (2). Suppose that $\Xi(\mathcal{M}_{1})\le\Xi(\mathcal{M}_{\ell})$.
By \eqref{05b} for $i=\ell-1$ and the first claim of the lemma,
\begin{equation}
\Theta(\mathcal{M}_{1},\,\mathcal{M}_{\ell})\ \le\ \max_{j\in\llbracket i,\,\ell-1\rrbracket}\,\Theta(\mathcal{M}_{j},\widetilde{\mathcal{M}_{j}})\ \le\ \Theta(\mathcal{M}_{1},\widetilde{\mathcal{M}_{1}})\;.\label{eq:ecc2}
\end{equation}

We claim that $U(\mathcal{M}_{1})\ge U(\mathcal{M}_{\ell})$. Indeed,
suppose by contradiction that $U(\mathcal{M}_{\ell})>U(\mathcal{M}_{1})$.
In this case, since by hypothesis $\Xi(\mathcal{M}_{1})\le\Xi(\mathcal{M}_{\ell})$,
\[
\Theta(\mathcal{M}_{1},\widetilde{\mathcal{M}_{1}})\ =\ \Xi(\mathcal{M}_{1})+U(\mathcal{M}_{1})\ <\ \Xi(\mathcal{M}_{\ell})+U(\mathcal{M}_{\ell})\ =\ \Theta(\mathcal{M}_{\ell},\,\widetilde{\mathcal{M}_{\ell}})\;.
\]
As $\mathcal{M}_{1}\subset\widetilde{\mathcal{M}_{\ell}}$, this last
quantity is bounded by $\Theta(\mathcal{M}_{\ell},\,\mathcal{M}_{1})$,
in contradiction with \eqref{eq:ecc2} and proving the claim.

Since $U(\mathcal{M}_{1})\ge U(\mathcal{M}_{\ell})$, $\mathcal{M}_{\ell}\subset\widetilde{\mathcal{\mathcal{M}}_{1}}$
so that $\Theta(\mathcal{M}_{1},\,\widetilde{\mathcal{M}_{1}})\le\Theta(\mathcal{M}_{1},\,\mathcal{M}_{\ell})$.
This inequality together with \eqref{eq:ecc2} yields that $\Theta(\mathcal{M}_{1},\,\widetilde{\mathcal{M}_{1}})=\Theta(\mathcal{M}_{1},\,\mathcal{M}_{\ell})$,
which completes the proof of the second assertion of the lemma.
\end{proof}
\begin{lem}
\label{l_bound_conn} Fix a simple and bound set $\mathcal{M}\subset\mathcal{M}_{0}$.
Let $H$ be a real number such that $\max_{\boldsymbol{m},\,\boldsymbol{m}'\in\mathcal{M}}\Theta(\boldsymbol{m},\,\boldsymbol{m}')<H$.
Then, there exists a connected component $\mathcal{W}$ of $\{U<H\}$
containing $\mathcal{M}$. Suppose, additionally, that $H\le\Theta(\mathcal{M},\,\widetilde{\mathcal{M}})$.
Then, $\mathcal{M}=\mathcal{M}^{*}(\mathcal{W})$.
\end{lem}

\begin{proof}
Fix $\boldsymbol{m}_{1}\in\mathcal{M}$. Let $\mathcal{W}$ be a connected
component of $\{U<H\}$ containing $\boldsymbol{m}_{1}$. Since, by
hypothesis, $\Theta(\boldsymbol{m}_{1},\,\boldsymbol{m})<H$ for all
$\boldsymbol{m}\in\mathcal{M}$, by Lemma \ref{lap01}-(2), $\mathcal{M}\subset\mathcal{W}$.
This proves the first assertion.

We turn to the second. Suppose that $H\le\Theta(\mathcal{M},\,\widetilde{\mathcal{M}})$
and suppose that $\bm{m}_{2}\in\mathcal{M}^{*}(\mathcal{W})\setminus\mathcal{M}$.
Since $\mathcal{M}$ is simple and $U(\boldsymbol{m}_{2})\le U(\mathcal{M})$,
by Lemma \ref{lap01}-(1)
\[
\Theta(\mathcal{M},\,\widetilde{\mathcal{M}})\ \le\ \Theta(\mathcal{M},\,\boldsymbol{m}_{2})\ <\ H\,
\]
This contradicts to $H\le\Theta(\mathcal{M},\,\widetilde{\mathcal{M}})$.
Hence, $\mathcal{M}^{*}(\mathcal{W})\subset\mathcal{M}$. This completes
the proof since $\mathcal{M}$ is simple and $\mathcal{M}\subset\mathcal{W}$.
\end{proof}

\subsection{Connected components of level sets\label{sec8b}}

In this subsection, we present some results on connected components
needed in the article. Denote by \textcolor{blue}{$\mathcal{A}^{o}$
the interior} of a subset $\mathcal{A}^{o}$ of $\mathbb{R}^{d}$.
\begin{lem}
\label{2-la0} Fix $H\in\mathbb{R}$. For each connected component
$\mathcal{A}$ of $\{U\le H\}^{o}$, let
\[
\widehat{\mathcal{A}}\ :=\ \mathcal{A}\setminus\{\boldsymbol{x}\in\mathcal{A}:\boldsymbol{x}\text{ is a local maximum such that }U(\boldsymbol{x})=H\}\ .
\]
\begin{enumerate}
\item If $U(\boldsymbol{x})=H$ for some $\boldsymbol{x}\in\{U\le H\}^{o}$,
then $\boldsymbol{x}$ is a local maximum.
\item Let $\mathcal{A}$ and $\mathcal{W}$ be connected components of $\{U\le H\}^{o}$
and $\{U<H\}$, respectively. If $\mathcal{A}\cap\mathcal{W}\ne\varnothing$,
then $\mathcal{W}=\widehat{\mathcal{A}}$. In particular, such $\mathcal{W}$
is unique and $\mathcal{M}_{0}\cap\mathcal{W}=\mathcal{M}_{0}\cap\mathcal{A}$
\item Let $\mathcal{W}$ be a connected component of $\{U<H\}$. Then, there
exists a unique connected component $\mathcal{A}$ of $\{U\le H\}^{o}$
such that $\mathcal{W}=\widehat{\mathcal{A}}$
\end{enumerate}
\end{lem}

\begin{proof}
(1) Let $\boldsymbol{x}\in\{U\le H\}^{o}$ satisfy $U(\boldsymbol{x})=H$.
Since $\boldsymbol{x}\in\{U\le H\}^{o}$, there exists $\delta>0$
such that $B(\boldsymbol{x},\,\delta)\subset\{U\le H\}^{o}$. Let
us choose $\delta>0$ sufficiently small so that there is no critical
point in $B(\boldsymbol{x},\,\delta)\setminus\{\boldsymbol{x}\}$.
Suppose that there exists $\boldsymbol{y}\in B(\boldsymbol{x},\,\delta)\setminus\{\boldsymbol{x}\}$
such that $U(\boldsymbol{y})=H$. Since $U(\boldsymbol{z})\le H$
for all $\boldsymbol{z}\in B(\boldsymbol{x},\,\delta)$, $\nabla U(\boldsymbol{y})=0$,
which is a contradiction. Hence, there is no $\boldsymbol{y}\in B(\boldsymbol{x},\,\delta)\setminus\{\boldsymbol{x}\}$
such that $U(\boldsymbol{y})=H$. Therefore, $\boldsymbol{x}$ is
a local maximum. \medskip{}

\noindent (2) Let $\boldsymbol{x}_{0}\in\mathcal{A}\cap\mathcal{W}$.
Let $\boldsymbol{x}_{1}\in\mathcal{W}$ so that there exists a continuous
path $\boldsymbol{z}_{1}:[0,\,1]\to\mathcal{W}$ such that
\[
\boldsymbol{z}_{1}(0)\ =\ \boldsymbol{x}_{0}\;,\;\;\;\boldsymbol{z}_{1}(1)\ =\ \boldsymbol{x}_{1}\;\;\mbox{and}\;\;U(\boldsymbol{z}_{1}(t))\ <\ H\;\mbox{ for all }t\in[0,\,1]
\]
Therefore, by Lemma \ref{l_level_boundary}, the path $\boldsymbol{z}_{1}([0,\,1])$
is contained in an connected component of $\{U\le H\}^{o}$ containing
$\boldsymbol{x}_{0}$, i.e., contained in $\mathcal{A}$. This implies
that $\boldsymbol{x}_{1}=\boldsymbol{z}_{1}(1)\in\mathcal{A}$. Furthermore,
since $U(\boldsymbol{x}_{1})<H$, we obtain $\boldsymbol{x}\in\widehat{\mathcal{A}}$.
This proves that $\mathcal{W}\subset\widehat{\mathcal{A}}$.

On the other hand, by part (1), $U(\boldsymbol{y})<H$ for all $\boldsymbol{y}\in\widehat{\mathcal{A}}$
so that $\widehat{\mathcal{A}}\subset\{U<H\}$. Since the set obtained
by deleting finite number of points from an open and connected non-empty
set in $\mathbb{R}^{d}$ is still open and connected non-empty, we
can conclude that $\widehat{\mathcal{A}}$ is open and connected.
By the above paragraph,, $\widehat{\mathcal{A}}\cap\mathcal{W}=\mathcal{W}\neq\varnothing$.
Therefore, since $\widehat{\mathcal{A}}\subset\{U<H\}$, by Lemma
\ref{l_level_connected}, $\widehat{\mathcal{A}}\subset\mathcal{W}$.\medskip{}

\noindent (3) Since $\{U<H\}\subset\{U\le H\}^{o}$, the set $\mathcal{W}$
must intersect with at least one of the connected components of $\{U\le H\}^{o}$.
Denote this connected component by $\mathcal{A}$. Then, by part (2),
$\mathcal{W}=\widehat{\mathcal{A}}$. The uniqueness is straightforward.
\end{proof}
\begin{lem}
\label{2la1-0}
Fix $H\in\mathbb{R}$, and let $\mathcal{K}$
be a connected component of $\{U\le H\}$ which is not a singleton.
Let $\mathcal{W}_{1},\,\dots,\,\mathcal{W}_{a}$ be all the connected
components of $\{U<H\}$ intersecting with $\mathcal{K}$. Then, we
have
\[
\mathcal{K}=\bigcup_{i=1}^{a}\overline{\mathcal{W}_{i}}\;\;\;\text{and}\;\;\;\mathcal{M}_{0}\cap\mathcal{K}\ =\ \mathcal{M}_{0}\cap\bigcup_{i=1}^{a}\mathcal{W}_{i}\;.
\]
\end{lem}

\begin{proof}
Since $\mathcal{K}$ is not a singleton, $\mathcal{K}^{o}$
is nonempty. Decompose $\mathcal{K}^{o}$ as
\[
\mathcal{K}^{o}\ :=\ \bigcup_{i=1}^{a}\mathcal{A}_{i}\ .
\]
If $\mathcal{A}_{i}\cap\{U<H\}$, $i\in\llbracket1,\,a\rrbracket$,
is empty, $U(\bm{x})=H$ for all $\bm{x}\in\mathcal{A}_{i}$. This
contradicts to the assumption that all critical points are nondegenerate.
Therefore, all $\mathcal{A}_{i}$ intersect with $\{U<H\}$.

By Lemma \ref{2-la0}, for each $i\in\llbracket1,\,a\rrbracket$,
there exists a connected component $\mathcal{W}_{i}$ of $\{U<H\}$
intersecting with $\mathcal{A}_{i}$. Since $\mathcal{A}_{i}\subset\mathcal{K}$,
we have $\mathcal{W}_{i}\cap\mathcal{K}\ne\varnothing$. On the other
hand, if $\mathcal{W}$ be a connected component of $\{U<H\}$ such
that $\mathcal{W}\cap\mathcal{K}\ne\varnothing$, then $\mathcal{W}$
intersects $\mathcal{A}_{i}$ for some $i\in\llbracket1,\,a\rrbracket$,
as $\mathcal{W}\cap\mathcal{K}\subset\mathcal{K}^{o}$ by Lemma \ref{l_level_boundary}.
Therefore, $\mathcal{W}_{1},\,\dots,\,\mathcal{W}_{a}$ are all the
connected components of $\{U<H\}$ intersecting $\mathcal{K}$.

Since $\mathcal{W}_{i}\subset\mathcal{A}_{i}\setminus\{\text{local maxima}\}$
and local maxima in $\mathcal{A}_{i}$ are accumulation points, we
get $\mathcal{M}_{0}\cap\mathcal{W}_{i}=\mathcal{M}_{0}\cap\mathcal{A}_{i}$
and $\overline{\mathcal{W}_{i}}=\overline{\mathcal{A}_{i}}$. Therefore,
by \cite[Lemma A.16-(2)]{LLS-1st}, $\mathcal{K}$ is path-connected
and we get
\[
\mathcal{K}=\overline{\mathcal{K}^{o}}=\overline{\bigcup_{i=1}^{a}\mathcal{A}_{i}}=\bigcup_{i=1}^{a}\overline{\mathcal{A}_{i}}=\bigcup_{i=1}^{a}\overline{\mathcal{W}_{i}}\;.
\]
For the second assertion of the lemma, suppose that $\bm{m}\in\partial\mathcal{K}\cap\mathcal{M}_{0}$.
Then, by Lemma \ref{l_level_boundary}, we have $U(\boldsymbol{m})=H$.
Since $\boldsymbol{m}$ is a local minimum of $U$, $\boldsymbol{m}$
should be an isolated point of the level set $\{U\le H\}$ so that
$\mathcal{K}=\{\bm{m}\}$ which contradicts to the assumption. This
proves the second assertion.
\end{proof}
\begin{lem}
\label{2-la1} Fix $H\in\mathbb{R}$, and let $\mathcal{K}$ be a
connected component of $\{U\le H\}$ which is not a singleton and
let $\{\mathcal{W}_{1},\,\dots,\,\mathcal{W}_{a}\}$ be a level set
decomposition of $\mathcal{K}^{o}$. Then,
\begin{enumerate}
\item For all $\boldsymbol{m}\in\mathcal{M}_{0}\cap\mathcal{W}_{i}$ and
$\boldsymbol{m}'\in\mathcal{M}_{0}\cap\mathcal{W}_{j}$ for some $i\ne j$,
$\Theta(\boldsymbol{m},\,\boldsymbol{m}')=H$.
\item If $\overline{\mathcal{W}_{i}}\cap\overline{\mathcal{W}_{j}}\neq\varnothing$
for some $i,\,j\in\llbracket1,\,a\rrbracket$, then the set $\overline{\mathcal{W}_{i}}\cap\overline{\mathcal{W}_{j}}$
is a collection of saddle points $\boldsymbol{\sigma}$ such that
$U(\boldsymbol{\sigma})=H$.
\item For all $i,\,j\in\llbracket1,\,a\rrbracket$, there exist $k_{1},\,\dots,k_{n}\in\llbracket1,\,a\rrbracket$
such that
\begin{equation}
\overline{\mathcal{W}_{i}}\cap\overline{\mathcal{W}_{k_{1}}},\;\overline{\mathcal{W}_{k_{1}}}\cap\overline{\mathcal{W}_{k_{2}}},\;\dots,\;\overline{\mathcal{W}_{k_{n-1}}}\cap\overline{\mathcal{W}_{k_{n}}},\;\overline{\mathcal{W}_{k_{n}}}\cap\overline{\mathcal{W}_{j}}\ \ne\ \varnothing\;.\label{e_4141}
\end{equation}
\end{enumerate}
\end{lem}

\begin{proof}
For part (1), by Lemma \ref{lap01}-(2), we get $\Theta(\boldsymbol{m},\,\boldsymbol{m}')\ge H$.
Now, since $\mathcal{K}$ is path connected, $\Theta(\boldsymbol{m},\,\boldsymbol{m}')=H$.
(2) is a direct consequence of Lemma \ref{l_cap_saddle}, and (3)
is explained in \cite[Proof of Lemma A.16-(2)]{LLS-1st}.
\end{proof}
\begin{lem}
\textcolor{red}{\label{lem_level}}For a non-empty simple bound set
$\mathcal{M}\subset\mathcal{M}_{0}$ such that $\widetilde{\mathcal{M}}\neq\varnothing$,
the following hold.
\begin{enumerate}
\item There exists a connected component $\mathcal{K}$ of level set $\{U\le\Theta(\mathcal{M},\,\widetilde{\mathcal{M}})\}$
containing $\mathcal{M}$ and intersecting with $\widetilde{\mathcal{M}}$.
\item Let $\{\mathcal{W}_{1},\,\dots,\,\mathcal{W}_{a}\}$ be a level set
decomposition of $\mathcal{K}$. Then, $a\ge2$.
\item There exists a saddle point $\boldsymbol{\sigma}\in\mathcal{S}_{0}\cap\mathcal{K}$
such that $U(\boldsymbol{\sigma})=\Theta(\mathcal{M},\,\widetilde{\mathcal{M}})$.
\end{enumerate}
\end{lem}

\begin{proof}
\noindent Let $H=\Theta(\mathcal{M},\,\widetilde{\mathcal{M}})$ and
pick $\bm{m}\in\mathcal{M}$ and $\bm{m}'\in\widetilde{\mathcal{M}}$
such that $\Theta(\bm{m},\,\bm{m}')=\Theta(\mathcal{M},\,\widetilde{\mathcal{M}})$.
\smallskip{}

\noindent (1) By Lemma \ref{l_105a-2}, there exists a connected component
$\mathcal{K}$ of $\{U\le H\}$ containing both of $\bm{m}$ and $\bm{m}'$.
If $\bm{m}''\in\mathcal{M}$, we have $\Theta(\bm{m},\,\bm{m}'')<\Theta(\mathcal{M},\,\widetilde{\mathcal{M}})=H$
since $\mathcal{M}$ is bound, and therefore $\bm{m}''\in\mathcal{K}$
by Lemmata \ref{l_level_connected} and \ref{l_105a-2} . Hence, we
have $\mathcal{M}\subset\mathcal{K}$. We also have $\mathcal{K}\cap\widetilde{\mathcal{M}}\neq\varnothing$
since $\bm{m}'\in\mathcal{K}\cap\widetilde{\mathcal{M}}$.

\smallskip{}

\noindent (2) Let $\mathcal{W}_{1}$ and $\mathcal{W}_{2}$ be an
element of level set decomposition of $\mathcal{K}$ containing $\bm{m}$
and $\bm{m}'$, respectively. If $\mathcal{W}_{1}=\mathcal{W}_{2}$,
by Lemma \ref{lap01}-(1), $\Theta(\bm{m},\,\bm{m}')<H$ which is
a contradiction. Therefore, $\mathcal{W}_{1}\ne\mathcal{W}_{2}$ so
that $a\ge2$.

\smallskip{}

\noindent (3) By Lemma \ref{2-la1}-(3) with $i=1$ and $j=2$, there
exists a connected component $\mathcal{W}_{3}$ of $\mathcal{K}^{o}$
such that
\[
\mathcal{W}_{1}\ \ne\ \mathcal{W}_{3}\ \ ,\ \ \overline{\mathcal{W}_{1}}\cap\overline{\mathcal{W}_{3}}\ \ne\ \varnothing\ .
\]
Hence, by Lemma \ref{2-la1}-(2), $\bm{\sigma}\in\overline{\mathcal{W}_{1}}\cap\overline{\mathcal{W}_{3}}$
satisfies $U(\boldsymbol{\sigma})=H=\Theta(\mathcal{M},\,\widetilde{\mathcal{M}})$.
\end{proof}
\begin{lem}
\label{lem_H,H+a} Fix $h\in\mathbb{R}$ and $r>0$. Suppose that
the interval $[h,\,h+r)$ does not contain the image of any critical
point of $U$: $\{U(\boldsymbol{c}):\boldsymbol{c}\in\mathcal{C}\}\cap[h,\,h+r)=\varnothing$.
Let $\mathcal{V}$ be a connected component of $\{U<h\}$ and let
$\mathcal{V}'$ be a connected component of $\{U<h+r\}$ containing
$\mathcal{V}$. Then, $\mathcal{M}_{0}\cap\mathcal{V}=\mathcal{M}_{0}\cap\mathcal{V}'$.
\end{lem}

\begin{proof}
Since $\mathcal{V}$ is a connected component of a level set, there
exists local minimum $\boldsymbol{m}_{1}$ of $U$ in $\mathcal{V}$.
Suppose, by contradiction, that there exists a local minimum $\boldsymbol{m}_{2}$
of $U$ in $\mathcal{V}'\setminus\mathcal{V}$. Then, by Lemma \ref{lap01},
\begin{equation}
h\ \le\ \Theta(\boldsymbol{m}_{1},\,\boldsymbol{m}_{2})\ <\ h+r\ .\label{e_H,H+a}
\end{equation}
Let $\mathcal{G}$ be a connected component of $\{U\le\Theta(\boldsymbol{m}_{1},\,\boldsymbol{m}_{2})\}$
containing both $\boldsymbol{m}_{1}$ and $\boldsymbol{m}_{2}$ whose
existence is guaranteed by Lemma \ref{l_105a-2}. Let $\{\mathcal{W}_{1},\,\dots,\,\mathcal{W}_{a}\}$
be a level set decomposition of $\mathcal{G}$.

By Lemma \ref{lap01}, $\bm{m}_{1}$ and $\bm{m}_{2}$ are contained
in different connected components of $\{U<\Theta(\boldsymbol{m}_{1},\,\boldsymbol{m}_{2})\}$
so that $a>1$. Bby Lemma \ref{2-la1}, there exist $i,\,j\in\llbracket1,\,a\rrbracket$
such that $\overline{\mathcal{W}_{i}}\cap\overline{\mathcal{W}_{j}}\ne\varnothing$,
and a saddle point $\boldsymbol{\sigma}\in\overline{\mathcal{W}_{i}}\cap\overline{\mathcal{W}_{j}}$
of $U$ such that $U(\boldsymbol{\sigma})=\Theta(\boldsymbol{m}_{1},\,\boldsymbol{m}_{2})$.
Thus, by \eqref{e_H,H+a}, $\boldsymbol{\sigma}$ is a critical point
of $U$ such that $U(\boldsymbol{\sigma})\in[h,\,h+r)$, which contradicts
to the assumption of the lemma.
\end{proof}

\subsection{Landscape near saddle points\label{sec8a}}

In this subsection, we present some results on saddle points needed
in the article. First, we state results from \cite{LLS-1st}.
\begin{lem}
\label{l_path_saddle} Fix a saddle point $\boldsymbol{\sigma}\in\mathcal{S}_{0}$.
Then,
\begin{enumerate}
\item There exists $r'>0$ such that the set $B(\boldsymbol{\sigma},\,r')\cap\{\boldsymbol{x}\in\mathbb{R}^{d}:U(\boldsymbol{x})<U(\boldsymbol{\sigma})\}$
has exactly two connected components. Denote by ${\color{blue}\mathcal{A}^{\pm}}$
these two connected components.
\item There exist local minima $\boldsymbol{m}_{\boldsymbol{\sigma}}^{\pm}\in\mathcal{M}_{0}$,
$t^{\pm}\in\mathbb{R}$, and heteroclinic orbits $\phi^{\pm}$ from
$\boldsymbol{\sigma}$ to $\boldsymbol{m}_{\boldsymbol{\sigma}}^{\pm}$
such that $\phi^{\pm}\left((-\infty,\,t^{\pm}]\right)\subset\mathcal{A}^{\pm}$.
\item Let $\mathcal{W}^{\pm}$ be the connected components of $\{\boldsymbol{x}\in\mathbb{R}^{d}:U(\boldsymbol{x})<U(\boldsymbol{\sigma})\}$
containing $\boldsymbol{m}_{\boldsymbol{\sigma}}^{\pm}$, respectively.
Then, $\phi^{\pm}(\mathbb{R})\subset\mathcal{W}^{\pm}$ and $\boldsymbol{\sigma}\in\partial\mathcal{W}^{\pm}$.
\end{enumerate}
\end{lem}

\begin{proof}
Part (1) is the content of \cite[Lemma A.7]{LLS-1st}. \smallskip{}

\noindent (2) This is a straightforward consequence of the Hartman--Grobman
theorem \cite[Section 2.8]{Perko} and hypothesis \eqref{hyp2}. \smallskip{}

\noindent (3) As $U$ strictly decreases along solutions of the ODE
\eqref{eq:x(t)}, $\{\phi^{\pm}(t):t\in\mathbb{R}\}$ is contained
in a connected component of the set $\{\boldsymbol{x}\in\mathbb{R}^{d}:U(\boldsymbol{x})<U(\boldsymbol{\sigma})\}$,
denoted by $\mathcal{V}^{\pm}$. Since $\phi^{\pm}(t)\to\boldsymbol{m}_{\boldsymbol{\sigma}}^{\pm}$
as $t\to\infty$, and $U(\boldsymbol{m}_{\boldsymbol{\sigma}})<U(\boldsymbol{\sigma})$,
$\boldsymbol{m}_{\boldsymbol{\sigma}}^{\pm}$ belongs to $\mathcal{V}_{\pm}$.
In particular, $\mathcal{W}^{\pm}=\mathcal{V}^{\pm}$.

On the other hand, since $\phi^{\pm}(t)\to\boldsymbol{\sigma}$ as
$t\to-\infty$, $\boldsymbol{\sigma}\in\overline{\mathcal{W}^{\pm}}$.
Since $U(\boldsymbol{\sigma})=H$, $\boldsymbol{\sigma}\notin\mathcal{W}^{\pm}$
so that $\boldsymbol{\sigma}\in\partial\mathcal{W}^{\pm}$.
\end{proof}
\begin{lem}
\label{l_assu_saddle} Fix $H\in\mathbb{R}$ and let $\mathcal{H}\subset\mathbb{R}^{d}$
be a connected component of the level set $\{U<H\}$. Suppose that
$\mathcal{S}_{0}\cap\partial\mathcal{H}\ne\varnothing$. Then,
\begin{enumerate}
\item For all $\boldsymbol{\sigma}\in\mathcal{S}_{0}\cap\partial\mathcal{H}$,
there exists $\boldsymbol{m}\in\mathcal{M}_{0}\cap\mathcal{H}$ such
that $\boldsymbol{\sigma}\curvearrowright\boldsymbol{m}$.
\item For any $\boldsymbol{m},\,\boldsymbol{m}'\in\mathcal{M}_{0}\cap\mathcal{H}$,
there exists saddle points $\boldsymbol{\sigma}_{1},\,\dots,\,\boldsymbol{\sigma}_{a}\in\mathcal{S}_{0}\cap\mathcal{H}$
and local minima $\boldsymbol{m}_{1},\,\dots,\,\boldsymbol{m}_{a-1}\in\mathcal{M}_{0}\cap\mathcal{H}$
satisfying
\[
\boldsymbol{m}\curvearrowleft\boldsymbol{\sigma}_{1}\curvearrowright\boldsymbol{m}_{1}\curvearrowleft\cdots\curvearrowright\boldsymbol{m}_{a-1}\curvearrowleft\boldsymbol{\sigma}_{a}\curvearrowright\boldsymbol{m}'\;.
\]
\item For all $\boldsymbol{m}\in\mathcal{M}_{0}\cap\mathcal{H}$ and $\boldsymbol{\sigma}\in\mathcal{S}_{0}\cap\partial\mathcal{H}$,
$\bm{\sigma}\rightsquigarrow\boldsymbol{m}$.
\end{enumerate}
\end{lem}

\begin{proof}
(1) Fix $\boldsymbol{\sigma}\in\mathcal{S}_{0}\cap\partial\mathcal{H}$.
Recall the notation introduced in the previous lemma. Since $\boldsymbol{\sigma}\in\partial\mathcal{H}$,
$\mathcal{H}\cap B(\boldsymbol{\sigma},r')\cap\{U<U(\boldsymbol{\sigma})\}\neq\varnothing$.
Hence, $\mathcal{H}$ intersects $\mathcal{A}^{+}$ or $\mathcal{A}^{-}$.
Assume, without loss of generality, that $\mathcal{H}\cap\mathcal{A}^{+}\neq\varnothing$.
Then, $\mathcal{H}\cap\mathcal{W}^{+}\neq\varnothing$. Since, they
are both connected components of the set $\{U<H\}$, they are equal
and the assertion follows from the previous lemma. \smallskip{}

\noindent (2) We prove this part by induction on the number of local
minima of $U$ in $\mathcal{H}$, denoted by $|\mathcal{M}_{0}\cap\mathcal{H}|$.
There is nothing to prove if $|\mathcal{M}_{0}\cap\mathcal{H}|=1$.

Assume that $|\mathcal{M}_{0}\cap\mathcal{H}|=2$. Write $\mathcal{M}_{0}\cap\mathcal{H}=\{\boldsymbol{m}_{1},\,\boldsymbol{m}_{2}\}$.
Let $\mathcal{V}_{1}$ be a connected component of $\{U<\Theta(\boldsymbol{m}_{1},\,\boldsymbol{m}_{2})\}$
containing $\boldsymbol{m}_{1}$. By Lemma \ref{lap01}-(1), $\Theta(\boldsymbol{m}_{1},\,\boldsymbol{m}_{2})<H$
so that $\mathcal{V}_{1}\subset\mathcal{H}$. By Lemma \ref{lap01}-(1)
again, $\boldsymbol{m}_{2}\notin\mathcal{V}_{1}$ so that $\boldsymbol{m}_{1}$
is the only local minimum of $\mathcal{V}_{1}$. By Lemma \ref{lap01}-(2),
for all $\boldsymbol{m}\in\mathcal{M}_{0}\setminus\{\boldsymbol{m}_{1},\,\boldsymbol{m}_{2}\}$,
\begin{equation}
\Theta(\boldsymbol{m}_{1},\,\boldsymbol{m})\ \ge\ H\ >\ \Theta(\boldsymbol{m}_{1},\,\boldsymbol{m}_{2})\label{2-la1-1}
\end{equation}
so that
\[
\min_{\boldsymbol{m}\in\mathcal{M}_{0}\setminus\{\boldsymbol{m}_{1}\}}\,\Theta(\boldsymbol{m}_{1},\,\boldsymbol{m})\ =\ \Theta(\boldsymbol{m}_{1},\,\boldsymbol{m}_{2})\ .
\]
Therefore, by \cite[Lemma A.2]{LLS-1st}, \footnote{In \cite{LLS-1st}, the definition of $\Gamma(\boldsymbol{m})$ is
$U(\boldsymbol{m})+\Gamma(\boldsymbol{m})=\min_{\boldsymbol{m}'\in\mathcal{M}_{0}\setminus\{\boldsymbol{m}\}}\Theta(\boldsymbol{m},\,\boldsymbol{m}')$} there exists a connected component $\mathcal{V}_{2}$ of $\{U<\Theta(\boldsymbol{m}_{1},\,\boldsymbol{m}_{2})\}$
such that
\[
\mathcal{V}_{1}\cap\mathcal{V}_{2}\ =\ \varnothing\ \ ,\ \ \overline{\mathcal{V}_{1}}\cap\overline{\mathcal{V}_{2}}\ \ne\ \varnothing\ .
\]
By Lemma \ref{l_cap_saddle}, $\boldsymbol{\sigma}\in\overline{\mathcal{V}_{1}}\cap\overline{\mathcal{V}_{2}}$
is a saddle point and $U(\boldsymbol{\sigma})=\Theta(\boldsymbol{m}_{1},\,\boldsymbol{m}_{2})$.
By the first assertion of this lemma for $H=\Theta(\boldsymbol{m}_{1},\,\boldsymbol{m}_{2})$,
there exists $\widehat{\boldsymbol{m}}_{i}\in\mathcal{V}_{i}$, $i=1$,
$2$, such that $\boldsymbol{\sigma}\curvearrowright\widehat{\boldsymbol{m}}_{i}$.
Since $\boldsymbol{m}_{1}$ is the unique local minima of $U$ in
$\mathcal{V}_{1}$, $\widehat{\boldsymbol{m}}_{1}=\boldsymbol{m}_{1}$.
By Lemma \ref{lem_not}-(3), $\Theta(\boldsymbol{m}_{1},\,\widehat{\boldsymbol{m}}_{2})\le U(\boldsymbol{\sigma})=\Theta(\boldsymbol{m}_{1},\,\boldsymbol{m}_{2})$.
Therefore, by \eqref{2-la1-1}, $\widehat{\boldsymbol{m}}_{2}=\boldsymbol{m}_{2}$
so that $\boldsymbol{m}_{1}\curvearrowleft\boldsymbol{\sigma}\curvearrowright\boldsymbol{m}_{2}$.
Finally, since $U(\boldsymbol{\sigma})<H$ and $\boldsymbol{\sigma}\curvearrowright\boldsymbol{m}_{1}$,
$\boldsymbol{\sigma}\in\mathcal{H}$. This proves the claim for $|\mathcal{M}_{0}\cap\mathcal{H}|=2$.

Suppose that the claim of part (2) holds for $|\mathcal{M}_{0}\cap\mathcal{H}|=2,\,\dots,\,n-1$
and assume that $|\mathcal{M}_{0}\cap\mathcal{H}|=n$. Let $\boldsymbol{m},\,\boldsymbol{m}'\in\mathcal{M}_{0}\cap\mathcal{H}$.
By Lemma \ref{lap01}-(1), $\Theta(\boldsymbol{m},\,\boldsymbol{m}')<H$.
Denote by $\mathcal{G}\subset\mathcal{H}$ the connected component
of $\{U\le\Theta(\boldsymbol{m},\,\boldsymbol{m}')\}$ containing
both $\boldsymbol{m}$ and $\boldsymbol{m}'$ whose existence is guaranteed
by Lemma \ref{l_105a-2}. By \cite[Lemma A.18]{LLS-1st}, $U(\boldsymbol{m}),\,U(\boldsymbol{m}')<\Theta(\boldsymbol{m},\,\boldsymbol{m}')$
so that by Lemma \ref{l_level_boundary}, $\boldsymbol{m},\,\boldsymbol{m}'\in\mathcal{G}^{o}$.
Let $\{\mathcal{W}_{1},\,\dots,\,\mathcal{W}_{b}\}$ be a level set
decomposition of $\mathcal{G}$. Without loss of generality, assume
that $\boldsymbol{m}\in\mathcal{W}_{1}$. By Lemma \ref{lap01}-(1),
$\boldsymbol{m}'\notin\mathcal{W}_{1}$ so that $b\ge2$. Without
loss of generality, let $\boldsymbol{m}'\in\mathcal{W}_{2}$. By Lemma
\ref{2-la1}-(3), there exist $k_{1},\,\dots,k_{\ell}\subset\llbracket1,\,m\rrbracket$
such that
\[
\overline{\mathcal{W}_{1}}\cap\overline{\mathcal{W}_{k_{1}}},\;\overline{\mathcal{W}_{k_{1}}}\cap\overline{\mathcal{W}_{k_{2}}},\;\dots,\;\overline{\mathcal{W}_{k_{\ell-1}}}\cap\overline{\mathcal{W}_{k_{\ell}}},\;\overline{\mathcal{W}_{k_{\ell}}}\cap\overline{\mathcal{W}_{2}}\ \ne\ \varnothing\;.
\]
Let $\mathcal{W}_{k_{0}}=\mathcal{W}_{1}$ and $\mathcal{W}_{k_{\ell+1}}=\mathcal{W}_{2}$
and let $\boldsymbol{\sigma}_{i}\in\overline{\mathcal{W}_{k_{i}}}\cap\overline{\mathcal{W}_{k_{i+1}}}$,
$i\in\llbracket0,\,\ell\rrbracket$. Note that the number of local
minima in each $\mathcal{W}_{k_{i}}$ is smaller than $n$ because
all these sets are contained in $\mathcal{G}\subset\mathcal{H}$ and
do not contain $\boldsymbol{m}$ or $\boldsymbol{m}'$.

By Lemma \ref{l_path_saddle}-(3), there exist local minima $\boldsymbol{m}_{i}\in\mathcal{W}_{i},\,\boldsymbol{m}_{i+1}'\in\mathcal{W}_{i+1}$
such that $\boldsymbol{m}_{i}\curvearrowleft\boldsymbol{\sigma}_{i}\curvearrowright\boldsymbol{m}_{i+1}'$.
By the induction hypothesis, we can find local minima $\boldsymbol{m}'_{i}=\boldsymbol{m}_{i,1},\dots,\boldsymbol{m}_{i,\ell_{i}}=\boldsymbol{m}_{i}$
and saddle points $\boldsymbol{\sigma}_{i,1},\dots,\boldsymbol{\sigma}_{i,\ell_{i}}$
such that
\[
\boldsymbol{m}_{i}'\curvearrowleft\boldsymbol{\sigma}_{i,1}\curvearrowright\boldsymbol{m}_{i,2}\curvearrowleft\cdots\curvearrowright\boldsymbol{m}_{i,\ell_{i}}\curvearrowleft\boldsymbol{\sigma}_{i,\ell_{i}}\curvearrowright\boldsymbol{m}_{i}\ .
\]
Concatenating these paths we complete the proof. \smallskip{}

\noindent (3) Let $\boldsymbol{m}'\in\mathcal{M}_{0}\cap\mathcal{H}$
and $\boldsymbol{\sigma}\in\mathcal{S}_{0}\cap\partial\mathcal{H}$.
By (1), there exists $\boldsymbol{m}\in\mathcal{M}_{0}\cap\mathcal{H}$
such that $\boldsymbol{\sigma}\curvearrowright\boldsymbol{m}$. Then,
by part (2), we can find $\boldsymbol{\sigma}_{1},\,\dots,\,\boldsymbol{\sigma}_{n}\in\mathcal{S}_{0}\cap\mathcal{H}$
and $\boldsymbol{m}_{1},\,\dots,\,\boldsymbol{m}_{n-1}\in\mathcal{M}_{0}\cap\mathcal{H}$
such that
\[
\boldsymbol{m}'\curvearrowleft\boldsymbol{\sigma}_{1}\curvearrowright\boldsymbol{m}_{1}\curvearrowleft\cdots\curvearrowright\boldsymbol{m}_{n-1}\curvearrowleft\boldsymbol{\sigma}_{n}\curvearrowright\boldsymbol{m}\ .
\]
Since $\boldsymbol{\sigma}_{i}\in\mathcal{H}$, $U(\boldsymbol{\sigma}_{i})<H=U(\boldsymbol{\sigma})$
for all $i$ and therefore we get $\boldsymbol{\sigma}\rightsquigarrow\boldsymbol{m}'$.
\end{proof}
We have the following inverse version of Lemma \ref{l_assu_saddle}-(3).
\begin{lem}
\label{l_squig_saddle} Suppose that $\boldsymbol{m}\in\mathcal{M}_{0}$
and $\boldsymbol{\sigma}\in\mathcal{S}_{0}$ satisfy $\boldsymbol{\sigma}\rightsquigarrow\boldsymbol{m}$.
Denote by $\mathcal{W}$ the connected component of $\{U<U(\boldsymbol{\sigma})\}$
containing $\boldsymbol{m}$. Then, $\boldsymbol{\sigma}\in\partial\mathcal{W}$.
\end{lem}

\begin{proof}
Since $\boldsymbol{\sigma}\rightsquigarrow\boldsymbol{m}$, there
exists a continuous path $\boldsymbol{z}(\cdot)$ from $\boldsymbol{\sigma}$
to $\boldsymbol{m}$ such that $\boldsymbol{z}(0)=\boldsymbol{\sigma}$,
$\boldsymbol{z}(1)=\boldsymbol{m}$ and $\boldsymbol{z}\left((0,\,1]\right)\subset\{U<U(\boldsymbol{\sigma})\}$.
Since $\boldsymbol{z}(1)\in\mathcal{W}$ and $\boldsymbol{z}\left((0,\,1]\right)$
is connected, by Lemma \ref{l_level_connected}, $\boldsymbol{z}\left((0,\,1]\right)\subset\mathcal{W}$.
Thus, $\boldsymbol{\sigma}\in\overline{\mathcal{W}}$. Since $\boldsymbol{\sigma}\not\in\mathcal{W}$
by definition of $\mathcal{W}$, $\boldsymbol{\sigma}\in\partial\mathcal{W}$.
\end{proof}
Recall that $\mathcal{W}(\mathcal{M})$ is a connected component of
$\{U<\Theta(\mathcal{M},\,\widetilde{\mathcal{M}})\}$ containing
$\mathcal{M}$.
\begin{lem}
\label{l_exist_saddle} Let $\mathcal{M}\subset\mathcal{M}_{0}$ be
a bound set satisfying $\widetilde{\mathcal{M}}\ne\varnothing$. Then,
there exists $\boldsymbol{\sigma}\in\mathcal{S}_{0}\cap\partial\mathcal{W}(\mathcal{M})$
satisfying \eqref{eq:SM}.
\end{lem}

\begin{proof}
Let $H=\Theta(\mathcal{M},\,\widetilde{\mathcal{M}})$. Pick $\boldsymbol{m}\in\mathcal{M}$
and $\boldsymbol{m}'\in\widetilde{\mathcal{M}}$ satisfying $\Theta(\boldsymbol{m},\,\boldsymbol{m}')=H$.
As $\Theta(\boldsymbol{m},\,\boldsymbol{m}')=H$, by Lemma \ref{lap01}-(2),
$\boldsymbol{m}'\not\in\mathcal{W}(\mathcal{M})$.

By Lemma \ref{l_105a-2}, there exists a connected component $\mathcal{H}$
of $\{U\le H\}$ containing both $\boldsymbol{m}$ and $\boldsymbol{m}'$.
As $U(\boldsymbol{m})$, $U(\boldsymbol{m}')$ are smaller than $H$,
they belongs to $\mathcal{H}^{o}$. Let $\{\mathcal{W}_{1},\,\dots,\,\mathcal{W}_{m}\}$
be level set decomposition of $\mathcal{H}$ and assume, without loss
of generality, that $\boldsymbol{m}\in\mathcal{W}_{1}$.

Since $\boldsymbol{m}\in\mathcal{W}_{1}\cap\mathcal{W}(\mathcal{M})$,
$\mathcal{W}_{1}=\mathcal{W}(\mathcal{M})$ by Lemma \ref{l_level_connected}.
In particular, since $\mathcal{M}\subset\mathcal{W}(\mathcal{M})$
and $\boldsymbol{m}'\notin\mathcal{W}(\mathcal{M})$, $\mathcal{M}\subset\mathcal{W}_{1}$
and $\boldsymbol{m}'\notin\mathcal{W}_{1}$. Thus, $m\ge2$ and hence
by Lemma \ref{2-la1}-(3), there exists $k\in\llbracket2,\,m\rrbracket$
such that $\overline{\mathcal{W}_{1}}\cap\overline{\mathcal{W}_{k}}\ne\varnothing$.
Finally, by Lemma \ref{l_cap_saddle}, all the points $\boldsymbol{\sigma}\in\overline{\mathcal{W}_{1}}\cap\overline{\mathcal{W}_{k}}=\overline{\mathcal{W}(\mathcal{M})}\cap\overline{\mathcal{W}_{k}}\subset\partial\mathcal{W}(\mathcal{M})$
satisfy \eqref{eq:SM}.
\end{proof}
\begin{lem}
\label{lem_exsig}Let $\mathcal{M}\subset\mathcal{M}_{0}$ be a bound
set such that $\widetilde{\mathcal{M}}\ne\varnothing$ and let $\mathcal{H}$
be a connected component of $\{U\le\Theta(\mathcal{M},\,\widetilde{\mathcal{M}})\}$
containing $\mathcal{M}$. Then, there exists $\boldsymbol{\sigma}\in\mathcal{H}\cap\mathcal{S}_{0}$
such that $U(\boldsymbol{\sigma})=\Theta(\mathcal{M},\,\widetilde{\mathcal{M}})$.
\end{lem}

\begin{proof}
By Lemma \ref{l_exist_saddle}, there exists a saddle point $\boldsymbol{\sigma}\in\partial\mathcal{W}(\mathcal{M})$.
By Lemma \ref{l_level_boundary}, $U(\boldsymbol{\sigma})=\Theta(\mathcal{M},\,\widetilde{\mathcal{M}})$.
Since $\mathcal{H}$ is the connected component of $\{U\le\Theta(\mathcal{M},\,\widetilde{\mathcal{M}})\}$
containing $\mathcal{M}$, $\mathcal{W}(\mathcal{M})\subset\mathcal{H}$.
Since $\mathcal{H}$ is closed, $\partial\mathcal{W}(\mathcal{M})\subset\mathcal{H}$
so that $\boldsymbol{\sigma}\in\mathcal{H}$.
\end{proof}

\section{\label{app:trace}Trace Process}

Let $\{\bm{z}(t)\}_{t\ge0}$ be a continuous-time
Markov process on a certain state space $E$. We note that $\bm{z}(\cdot)$
might be either a continuous time Markov process or diffusion process.

We suppose that $\bm{z}(\cdot)$ is non-explosive.
Then, for a subset $F$ of $E$ with a  \emph{good
property}  (cf. \eqref{e_F_assu_trace}) which will
be specified later, we shall define a trace of the process $\bm{z}(\cdot)$
on $F$. Heuristically, trace process is a process obtained from $\boldsymbol{z}(\cdot)$
by turning off the clock when the process is at the outside of $F$.
To define this rigorously, let $T^{F}:[0,\infty)\to[0,\infty)$ be
a time spent by $\bm{z}(\cdot)$ staying in $F$:
\[
{\color{blue}T^{F}(t)}\ :=\ \int_{0}^{t}\,{\bf 1}\left\{ \,\bm{z}(s)\in F\,\right\} \,ds\ .
\]
It is obvious that $T^{F}(\cdot)$ is increasing. Suppose moreover
that (cf. \cite[Section 2.2]{BL1}).
\begin{equation}
\lim_{t\to\infty}\,T^{F}(t)\ =\ \infty\;\ \text{almost surely}\ .\label{e_F_assu_trace}
\end{equation}
Then, we can define generalized inverse $S^{F}:[0,\infty)\to[0,\infty)$
by
\[
{\color{blue}S^{F}(t)}\ :=\ \sup\,\left\{ \,s\ge0:T^{F}(t)\le s\,\right\} \ .
\]
Finally, the trace process $\bm{z}^{F}(\cdot)$ of $\boldsymbol{z}(\cdot)$
on the set $F$ is defined by
\[
\bm{z}^{F}(t)\ :=\ \bm{z}(S^{F}(t))\ .
\]
It is well-known (cf. \cite[Section 6.1]{BL1}) that the trace process
defined in this manner again becomes a Markov process.

Now, we are interested in under which assumption on
$\bm{z}(\cdot)$ and $F$ the condition \eqref{e_F_assu_trace} holds.
Firstly, it is clear that the condition \eqref{e_F_assu_trace} holds
if $\bm{z}(\cdot)$ is either a positive recurrent Markov process
on discrete space and $F\ne\varnothing$ or a positive recurrent diffusion
process and $F$ contains an open ball. Otherwise, we need to check
\eqref{e_F_assu_trace} to define the order process.
\begin{lem}
\label{l_T_diverge}Suppose that $\boldsymbol{z}(\cdot)$
is a Markov chain on the finite set $E$ and that $F$ contains at
least one element of each irreducible class of $\boldsymbol{z}(\cdot)$.
Then, the condition \eqref{e_F_assu_trace} holds.
\end{lem}

\begin{proof}
Since the state space is finite, the process $\boldsymbol{z}(\cdot)$
eventually arrives at an irreducible class. Therefore, we may suppose
that the process starts at a state in irreducible class. Let $E_{0}\subset E$
be an irreducible class and $F_{0}=F\cap E_{0}$. By the assumption,
$F_{0}\ne\varnothing$. Since the process $\boldsymbol{z}(\cdot)$
starting at $E_{0}$ cannot escape from there, we can regard the process
$\boldsymbol{z}(\cdot)$ as the one defined only in $E_{0}$. Then,
the process $\boldsymbol{z}(\cdot)$ is positive recurrent and thus
$\lim_{t\to\infty}T_{0}^{F_{0}}(t)=\infty$ almost surely. Since $T^{F}(t)=T^{F_{0}}(t)$,
we are done.
\end{proof}
The following is a general result on Markov chains which is the main
idea behind the proof of the local reversibility.
\begin{prop}
\label{prop:revpre}Suppose that $E$ is finite. Let $\mathbf{z}(\cdot)$
be a continuous-time irreducible Markov chain on $E$ with jump rate
$R:E\times E\rightarrow[0,\,\infty)$. Suppose that we have a decomposition
\[
E\ =\ \bigcup_{i=1}^{n}E_{i}\;,
\]
of $E$ and $x_{i}\in E_{i}$ for $i\in\llbracket1,\,n\rrbracket$
such that,
\begin{equation}
\begin{cases}
R(x,\,E_{i}^{c})\ =\ 0 & \text{if }x\in E_{i}\setminus\{x_{i}\}\\
R(x,\,E_{i}^{c})\ >\ 0 & \text{if }x=x_{i}
\end{cases}\;\;\;\text{ for all }i\in\llbracket1,\,n\rrbracket\;,\label{eq:revm1}
\end{equation}
where $R(x,\,A)=\sum_{y\in A}R(x,\,y)$. Suppose in addition that,
there exists a probability measure $\rho(\cdot)$ on $F=\{x_{1},\,\dots,\,x_{n}\}$
such that we assume further that
\begin{equation}
\rho(x_{i})\,R(x_{i},\,E_{j})\ =\ \rho(x_{j})\,R(x_{j},\,E_{i})\;\;\;\text{for all }i,\,j\in\llbracket1,\,n\rrbracket\;.\label{eq:revm2}
\end{equation}
Then, the trace process of $\mathbf{z}(\cdot)$ on $F$ is reversible
with respect to $\rho(\cdot)$.
\end{prop}

\begin{proof}
Denote by $\overline{R}:E\times E\rightarrow[0,\,\infty)$ the jump
rate of the trace of $\mathbf{z}(\cdot)$ on $F$. Then, in view of
\eqref{eq:revm2}, it suffices to show that $\overline{R}(x_{i},\,x_{j})=R(x_{i},\,E_{j})$
for all $i,\,j\in\llbracket1,\,n\rrbracket$. To that end, recall
from \cite[Corollary 6.2]{BL1} that
\begin{equation}
\overline{R}(x_{i},\,x_{j})\ =\ \sum_{k=1}^{n}\,\sum_{x\in E_{k}}\,R(x_{i},\,x)\,\mathbb{P}_{x}[\,\tau_{S_{0}}=\tau_{x_{j}}\,]\;.\label{eq:revm3}
\end{equation}
By \eqref{eq:revm1},
\[
\mathbb{P}_{x}[\,\tau_{F}=\tau_{x_{j}}\,]\ =\ \mathbf{1}\{\,x\in E_{j}\,\}
\]
and therefore, we can deduce from \eqref{eq:revm3} that
\[
\overline{R}(x_{i},\,x_{j})\ =\ \sum_{x\in E_{j}}\,R(x_{i},\,x)\ =\ R(x_{i},\,E_{j})\;.
\]
This completes the proof.
\end{proof}

\noindent\textbf{Acknowledgement. } C. L. has
been partially supported by FAPERJ CNE E-26/201.117/2021, and by CNPq
Bolsa de Produtividade em Pesquisa PQ 305779/2022-2.  J. L. was supported by the KIAS Individual Grant (HP093101) at Korea Institute for Advanced Study, the National Research Foundation of Korea (NRF) grant funded by the Korea government (No. RS-2019-NR040050, No. 2022R1F1A106366811, No. 2022R1A6A3A13065174), and Seoul National
University Research Grant in 2022. I.S was supported by NRF grant
funded by the Korea government (MSIT) (No. 2022R1F1A106366811,
2022R1A5A600084012, 2023R1A2C100517311) and Samsung Science and
Technology Foundation (Project Number SSTF-BA1901-03), and Seoul
National University Research Grant in 2022. I. S and J.  L are
grateful for the invitations from CIRM and the University of Rouen,
where significant portion of this collaboration took place.
I. S. thank to KIAS for their support as an associate member.

\end{document}